\documentclass[11pt]{article}
\usepackage[margin=.8in,left=.8in]{geometry}
\usepackage{mathtools}
\usepackage{amsthm}
\usepackage{xcolor}
\usepackage{stmaryrd}

\usepackage{bbm}

\usepackage{enumitem} 
\setlist{
  topsep=1pt, 
  itemsep=0pt,
  parsep=1pt
}

\usepackage{adjustbox}

\usepackage{hyperref}
\hypersetup{
  colorlinks = true,
  allcolors=.  
}

\let\PLAINthebibliography\thebibliography
\renewcommand\thebibliography[1]{
  \PLAINthebibliography{#1}
  \setlength{\parskip}{1pt}
  \setlength{\itemsep}{1pt plus .3ex}
}

\usepackage{tikz}
\usetikzlibrary{decorations.pathmorphing}
\usetikzlibrary{arrows}
\usetikzlibrary{cd}

\tikzcdset{
  arrow style=tikz, 
  diagrams={
    trim left=
    {([xshift=5pt]current bounding box.west)},
    trim right=
    {([xshift=-5pt]current bounding box.east)},
    baseline=
    {([yshift=-2pt]current bounding box.center)},
    >={
      Computer Modern Rightarrow[
        length=4pt, width=4pt
      ]
    },
    hook/.style={
      {Hooks[right, length=2pt, width=8pt]}->
    },
    hook'/.style={
      {Hooks[left, length=2pt, width=8pt]}->
    },
    shorten=0pt,
  }
}

\DeclareMathAlphabet{\mathpzc}{OT1}{pzc}{m}{it} 

\newcommand\mathscr[1]{\scalebox{1.1}{$\mathpzc{#1}$}}

\newcommand{\proofstep}[1]{\scalebox{.85}{ #1 }}

\usepackage[titletoc]{appendix}

\usepackage[safe]{tipa} 

\definecolor{darkblue}{rgb}{0.05,0.25,0.65}
\definecolor{greenii}{RGB}{20,140,10}
\definecolor{lightgray}{rgb}{0.9,0.9,0.9}
\definecolor{orangeii}{RGB}{200,100,5}

\usepackage{multirow}

\usepackage{stmaryrd}    

\usepackage{amssymb,amsmath}

\newcounter{sqindex}

\usepackage{mathptmx}
\usepackage{amsmath}
\usepackage{graphicx}
\DeclareRobustCommand{\coprod}{\mathop{\text{\fakecoprod}}}
\newcommand{\fakecoprod}{%
  \sbox0{$\prod$}%
  \smash{\raisebox{\dimexpr.9625\depth-\dp0}{\scalebox{1}[-1]{$\prod$}}}%
  \vphantom{$\prod$}%
}

\newcommand{\isa}{\in}

\newcommand{\GroundField}{\mathbb{K}}

\newcommand{\Sets}{\mathrm{Set}}
\newcommand{\Set}{\Sets}

\newcommand{\Groups}{
  \mathrm{Grp}
}

\newcommand{\ZTwo}{\mathbb{Z}_{/2}}

\newcommand{\Modules}[1]{\mathrm{Mod}_{#1}}

\newcommand{\Groupoids}{\mathrm{Grpd}}
\newcommand{\Groupoid}{\Groupoids}

\newcommand{\AGroupoid}[1]{\mathcal{#1}}

\newcommand{\Categories}{\mathrm{Cat}}

\newcommand{\SimplicialGroupoids}{\mathrm{sSet}\mbox{-}\Groupoids}

\newcommand{\Actions}[1]{
  {#1}\,\mathrm{Act}
}

\newcommand{\RealNumbers}{\mathbb{R}}
\newcommand{\ComplexNumbers}{\mathbb{C}}

\newcommand{\shape}{
  \raisebox{1pt}{\rm\textesh}
}

\usepackage[new]{old-arrows}   

\def\acts{\raisebox{1.4pt}{\;\rotatebox[origin=c]{90}{$\curvearrowright$}}\hspace{.5pt}}

\makeatletter
\newif\if@sup
\newtoks\@sups
\def\append@sup#1{\edef\act{\noexpand\@sups={\the\@sups #1}}\act}%
\def\reset@sup{\@supfalse\@sups={}}%
\def\mk@scripts#1#2{\if #2/ \if@sup ^{\the\@sups}\fi \else%
  \ifx #1_ \if@sup ^{\the\@sups}\reset@sup \fi {}_{#2}%
  \else \append@sup#2 \@suptrue \fi%
  \expandafter\mk@scripts\fi}
\def\tensor#1#2{\reset@sup#1\mk@scripts#2_/}
\def\multiscripts#1#2#3{\reset@sup{}\mk@scripts#1_/#2%
  \reset@sup\mk@scripts#3_/}
\makeatother

\makeatletter
\newbox\slashbox \setbox\slashbox=\hbox{$/$}
\def\itex@pslash#1{\setbox\@tempboxa=\hbox{$#1$}
  \@tempdima=0.5\wd\slashbox \advance\@tempdima 0.5\wd\@tempboxa
  \copy\slashbox \kern-\@tempdima \box\@tempboxa}
\def\slash{\protect\itex@pslash}
\makeatother

\def\clap#1{\hbox to 0pt{\hss#1\hss}}
\def\mathllap{\mathpalette\mathllapinternal}
\def\mathrlap{\mathpalette\mathrlapinternal}
\def\mathclap{\mathpalette\mathclapinternal}
\def\mathllapinternal#1#2{\llap{$\mathsurround=0pt#1{#2}$}}
\def\mathrlapinternal#1#2{\rlap{$\mathsurround=0pt#1{#2}$}}
\def\mathclapinternal#1#2{\clap{$\mathsurround=0pt#1{#2}$}}

\let\oldroot\root
\def\root#1#2{\oldroot #1 \of{#2}}
\renewcommand{\sqrt}[2][]{\oldroot #1 \of{#2}}

\DeclareSymbolFont{symbolsC}{U}{txsyc}{m}{n}
\SetSymbolFont{symbolsC}{bold}{U}{txsyc}{bx}{n}
\DeclareFontSubstitution{U}{txsyc}{m}{n}

\DeclareSymbolFont{stmry}{U}{stmry}{m}{n}
\SetSymbolFont{stmry}{bold}{U}{stmry}{b}{n}

\DeclareFontFamily{OMX}{MnSymbolE}{}
\DeclareSymbolFont{mnomx}{OMX}{MnSymbolE}{m}{n}
\SetSymbolFont{mnomx}{bold}{OMX}{MnSymbolE}{b}{n}
\DeclareFontShape{OMX}{MnSymbolE}{m}{n}{
    <-6>  MnSymbolE5
   <6-7>  MnSymbolE6
   <7-8>  MnSymbolE7
   <8-9>  MnSymbolE8
   <9-10> MnSymbolE9
  <10-12> MnSymbolE10
  <12->   MnSymbolE12}{}

\usepackage{cleveref}

\crefformat{section}{\S#2#1#3} 
\crefformat{subsection}{\S#2#1#3}
\crefformat{subsubsection}{\S#2#1#3}

\theoremstyle{plain}
\newtheorem{theorem}{Theorem}[section]
\newtheorem{lemma}[theorem]{Lemma}
\newtheorem{proposition}[theorem]{Proposition}

\newtheorem{corollary}[theorem]{Corollary}
\theoremstyle{definition}
\newtheorem{definition}[theorem]{Definition}

\newtheorem{example}[theorem]{Example}

\newtheorem{remark}[theorem]{Remark}

\newcommand{\defneq}{\equiv}

\newcommand{\extmap}{\raisebox{.0pt}{\---}\scalebox{.8}{$\!\Box$}}

\begin{document}

\setlength{\abovedisplayskip}{3.5pt}
\setlength{\belowdisplayskip}{3.5pt}
\setlength{\abovedisplayshortskip}{-5pt}
\setlength{\belowdisplayshortskip}{3pt}

\title{A Global Model Structure for $\mathbb{K}$-Linear $\infty$-Local Systems}

\author{
  Hisham Sati${}^{\ast, \dagger}$
\qquad 
  Urs Schreiber${}^\ast$
}

\maketitle

\begin{abstract}
Parameterized stable homotopy theory organizes local systems of spectra over homotopy types, governed by a ``yoga'' of six functors. To provide semantics for the recently developed \emph{Linear Homotopy Type Theory} (LHoTT), good model categories of these spectra are required, preferably monoidal with respect to the external smash product. 

We focus on the case of parameterized $H\scalebox{.95}{$\mathbb{K}$}$-module spectra ($\infty$-local systems), motivated by recent applications of parameterized homotopy to topological quantum computing. While traditionally treated via dg-categories, we leverage combinatorial model structures on simplicial chain complexes to construct the first dedicated global model structure for $\scalebox{.95}{$\mathbb{K}$}$-linear $\infty$-local systems, which offers better control than existing models for general parameterized spectra. In particular, when restricted to base 1-types, our model structure is monoidal with respect to the external tensor product, making it a candidate target semantics for the multiplicative fragment of LHoTT.
\end{abstract}

\medskip

 \begin{center}
 \begin{minipage}{14cm}
  \tableofcontents  
  \end{minipage}
 \end{center}

\vfill

\hrule
\vspace{5pt}

{
\footnotesize
\noindent
\def\arraystretch{1}
\tabcolsep=0pt
\begin{tabular}{ll}
${}^*$\,
&
Mathematics, Division of Science; and
\\
&
Center for Quantum and Topological Systems,
\\
&
NYUAD Research Institute,
\\
&
New York University Abu Dhabi, UAE.  
\end{tabular}
\hfill
\adjustbox{raise=-15pt}{
\href{https://ncatlab.org/nlab/show/Center+for+Quantum+and+Topological+Systems}{
\includegraphics[width=3cm]{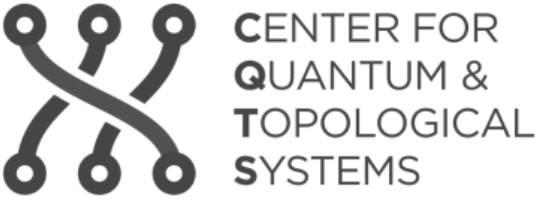}}
}

\vspace{1mm} 
\noindent ${}^\dagger$The Courant Institute for Mathematical Sciences, NYU, NY

\vspace{.05cm}

\noindent
The authors acknowledge the support by {\it Tamkeen UAE} under the 
{\it NYU Abu Dhabi Research Institute grant} {\tt CG008}.
}

\newpage

\section{Introduction and Overview}

\noindent
{\bf Parameterized stable homotopy theory.}
The two major branches of homotopy theory, namely the unstable plain homotopy theory of spaces and the stable homotopy theory of spectra, unify in the \emph{parameterized stable homotopy theory} of \emph{parameterized spectra} over spaces (flat bundles of spectra). The tightness of this unification is witnessed by the remarkable fact (due to G. Biedermann 2007, cf. \cite[\S 4.2.3]{AnelJoyal2021}) that the global homotopy theory of parameterized spectra (amalgamated over varying base spaces) is again an $\infty$-topos like the plain homotopy theory of spaces --- even though each of its fiber homotopy theories of spectra parameterized over a fixed base is stable and as such far from being an $\infty$-topos. But in fact, globally parameterized spectra are ``doubly closed monoidal'', with the canonical Cartesian product (which is the ``external direct sum'' of parameterized spectra) now accompanied by the ``external smash tensor product'' of spectra (the fiberwise smash product after pullback to the Cartesian product of base spaces). 

\medskip

\noindent
{\bf The quest for good models.}
A powerful tool for accurately dealing with homotopy theories famously is \emph{model category theory} (cf. \cite{Quillen67}\cite{Hovey99}\cite{Hirschhorn02}), and specifically \emph{monoidal model category theory} if further (tensor) products are involved. However, global model categories of parameterized spectra, pioneered by May \& Sigurdsson \cite{MaySigurdsson06} (see also \cite{HSS20}\cite{BM2021}), are notoriously hard to deal with. Considerable streamlining was achieved by Malkiewich \cite{Malkiewich19}\cite{Malkiewich23}, but not quite including the monoidal model property of the external smash product of spectra. Such a monoidal model structure is expected to be needed notably for realizing categorical semantics for the recently described \emph{Linear Homotopy Type Theory} (LHoTT, \cite{RFL21}\cite{Riley22} following \cite[\S 3]{Schreiber14}), 
which promises to be a formal language for parameterized stable homotopy theory generalizing how plain Homotopy Type Theory (HoTT) is now famously known to be a formal language for plain homotopy theory (cf. \cite{Shulman21}).

\medskip

\noindent
{\bf Specialization to the $\mathbb{K}$-Linear case.}
However, in prominent applications of parameterized stable homotopy theory and LHoTT, such as topological quantum processes (cf. \cite{Schreiber14}\cite{TQP}\cite{QS}\cite[\S 2.3.1]{SS25-Complete}), one deals only with the comparatively simpler $\mathbb{K}$-linear case of parameterized $H\mathbb{K}$-module spectra over some field $\mathbb{K}$ (not necessarily of characteristic zero). Namely, under the \emph{stable Dold-Kan correspondence} which says here that $H\mathbb{K}$-module spectra are equivalently (simplicial) chain complexes of $\mathbb{K}$-vector spaces,
\begin{equation}
  \label{ChainComplexesAsModuleSpectra}
  \mathrm{Mod}_{H\GroundField}(\mathrm{Spectra})
  \underset{
    \mathclap{
    \scalebox{.7}{
      \def\arraystretch{.9}
      \begin{tabular}{c}
        \cite[Thm. 1.1]{Shipley07}
        \\
        \cite[Thm. 5.1.6]{SchwedeShipley03}
      \end{tabular}
    }
    }
  }{
  \;\;\;\;\;\;
  \simeq_{{}_{\mathrm{Qu}}}
  \;\;\;\;\;\;
  }
  \mathrm{Ch}_{\GroundField}
  \underset{
    \mathclap{
      \scalebox{.7}{
        \def\arraystretch{.9}
        \begin{tabular}{c}
          \cite[p. 10]{RezkSchwedeShipley01}
          \\
          Thm. \ref{ModelCategoryOfSimplicialChainComplexes}
        \end{tabular}
      }
    }
  }{
  \;\;\;\;\;\;
  \simeq_{{}_{\mathrm{Qu}}}
  \;\;\;\;\;\;
  }
  \mathrm{sCh}_{\GroundField}
  \mathrlap{\,,}
\end{equation}
the parameterized $H\mathbb{K}$-module spectra are equivalently identified with homotopically parameterized chain complexes of vector spaces: ``flat $\infty$-vector bundles'' (cf. \S\ref{SimplicialLocalSystems} below): 
These have been studied with the alternative dg-algebraic tools of dg-category theory under the name ``$\infty$-local systems'' (\cite{RiveraZeinalian20}\cite{AMV24} following \cite{BlockSmith14}). 

\medskip

\noindent
{\bf Need for a good $\mathbb{K}$-linear model.}
What has been missing is a combination of these approaches: A model category theoretic formulation of the homotopy theory of $\infty$-local systems, akin to the model structures for parameterized spectra but capitalizing on the relative simplifications that occur after $\mathbb{K}$-linearization and reducing fiberwise to common dg-algebraic constructions on chain complexes of vector spaces. 

\medskip

\noindent
{\bf Outline of results.}
To that end, we start in \S\ref{SimplicialLocalSystems} by making explicit a good (left proper simplicial combinatorial) model structure on the category $\mathbf{sCh}_{\mathbb{K}}$ of simplicial chain complexes (Thm. \ref{ModelCategoryOfSimplicialChainComplexes}). This readily implies (Prop. \ref{ModelStructureOnSimplicialFunctors}) a combinatorial model structure on each of the simplicial categories $\mathbf{sCh}_{\mathbb{K}}^{\mathbf{X}}$ of $\infty$-local systems over any fixed base homotopy type modeled as a simplicial groupoid $\mathbf{X}$, naturally defined the way one will want to use them in practice (Def. \ref{MonoidalSimplicialFunctorCategoryIntoSimplicialChainComplexes}): as simplicial functors $\begin{tikzcd}[sep=small] \mathbf{X} \ar[r] & \mathbf{sCh}_{\mathbb{K}}\end{tikzcd}$. The main step, then, is in \S\ref{ParameterizationOverVaryingBaseSpaces} the verification of the integral model structure on the Grothendieck construction 
\begin{equation}
  \mathbf{Loc}_{\mathbb{K}}
  \;:=\;
  \;\;\;
  \underset
    {\mathclap{
      \mathbf{X} \in \mathrm{sSet}\text{-}\mathrm{Grpd}_{\mathrm{skl}}
      }
    }
    \int
    \;\;\;
  \mathbf{sCh}
    ^{\mathbf{X}}
    _{\mathbb{K}}
  \mathrlap{\,,}
\end{equation}
amalgamating all these local model structures as the base homotopy type varies (Thm. \ref{GlobalModelStructure}). On this global model structure we discuss in \S\ref{ExternalTensorOnSimplicialLocalSystems} aspects of the 6-functor formalism (Grothendieck's ``motivic yoga'') and specifically of the \emph{external tensor product} of $\infty$-local systems (Def. \ref{ExternalTensorProductOfSimplicialLocalSystems},  covering the Cartesian product of base homotopy types),
\begin{equation}
  \begin{tikzcd}[
    row sep=0pt
  ]
  \mathbf{Loc}_{\mathbb{X}}
  \times
  \mathbf{Loc}_{\mathbb{K}}
  \ar[
    r,
    "{ \boxtimes }"
  ]
  &
  \mathbf{Loc}_{\mathbb{K}}
  \mathrlap{\,.}
  \\
  \mathbf{sCh}^{\mathbf{X}}
    _{\mathbb{K}}
  \times
  \mathbf{sCh}^{\mathbf{Y}}
    _{\mathbb{K}}
  \ar[r]
  &
  \mathbf{sCh}^{\mathbf{X} \times \mathbf{Y}}
    _{\mathbb{K}}
  \mathrlap{\,.}
  \end{tikzcd}
\end{equation}
While this cannot be a Quillen bifunctor (since already the underlying Cartesian product of simplicial groupoids is not), we show in Thm. \ref{ExternalTensorProductIsHomotopical} that it comes as close as it gets under these circumstances. In fact, in \S\ref{ExternalModulesOverOneTypes} we show that when restricted to base homotopy 1-types, the external tensor product does form a monoidal model category structure (Thm. \ref{ExternalMonoidalModelStructureOnLocalSystemsOverOneTypes}).

\medskip

\noindent
{\bf Conclusion and outlook.}
Under the equivalence \eqref{ChainComplexesAsModuleSpectra}, the general form of our result compares favorably with the global model categories for parameterized plain spectra ($\mathbb{S}$-module spectra) indicated in \cite[Rem. 5.4.3]{Malkiewich23}:
 The analogue for $\mathbb{S}$-modules of our sub-result \eqref{ExternalTensorIsHomotopical} in Thm. \ref{ExternalTensorProductIsHomotopical} 
  is \cite[Rem. 6.1.4]{Malkiewich23}, 
  and the analogue of our sub-result \eqref{ExternalTensorIsQuillenBifunctorOverFixedBaseObjects} is
  \cite[Lem. 6.4.3]{Malkiewich19}\cite[Lem. 5.4.5]{Malkiewich23}, following \cite{MaySigurdsson06},
  which is referred to there as the ``perhaps most convenient category of parameterized spectra''.  At the same time, our model makes full use of the simplifications brought about by $\mathbb{K}$-linearization and will lend itself to applications in this context, some of which we indicate in the concluding \S\ref{DiscussionAndOutlook}.

\medskip

\section{The Model Structures}
 \label{ExternalTensorProductAsAQuillenBifunctor}

\subsection{\texorpdfstring{Flat $\infty$-vector bundles ($\infty$-local systems)}{Flat infinity-vector bundles (infinity-local systems)}}
\label{SimplicialLocalSystems}

The idea of {\it flat $\infty$-vector bundles} ($\infty$-local systems) over any space $X$ is that 

\begin{itemize}
  \item[{\bf (i)}] to any point $x \in X$ is assigned a chain complex $V_x$ (cf. Def. \ref{CategoryOfChainComplexes} below);
  \item[{\bf (ii)}] to any path $p : \begin{tikzcd}[decoration=snake] x \ar[r, decorate] & y\end{tikzcd}$ is assigned
  a chain map $\phi_p : V_x \to V_y$;
  \item[{\bf (iii)}] to any path-of-paths 
  $  
  \begin{tikzcd}[decoration=snake, row sep=6pt]
      & 
      y
      \ar[
        dr, 
        decorate, 
        bend left=20,
        "{ p' }",
        "{\ }"{name=s, pos=.1, swap}
      ]
      \\
      x
      \ar[
        ur, 
        decorate,
        "{ p }",
        bend left=20
      ]
      \ar[
        rr, 
        decorate,
        bend right=20,
        "{ p }"{swap},
        "{\ }"{pos=.4, name=t}
      ]
      &&
      z
      \ar[
        from=s,
        to=t,
        Rightarrow,
        "{ \sigma }"
      ]
    \end{tikzcd}
      $
  is assigned a chain homotopy
  $ 
    \begin{tikzcd}[decoration=snake, row sep=2pt]
      & 
      V_y
      \ar[
        dr, 
        bend left=20,
        "{ \phi_{p'} }",
        "{\ }"{name=s, pos=.1, swap}
      ]
      \\
      V_x
      \ar[
        ur, 
        "{ \phi_{p} }",
        bend left=15
      ]
      \ar[
        rr, 
        bend right=15,
        "{ \phi_{p''} }"{swap},
        "{\ }"{pos=.4, name=t}
      ]
      &&
      V_z
      \ar[
        from=s,
        to=t,
        Rightarrow,
        "{ \sigma }"
      ]
    \end{tikzcd}
  $
  
  (In the special case where all chain complexes here are concentrated in degree 0, then these chain homotopies are necessarily 
  identities and the structure reduces to that of an ordinary flat vector bundle.)
  \item[\bf ($\infty$)] ``and so on''...
\end{itemize}
But the explicit continuation of this pattern is not as straightforward as the previous steps, since there is not {\it directly} a notion of higher chain homotopy.

On the other hand, in the conceptual perspective of $\infty$-category theory it is evident that, abstractly, flat $\infty$-vector bundles 
should equivalently be $\infty$-functors from the fundamental $\infty$-groupoid of their base space to the $\infty$-category of chain complexes:
\begin{equation}
  \label{LocalSystemsAsInfinityFunctors}
  \overset{
    \mathclap{
      \raisebox{6pt}{
        \scalebox{.7}{
          \color{darkblue}
          \bf
          $\infty$-local systems
        }
      }
    }
  }{
  \mathrm{LocSys}_\infty(
    \underset{
      \mathclap{
        \rotatebox{-30}{
          \rlap{
            \hspace{-10pt}
            \scalebox{.6}{
              \color{gray}
              \bf
              base space
            }
          }
        }
      }
    }{
      X
    }
  )
  }
  \;\;
  :=
  \;\;
  \overset{
    \mathclap{
      \raisebox{2pt}{
        \scalebox{.7}{
          \color{darkblue}
          \bf
          $\infty$-functors
        }
      }
    }
  }{
  \mathrm{Func}_\infty\Big(
    \underset{
      \mathclap{
        \hspace{14pt}
        \rotatebox{-30}{
          \rlap{
            \hspace{-14pt}
            \scalebox{.6}{
              \color{gray}
              \bf
              \def\arraystretch{.75}
              \begin{tabular}{c}
                fundamental 
                \\
                \hspace{-40pt}$\infty$-groupoid
              \end{tabular}
            }
          }
        }
      }
    }{
    \shape X
    }
\;\;    ,\;\;
    \underset{
      \mathclap{
        \hspace{19pt}
        \rotatebox{-30}{
          \rlap{
            \hspace{-14pt}
            \scalebox{.6}{
              \color{gray}
              \bf
              \def\arraystretch{.75}
              \begin{tabular}{c}
                $\infty$-category of
                \\
                \hspace{-40pt}chain complexes
              \end{tabular}
            }
          }
        }
      }
    }{
    N \mathrm{Ch}_{\mathbb{K}} [\mathrm{W}_{\mathrm{qi}}^{-1}]
    }
  \Big)
  }
  \;\;\simeq\;\;
  \overset{
    \mathclap{
      \raisebox{2pt}{
        \scalebox{.7}{
          \color{darkblue}
          \bf
          simplicial functors
        }
      }
    }
  }{
  N
  \Big(
  \mathrm{sFunc}\big(
    \underset{
      \mathclap{
        \hspace{13pt}
        \rotatebox{-30}{
          \rlap{
            \hspace{-10pt}
            \scalebox{.6}{
              \color{gray}
              \bf
              \def\arraystretch{.8}
              \begin{tabular}{c}
                fundamental 
                \\
                \hspace{-30pt}simplicial groupoid
              \end{tabular}
            }
          }
        }
      }
    }{
    \mathcal{G}(\shape X)
    }
   \;\; ,\;\;
    \underset{
      \mathclap{
        \hspace{13pt}
        \rotatebox{-30}{
          \rlap{
            \hspace{-10pt}
            \scalebox{.6}{
              \color{gray}
              \bf
              \def\arraystretch{.8}
              \begin{tabular}{c}
                simplicial model category
                \\
                \hspace{-70pt}of chain complexes
              \end{tabular}
            }
          }
        }
      }
    }{
    \mathrm{sCh}_{\mathbb{K}}
    }
  \big)^\circ
  \Big)
  \,.
  }
\end{equation}
\vspace{.7cm}

A quick but unwieldy way to define such $\infty$-functors is as maps from the singular simplicial complex $\shape X$ of $X$ into the suitably-defined 
``homotopy-coherent simplicial nerve'' of $\mathrm{Ch}_{\mathbb{K}}$ (which is essentially tantamount to defining a simplicial notion of ``higher chain homotopy''). 
Taking this to be the ``dg-nerve'' \cite[\S 1.3.1.6]{LurieAlgebra} of the dg-self-enrichment of $\mathrm{Ch}_{\mathbb{K}}$ is the approach 
in the existing literature on $\infty$-local systems \cite[\S 2]{BlockSmith14}\cite[\S 5]{RiveraZeinalian20}. 

\smallskip 
However, here we need tighter control over these $\infty$-categories of $\infty$-local systems, namely we need a good model category 
presentation by simplicial functors --- which had been missing but which we establish now. 

\begin{remark}[Equivalence of models of $\infty$-local systems]
\label{EquivalenceOfModelsOfInfinityLocalSystems}
That our new construction of $\infty$-local systems (in Def. \ref{MonoidalSimplicialFunctorCategoryIntoSimplicialChainComplexes} below) is equivalent, 
as an $\infty$-category, to the existing constructions (\cite[\S 2]{BlockSmith14}\cite[\S 5]{RiveraZeinalian20})
can be seen via the comparison maps between the dg-nerve  and the simplicial nerve established in
\cite[Prop. 1.3.4.5]{LurieAlgebra}\cite[Prop. 5.17]{GwilliamPavlov18}.\footnote{We thank Dmitri Pavlov for pointing out this result.}
\end{remark}

\noindent
{\bf Simplicial model categories.}
In the following, we make free use of notions and facts of model category theory \cite{Quillen67} (textbook accounts include \cite{Hovey99}\cite{Hirschhorn02}\cite{SSS23Character}) 
in particular as concerns (combinatorial) simplicial model categories (see \cite[\S A]{Lurie09} and especially the comprehensive introductory account in \cite{Riehl14}).

\medskip 

\noindent
{\bf $\infty$-Vector spaces.} We consider a particularly good model category presentation of the $\infty$-category of $\infty$-vector 
spaces modeled by chain complexes of vector spaces.
While the model category $\mathrm{Ch}_{\mathbb{K}}$ of unbounded chain complexes is fairly familiar (recalled below), in its 
plain form as an ordinary category it is not a useful ingredient in the construction of flat $\infty$-vector bundles ($\infty$-local systems), 
since the ordinary functors from 1-groupoids into it only model flat vector bundles over homotopy 1-types. This may
be one reason why existing literature on $\infty$-local systems has made no use of model category theoretic tools.
However, there is a well-known general approach to such situations which does stay within (and hence retains the power of)
model category theory: This is to find an enhancement to a {\it combinatorial simplicial model} category whose underlying 
ordinary model category is Quillen equivalent to the original one: In this case, the category of simplicial functors with 
this codomain still presents the desired $\infty$-functor $\infty$-category but now does inherit a supporting model category structure itself.
The fact that $\mathrm{Ch}_{\mathbb{K}}$ does admit such a simplicial enhancement is essentially well-known, though 
some of the details, such as its compatibility with the monoidal structure, are not explicit in the literature;
therefore we spell it out:

\begin{definition}[Category of simplicial chain complexes\footnote{Beware that some authors say ``simplicial chain complex''
for the chain complexes that compute singular homology groups. Here we properly mean ``simplicial objects in the category 
of chain complexes'' ---  which is of course not unrelated but different and/or more general.}]
\label{CategoryOfChainComplexes} $\,$ \newline
  For $\mathbb{K}$ a field (not necessarily of characteristic zero), we write:
  \begin{itemize}
  \item[{\bf (i)}]
  \fbox{$\big(
    \mathrm{Mod}_{\mathbb{K}}
    ,\,
    \otimes
  \big)$}
  for the category of $\mathbb{K}$-vector spaces with $\mathbb{K}$-linear maps between them (which below we frequently think 
  of as the special case of $R$-modules for $R = \mathbb{K}$ a field, whence the notation), and equipped with the usual 
  closed monoidal category structure given by the ordinary tensor product
  whose tensor unit is $\mathbb{K}$. We denote
  the linear mapping vector space between a pair of vector spaces by angular brackets 
  \begin{equation}
    \label{ClosedMonoidalStructureOnKMod}
    (-) \otimes (-)
    \,:\,
    \mathrm{Mod}_{\mathbb{K}}
    \times
    \mathrm{Mod}_{\mathbb{K}}
    \xrightarrow{\quad}
    \mathrm{Mod}_{\mathbb{K}}
    \,,
    \;\;\;\;\;\;\;\;\;\;
    [-,-]
    \,:\,
    \mathrm{Mod}_{\mathbb{K}}^{\mathrm{op}}
    \times
    \mathrm{Mod}_{\mathbb{K}}
    \xrightarrow{\quad}
    \mathrm{Mod}_{\mathbb{K}}    
  \end{equation}
  and will successively overload this notation as this category is incrementally generalized in the following. The linear 
  mapping space is of course the {\it internal hom} for this closed monoidal category, in that we have natural isomorphisms of hom-sets of this form:
  \begin{equation}
    \label{InternalHomIsomorphismForVectorSpaces}
    \mathrm{Mod}_{\mathbb{K}}
    \big(
      T
      \,\otimes\,
      V
      ,\,
      W
    \big)
    \;\simeq\;
    \mathrm{Mod}_{\mathbb{K}}\big(
      T
      ,\,
      [V,\,W]
    \big)
    \,.
  \end{equation}
  The category $\mathrm{Mod}_{\mathbb{K}}$ is complete and cocomplete, in particular it is canonically 
{\it tensored} and {\it powered} over $\mathrm{Set}$
  \begin{equation}
    \label{CoTensoringOfVectorSpacesOverSets}
    \begin{tikzcd}[row sep=-2pt, column sep=small]
      \mathrm{Set}
      \times
      \mathrm{Mod}_{\mathbb{K}}
      \ar[
        rr,
        "{
          (\mbox{-})
          \cdot
          (\mbox{-})
        }"
      ]
      &&
      \mathrm{Mod}_{\mathbb{K}}
      \mathrlap{\,,}
      \\
      (S,\, V) &\longmapsto&
      \underset{s \in S}{\coprod} V
    \end{tikzcd}
    \hspace{1.2cm}
    \begin{tikzcd}[row sep=-2pt, column sep=small]
      \mathrm{Set}^{\mathrm{op}}
      \times
      \mathrm{Mod}_{\mathbb{K}}
      \ar[
        rr,
        "{
          (\mbox{-})^{(\scalebox{.8}{-})}
        }"
      ]
      &&
      \mathrm{Mod}_{\mathbb{K}}
      \\
      (S,\, V) &\longmapsto&
      \underset{s \in S}{\prod} V
    \end{tikzcd}
  \end{equation}
  such that there are natural isomorphisms of
  hom-sets of the following form
  \begin{equation}
    \label{TensoringIsomorphismsForKMod}
    S \,\in\, \,\mathrm{Set}
    ,\,
    V,W \,\in\, \mathrm{Mod}_{\mathbb{K}}
    \hspace{.5cm}
      \vdash
    \hspace{.5cm}
    \mathrm{Mod}_{\mathbb{K}}\big(
      S \cdot V
      ,\,
      W
    \big)
    \;\simeq\;
    \mathrm{Set}\big(
      S 
      ,\,
      \mathrm{Mod}_{\mathbb{K}}(
        V
        ,\,
        W
      )
    \big)
    \;\simeq\;
    \mathrm{Mod}_{\mathbb{K}}\big(
      V
      ,\,
      W^S
    \big)
    \,.
  \end{equation}
  This may also be understood as the natural isomorphism \eqref{InternalHomIsomorphismForVectorSpaces}
  partially restricted along the free $\mathbb{K}$-module functor 
  \[
    \begin{tikzcd}[row sep=-3pt, column sep=small]
      \mathllap{
        \mathbb{K}[-] 
        \,:\, 
      }
      \mathrm{Set} 
        \ar[rr]
        &&
      \mathrm{Mod}_{\mathbb{K}}
      \\
      S 
        &\longmapsto& 
      \underset{s \in S}{\coprod}
      \mathbb{K}
    \end{tikzcd}
  \]
 in that
  \begin{equation}
    \label{EquivalentWaysOfTensoringVectorSpaces}
    \begin{tikzcd}
      \mathrm{Set}
        \times
      \mbox{\bf{Mod}}_{\mathbb{K}}
      \ar[
        rr,
        "{
          \mathbb{K}{[\mbox{-}]}
          \,\times\,
          \mathrm{id}
        }",
        "{\ }"{pos=.5, swap, name=s}
      ]
      \ar[
        dr,
        "{ (\mbox{-})\cdot(\mbox{-}) }"{swap},
        "{\ }"{name=t}
      ]
      \ar[
        from=s,
        to=t,
        Rightarrow, 
        "{ \sim }"{sloped, swap}
      ]
      &&
      \mbox{\bf{Mod}}_{\mathbb{K}}
        \times
      \mbox{\bf{Mod}}_{\mathbb{K}}
      \ar[
        dl, 
        "{ (\mbox{-})\otimes(\mbox{-}) }"
      ]
      \\
      &
      \mbox{\bf{Mod}}_{\mathbb{K}}
    \end{tikzcd}
    \hspace{1cm}
    \mathbb{K}[S]
      \,\otimes\,
    V
    \;\simeq\;
    S \cdot V
    \,.
  \end{equation}

  In particular, for $G \in \mathrm{Grp} \to \mathrm{Set}$ a set equipped with a group structure $\mu : G \times G \to G$, $\mathrm{e} : \ast \to G$, 
  whose group algebra is $\mathbb{K}[G] \,\in\, \mathrm{Alg}_{\mathbb{K}} \to \mathrm{Mod}_{\mathbb{K}}$, a $\mathbb{K}[G]$-module structure on
  a vector space may equivalently be thought of as a ``tensoring action'' via a morphism
  $\rho : G \cdot V \longrightarrow V$ making the following diagrams commute:
  \[
    \begin{tikzcd}[column sep=large]
      G \cdot (G \cdot V) 
      \ar[r, phantom, "{ \simeq }"]
      \ar[
        d,
        "{
          \mathrm{id}
          \cdot 
          \rho
        }"
      ]
      &[-20pt]
      (G \times G) \cdot V
      \ar[
        r,
        "{ \mu \,\cdot\, \mathrm{id} }"
      ]
      &
      G \cdot V
      \ar[
        d, 
        "{ \rho }"
      ]
      \\
      G \cdot V
      \ar[rr, "{\rho}"]
      &&
      V
    \end{tikzcd}
    \hspace{1cm}
    \begin{tikzcd}
      \ast \cdot \mathscr{V}
      \ar[rr, "{ \mathrm{e} \,\cdot\, \mathrm{id} }"]
      \ar[dr, "{ \sim }"{swap, sloped}]
      &&
      G \cdot \mathscr{V}
      \ar[dl, "\rho"]
      \\
      & \mathscr{V}
    \end{tikzcd}
  \]


  \item[{\bf (ii)}] 
  \fbox{$\big(\mathrm{Ch}_{\mathbb{K}}, \, \otimes\big)$} for the category of (unbounded) chain complexes of $\mathbb{K}$-vector 
  spaces with chain maps between them as morphisms (cf. \cite[\S 1]{Weibel94})
  \begin{equation}
    \label{ChainMap}
    \begin{tikzcd}[column sep=large]
      V
      \ar[d, "\phi"]
      \ar[r, phantom, ":\defneq"]
      &
      \big[
      \;
      \cdots
      \ar[r, "{ \partial^V_{1} }"]
      \ar[d, phantom, shift left=5pt, "{\cdots}"]
      &
      V_1 
      \ar[r, "{ \partial^V_{0} }"]
      \ar[d, "{ \phi_1 }"]
      &
      V_0 
      \ar[r, "{ \partial^V_{-1} }"]
      \ar[d, "{ \phi_0 }"]
      &
      V_{-1}
      \ar[r, "{ \partial^V_{-2} }"]
      \ar[d, "{ \phi_{-1} }"]
      &
      \cdots
      \ar[d, phantom, shift right=5pt, "{\cdots}"]
      \;
      \big]
      \\
      W
      \ar[r, phantom, ":\defneq"]
      &
      \big[
      \;
      \cdots
      \ar[r, "{ \partial^W_{1} }"]
      &
      W_1 
      \ar[r, "{ \partial^W_{0} }"]
      &
      W_0 
      \ar[r, "{ \partial^W_{-1} }"]
      &
      W_{-1}
      \ar[r, "{ \partial^W_{-2} }"]
      &
      \cdots      
      \;
      \big]
    \end{tikzcd}
  \end{equation} 
  and equipped with monoidal category structure given by
  the usual tensor product of chain complexes (cf. \cite[\S 2.7]{Weibel94}):
  \begin{equation}
    \label{TensorProductOfChainComplexes}
    \hspace{-6mm} 
    \begin{tikzcd}[column sep=21pt]
      V
      \otimes 
      W
    \ar[r, phantom, ":\defneq"]
    &[-9pt]
   \Big[
    \;
    \cdots
    \ar[r]
    &[-5pt]
    \underset{n \in \mathbb{Z}}{\bigoplus}
    \big(
      V_{n+1} 
      \!\otimes\!
      W_{-n}
    \big)
    \ar[
      rr,
      "{
        \def\arraystretch{.8}
        \begin{array}{c}
        \partial^V \otimes\, \mathrm{id} -
        \\
        (-1)^n
        \mathrm{id} \,\otimes\, \partial^W
        \end{array}
      }"
    ]
    & &
    \underset{n \in \mathbb{Z}}{\bigoplus}
    \big(
      V_n 
      \!\otimes\!
      W_{-n}
    \big)
    \ar[
      rr,
      "{
        \def\arraystretch{.8}
        \begin{array}{c}
        \partial^V \otimes\, \mathrm{id} -
        \\
        (-1)^n
        \mathrm{id} \,\otimes\, \partial^W
        \end{array}
      }"
    ]
    &&
    \underset{n \in \mathbb{Z}}{\bigoplus}
    \big(
      V_{n-1} 
      \!\otimes\!
      W_{-n}
    \big)
    \ar[r]
    &[-5pt]
    \cdots \;\Big]
    \end{tikzcd}
  \end{equation}
  whose tensor unit is
  \begin{equation}
    \label{TensorUnitAmongChainComplexes}
    \mathbbm{1}
    \;\;:\defneq \;
    \begin{tikzcd}[sep=20pt]
      [\;
        \cdots
        \ar[r]
        &
        0 
        \ar[r]
        &
        0 
        \ar[r]
        &
        \underset{
          \mathclap{
            \raisebox{-3pt}{
              \scalebox{.7}{
                \color{gray}
                $\mathrm{deg} = 0$
              }
            }
          }
        }{
          \mathbb{K}
        }
        \ar[r]
        &
        0
        \ar[r]
        &
        0
        \ar[r]
        &
        \cdots
      \;]\,,
    \end{tikzcd}
  \end{equation}
  and which is closed with internal-hom (``mapping complex'') given by:
  \begin{equation}
    \label{InternalHomOfChainComplexes}
    \hspace{-3mm} 
    [V,\,W]
    \;:\defneq\;
    \Big[
    \begin{tikzcd}[column sep=50pt]
      \cdots
      \ar[r]
      &[-30pt]
      \underset{
        n \in \mathbb{X}
      }{\bigtimes}
      \big[
        V_n 
        ,\,
        W_{n+1}
      \big]
      \ar[
        r,
        "{
          \def\arraystretch{.8}
          \begin{array}{l}
            \hspace{3pt} 
            \partial^{W} \,\circ\, (\mbox{-})
            \;-
            \\
            (-1)^n 
            (\mbox{-}) \,\circ\, \partial^V
           \end{array}
        }"
      ]
      &
      \underset{
        n \in \mathbb{X}
      }{\bigtimes}
      \big[
        V_n 
        ,\,
        W_{n}
      \big]
      \ar[
        r,
        "{
          \def\arraystretch{.8}
          \begin{array}{l}
            \hspace{3pt} 
            \partial^{W} \,\circ\, (\mbox{-})
            \;-
            \\
            (-1)^n 
            (\mbox{-}) \,\circ\, \partial^V
           \end{array}
        }"
      ]
      &
      \underset{
        n \in \mathbb{X}
      }{\bigtimes}
      \big[
        V_n 
        ,\,
        W_{n-1}
      \big]
      \ar[r]
      &[-30pt]
      \cdots
    \end{tikzcd}
    \Big]
    \,,
  \end{equation}
  where the tensor symbol and the angular brackets on the right denote the corresponding operations \eqref{ClosedMonoidalStructureOnKMod} on component vector spaces.
  
  This category is (co)complete with (co)limits formed degreewise in $\mathrm{Mod}_{\mathbb{K}}$; in particular it is (co)tensored over $\mathrm{Set}$, degreewise as in \eqref{CoTensoringOfVectorSpacesOverSets}  
  \begin{equation}
    \label{SetTensoringOfChainComplexes}
    \begin{tikzcd}[row sep=-3pt, column sep=small]
      \mathrm{Set} \,\times\, \mathrm{Ch}_{\mathbb{K}}
      \ar[
        rr,
        "{
          (\mbox{-}) \cdot (\mbox{-})
        }"
      ]
      &&
      \mathrm{Ch}_{\mathbb{K}}
      \\
      (
        S ,\, V
      )
      &\longmapsto&
      \underset{s \in S}{\coprod} V
    \end{tikzcd}
    \hspace{1cm}
    \begin{tikzcd}[sep=0pt]
      \mathrm{Set}^{\mathrm{op}} 
        \,\times\, 
      \mathrm{Ch}_{\mathbb{K}}
      \ar[
        rr,
        "{
          (\mbox{-})^{(\mbox{-})}
        }"
      ]
      &&
      \mathrm{Ch}_{\mathbb{K}}
      \\
      (
        S ,\, V
      )
      &\longmapsto&
      \underset{s \in S}{\prod} V
    \end{tikzcd}
  \end{equation}
  
  \item[{\bf (iii)}] 
  \fbox{$\big(\mathrm{sCh}_{\mathbb{K}}, \otimes \big) \,\coloneqq\, \big(\mathrm{Ch}_{\mathbb{K}}, \otimes\big)^{\Delta^{\mathrm{op}}}$} 
  for the category of {\it simplicial objects} (cf. \cite[\S 1]{May67}\cite[\S A.1]{TamakiKono06}) 
  in the previous category of unbounded chain complexes:
  $$
    \hspace{-.5cm}
    \mathscr{V}
    :\defneq
    \left[
    \adjustbox{raise=7pt}{
    \begin{tikzcd}[
      row sep=43pt,
      column sep=55pt
    ]
      {}
      \ar[d, -, shift left=18pt, dotted]
      \ar[d, -, shift left=12pt, dotted]
      \ar[d, -, shift left=6pt, dotted]
      \ar[d, -, dotted]
      \ar[d, -, shift right=6pt, dotted]
      \ar[d, -, shift right=12pt, dotted]
      \ar[d, -, shift right=18pt, dotted]
      &[-10pt]
      &
      \ar[d, -, shift left=18pt, dotted]
      \ar[d, -, shift left=12pt, dotted]
      \ar[d, -, shift left=6pt, dotted]
      \ar[d, -, dotted]
      \ar[d, -, shift right=6pt, dotted]
      \ar[d, -, shift right=12pt, dotted]
      \ar[d, -, shift right=18pt, dotted]
      &
      \ar[d, -, shift left=18pt, dotted]
      \ar[d, -, shift left=12pt, dotted]
      \ar[d, -, shift left=6pt, dotted]
      \ar[d, -, dotted]
      \ar[d, -, shift right=6pt, dotted]
      \ar[d, -, shift right=12pt, dotted]
      \ar[d, -, shift right=18pt, dotted]
      &
      \ar[d, -, shift left=18pt, dotted]
      \ar[d, -, shift left=12pt, dotted]
      \ar[d, -, shift left=6pt, dotted]
      \ar[d, -, dotted]
      \ar[d, -, shift right=6pt, dotted]
      \ar[d, -, shift right=12pt, dotted]
      \ar[d, -, shift right=18pt, dotted]
      \\[+5pt]
      \mathscr{V}_{2}
      \ar[r, phantom, ":\defneq"{pos=.6}]
      \ar[
        d,
        shift left=22pt,
        "{ d^2_1  }"{description}
      ]
      \ar[
        from=d, 
        shift left=11pt, 
        "{s^0_1}"{description}
      ]
      \ar[
        from=d, 
        shift right=11pt, 
        "{s^1_1}"{description}
      ]
      \ar[
        d,
        "{ d^1_1 }"{description}
      ]
      \ar[
        d,
        shift right=22pt,
        "{ d^0_1  }"{description}
      ]
      &
      \big[\;\cdots
      \ar[
        r,
        "{ \partial^{\mathscr{V}_2}_1 }"{description}
      ]
      &
      \mathscr{V}_{2,1}
      \ar[
        r,
        "{ \partial^{\mathscr{V}_2}_0 }"{description}
      ]
      \ar[
        d,
        shift left=22pt,
        "{ d^2_{1,1}  }"{description, pos=.65}
      ]
      \ar[
        d,
        "{ d^1_{1,1} }"{description, pos=.65}
      ]
      \ar[
        d,
        shift right=22pt,
        "{ d^0_{1,1}  }"{description, pos=.65}
      ]
      \ar[
        from=d, 
        shift left=11pt, 
        "{s^0_{1,1}}"{description, pos=.65}
      ]
      \ar[
        from=d, 
        shift right=11pt, 
        "{s^1_{1,1}}"{description, pos=.65}
      ]
      &
      \mathscr{V}_{2,0}
      \ar[
        r,
        "{ \partial^{\mathscr{V}_2}_{-1} }"{description}
      ]
      \ar[
        d,
        shift left=22pt,
        "{ d^2_{1,0}  }"{description, pos=.65}
      ]
      \ar[
        d,
        "{ d^1_{1,0} }"{description, pos=.65}
      ]
      \ar[
        d,
        shift right=22pt,
        "{ d^0_{1,0}  }"{description, pos=.65}
      ]
      \ar[
        from=d, 
        shift left=11pt, 
        "{s^0_{1,0}}"{description, pos=.65}
      ]
      \ar[
        from=d, 
        shift right=11pt, 
        "{s^1_{1,0}}"{description, pos=.65}
      ]
      &
      \mathscr{V}_{2,-1}      
      \ar[
        r,
        "{ \partial^{\mathscr{V}_2}_{-2} }"{description}
      ]
       \ar[
        d,
        shift left=22pt,
        "{ d^2_{1,-1}  }"{description, pos=.65}
      ]
      \ar[
        d,
        "{ d^1_{1,-1} }"{description, pos=.65}
      ]
      \ar[
        d,
        shift right=22pt,
        "{ d^0_{1,-1}  }"{description, pos=.65}
      ]
      \ar[
        from=d, 
        shift left=11pt, 
        "{s^0_{1,-1}}"{description, pos=.65}
      ]
      \ar[
        from=d, 
        shift right=11pt, 
        "{s^1_{1,-1}}"{description, pos=.65}
      ]
     &
      \cdots\;\big]
      \\
      \mathscr{V}_{1}
      \ar[r, phantom, ":\defneq"{pos=.6}]
      \ar[
        d,
        shift left=11pt,
        "{ d^1_{0}  }"{description}
      ]
      \ar[
        from=d, 
        "{s^0_{0}}"{description}
      ]
      \ar[
        d,
        shift right=11pt,
        "{ d^0_0  }"{description}
      ]
      &
      \big[\;\cdots
      \ar[
        r,
        "{ \partial^{\mathscr{V}_1}_1 }"{description}
      ]
      &
      \mathscr{V}_{1,1}
      \ar[
        r,
        "{ \partial^{\mathscr{V}_1}_0 }"{description}
      ]
      \ar[
        d,
        shift left=11pt,
        "{ d^1_{0,1}  }"{description, pos=.65}
      ]
      \ar[
        from=d, 
        "{s^0_{0,1}}"{description, pos=.65}
      ]
      \ar[
        d,
        shift right=11pt,
        "{ d^0_{0,1}  }"{description, pos=.65}
      ]
      &
      \mathscr{V}_{1,0}
      \ar[
        r,
        "{ \partial^{\mathscr{V}_1}_{-1} }"{description}
      ]
      \ar[
        d,
        shift left=11pt,
        "{ d^1_{0,0}  }"{description, pos=.65}
      ]
      \ar[
        from=d, 
        "{s^0_{0,0}}"{description, pos=.65}
      ]
      \ar[
        d,
        shift right=11pt,
        "{ d^0_{0,0}  }"{description, pos=.65}
      ]
      &
      \mathscr{V}_{1,-1}      
      \ar[
        r,
        "{ \partial^{\mathscr{V}_1}_{-2} }"{description}
      ]
      \ar[
        d,
        shift left=11pt,
        "{ d^1_{0,-1}  }"{description, pos=.65}
      ]
      \ar[
        from=d, 
        "{s^0_{0,-1}}"{description, pos=.65}
      ]
      \ar[
        d,
        shift right=11pt,
        "{ d^0_{0,-1}  }"{description, pos=.65}
      ]
      &
      \cdots\;\big]
      \\
      \mathscr{V}_{0}
      \ar[r, phantom, ":\defneq"{pos=.6}]
      &
      \big[\;\cdots
      \ar[
        r,
        "{ \partial^{\mathscr{V}_0}_1 }"{description}
      ]
      &
      \mathscr{V}_{0,1}
      \ar[
        r,
        "{ \partial^{\mathscr{V}_0}_0 }"{description}
      ]
      &
      \mathscr{V}_{0,0}
      \ar[
        r,
        "{ \partial^{\mathscr{V}_0}_{-1} }"{description}
      ]
      &
      \mathscr{V}_{0,-1}      
      \ar[
        r,
        "{ \partial^{\mathscr{V}_0}_{-2} }"{description}
      ]
      &
      \cdots\;\big]
    \end{tikzcd}
    }
  \!\!\!\!\!  \right]
  $$  
  with simplicial chain maps between these as morphisms, i.e.,
  arrays of linear maps respecting all these structure maps:
  $$
    \begin{tikzcd}[row sep=20pt]
      \mathscr{V}
      \ar[d, "{ \phi }"]
      \\
      \mathscr{W}
    \end{tikzcd}
    \;\;\;\;
    :\defneq
    \;\;\;
    \left(\!\!\!\!\!
    \adjustbox{raise=3pt}{
    \begin{tikzcd}
      \mathscr{V}_{i}
      \ar[d, " \phi_{i} "]
      \\
      \mathscr{W}_{i}
    \end{tikzcd}
    }
    \;\in\;
    \mathrm{Mor}\big(\mathrm{Ch}_{\mathbb{K}}\big)
   \!\!\! \right)_{
      {i \in \mathbb{N}}
   }
   \;\;\;
   \defneq
   \;\;\;
    \left(\!\!\!\!\!\!
    \adjustbox{raise=3pt}{
    \begin{tikzcd}
      \mathscr{V}_{i, j}
      \ar[d, " \phi_{i,j} "]
      \\
      \mathscr{W}_{i,j}
    \end{tikzcd}
    }
    \;\in\;
    \mathrm{Mor}\big(\mathrm{Mod}_{\mathbb{K}}\big)
   \!\!\! \right)_{
     \substack{
      {i \in \mathbb{N}}
      \\
      {j \in \mathbb{Z}}
      }
   }
  $$
  and equipped with the simplicial-degree-wise tensor product of chain complexes from above:
  $$
    \mathscr{V} \otimes \mathscr{W}
    \;\;\;:\defneq\;\;\;
    \left[
    \adjustbox{raise=+5pt}{
    \begin{tikzcd}[row sep=50pt]
      {}
      \ar[d, shift left=10pt, -, dotted]
      \ar[d, shift left=5pt, -, dotted]
      \ar[d, -, dotted]
      \ar[d, shift right=5pt, -, dotted]
      \ar[d, shift right=10pt, -, dotted]
      \\[-26pt]
      \mathscr{V}_1 \otimes \mathscr{W}_1
      \ar[
        from=d, 
        "{ s^0_{0} \otimes s^0_0 }"{pos=.65,description}
      ]
      \ar[
        d,
        shift right=17pt,
        "{ d^0_0 \otimes d^0_0 }"{pos=.65,description}
      ]
      \ar[
        d,
        shift left=17pt,
        "{ d^1_0 \otimes d^1_0 }"{pos=.65,description}
      ]
      \\
      \mathscr{V}_0 \otimes \mathscr{W}_0
    \end{tikzcd}
    }
    \right]
    \,,
  $$
  which is\footnote{
    This follows by using Lem. \ref{EnhancedEnrichmentOfFunctorCategories} for: 
  $\mathbf{V} :\defneq \mathbf{C} :\defneq \mathrm{Ch}_{\mathbb{K}}$ and
  $\mathbf{X} :\defneq \Delta^{\mathrm{op}}$ regarded as a $\mathrm{Set}\xhookrightarrow{\mathbb{K}[-]}\mathrm{Ch}_{\mathbb{K}}$-enriched category).
  } 
  monoidal closed with 
  internal hom (``simplicial mapping complex'') given by the end-formula
  \begin{equation}
    \label{InternalHomOfSimplicialChainComplexes}
    \mathscr{V}, \mathscr{W}
    \;\in\;
    \mathrm{sCh}_{\mathbb{K}}
    \hspace{.6cm}
    \vdash
    \hspace{.6cm}
    [\mathscr{V},\,\mathscr{W}]
    \;:=\;
    \bigg(
    [k]
    \;\longmapsto\;
    \int_{[s] \in \Delta}
    \Big[
      \Delta[k]_s \cdot \mathscr{V}_s
      ,\,
      \mathscr{W}_s
    \Big]
    \bigg)
    \;\;\;
    \in
    \;
    \mathrm{sCh}_{\mathbb{K}}
    \,,
  \end{equation}
  (where in the integrand on the right the angular brackets refer to the mapping complex \eqref{InternalHomOfChainComplexes} and the dot to the tensoring \eqref{SetTensoringOfChainComplexes} of component chain complexes).
  
  \item[{\bf (iv)}]
  \fbox{$\mathbf{sCh}_{\mathbb{K}}$} for the $\mathrm{sSet}$-enriched category (``simplicial category'') whose objects are those of $\mathrm{sCh}_{\mathbb{K}}$ and whose simplicial hom-sets are\footnote{
    This follows using Lem. \ref{EnhancedEnrichmentOfFunctorCategories} for $\mathbf{V} :\defneq \mathrm{Set}$, $\mathbf{C} :\defneq \mathrm{Ch}_{\mathbb{K}}$ and $\mathbf{X} :\defneq \Delta^{\mathrm{op}}$.
  }
  \begin{equation}
    \label{SimplicialHomSetsBetweenSimplicialChainComplexes}
    \mathscr{V}, \mathscr{W}
    \;\in\;
    \mathrm{sCh}_{\mathbb{K}}
    \hspace{.6cm}
    \vdash
    \hspace{.6cm}
    \mbox{\bf{sCh}}_{\mathbb{K}}
    \big(
    {
      \mathscr{V}
      ,\,
      \mathscr{W}
    }
    \big)
    \;\;:=\;\;
    \bigg(
    [k] 
    \,\mapsto\, 
    \int_{[s] \in \Delta}
    \mathrm{Ch}_{\mathbb{K}}\big(
      \Delta[k]_s \cdot \mathscr{V}_s
      ,\,
      \mathscr{W}_s
    \big)
    \bigg)
    \;\;\;
    \in
    \;
    \mathrm{sSet}
    \,,
  \end{equation}
 and with (co)tensoring given by
 \begin{equation}
    \label{sSetTensoringOfSimplicialChainComplexes}
    \hspace{-8mm} 
    \begin{tikzcd}[sep=0pt]
      \mathrm{sSet} \times \mathrm{sCh}_{\mathbb{K}}
      \ar[
        rr, 
        "{ 
          (\mbox{-})
          \cdot
          (\mbox{-})
        }"
      ]
      &&
      \mathrm{sCh}_{\mathbb{K}}
      \\
      \big( 
        \mathcal{S} 
        ,\, 
        \mathscr{V} 
      \big)
      &\mapsto&
      \big(
        [k] 
          \,\mapsto\, 
        \mathcal{S}_k \cdot \mathscr{V}_k
      \big)
      \,,
    \end{tikzcd}
    \hspace{1.2cm}
    \begin{tikzcd}[sep=0pt]
      \mathrm{sSet}^{\mathrm{op}} 
        \times 
      \mathrm{sCh}_{\mathbb{K}}
      \ar[
        rr, 
        "{ 
          (\mbox{-})^{(\mbox{-})} 
        }"
      ]
      &&
      \mathrm{sCh}_{\mathbb{K}}
      \\
            \big( 
        \mathcal{S} 
        ,\, 
        \mathscr{V} 
      \big)
      &\mapsto&
      \Big(
        [k] 
          \,\mapsto\, 
        \int_{s \in \Delta}
        \big(
          \mathscr{V}_s
        \big)^{
          (\mathcal{S} \times \Delta[k])_s
        }
      \Big)
      \,,
    \end{tikzcd}
  \end{equation}
  \begin{equation}
    \label{TensoringIsomorphismsForSCh}
    \hspace{-5mm} 
    \mathcal{S} \,\in\, \,\mathrm{sSet}
    ;\,
    \mathscr{V}
    ,
    \mathscr{W}
    \,\in\, 
    \mathrm{sCh}_{\mathbb{K}}
    \hspace{.5cm}
      \vdash
    \hspace{.5cm}
    \mbox{\bf{sCh}}_{\mathbb{K}}\big(
      \mathcal{S} 
      \!\cdot\! 
      \mathscr{V}
      ,\,
      \mathscr{W}
    \big)
    \;\simeq\;
    \mbox{\bf{sSet}}\big(
      \mathcal{S}
      ,\,
      \mbox{\bf{sCh}}(
        \mathscr{V}
        ,\,
        \mathscr{W}
      )
    \big)
    \;\simeq\;
    \mbox{\bf{sCh}}_{\mathbb{K}}\big(
      \mathscr{V}
      ,\,
      \mathscr{W}^{\mathcal{S}}
    \big)
  \end{equation}
  and regarded as an {\it $\mathrm{sSet}$-enriched monoidal category} (Def. \ref{sSetEnrichedMonoidalCategory}) 
  by enhancement of the previous tensor functor to an $\mathrm{sSet}$-enriched functor with the following components:
  \begin{equation}
    \label{sSetEnrichedTensorProduct}
    \hspace{-.5cm}
    \begin{tikzcd}[sep=-4pt]
      \mbox{\bf{sCh}}_{\mathbb{K}}\big(
        \mathscr{V}
        ,\,
        \mathscr{W}
      \big)
      \times
      \mbox{\bf{sCh}}_{\mathbb{K}}\big(
        \mathscr{V}'
        ,\,
        \mathscr{W}'
      \big)
      \ar[
        rr,
        "{
          \otimes_{ 
            (\mathscr{V} ,\, \mathscr{V}')
            ,\,
            (\mathscr{W} ,\, \mathscr{W}')
          }
        }"
      ]
      &&
      \mbox{\bf{sCh}}_{\mathbb{K}}\big(
        \mathscr{V} \,\otimes\, \mathscr{V}'
        ,\,
        \mathscr{W} \,\otimes\, \mathscr{W}'
      \big)  
      \\
      \Big(\!
      \big(
        \Delta[k]
        \cdot \mathscr{V}
        \overset{ \phi }{\to}
        \mathscr{W}
      \big)
      ,\,
      \big(
        \Delta[k]
        \cdot \mathscr{V}'
        \overset{ \phi' }{\to}
        \mathscr{W}'
      \big)
     \! \Big)
      &\mapsto&
      \Big(
        \Delta[k]
        \!\cdot\!
        \big(\mathscr{V} \otimes \mathscr{V}'\big)
        \xrightarrow{
          \mathrm{diag}
          \cdot(\cdots)
        }
        \big(
        \Delta[k]
        \cdot
        \mathscr{V} 
        \big)
          \otimes 
        \big(
          \Delta[k] \cdot \mathscr{V}'
        \big)
          \xrightarrow{
            \phi \,\otimes\, \phi'
          }
        \mathscr{W} \otimes \mathscr{W}'
      \Big)
    \end{tikzcd}
  \end{equation}
  and dually for the simplicial enhancement of the internal hom-functor:
  \begin{equation}
    \label{sSetEnrichedInternalHom}
    \hspace{-.4cm}
    \begin{tikzcd}[sep=0pt]
      \mbox{\bf{sCh}}^{\mathrm{op}}_{\mathbb{K}}\big(
        \mathscr{V}
        ,\,
        \mathscr{W}
      \big)
      \,\times\,
      \mbox{\bf{sCh}}_{\mathbb{K}}\big(
        \mathscr{V}'
        ,\,
        \mathscr{W}'
      \big)
      \ar[
        rr,
        "{
          [
            -
            ,\
            -
          ]_{
            (\mathscr{V} ,\, \mathscr{V}')
            ,\,
            (\mathscr{W} ,\, \mathscr{W}')
          }
        }"
      ]
      &&
      \mbox{\bf{sCh}}_{\mathbb{K}}\big(
        [\mathscr{V} ,\, \mathscr{V}']
        ,\,
        [\mathscr{W} ,\, \mathscr{W}']
      \big)
      \\
      \Big(
      \big(
       \Delta[k]
       \!\cdot\!
       \mathscr{W} 
         \xrightarrow{\phi}
       \mathscr{V}
       \big)
       ,
       \big(
       \mathscr{V}'
         \xrightarrow{\phi'}
       (\mathscr{W}')^{
         \Delta[k]
       }
       \big)
       \Big)
       &\mapsto&
       \Big(
         [\mathscr{V},\,\mathscr{W}]
         \xrightarrow{
           \scalebox{.7}{$
             \big[
               \phi
               ,\,
               \phi'
             \big]
           $}
         }
         \big[
           \Delta[k] \!\cdot\! \mathscr{V}'
           ,\,
           (\mathscr{W}')^{\Delta[k]}
         \big]
         \xrightarrow{
           (\cdots)^{\mathrm{diag}}
         }
         [\mathscr{V},\,\mathscr{W}]
       \Big)\,.
    \end{tikzcd}
  \end{equation}
  \item[{\bf (v)}]
  \adjustbox{fbox}{
      \begin{tikzcd}
        \mathrm{Ch}_{\mathbb{K}}
        \ar[
          r,
          shift left=5pt,
          "{ \mathrm{const} }"
        ]
        \ar[
          from=r,
          shift left=5pt,
          "{ \mathrm{ev}_0 }"
        ]
        \ar[
          r,
          phantom,
          "{ \scalebox{.7}{$\bot$} }"
        ]
        &
        \mathrm{sCh}_{\mathbb{K}}
      \end{tikzcd}
  }
  for the pair of strong monoidal adjoint functors, where
  \begin{itemize}
    \item the left adjoint, $\mathrm{const}$, sends a chain complex $V$ to the simplicial chain complex all of whose entries are $V$ and all whose simplicial maps are $\mathrm{id}_V$,
    \item the right adjoint, $\mathrm{ev}_0$, sends a simplicial chain complex $\mathscr{V}$ to its 0-component $\mathscr{V}_0$ (this being the limit over the simplicial diagram),
  \end{itemize}

  \item[{\bf (vi)}] 
  \fbox{$\mathrm{sCh}_{\mathbb{K}} 
    \overset{\mathrm{tot}}{\longrightarrow}
  \mathrm{Ch}_{\mathbb{K}}$} 
  for the {\it total chain complex} functor 
  \begin{equation}  
     \label{Totalization}
      \mathrm{tot}(\mathscr{V})
      \;\;
      :\defneq
      \;\;
      \begin{tikzcd}[column sep=50pt]
        \Big[\;
        \cdots
        \ar[r]
        &[-20pt]
        \underset{
          s+t = 1
        }{\bigoplus}
        \mathscr{V}_{s,t}
        \ar[
          r,
          "{
            \partial 
            + 
            \sum_s
            (-1)^s d^s 
          }"
        ]
        &
        \underset{
          s+t = 0
        }{\bigoplus}
        \mathscr{V}_{s,t}        
        \ar[
          r,
          "{
            \partial 
            + 
            \sum_s
            (-1)^s d^s 
          }"
        ]
        &
        \underset{
          s+t = -1
        }{\bigoplus}
        \mathscr{V}_{s,t}
        \ar[r]
        &[-20pt]
        \cdots
        \;\Big].
      \end{tikzcd}
  \end{equation}
\end{itemize}

\end{definition}
\begin{theorem}[Model category of simplicial chain complexes]
\label{ModelCategoryOfSimplicialChainComplexes}
The categories of chain complexes of vector spaces from Def. \ref{CategoryOfChainComplexes} carry the following model category structures:

\begin{enumerate}[label = (\roman*)]
  \item
  \fbox{$(\mathrm{Ch}_{\mathbb{K}},\,\otimes)$} carries a monoidal model category structure with the following properties: 
  \begin{itemize}
    \item the model structure has
    \begin{itemize}
    \item weak equivalences the quasi-isomorphisms (the isomorphisms on chain homology),
    \item fibrations being the degreewise surjections (in particular all objects are fibrant),
    \item cofibrations the degreewise injections (in particular all objects are cofibrant),
    \end{itemize}
    \item the model structure is:
    \begin{itemize}
      \item proper,
      \item combinatorial      
      with sets of generating (acyclic) cofibrations
       $\mathrm{I}_{\mathbb{K}} 
         \,:=\, 
        \{ 
          i_n \;\big\vert\; n \in \mathbb{Z} 
        \}
      $
       ($\mathrm{J}_{\mathbb{K}} \,:=\, \{ j_n \,\big\vert\, n \in \mathbb{Z} \}$) given by:
  \begin{equation}
    \label{GeneratingCofibrationsOfChainComplexes}
    \hspace{-1.2cm}
    \adjustbox{scale=.9}{
    \begin{tikzcd}[sep=10pt]
      \mathbb{S}^{n-1}
      \ar[d, "{ i_n }"]
      \ar[r, phantom, ":\equiv"]
      &
      \big[\;
      \cdots
      \ar[r]
      &
      0 
      \ar[r]
      \ar[d]
      &
      0
      \ar[r]
      \ar[d]
      &
      0
      \ar[r]
      \ar[d]
      &
      \mathbb{K}
      \ar[r]
      \ar[d, "{\mathrm{id}}"]
      &
      0
      \ar[r]
      \ar[d]
      &
      \cdots
      \;\big]
      \\
      \mathbb{D}^n
      \ar[r, phantom, ":\equiv"]
      &
      \big[\;
      \cdots
      \ar[r]
      &
      0 
      \ar[r]
      &
      0
      \ar[r]
      &
      \underset{
        \mathclap{
          \raisebox{-3pt}{
            \scalebox{.7}{
              $\mathrm{deg} = n$
            }
          }
        }
      }{
        \mathbb{K}
      }
      \ar[r, "{ \mathrm{id} }"]
      &
      \mathbb{K}
      \ar[r]
      &
      0
      \ar[r]
      &
      \cdots
      \;\big]
    \end{tikzcd}
    \hspace{.3cm}
    \begin{tikzcd}[sep=10pt]
      0
      \ar[d, "{ j_n }"]
      \ar[r, phantom, ":\equiv"]
      &
      \big[\;
      \cdots
      \ar[r]
      &
      0 
      \ar[r]
      \ar[d]
      &
      0
      \ar[r]
      \ar[d]
      &
      0
      \ar[r]
      \ar[d]
      &
      0
      \ar[r]
      \ar[d]
      &
      0
      \ar[r]
      \ar[d]
      &
      \cdots
      \;\big]
      \\
      \mathbb{D}^n
      \ar[r, phantom, ":\equiv"]
      &
      \big[\;
      \cdots
      \ar[r]
      &
      0 
      \ar[r]
      &
      0
      \ar[r]
      &
      \underset{
        \mathclap{
          \raisebox{-3pt}{
            \scalebox{.7}{
              $\mathrm{deg} = n$
            }
          }
        }
      }{
        \mathbb{K}
      }
      \ar[r, "{\mathrm{id}}"]
      &
      \mathbb{K}
      \ar[r]
      &
      0
      \ar[r]
      &
      \cdots
      \;\big]
    \end{tikzcd}
    }
  \end{equation}
    \end{itemize}
  \end{itemize}

  \item
  \fbox{$(\mathbf{sCh}_{\mathbb{K}},\,\otimes)$} carries a simplicial monoidal model category structure
  with the following properties:
  \begin{itemize}
  \item the model structure has
  \begin{itemize}
    \item weak equivalences the total-quasi-isomorphisms, 
    i.e., the quasi-isos on total chain complexes \eqref{Totalization},
    \item all objects cofibrant;
  \end{itemize}
    \item 
    the adjunction 
    $\mathrm{const} : \mathrm{Ch}_{\mathbb{K}}
    \rightleftarrows \mathrm{sCh}_{\mathbb{K}} : \mathrm{ev}_0$
    is a monoidal Quillen equivalence;
  \item the model structure is:
  \begin{itemize}
    \item left proper,
    \item combinatorial,

    with sets of generating (acyclic) cofibrations to be denoted
    \begin{equation}
      \label{GeneratingCofibrationsOfSimplicialChainComplexes}
      \mathrm{sI}_{\mathbb{K}}
      \,, \quad
      \mathrm{sJ}_{\mathbb{K}}
      \,.
    \end{equation}
  \end{itemize}
  \end{itemize}
  \end{enumerate}
\end{theorem}
\begin{proof}
  We discuss the claims successively:

  \noindent
  {\bf Combinatorial model structure.}
  Generally, for $R$ a commutative ring, the projective model structure on unbounded chain complexes of $R$-modules (i.e.,
  with weak equivalences the quasi-isomorphisms and fibrations the degreewise surjection) exists 
  as a proper and cofibrantly generated model category
  by arguments given in
  \cite[after Thm. 9.3.1]{HoveyPalmieriStrickland97}\cite[p. 41]{Hovey99}\cite[p. 7]{SchwedeShipley00}\cite[Thm. 3.2]{Fausk06}. For the special case where all submodules of free $R$-modules are themselves 
  free (such as for $R = \mathbb{Z}$ but also in our case where $R = \mathbb{K}$ is a field), an alternative proof is spelled out in \cite{Strickland20}.
  
  That the underlying category is locally presentable, hence that this cofibrantly generated  model structure is in fact combinatorial, 
  follows from classical facts:
  A category of $R$-modules is a Grothendieck abelian category (cf. \cite[Thm. 8.11]{Johnstone77}), the category of chain complexes in 
  a Grothendieck abelian category is itself Grothendieck abelian (cf. \cite[p. 3]{Hovey99a}) and every Grothendieck abelian category 
  is locally presentable, by \cite[Prop. 3.10]{Beke00} (cf. \cite[Cor. 5.2]{Krause15}).

\smallskip 
  \noindent
  {\bf Characterization of the cofibrations.}
   Still in the generality of any commutative ground ring $R$,  \cite[Prop. 2.3.9]{Hovey99} shows that the cofibrations in this projective 
   model structure are the degreewise {\it split}-injections {\it with  cofibrant cokernel}, and \cite[Lem 2.3.6]{Hovey99} shows that at least all bounded-below chain complexes of {\it projective} modules are cofibrant.

  Now to specialize this characterization to our case where $R$ is a field $\mathbb{K}$, so that the above $R$-modules become $\mathbb{K}$-vector spaces: 
  Here, by the basis theorem, every module is free and in particular projective, and every short exact sequence splits. So in this case it follows, 
  first, that at least every bounded-below chain complex is cofibrant and that the cofibrations are the degreewise injections with cofibrant cokernel.
  So to see that in this case the cofibrations in fact coincide with all degreewise injections it is now sufficient to see that actually every
  chain complex over a field is projectively cofibrant:  
    Observe that every chain complex $V$ is the colimit over its cotower of $k$-connective covers
  $
    V \,\simeq\,
    {\mathrm{colim}}
    \big(
      \mathrm{cn}_0 V
      \hookrightarrow
      \mathrm{cn}_{-1} V
      \hookrightarrow
      \mathrm{cn}_{-2} V
      \hookrightarrow
      \cdots
    \big)
    \,,
  $
  which, by the remarks just made, is a transfinite composition of cofibrations, so that the cocone $\mathrm{cn}_0 V \hookrightarrow V$ is itself 
  a cofibration. Since we already know that $\mathrm{cn}_0 V$ is cofibrant (being bounded below), it follows that also $V$ is.

\smallskip 
  \noindent
  {\bf Monoidal model structure.} That the tensor product of chain complexes makes the above model category into a monoidal model category is 
  discussed in \cite[Cor. 3.7]{Hovey01}\cite[Thm. 6.1]{Fausk06} and in \cite{Strickland20} (there in the generality where submodules of 
  free $R$-modules are themselves free). Explicitly, notice that for checking the pushout-product axiom \eqref{PushoutProductAxiom} in
  a closed tensor product on a cofibrantly generated model category, it is sufficient to check it on generating (acyclic) cofibrations, 
  which in our case  \eqref{GeneratingCofibrationsOfChainComplexes} is fairly immediate.
  
  Much of the model category structure listed so far, that is not specific to $R$ being a field, is also summarized in \cite[\S 1]{MuroRoitzheim19}.

\smallskip 
  \noindent
  {\bf Simplicial model structure.} Again in the generality of any commutative ground ring, \cite[p. 10]{RezkSchwedeShipley01} show that a 
  simplicial enhancement of the projective model structure on unbounded chain complexes is given by the Reedy model structure on
  $\mathrm{sCh}_{R}$ left Bousfield-localized at the total-quasi-isomorphisms, making the adjoint pair 
  $\mathrm{const} \dashv \mathrm{ev}_0$ a Quillen equivalence.

  For us, it remains to see that over a ground field $R = \mathbb{K}$ all objects in this simplicial model structure are cofibrant. 
  But since left Bousfield localization does not change the class of cofibrations, and since we already saw above that all objects 
  in $\mathrm{Ch}_{\mathbb{K}}$ are cofibrant, it is sufficient to see that: \footnote{This argument was pointed out by Charles 
  Rezk, and we thank Dmitri Pavlov for further discussion. The details may be found spelled out 
  at: \href{ https://ncatlab.org/nlab/show/Reedy+model+structure\#WithValuesInAnAbelianCategory}{\tt ncatlab.org/nlab/show/Reedy+model+structure\#WithValuesInAnAbelianCategory}.}
  
  Every simplicial diagram in a model category with underlying abelian category and all whose objects are cofibrant is itself Reedy cofibrant. 
  This follows by appeal to the Dold-Kan correspondence, which exhibits any such simplicial object degreewise as a direct sum of objects of 
  degenerate and of non-degenerate simplices. Inspection shows that these summands of degenerate simplices are isomorphically the 
  ``latching objects'' appearing in the definition of the Reedy model structure, which implies that a simplicial diagram in the 
  given case is Reedy cofibrant as soon as the degreewise sub-objects of non-degenerate simplices are cofibrant.

\smallskip 
  \noindent
  {\bf Monoidal simplicial model structure.} 
  First, the plain Reedy model structure on simplicial objects in a symmetric monoidal model category is itself monoidal model under the degree-wise tensor product, by
  \cite[Thm. 3.51]{Barwick10}.
  To check that this monoidal model structure is preserved by left Bousfield localization at the total-quasi-isomorphisms, we check the
  sufficient criterion given in \cite[Thm. 3.51]{Barwick10}\footnote{We thank Dmitri Pavlov for pointing out this result.}. Indeed, observing that:
  \begin{itemize}
  \item every object $\mathscr{V} \in \mathrm{sCh}_{\mathbb{K}}$ is a homotopy colimit of simplicially constant objects $\mathrm{const}(V)$ 
  (since these are Reedy cofibrant, by the above, so that $\mathrm{hocolim}_{[k] \in \Delta} \mathrm{const}(\mathscr{V}_k)$ is computed by
  the coend $\int^{[k] \in \Delta} \Delta[k] \cdot \mathrm{const}(\mathscr{V})_k$, which is $\mathscr{V}$),
  \item for a Reedy fibrant object $\mathscr{W}$ to be local in $\mathrm{sCh}_{\mathbb{K}}$ with respect to total-quasi-isomorphisms 
  means \cite{RezkSchwedeShipley01} to be  {\it homotopically constant} in that all the simplicial maps 
  $d_i : \mathscr{W}_{i + 1} \longrightarrow \mathscr{W}_i$ and $s_i : \mathscr{W}_i \longrightarrow \mathscr{W}_{i+1}$ are quasi-isomorphisms,
  \end{itemize}
  this criterion says it is sufficient to check that for $\mathscr{W}$ homotopically constant, also the internal hom 
  $\big[\mathrm{const}(V),\,\mathscr{W}\big]$ \eqref{InternalHomOfSimplicialChainComplexes} 
  is homotopically constant. Now for constant domain, the internal hom reduces (essentially by an incarnation of the Yoneda Lemma) to
  \[
  \hspace{2cm} 
    \def\arraystretch{1.5}
    \begin{array}{ll}
      \mathllap{
      \big[
        \mathrm{const}(V)
        ,\,
        \mathscr{W}
      \big]
      \;:\; [k] \;\longmapsto\;      
      }
      \int_{[s] \in \Delta}
      \Big[
        \big(
          \Delta[k] \cdot 
          \mathrm{const}(V)
        \big)_{s}
        ,\,
        \mathscr{W}_s
      \Big]
   &     \;\simeq\;
      \int_{[s] \in \Delta}
      \big[
          \Delta[k]_s 
          \cdot 
          V
        ,\,
        \mathscr{W}_s
      \big]
      \\
   &   \;\simeq\;
      \int_{[s] \in \Delta}
      \Big[
        V
        ,\,
        (\mathscr{W}_s)^{\Delta[k]_s}
      \Big]
      \\
   &   \;\simeq\;
      \Big[
        V
        ,\,
        \int_{[s] \in \Delta}
        (\mathscr{W}_s)^{\Delta[k]_s}
      \Big]
      \\
   &   \;\simeq\;
      \big[
        V
        ,\,
        \mathscr{W}_k
      \big]
      \,.
    \end{array}
  \]
Therefore, it is sufficient now to observe that $[V,-] \,:\,\mathrm{Ch}_{\mathbb{K}} \longrightarrow \mathrm{Ch}_{\mathbb{K}}$
preserves all quasi-isomorphisms. But this is the case because, by the above discussion, (1.) all objects in $\mathrm{Ch}_{\mathbb{K}}$, 
such as $V$ here, are cofibrant so that $[V,-]$ is a right Quillen functor by the pullback-power axioms satisfied in the monoidal model 
category, and (2.)   all objects, such as the $\mathscr{W}_k$ here, are also fibrant, so that weak equivalences between them are preserved 
by right Quillen functors, according to Ken Brown's lemma \ref{KenBrownLemma}.

  It just remains to observe that the Quillen equivalence $\mathrm{const} \dashv \mathrm{ev}_0$ is a monoidal Quillen adjunction 
  (according to \cite[Def. 4.2.16]{Hovey99}), which is immediate since $\mathrm{const}$ is already a strong monoidal functor and 
  since the tensor unit is cofibrant (like all objects).

\smallskip 
 \noindent
 {\bf Left proper combinatorial simplicial model structure.} Finally, that this model structure on $\mathrm{sCh}_{\mathbb{K}}$ is left proper 
  and combinatorial follows by general results from the above fact that $\mathrm{Ch}_{\mathbb{K}}$ is so:
  \begin{itemize}
    \item[1.] Any Reedy model category with coefficients in a locally presentable model category (with small domains of generating cofibrations) 
    is itself locally presentable, by \cite[Thm. 15.6.27]{Hirschhorn02}.
    \item[2.] Any functor category out of a small category into a locally presentable category is itself locally presentable
    \cite[Cor. 1.54]{AdamekRosicky94}.
    \item[3.]
      Any Reedy model category with coefficients in a left (or right) proper model category is itself left (or right) proper
      \cite[Thm. 15.3.4 (2)]{Hirschhorn02}.
    \item[4.] 
    Any left Bousfield localization of a left proper combinatorial model category is itself
    left proper combinatorial \cite[Thm. 4.7]{Barwick10}. 
    \qedhere
  \end{itemize}
\end{proof}

\medskip

\noindent
{\bf The upshot of Theorem \ref{ModelCategoryOfSimplicialChainComplexes}} is that the simplicial monoidal model category $\mathbf{sCh}_{\mathbb{K}}$ 
is a good model category theoretic enhancement of the coefficient $\infty$-category needed to define flat $\infty$-vector bundles.
Beyond giving a good handle on $\infty$-local systems over fixed base space, we use this below to construct the global theory of flat $\infty$-vector
bundles over {\it varying} parameter spaces (Thm. \ref{GlobalModelStructure} below), on which the $\boxtimes$-tensor product will exist as 
a decently homotopical functor (Thm. \ref{ExternalTensorProductIsHomotopical} below).

\begin{remark}[Fibrant replacement]
  Since every object of $\mathrm{sCh}_{\mathbb{K}}$ is cofibrant, the notion of higher chain homotopy encoded by this model category is all 
  given by Reedy {\it fibrant} replacement of chain complexes $V \,\in\, \mathrm{Ch}_{\mathbb{K}} \xhookrightarrow{\mathrm{const}} \mathrm{sCh}_{\mathbb{K}}$.
\end{remark}

\smallskip 
\noindent
{\bf Linear $\infty$-representations of $\infty$-categories}.
Specifically, we can now invoke the following general constructions:

\begin{definition}[Category of simplicial local systems] 
\label{MonoidalSimplicialFunctorCategoryIntoSimplicialChainComplexes}
Given $\mathbf{X} \,\in\, \mathrm{sSet}\mbox{-}\mathrm{Cat}$ a small simplicial category, we write

\noindent
\fbox{$\Big(\mathbf{sCh}_{\mathbb{K}}^{\mathbf{X}} \,:\defneq\, \mathbf{sFunc}\big(\mathbf{X},\,\mathbf{sCh}_{\mathbb{K}}\big),\, \otimes_{\mathbf{X}}\Big)$} 
for the closed monoidal simplicial category whose \footnote{
  This uses Lem. \ref{EnhancedEnrichmentOfFunctorCategories} for $\mathbf{V} :\defneq \mathrm{sSet}$, $\mathbf{C} :\defneq \mathrm{sCh}_{\mathbb{K}}$ 
  and the given $\mathbf{X}$.
} 
\begin{itemize}
  \item objects are $\mathrm{sSet}$-enriched functors from $\mathbf{X}$ to $\mathbf{sCh}_{\mathbb{K}}$ (Def. \ref{CategoryOfChainComplexes}), 
  to be denoted
  \[
    \begin{tikzcd}[row sep=-3pt, column sep=small]
    \mathllap{
    \mathscr{V}_{\!\mathbf{X}} 
    \;:\;\;
    }
    \mathbf{X}
    \ar[rr]
    &&
    \mathbf{sCh}_{\mathbb{K}}
    \\
    x &\longmapsto& \mathscr{V}_{\!x}
    \end{tikzcd}
  \]
  \item hom-complexes are
  \[
    \mbox{\bf{sCh}}^{
      \scalebox{.6}{\bf{X}}
    }_{\mathbb{K}}
    \big(
     \mathscr{V}_{\!\scalebox{.7}{\bf{X}}}
     ,\,
     \mathscr{W}_{\!\scalebox{.7}{\bf{X}}}
    \big)
    \;\;
    :\defneq
    \;\;
    \int_{x \in \scalebox{.6}{\bf{X}}}
    \mbox{\bf{sCh}}\big(
      \mathscr{V}_{\!x}
      ,\,
      \mathscr{W}_{\!x}
    \big)
    \;\;\;
    \in
    \;
    \mathrm{sSet}
    \,,
  \]
and\footnote{
  Now using Lem. \ref{EnhancedEnrichmentOfFunctorCategories}
  for $\mathbf{V} :\defneq \mathbf{C} :\defneq \mathrm{sCh}_{\mathbb{K}}$ and $\mathbf{X}$ regarded as a 
  $\mathrm{sSet} \xhookrightarrow{\mathbb{K}[-]}\mathrm{sCh}_{\mathbb{K}}$-enriched category.
}

\begin{itemize}
\item 
equipped with the cup-tensor product induced from the tensor product carried by  $\mathbf{sCh}_{\mathbb{K}}$
\begin{equation}
  \label{CupTensorProductOnsChValuedSimplicialFunctors}
  \begin{tikzcd}[row sep=-3pt, column sep=small]
    \mathbf{sCh}_{\mathbb{K}}^{\mathbf{X}}
    \,\times\,
    \mathbf{sCh}_{\mathbb{K}}^{\mathbf{X}}
    \ar[
      rr,
      "{ \otimes_{\mathbf{X}} }"
    ]
    &&
    \mathrm{sCh}_{\mathbb{K}}^{\mathbf{X}}
    \\
    \big(
      \mathscr{V}_{\!\mathbf{X}}
      ,\,
      \mathscr{V}'_{\!\mathbf{X}}
    \big)
    &\longmapsto&
    \mathbf{X}
    \xrightarrow{ \mathrm{diag} }
    \mathbf{X} \times \mathbf{X}
    \xrightarrow{
      \mathscr{V}_{\!\mathbf{X}}
      \times
      \mathscr{V}'_{\!\mathbf{X}}
    }
    \mathbf{sCh}_{\mathbb{K}}
    \times
    \mathbf{sCh}_{\mathbb{K}}
    \xrightarrow{ \otimes }
    \mathbf{sCh}_{\mathbb{K}}    
  \end{tikzcd}
\end{equation}
\item and similarly with the following cup-tensoring
\begin{equation}
  \label{CupTensoringOnsChValuedSimplicialFunctors}
  \begin{tikzcd}[row sep=-3pt, column sep=small]
    \mathbf{sSet}^{\mathbf{X}}
    \,\times\,
    \mathbf{sCh}_{\mathbb{K}}^{\mathbf{X}}
    \ar[
      rr,
      "{ (\mbox{-})\cdot_{{}_{\mathbf{X}}}(\mbox{-}) }"
    ]
    &&
    \mathbf{sCh}_{\mathbb{K}}^{\mathbf{X}}
    \\
    \big(
      S_{\mathbf{X}}
      ,\,
      \mathscr{V}_{\!\mathbf{X}}
    \big)
    &\longmapsto&
    \mathbf{X}
    \xrightarrow{ \mathrm{diag} }
    \mathbf{X} \times \mathbf{X}
    \xrightarrow{
      S_{\mathbf{X}}
      \times
      \mathscr{V}'_{\!\mathbf{X}}
    }
    \mathbf{sSet}
    \times
    \mathbf{sCh}_{\mathbb{K}}
    \xrightarrow{ (\mbox{-})\cdot(\mbox{-}) }
    \mathbf{sCh}_{\mathbb{K}}    
  \end{tikzcd}
\end{equation}
\item whose corresponding internal hom is given by
\begin{equation}
  \begin{tikzcd}[row sep=-3pt, column sep=small]
    \big(
      \mathbf{sCh}_{\mathbb{K}}^{\mathbf{X}}
    \big)^{\mathrm{op}}
    \,\times\,
    \mathbf{sCh}_{\mathbb{K}}^{\mathbf{X}}
    \ar[
      rr,
      "{ [-,-]_{{}_{\mathbf{X}}} }"
    ]
    &&
    \mathrm{sCh}_{\mathbb{K}}^{\mathbf{X}}
    \\
    \big(
      \mathscr{V}_{\!\mathbf{X}}
      ,\,
      \mathscr{W}_{\!\mathbf{X}}
    \big)
    &\longmapsto&
    \Big(
  x 
    \;\mapsto\;
  \int_{x' \in \scalebox{.6}{\bf{X}}}
  \big[
    \mbox{\bf{X}}(x',x)
    \cdot
    \mathscr{V}_{\!x'}
    ,\,
    \mathscr{W}_{\!x'}
  \big]
  \Big)
  \end{tikzcd}
\end{equation}
\end{itemize}
\end{itemize}
\begin{equation}
  \label{InternalHomAdjunctionForSimplicialLocalSystems}
  \mathscr{V}
  \,\in\,
  \mathbf{sCh}_{\mathbb{K}}^{\mathbf{X}}
  \hspace{1.2cm}
   \dashv
  \hspace{1.2cm}
  \begin{tikzcd}
    \mathbf{sCh}_{\mathbb{K}}^{\mathbf{X}}
    \ar[
      rr,
      shift left=5pt,
      "{
        \mathscr{V} 
          \otimes_{\mathbf{X}} 
        (\mbox{-})
      }"
    ]
    \ar[
      from=rr,
      shift left=5pt,
      "{
        [
          \mathscr{V}
          ,\,
          \mbox{-}
        ]_{\mathbf{X}}
      }"
    ]
    \ar[
      rr,
      phantom,
      "{ \scalebox{.7}{$\bot$} }"
    ]
    &&
    \mathbf{sCh}_{\mathbb{K}}^{\mathbf{X}}\;.
  \end{tikzcd}
\end{equation}
\end{definition}

\begin{proposition}[Model category of simplicial local systems over a fixed base space]
\label{ModelStructureOnSimplicialFunctors} 
$\,$ \newline
  For $\mathbf{X} \,\in\, \mathrm{sSet}\mbox{-}\mathrm{Cat}$ a small simplicial category, the 
  monoidal simplicial functor category (Def. \ref{MonoidalSimplicialFunctorCategoryIntoSimplicialChainComplexes}) 
  
  \noindent
    \fbox{$\mathbf{sCh}^{\mathbf{X}}_{\mathbb{K}}$} carries a model category structure with the following properties:
  \begin{itemize}
    \item[{\bf (i)}] the model structure has
    \begin{itemize}
    \item weak equivalences the $\mathbf{X}$-objectwise weak equivalences in $\mathrm{sCh}_{\mathbb{K}}$,
    \item fibrations the $\mathbf{X}$-objectwise
     fibrations in $\mathrm{sCh}_{\mathbb{K}}$
     (both according to Prop. \ref{ModelCategoryOfSimplicialChainComplexes});
    \end{itemize}
    \item[{\bf (ii)}] 
    the model structure is 
    \begin{itemize}
      \item combinatorial
      
      with sets of generating (acyclic) cofibrations
      those of \eqref{GeneratingCofibrationsOfSimplicialChainComplexes} tensored to representables:
      \begin{equation}
        \label{GeneratingCofibrationsOfSimplicialFunctors}
        \mathrm{sI}_{\mathbb{K}}^{\mathbf{X}} 
        := 
        \big\{
          \mathbf{X}(x,-)
          \cdot
          i
          \;\big\vert\;
          x \in \mathbf{X},\;
          i \in \mathrm{sI}_{\mathbb{K}}
        \big\}
        ,\;
        \hspace{1cm}
        \mathrm{sJ}_{\mathbb{K}}^{\mathbf{X}} 
        := 
        \big\{
          \mathbf{X}(x,-)
          \cdot
          j
          \;\big\vert\;
          x \in \mathbf{X},\;
          j \in \mathrm{sJ}_{\mathbb{K}}
        \big\}
        \,.
      \end{equation}

    \end{itemize}
  \end{itemize}
\end{proposition}
\begin{proof}
  Since $\mathbf{sCh}_{\mathbb{K}}$ is simplicial combinatorial, by Prop. \ref{ModelCategoryOfSimplicialChainComplexes}, this is the 
  existence statement of the projective model structure on enriched functors, see for instance \cite[Prop. A.3.3.2]{Lurie09}.
\end{proof}

\begin{remark}[Base change between model structures of simplicial local systems]
  \label{BaseChangeBetweenModelStructuresOfSimplicialLocalSystems}
  For $\mathbf{f} \,:\, \mathbf{X} \longrightarrow \mathbf{X}'$ a morphism of simplicial groupoids,
  the induced pair of adjoint functors
  \eqref{KanExtensionAdjointTriple}
  \[
    \begin{tikzcd}
      \mathbf{sCh}^{\mathbf{X}}_{\mathbb{K}}
      \ar[
        rr,
        shift left=6pt,
        "{ \mathbf{f}_! }"
      ]
      \ar[
        from=rr,
        shift left=6pt,
        "{ \mathbf{f}^\ast }"
      ]
      \ar[
        rr,
        phantom,
        "{ \scalebox{.7}{$\bot$} }"
      ]
      &&
      \mathbf{sCh}^{\mathbf{X}'}_{\mathbb{K}}
    \end{tikzcd}
  \]
  is a Quillen adjunction with respect to the model structure from Prop. \ref{ModelStructureOnSimplicialFunctors} 
(since the right adjoint $\mathbf{f}^\ast$, 
  which acts by precomposition, clearly preserves the objectwise defined weak equivalences and fibrations).
\end{remark}

\begin{remark}[Linear $\infty$-category-representations]
  In the following, we focus on the specialization of Prop. \ref{ModelStructureOnSimplicialFunctors} to domains $\mathcal{X}$ which 
  model $\infty$-groupoids. But it is clear that Prop. \ref{ModelStructureOnSimplicialFunctors} is relevant more generally. 

  For example, already in one of the simplest examples of a small (simplicial) category $\mathcal{X}$ which is not a (simplicial) 
  groupoid, namely the category $\mathrm{FinSet}_{\mathrm{inj}}$ of finite sets and {\it injective} maps between them, Prop. \ref{ModelStructureOnSimplicialFunctors} provides the homotopy-coherent enhancement of the notion of {\it FI-modules} (known to have deep relation to braid group representations and hence to topological quantum computation, see eg. \cite{Wilson23} for pointers). 
  A dedicated discussion of such homotopical FI-modules has recently appeared in \cite{Arro23}.
\end{remark}

\medskip

\noindent
{\bf $\infty$-Group representations.}
With this in hand and by the fact that every $\infty$-group is presented by a simplicial group, we immediately obtain 
a model category of ``$\mathbb{K}$-linear $\infty$-representations of $\infty$-groups'', identified with $\infty$-local systems over a delooping $\infty$-groupoid:

\begin{definition}[Simplicial delooping groupoids]
 \label{SimplicialDeloopingGroupoid}
  For a simplicial group
  $\mathcal{G} \,\in\, \mathrm{Grp}(\mathrm{sSet})$ we write 
  \[
    \mathbf{B}\mathcal{G}
    \,\in\,
    \mathrm{sSet}\mbox{-}\mathrm{Grpd}
    \longhookrightarrow
    \mathrm{sSet}\mbox{-}\mathrm{Cat}
  \]
  for the simplicial groupoid which has
  \begin{itemize}
    \item a single object,
    \item single hom-object identified with the simplicial group,
    whose composition operation and identity element is given by the group operation and the neutral element on $\mathcal{G}$.
  \end{itemize}
\end{definition}

\begin{remark}[Simplicial functors between $\mathbf{B}\mathcal{G}$s are simplicial group homomorphisms]
\label{SimplicialFunctorsBetweenDeloopingGroupoids}
$\,$ \newline  
  The $\mathrm{sSet}$-enriched functors between simplicial delooping groupoids (Def. \ref{SimplicialDeloopingGroupoid}) 
  correspond naturally to homomorphisms between the corresponding simplicial groups
  \[
    \mathrm{sSet}\mbox{-}\mathrm{Grpd}
    \big(
      \mathbf{B}\mathcal{G}
      ,\,
      \mathbf{B}\mathcal{G}'
    \big)
    \;\;
    \simeq
    \;\;
    \mathrm{sGrp}(\mathcal{G},\,\mathcal{G}')
    \,.
  \]
\end{remark}

\begin{remark}[Simplicial local systems on $\mathbf{B}\mathcal{G}$ are simplicial group representations]
\label{SimplicialGroupActions}
$\,$
\newline
By the (co)tensoring adjunctions, a simplicial functor $\mathscr{V}_{\mathbf{B}\mathcal{G}} \,\in\, \mathbf{sCh}^{\mathbf{B}\mathcal{G}}_{\mathbb{K}}$ 
on the simplicial delooping groupoid $\mathbf{B}\mathcal{G}$ (Def. \ref{SimplicialDeloopingGroupoid})
is equivalently a simplicial $\mathcal{G}$-action $\rho_{\mathscr{V}}$ on some object $\mathscr{V} \,\in\, \mathbf{sCh}_{\mathbb{K}}$:
\begin{equation}
  \label{SimplicialGroupActionAdjointness}
  \begin{tikzcd}[sep=0pt]
  \mbox{\bf{sCh}}_{\mathbb{K}}^{\mathbf{B}\mathcal{G}}
  \ar[rr, "{\sim}"]
  &&
  \mathcal{G}\mathrm{Act}
  \big(
    \mathbf{sCh}_{\mathbb{K}}
  \big)
  \\
  \mathscr{V}_{\mathbf{B}\mathcal{G}}
  &\mapsto&
  \big(
    \mathscr{V}
    ,\,
    \rho_{\mathscr{V}}
  \big)
  \end{tikzcd}
  \hspace{2cm}
  \begin{tikzcd}[row sep=3pt]
    \mathcal{G} 
    \!\cdot\!
    \mathscr{V}
    \ar[
      rr,
      "{
        \scalebox{.8}{$
          \rho_{\mathscr{V}}
        $}
      }"{description}
    ]
    &&
    \mathscr{V}    
    \\
    \hline
    \mathllap{
      \mathscr{V}_{\mathbf{B}\mathcal{G}}
      \;:\;\;
    }
    \mathcal{G}
    \ar[
      rr
    ]
    &&
    \mbox{\bf{sCh}}_{\mathbb{K}}\big(
      \mathscr{V}
      ,\,
      \mathscr{V}
    \big)
    \\
    \hline
    \mathscr{V}
    \ar[
      rr,
      "{
        \scalebox{.8}{$
          \tilde\rho_{\mathscr{V}}
        $}
      }"{description}
    ]
    &&
    \mathscr{V}^{\mathcal{G}}
  \end{tikzcd}
\end{equation}
\end{remark}

\medskip

\begin{proposition}
[Closed monoidal structure of simplicial local systems over simplicial delooping groupoids]
\label{ClosedMonoidalStructureOfLocalSystemsOverSimplicialDelooping}
Over a simplicial delooping groupoid, the monoidal structure from Def. \ref{MonoidalSimplicialFunctorCategoryIntoSimplicialChainComplexes} has 

\begin{itemize}
\item [{\bf (i)}] 
tensor product given by
\[
  \hspace{1.4cm}
  \adjustbox{scale=.9}{
  \begin{tikzcd}[row sep=-3pt, column sep=0pt]
    \mathllap{
    \mathscr{V}_{\mathbf{B}\mathcal{G}}
    \otimes
    \mathscr{W}_{\mathbf{B}\mathcal{G}}
    \;:\;\;
    }
    \mathcal{G}
    \ar[rr]
    &&
    \mbox{\bf{sCh}}^{\scalebox{.7}{\bf{B}}\mathcal{G}}_{\mathbb{K}}
    \big(
      \mathscr{V}\otimes\mathscr{W}
      \;,\;
      \mathscr{V}\otimes\mathscr{W}
    \big)
    \\
    \big(
      \Delta[k]
      \xrightarrow{g}
      \mathcal{G}
    \big)
    &\mapsto&
    \bigg(
    \Delta[k]
    \cdot
    \mathscr{V}\otimes\mathscr{W}
    \xrightarrow{
      \scalebox{.7}{$
        \mathrm{diag}
        \cdot
        \mathscr{V} 
          \otimes
        \mathscr{W}
      $}
    }
    \big(
      \Delta[k] \cdot\mathscr{V}
    \big)
      \otimes
    \big(
      \Delta[k] \cdot \mathscr{W}    
    \big)
    \xrightarrow{
      \scalebox{.7}{$
      \tilde\rho_{\mathscr{V}}
      \otimes
      \tilde\rho_{\mathscr{W}}
      $}
    }
    (\mathscr{G}\cdot\mathscr{V})
      \otimes
    (\mathscr{G}\cdot\mathscr{W})    
    \xrightarrow{
      g\ \otimes g
    }
    \mathscr{V} 
      \otimes 
    \mathscr{W}
    \bigg)
  \end{tikzcd}
  }
\]

\item[{\bf (ii)}] and internal hom given by
  \[
    \hspace{.8cm}
    \adjustbox{scale=.9}{
    \begin{tikzcd}[row sep=-3pt, column sep=0pt]
      \mathllap{
      \big[
        \mathscr{V}_{\mathbf{B}\mathcal{G}}
        ,\,
        \mathscr{W}_{\mathbf{B}\mathcal{G}}
      \big]
      \;:\;\;
      }
      \mathcal{G}
      \ar[rr]
      &&
      \mbox{\bf{sCh}}_{\mathbb{K}}
      \big(
      [
        \mathscr{V}
        ,\,
        \mathscr{W}
      ]
      ,\,
      [
        \mathscr{V}
        ,\,
        \mathscr{W}
      ]
      \big)
      \\
      \big(
        \Delta[k] 
          \xrightarrow{g} 
        \mathcal{G}
      \big)
      &\longmapsto&
      \Big(
        [\mathscr{V},\,\mathscr{W}]
        \xrightarrow{
          \scalebox{.7}{$
          \big[
            \rho_{\mathscr{V}}
            ,\,
            \tilde\rho_{\mathscr{W}}
          \big]
        $}
        }
        \big[
          \mathcal{G}
          \!\cdot\!
          \mathscr{V}
          ,\,
          \mathscr{V}^{\mathcal{G}}
        \big]
        \xrightarrow{\scalebox{.7}{$
          \big[
            g^{-1} \!\cdot\! \mathscr{V}
            ,\,
            \mathscr{V}^g
          \big]
        $}
        }
        \big[
          \Delta[k]
          \!\cdot\!
          \mathscr{V}
          ,\,
          \mathscr{W}^{\Delta[k]}
        \big]
        \xrightarrow{
          \scalebox{.73}{$
          [\mathscr{V},\mathscr{W}]^{\mathrm{diag}}
          $}
        }
        [\mathscr{V},\,\mathscr{W}]^{
          \Delta[k]
        }
      \Big).
    \end{tikzcd}
    }
  \]
\end{itemize}
\end{proposition}
\begin{proof}
  The first statement follows readily by unwinding the definitions. This makes the adjunction property of the second formula essentially manifest.
\end{proof}

\medskip

\noindent
{\bf Skeletal simplicial groupoids.}
Intermediate between general $\mathrm{sSet}$-enriched groupoids in Prop. \ref{ModelStructureOnSimplicialFunctors} 
and simplicial delooping groupoids in Rem. \ref{SimplicialGroupActions} are skeletal simplicial groupoids.

\begin{definition}[Skeletal simplicial groupoids]
  \label{SkeletalSimplicialGroupoids} $\,$

\begin{itemize}
  \item[{\bf (i)}] 
  A simplicial groupoid $\mathbf{X} \,\in\, \mathrm{sSet}\mbox{-}\mathrm{Grpd}$ is {\it skeletal} if its only non-empty 
  hom-complexes are those from a given object to itself.
  In other words, if and only if it is isomorphic to a disjoint union of simplicial delooping groupoids (Def. \ref{SimplicialDeloopingGroupoid}):
  \begin{equation}
   \label{ConditionForSkeletalSimplicialGroupoid}
   \mathbf{X}
   \,\in\,
   \mathrm{sSet}\mbox{-}\mathrm{Grpd}
   \hspace{1cm} 
   \vdash
   \hspace{1cm}
   \mbox{$\mathbf{X}$ is skeletal}
   \;\;\;\;\;
     \Leftrightarrow
   \;\;\;\;\;
   \mathbf{X}
   \;\;
     \underset{\mathrm{iso}}{\simeq}
   \;\;
   \underset{[x] \in \pi_0(\mathbf{X})}{\coprod}
   \mathbf{B}\big(
     \mathbf{X}(x,x)
   \big)
   \,.
  \end{equation}

\item[{\bf (ii)}] 
We denote the full subcategory of skeletal simplicial groupoids by
  \begin{equation}
    \label{SubcategoryOfSkeletalSimplicialGroupoids}
    \mathrm{sSet}\mbox{-}\mathrm{Grpd}_{\mathrm{skl}}
   \, \longhookrightarrow \, 
    \mathrm{sSet}\mbox{-}\mathrm{Grpd}
    \,.
  \end{equation}
\end{itemize}
\end{definition}
Given any $\mathbf{X} \,\in\, \mathrm{sSet}\mbox{-}\mathrm{Grpd}$ we say that a {\it skeleton of $\mathbf{X}$} is a full inclusion (hence a Dwyer-Kan equivalence) out of a skeletal groupoid \eqref{ConditionForSkeletalSimplicialGroupoid}
\begin{equation}
  \label{SkeletonForSimplicialGroupoid}
  \begin{tikzcd}
  \mathbf{X}_{\mathrm{skl}}
  \ar[
    r, 
    hook, 
    "{ 
      \in \mathrm{W}_{\mathrm{DK}} 
     }"{swap}
  ]
  &
  \mathbf{X}\;.
  \end{tikzcd}
\end{equation}

\begin{lemma}[Skeletal implies fibrant]
  \label{SkeletalSimplicialGroupoidsAreFibrant}
  Every skeletal simplicial groupoid (Def. \ref{SkeletalSimplicialGroupoids}) is fibrant.
\end{lemma}
\begin{proof}
  Unwinding the definitions, this is tantamount to saying that underlying any simplicial group is a Kan-fibrant simplicial set. This is the case by Moore's theorem \cite[Thm. 3, p. 18-04]{Moore54}\cite[\S II 3.8]{Quillen67}.
\end{proof}

\medskip

The following Lemmas \ref{SkeletizationOfSimplicialGroupoids},  \ref{SkeletalizationOfLocalSystems}, which are standard arguments, use the Axiom of Choice in the underlying category of sets, which we assume throughout, as usual.
\begin{lemma}[Skeletization of simplicial groupoids]
 \label{SkeletizationOfSimplicialGroupoids}
  Every $\mathbf{X} \in\, \mathrm{sSet}\mbox{-}\mathrm{Grpd}$ admits an adjoint equivalence and deformation retraction onto a skeleton \eqref{SkeletonForSimplicialGroupoid}.
\end{lemma}
\begin{proof}
  If $\mathbf{X}$ is empty then the statement is trivial. So assume $\mathbf{X}$ is inhabited, whence it is the disjoint union of its inhabited connected components
  \[
    \mathbf{X}
    \;\;
    \underset{\mathrm{iso}}{\simeq}
    \;\;
    \underset{i \in \pi_0(\mathbf{X})}{\coprod}
    \mathbf{X}_{i}
    \,.
  \]
  Choosing a base-point in each component
  \[
    i \,\in\, \pi_0(\mathbf{X})
    \;\;\;\;
      \vdash
    \;\;\;\;
    x_i \,\in\, \mathrm{Obj}(\mathbf{X}_i)
  \]  
  induces a $\mathrm{sSet}$-enriched full inclusion
  \[
    \iota
    \;:\;
    \begin{tikzcd}
      \underset{i \in \pi_0(\mathbf{X})}{\coprod}
      \mathbf{B}
      \big(
        \mathbf{X}(x_i, x_i)
      \big)
      \ar[r, hook]
      &
      \mathbf{X}
    \end{tikzcd}
  \]  
  and choosing a 1-morphism from each object to this basepoint:
  \[
    x \,\in\, \mathrm{Obj}(\mathbf{X})
    \;\;\;\;\;\;\;
    \vdash
    \;\;\;\;\;\;\;
    \gamma_x 
    \,:\,
    \ast 
      \xrightarrow{\;}
    \mathbf{X}(x_{[x]}, x)
  \]
  (where by $[x] \,\in\, \pi_0(\mathbf{X})$ we denote the connected component of the given object $x$) induces a reverse enriched functor
  \[
    \begin{tikzcd}[row sep=0pt, column sep=large]
      \mathbf{X} 
      \ar[
        rr, 
        "{ \mathbf{p} }"
      ]
      &&
      \underset{i \in \pi_0(\mathbf{X})}{\coprod}
      \mathbf{B}\big(
        \mathbf{X}(x_i, x_i
      \big)
      \\
      \mathbf{X}(x,y)
      \ar[rr, "{ \mathbf{p}_{x,y} }"]
      \ar[d, "{ \sim }"{sloped}]
      &&
      \mathbf{X}\big(x_{[x]},\, x_{[y]}\big)
      \\[+20pt]
      \ast
      \times
      \mathbf{X}(x,y)
      \times
      \ast
      \ar[
        rr,
      ]
      \ar[
        rr,
        "{
          (\gamma_y)^{-1}
          \times
          \mathrm{id}
          \times
          \gamma_x
        }"
      ]
      &&
      \mathbf{X}\big(
        y ,\, x_{[y]}
      \big)
      \times
      \mathbf{X}(x,y)
      \times
      \mathbf{X}\big(
        x_{[x]},\,x
      \big)
      \ar[
        u,
        "{ \circ }"{swap}
      ]
    \end{tikzcd}
  \]
  such that the $\gamma_x$ serve as components of an $\mathrm{sSet}$-enriched natural transformation 
  $
    \gamma 
      \,:\, 
    \iota \circ \mathbf{p} 
    \xrightarrow{\;} 
    \mathrm{id}_{\mathbf{X}}
      $.
  
  Similarly, the inverse components define a converse transformation, but if we choose, as we may, 
  $\gamma_{\,[x]} \,=\, \mathrm{id}_{[x]}$, then there is already an equality $\mathbf{p} \circ \iota \,=\, \mathrm{id}$. 
  This means that we have a {\it deformation retraction} of $\mathbf{X}$ into its skeleton
  \begin{equation}
    \label{SkeletaAreRetractions}
    \begin{tikzcd}
      \mathbf{X}_{\mathrm{sk}}
      \ar[r, "{ \iota }"]
      \ar[
        rr,
        rounded corners,
        to path={
             ([yshift=+0pt]\tikztostart.north)
         --  ([yshift=+9pt]\tikztostart.north)
         --  node[yshift=-5pt]{\scalebox{.7}{$\mathrm{id}$}}
             ([yshift=+9pt]\tikztotarget.north)
         --  ([yshift=+0pt]\tikztotarget.north)
        }
      ]
      &
      \mathbf{X}
      \ar[r, "{ \mathbf{p} }"]
      &
      \mathbf{X}_{\mathrm{skl}}
      \mathrlap{\,,}
    \end{tikzcd}
    \hspace{1cm}
    \begin{tikzcd}
      \mathbf{X}
      \ar[r, "{ \mathbf{p} }"]
      \ar[
        rr,
        bend right=45,
        "{
          \mathrm{id}
        }"{swap},
        "{\ }"{name=t}
      ]
      &
      \mathbf{X}_{\mathrm{skl}}
      \ar[r, "{ \iota }"]
      \ar[
        to=t,
        Rightarrow,
        "{ \gamma }"
      ]
      &
      \mathbf{X}
    \end{tikzcd}
  \end{equation}
  manifestly satisfying
  \[
    \mathbf{p}(\gamma_{x}) = \mathrm{id}_x
    \;\;\;
    \mbox{and}
    \;\;\;
    \gamma_{\mathbf{p}(x)} = \mathrm{id}_{\mathbf{p}(x)}
  \]
  and thus exhibiting $\iota$ as the left adjoint in an enriched adjoint equivalence.
\end{proof}

\begin{lemma}[Bifibrant resolution by skeletal simplicial groupoids]
  \label{BifibrantResolutionBySkeletalSimplicialGroupoids}
  Every $\mathbf{X} \,\in\, \mathrm{sSet}\mbox{-}\mathrm{Grpd}$ (Prop. \ref{DwyerKanModelStructures}) admits 
  a bifibrant replacement by a skeletal simplicial groupoid (Def. \ref{SkeletalSimplicialGroupoids}).
\end{lemma}
\begin{proof}
  By the existence of the model structure, all objects of $\mathrm{sSet}\mbox{-}\mathrm{Grpd}$ admit some cofibrant resolution, 
  and by Lem. \ref{SkeletizationOfSimplicialGroupoids} this in turn admits a deformation retraction along a weak equivalence onto 
  a skeletal object. The latter is still cofibrant since cofibrations are closed under retractions, and it is fibrant by 
  Lem. \ref{SkeletalSimplicialGroupoidsAreFibrant}. 
\end{proof}

\begin{lemma}[Skeletalization of simplicial local systems]
  \label{SkeletalizationOfLocalSystems}
  $\,$ \newline
  Given $\mathbf{X} \,\in\, \mathrm{sSet}\mbox{-}\mathrm{Grpd}$, with skeleton $\mathbf{X}_{\mathrm{skl}} \,\in\, \mathrm{sSet}\mbox{-}\mathrm{Grpd}_{\mathrm{skl}}$, 
  $\begin{tikzcd}\!\!\!\mathbf{X}_{\mathrm{skl}} \! \ar[r, hook, "{\iota}", "{ \in  \mathrm{W}_{\mathrm{DK}}}"{swap}] & \!\mathbf{X}\!\!\end{tikzcd}$ \eqref{SkeletonForSimplicialGroupoid}:
  
  \noindent {\bf (i)} We have an adjoint equivalence of categories of local systems
  \begin{equation}
    \label{SimplicialLocalSystemsEquivalentToRestrictionToSkeleton}
    \begin{tikzcd}[column sep=large]
      \mathbf{sCh}^{\mathbf{X}_{\mathrm{skl}}}_{\mathbb{K}}
      \ar[r, shift left=7pt, "{ \iota_! }"]
      \ar[from=r, shift left=7pt, "{ \iota^\ast }"]
      \ar[r, phantom, "{ \scalebox{.7}{$\bot_{\mathrlap{\simeq}}$} }"]
      &
      \mathbf{sCh}^{\mathbf{X}}_{\mathbb{K}}
      \,,
    \end{tikzcd}
  \end{equation}
  which is a Quillen equivalence with respect to the projective model structures (from Prop. \ref{ModelStructureOnSimplicialFunctors}).
  
 \noindent {\bf (ii)} Moreover, given a morphism $\mathbf{f} : \mathbf{X}' \xrightarrow{\;} \mathbf{X}$ 
 then these equivalences may be chosen such as to make a square of adjunctions commute
  \begin{equation}
    \label{SkeletizationOfLocalSystemsAlongAMap}
    \begin{tikzcd}[column sep=large]
      \mathbf{sCh}_{\mathbb{K}}^{\mathbf{X}'_{\mathrm{skl}}}
      \ar[
        rr,
        shift left=7pt,
        "{
          \mathbf{f}_!
        }"
      ]
      \ar[
        from=rr,
        shift left=7pt,
        "{
          \mathbf{f}^\ast
        }"
      ]
      \ar[rr, phantom, "{ \scalebox{.7}{$\bot$} }"]
      \ar[
        dd,
        shift left=7pt,
        "{ (\iota')_! }"
      ]
      \ar[
        from=dd,
        shift left=7pt,
        "{ (\iota')^\ast }"
      ]
      \ar[
        dd,
        phantom,
        "{ 
          \scalebox{.7}{$\bot_{\mathrlap{\simeq}}$} 
         }"{rotate=-90}
      ]
      &&
      \mathbf{sCh}_{\mathbb{K}}^{\mathbf{X}_{\mathrm{skl}}}
      \ar[
        dd,
        shift left=7pt,
        "{ \iota_! }"
      ]
      \ar[
        from=dd,
        shift left=7pt,
        "{ \iota^\ast }"
      ]
      \ar[
        dd,
        phantom,
        "{ 
          \scalebox{.7}{$\bot_{\mathrlap{\simeq}}$} 
         }"{rotate=-90}
      ]
      \\
      \\
      \mathbf{sCh}_{\mathbb{K}}^{\mathbf{X}'}
      \ar[
        rr,
        shift left=7pt,
        "{
          \mathbf{f}_!
        }"
      ]
      \ar[
        from=rr,
        shift left=7pt,
        "{
          \mathbf{f}^\ast
        }"
      ]
      \ar[
        rr, 
        phantom, 
        "{ \scalebox{.7}{$\bot$} }"
      ]
      &&
      \mathbf{sCh}_{\mathbb{K}}^{\mathbf{X}}
    \end{tikzcd}
  \end{equation}
\end{lemma}
\begin{proof}  
  The adjoint equivalence from the proof of Lem. \ref{SkeletizationOfSimplicialGroupoids} induces the claimed adjoint equivalence on local systems with $\iota_! \,=\, \mathbf{p}^\ast$
  \[
  \begin{tikzcd}
    \big(\mathbf{p}^\ast \iota^\ast \mathscr{V}\big)_x
    \;\simeq\;
    \mathscr{V}_{\iota\circ \mathbf{p} (x)}
    \ar[
      r, 
      "\mathscr{V}_{\gamma}"
    ]
    &
     \mathscr{V}_{x}
     \,.
  \end{tikzcd}
  \]
  Moreover, with $\iota_! \dashv \iota^\ast$ being an adjoint equivalence so is $\iota^\ast \dashv \iota_!$ and since the projective model structure on $\mathbf{sCh}_{\mathbb{K}}^{\mathbf{X}_{\mathrm{skl}}}$ is evidently right-transferred along $\iota_! = p^\ast$ from that of $\mathbf{sCh}_{\mathbb{K}}^{\mathbf{X}}$ it follows (Prop. \ref{ModelStructureTransferAlongAdjointEquivalence}) that $(\iota^\ast \dashv p^\ast)$ is a Quillen equivalence, whence also $(p^\ast \dashv \iota^\ast) \,\simeq\,(\iota_! \dashv \iota^\ast)$ is a Quillen equivalence:
\begin{equation}
  \label{QuillenEquivalencesRestrictionExtensioToSkeleton}
  \begin{tikzcd}
    \mathbf{sCh}_{\mathbb{K}}^{\mathbf{X}}
    \ar[
      rr,
      shift right=6pt,
      "{ \iota^\ast }"{swap}
    ]
    \ar[
      from=rr,
      shift right=6pt,
      "{ p^\ast }"{swap}
    ]
    \ar[
      rr,
      phantom,
      "{ \scalebox{.7}{$\simeq_{\mathrlap{\mathrm{Qu}}}$} }"
    ]
    &&
    \mathbf{sCh}_{\mathbb{K}}^{\mathbf{X}_{\mathrm{sk}}}
    \ar[
      rr,
      shift right=6pt,
      "{ p^\ast }"{swap}
    ]
    \ar[
      from=rr,
      shift right=6pt,
      "{ \iota^\ast }"{swap}
    ]
    \ar[
      rr,
      phantom,
      "{ \scalebox{.7}{$\simeq_{\mathrlap{\mathrm{Qu}}}$} }"
    ]
    &&
    \mathbf{sCh}_{\mathbb{K}}^{\mathbf{X}}\;.
  \end{tikzcd}
\end{equation}
While this construction is far from natural, due to the choices of $x_{i}$ involved, these choices can be made consistently with respect to a single map $\mathbf{f} : \mathbf{X}' \to \mathbf{X}$, by choosing $x'_{i'} \,\in\, \mathbf{f}^{-1}\big(\{x_i\}\big)$ for $i' \in [\mathbf{f}]^{-1}\big(\{i\}\big)$, thereby producing commuting diagrams of this form:
\[
  \begin{tikzcd}
    \mathbf{B}
    \big(
      \mathbf{X}'(x'_{i'}, x'_{i'})
    \big)
    \ar[
      rr,
      "{
        \mathbf{B}\big(
          \mathbf{f}_{x'_{i'}, x'_{i'}}
        \big)
      }"
    ]
    \ar[
      d, 
      hook, 
      "{ \iota' }"
    ]
    &&
    \mathbf{B}
    \big(
      \mathbf{X}(x_{i}, x_{i})
    \big)
    \ar[d, hook, "{ \iota }"]
    \\
    \mathbf{X}'
    \ar[rr, "{ \mathbf{f} }"]
    &&
    \mathbf{X}
    \mathrlap{\,.}
  \end{tikzcd}
\]
This is enough to obtain the diagram \eqref{SkeletizationOfLocalSystemsAlongAMap} commuting up to enriched natural isomorphism. But furthermore we may choose $\gamma'_{x'} \,\in\, \mathbf{f}^{-1}\big(\{\gamma_x\}\big)$ for $x' \in f^{-1}\big(\{x\}\big)$, which makes the diagram commute strictly.
\end{proof}

\medskip

\noindent
{\bf Simplicial local systems over skeletal simplicial groupoids.}
\begin{remark}[Simplicial local systems on skeletal groupoids]  
\label{SimplicialLocalySystemsOnSkeletalGroupoids}
  Over a skeletal simplicial groupoid (Def. \ref{SkeletalSimplicialGroupoids}), the model category
  of simplicial local systems (Prop. \ref{ModelStructureOnSimplicialFunctors}) is the product model structure on 
  the product of categories of simplicial local systems on the connected components:
  $$
    \mathbf{X}
    \;\simeq\; 
    \underset{s \in S}{\coprod} 
    \,
    \mathbf{B}\mathcal{G}_s
    \hspace{1cm}
      \vdash
    \hspace{1cm}
    \mathbf{sCh}^{\mathbf{X}}_{\mathbb{K}}
    \;\simeq\;
    \underset{s \in S}{\prod}
    \,
    \mathbf{sCh}^{\mathbf{B}\mathcal{G}_s}_{\mathbb{K}}    
    \;\;\;
    \in
    \;
    \mathrm{ModCat}.
  $$
  This is immediate from the fact that $\mathbf{Func}\big(\coprod_s \mathbf{D}_s ,\, \mathbf{C}\big) \,\simeq\, \prod_s \mathbf{Func}(\mathbf{D}_s, \mathbf{C})$ 
  and since the weak equivalences and fibrations in the projective model structure on functors are defined objectwise.
\end{remark}
\begin{proposition}[Monoidal model structure on simplicial local systems] 
\label{MonoidalModelStructureOnLinearSimplicialGroupRepresentations}
  For $\mathbf{X} \,\in\, \mathrm{sSet}\mbox{-}\mathrm{Grpd}$,
the simplicial model structure 
  \fbox{$\big(\mathbf{sCh}^{\mathbf{X}}_{\mathbb{K}}, \otimes\big)$} from Prop. \ref{ModelStructureOnSimplicialFunctors} is monoidal model
  (with respect to the monoidal structure from Prop. \ref{MonoidalSimplicialFunctorCategoryIntoSimplicialChainComplexes}).
\end{proposition}
\begin{proof}
By the equivalences of categories \eqref{SimplicialLocalSystemsEquivalentToRestrictionToSkeleton}
it is sufficient to show this for 
skeletal $\mathbf{X}$ (Def. \ref{SkeletalSimplicialGroupoids}).
 Moreover, by Rem. \ref{SimplicialLocalySystemsOnSkeletalGroupoids} and since the tensor product is defined objectwise, the simplicial local systems
  over skeletal $\mathbf{X}$ form a product model category equipped factorwise with the closed monoidal structure
  from Prop. \ref{ClosedMonoidalStructureOfLocalSystemsOverSimplicialDelooping}:
  \[
    \big(
    \mathbf{sCh}^{\mathbf{X}}_{\mathbb{K}}
    ,\,
    \otimes_{\mathbf{X}}
    \big)
    \;\;\simeq\;\;
    \underset{i \in I}{\prod}
    \big(
    \mathbf{sCh}^{\mathbf{B}\mathcal{G}_i}_{\mathbb{K}}
    ,\,
    \otimes_{\mathbf{B}\mathcal{G}_i}
    \big)    
    \,.
  \]
  Therefore, it is in fact sufficient to check that a model category of simplicial local systems over a simplicial delooping groupoid
  $
    \big(
    \mathbf{sCh}^{\mathbf{B}\mathcal{G}_i}_{\mathbb{K}}
    ,\,
    \otimes_{\mathbf{B}\mathcal{G}_i}
    \big)    
  $
  is monoidal as a model category.

  This follows by \cite[p. 6]{BergerMoerdijk06}. For the record, we spell out the argument.
    The point is that over a delooping groupoid $\mathbf{B}\mathcal{G}$ the generating (acyclic) cofibrations \eqref{GeneratingCofibrationsOfSimplicialChainComplexes} 
  are given by tensoring a generating (acyclic) cofibration of $\mathbf{sCh}_{\mathbb{K}}$ with $\mathcal{G}$ equipped with its own multiplication action:
  \begin{equation}
    \label{GeneratingCofibrationForSimplicialGroupRepresentations}
    \mathrm{sI}^{\mathbf{B}\mathcal{G}}_{\mathbb{K}}
    \,=\,
    \big\{
     \mathcal{G} \cdot i
     \,\big\vert\,
     i \,\in\, \mathrm{sI}_{\mathbb{K}}
    \big\}
    ,\,
    \;\;\;\;\;
    \mathrm{sJ}^{\mathbf{B}\mathcal{G}}_{\mathbb{K}}
    \,=\,
    \big\{
     \mathcal{G} \cdot j
     \,\big\vert\,
     j \,\in\, \mathrm{sJ}_{\mathbb{K}}
    \big\}
    \,.
  \end{equation}
  This happens to coincide with the free construction (forming simplicial ``regular representations'') which is 
left adjoint to the functor $\mathrm{undrl}$ that forgets the $\mathcal{G}$-action \eqref{SimplicialGroupActionAdjointness}:
  \begin{equation}
    \label{FreeSimplicialGroupAction}
    \begin{tikzcd}[column sep=large]
    \mathbf{sCh}_{\mathbb{K}}^{\mathbf{B}\mathcal{G}}
    \,\simeq\,
    \mathcal{G}\mathrm{Act}\big(
      \mathbf{sCh}_{\mathbb{K}}
    \big)
    \ar[
      from=rr,
      shift right=6pt,
      "{
        \mathcal{G}\cdot \mathscr{V}
        \,\mapsfrom\,
        \mathscr{V}
      }"{swap}
    ]
    \ar[
      rr,
      shift right=7pt,
      "{
         \mathrm{undrl}
      }"{description}
    ]
    \ar[
      rr,
      phantom,
      "{\scalebox{.7}{$\bot$}}"
    ]
    &&
    \mathbf{sCh}_{\mathbb{K}}
    \end{tikzcd}
  \end{equation}
  Another conclusion from  \eqref{GeneratingCofibrationForSimplicialGroupRepresentations} is that the underlying functor is also {\it left Quillen}
  \begin{equation}
    \label{UnderyinFunctorOnSimplicialRepresentationsIsLeftQuillen}
    \begin{tikzcd}[column sep=large]
    \mathbf{sCh}_{\mathbb{K}}^{\mathbf{B}\mathcal{G}}
    \,\simeq\,
    \mathcal{G}\mathrm{Act}\big(
      \mathbf{sCh}_{\mathbb{K}}
    \big)
    \ar[
      from=rr,
      shift left=8pt,
      "{
        (\ast \to \mathbf{B}\mathcal{G})_\ast
      }"
    ]
    \ar[
      rr,
      shift left=7pt,
      "{
         \mathrm{undrl}
      }"{description}
    ]
    \ar[
      rr,
      phantom,
      shift right=1pt,
      "{\scalebox{.7}{$\bot_{\mathrlap{\mathrm{Qu}}}$}}"
    ]
    &&
    \mathbf{sCh}_{\mathbb{K}}
    \end{tikzcd}
  \end{equation}
  since $\mathcal{G}\cdot i$ (resp. $\mathcal{G}\cdot j$) are still (acyclic) cofibrations in $\mathbf{sCh}_{\mathbb{K}}$ due to its $\mathrm{sSet}$-enriched model structure (Prop. \ref{ModelCategoryOfSimplicialChainComplexes}).

  \smallskip

  With these preliminaries in hand, we check the pushout-product axiom \eqref{PushoutProductAxiom}: Consider a pair of generating cofibrations
  $\mathscr{V} \to \mathscr{V}'$ and $\mathscr{W} \to \mathscr{W}'$ in $\mathbf{sCh}^{\mathbf{B}\mathcal{G}}_{\mathbb{K}}$; 
  we need to show that their pushout-product morphism on the far left of the following diagrams is a cofibration, which equivalently means 
  that for any acyclic fibration $\mathscr{R} \to \mathscr{R}'$ and commuting diagrams as on the left, there exists a lift as shown by the 
  dashed arrow on the left, and by a standard argument (cf. \cite[Rem. A.3.1.6]{Lurie09}) this exists if and only if a lift in the 
  corresponding diagram on the right exists:
  \[
    \adjustbox{raise=6pt}{
    \begin{tikzcd}[row sep=small, column sep=large]
      \scalebox{.9}{$
      (\mathscr{V} \otimes \mathscr{W}')
      $}
      \overset{
        \mathclap{
          \scalebox{.7}{$
          \mathscr{V} \otimes \mathscr{W}
          $}
        }
      }{\phantom{A} \coprod \phantom{A}}
      \scalebox{.9}{$
      (\mathscr{V}' \otimes \mathscr{W})
      $}
      \ar[d]
      \ar[r]
      &
      \mathscr{R}
      \ar[d]
      \\
      \mathscr{V}' \otimes \mathscr{W}'
      \ar[r]
      \ar[ur, dashed]
      &
      \mathscr{R}'
    \end{tikzcd}
    }
    \hspace{1cm}
    \Leftrightarrow
    \hspace{1cm}
    \adjustbox{raise=-6pt}{
    \begin{tikzcd}[row sep=small, column sep=large]
      \mathscr{V}
      \ar[d]
      \ar[r]
      &
      {[\mathscr{W}', \mathscr{R}]}
      \ar[d]
      \\
      \mathscr{V}'
      \ar[r]
      \ar[ur, dashed]
      &
      \scalebox{.9}{$[\mathscr{W}',\mathscr{R}']$}
      \,
      \underset{
        \mathclap{
        \scalebox{.7}{$[\mathscr{W},\mathscr{R}']$}
        }
      }{\phantom{A} \prod \phantom{A}}
      \,
      \scalebox{.9}{$[\mathscr{W}, \mathscr{R}]$}
    \end{tikzcd}
    }
  \]
  But by the previous observation, the left morphism in the diagram on the right is in the image of the left adjoint functor \eqref{FreeSimplicialGroupAction}, 
  which finally means that the dashed lift on the right exists in $\mathbf{sCh}^{\mathbf{B}\mathcal{G}}_{\mathbb{K}}$ as soon as such exists
  for the underlying diagram in $\mathbf{sCh}_{\mathbb{K}}$.   
  Now by definition of the projective model structure in Prop. \ref{ModelStructureOnSimplicialFunctors} the underlying map of $\mathscr{R} \to \mathscr{R}'$ 
  is still a fibration, and
  using that $\mathbf{sCh}_{\mathbb{K}}$ is $\mathrm{sSet}$-enriched model (Prop. \ref{ModelCategoryOfSimplicialChainComplexes}) 
  it follows also that underlying the generating cofibrations \eqref{GeneratingCofibrationForSimplicialGroupRepresentations} are cofibrations in
  $\mathbf{sCh}_{\mathbb{K}}$, 
  and then that lifting in the diagram on the right above is that of a cofibration against an acyclic fibration in $\mathbf{sCh}_{\mathbb{K}}$ and hence exists.
  Verbatim the argument with the evident substitutions shows that the same kind of lifts exist if $\mathscr{W} \to \mathscr{W}'$ is actually an
  acyclic cofibration and $\mathscr{R} \to \mathscr{R}'$ is any fibration. In summary this establishes the pushout-product axiom 
  in $\mathbf{sCh}^{\mathbf{B}\mathcal{G}}_{\mathbb{K}}$.
\end{proof}

\subsection{Parameterization over varying base spaces}
\label{ParameterizationOverVaryingBaseSpaces}

We now glue all the model categories of simplicial local systems over fixed base spaces to an integral model structure on simplicial 
local systems over varying base spaces.
First, to recall (references in the following proof, p. \pageref{ProofOfKanQuillenModelStructures}):
\begin{definition}[Free maps of simplicial groupoids]
  \label{FreeMapsOfSimplicialGroupoids}
  A morphism $\mathbf{f} \,\isa\, \mathbf{X} \to \mathbf{Y}$ of simplicial groupoids \eqref{sSetLocalizationAndCore}
  is called {\it free} if it is degreewise injective and there exists a subset   
  $
    \Gamma 
    \,\subset\,
    \mathrm{Mor}(\mathbf{Y})
  $
  of ``freely generating'' morphisms (of any degree)
  such that:
  \begin{itemize}
    \item[(1)] no identity morphism is in $\Gamma$, but $\Gamma$ is closed under the degeneracy maps of $\mathbf{Y}$,
    \item[(2)] every non-identity morphism in $\mathbf{Y}$ is {\it uniquely} the composition of a {\it reduced} (see below) sequence of composable
    \begin{itemize}
      \item[(i)] morphisms in $\Gamma$ and their inverses,
      \item[(ii)] morphisms in the image under $\mathbf{f}$ of non-identity morphisms in $\mathbf{X}$,
    \end{itemize}
  \end{itemize}
  where {\it reduced} means that
  \begin{itemize}
    \item[(i)] no morphism in the sequence is consecutive with its own inverse,
    \item[(ii)] no two consecutive morphisms in the sequence are both in the image of $\mathbf{f}$.
  \end{itemize}
\end{definition}
For example, the point inclusion $\ast \to \mathbf{B} F(S)$ into the delooping groupoid of a free group, free on a set $S$, is a free map, in the sense of Def. \ref{FreeMapsOfSimplicialGroupoids}, with $\Gamma \defneq S$.
\begin{proposition}[Dwyer-Kan model structures]
\label{DwyerKanModelStructures}
$\,$

\noindent {\bf (i)} We have the following classical model categories:
\begin{itemize}
  \item[{\bf (a)}] The category of simplicial sets
  
  \fbox{$\mathrm{sSet}$} carries the Kan-Quillen model structure whose
  
  \begin{itemize}
    \item weak equivalences are the simplicial weak homotopy equivalences,
    \item fibrations are the Kan fibrations,
    \item cofibrations are the monomorphisms
  \end{itemize}

  \item[{\bf (b)}]
  The category of (small) $\mathrm{sSet}$-enriched categories (often known as ``simplicial categories'')

  \fbox{$\mathrm{sSet}\mbox{-}\mathrm{Cat}$} carries a model structure whose
  \begin{itemize}
    \item weak equivalences are the {\it Dwyer-Kan equivalences}, namely the $\mathrm{sSet}$-functors which on isomorphism classes of 
    objects in the homotopy category are surjective and on all hom-complexes are simplicial weak homotopy equivalences of underlying 
    simplicial sets,
    \item fibrations are the $\mathrm{sSet}$-functors which are isofibrations on homotopy categories and on all hom-complexes are
    Kan fibrations of underlying simplicial sets.
  \end{itemize}
  
\item[{\bf (c)}]
The category of small $\mathrm{sSet}$-enriched groupoids (often known as ``simplicial groupoids'')

\noindent
  \fbox{$\mathrm{sSet}\mbox{-}\mathrm{Grpd}$} carries a model structure whose
  \begin{itemize}
    \item weak equivalences are the Dwyer-Kan equivalences as above,
    \item fibrations are the maps that admit lifting of 1-morphisms and are Kan fibrations on underlying simplicial sets 
    of all automorphism groups,
    \item 
    cofibrations are in particular injective on objects and degreewise on all hom-complexes.
  \end{itemize}

  \item[{\bf (d)}] 
  The category of simplicial groups 

  \noindent
  \fbox{$\mathrm{sGrp} := \mathrm{Grp}(\mathrm{sSet})$} carries a model structure whose
  \begin{itemize}
    \item  weak equivalences are the simplicial weak homotopy equivalences of underlying simplicial sets,
    \item fibrations are the Kan fibrations of underlying simplicial sets,
    \item cofibrations are the retracts of ``almost free'' (cf. \cite[p. 270]{GoerssJardine09}) simplicial group inclusions,
    in particular all cofibrations are monomorphisms (simplicial subgroup inclusions).
    \end{itemize}
\end{itemize}
\noindent {\bf (ii)} We have the following functors relating these:
\begin{itemize}
\item
The canonical full inclusions are compatible with this model structure
\begin{equation}
  \label{DiagramOfDKModelStructures}
  \begin{tikzcd}[row sep=-2pt, column sep=large]
    \mathrm{sGrp}
    \ar[r, hook]
    &
    \mathrm{sSet}\mbox{-}\mathrm{Grpd}
    \ar[
      from=r, 
      shift right=7pt,
      "{ \mathrm{Loc} }"{swap}
    ]
    \ar[
      r, 
      shift right=7pt,
      hook,
      "{ \iota }"{swap}
    ]
    \ar[
      r,
      phantom,
      "{
        \scalebox{.7}{$\bot_{\mathrlap{\mathrm{Qu}}}$}
      }"
    ]
    &
    \mathrm{sSet}\mbox{-}\mathrm{Cat}    
    \\
    \mathcal{G}
    \ar[
      r,
      phantom,
      "{\longmapsto}"
    ]
    &
    \mathbf{B}\mathcal{G}
  \end{tikzcd}
\end{equation}
in that
\begin{itemize}
  \item[--] $\mathrm{Loc} \dashv \iota$ is a Quillen adjunction 
  
  (here $\mathrm{Loc}$ is degreewise the free groupoid construction on or equivalently the full localization of a category);
  \item[--] $\mathbf{B}(-)$ (Def. \ref{SimplicialDeloopingGroupoid})
  preserves weak equivalences and fibrations
    (but has no left adjoint).
\end{itemize}
\item
There is a Quillen equivalence
\begin{equation}
  \label{DKGroupoidQuillenEquivalence}
  \begin{tikzcd}[column sep=large]
    \mathrm{sSet}\mbox{-}\mathrm{Grpd}
    \ar[
      r,
      shift right=7pt,
      "{
        \overline{\mathrm{W}}
      }"{swap}
    ]
    \ar[
      from=r,
      shift right=7pt,
      "{
        \mathbf{G}
      }"{swap}
    ]
    \ar[
      r,
      phantom,
      "{
        \scalebox{.7}{$\simeq_{\mathrlap{\mathrm{Qu}}}$}
      }"
    ]
    &
    \mathrm{sSet}
  \end{tikzcd}
\end{equation}
and a Quillen adjunction
\begin{equation}
  \label{RigidificationQuillenAdjunction}
  \begin{tikzcd}
    \mathrm{sSet}\mbox{-}\mathrm{Grpd}
    \ar[
      rrr,
      shift right=7pt,
      "{
        \widehat{\mathrm{W}}
        \,:\defneq\,
        N \,\circ\, \iota
      }"{swap}
    ]
    \ar[
      from=rrr,
      shift right=7pt,
      "{
        \widetilde{\mathbf{G}}
        \;:\defneq\;
        \mathrm{Loc}
        \,\circ\,
        \mathfrak{C}
      }"{swap}
    ]
    \ar[
      rrr,
      phantom,
      "{ \scalebox{.7}{$\bot_{\mathrlap{\mathrm{Qu}}}$} }"
    ]
    &&&
    \mathrm{sSet}
  \end{tikzcd}
  \;\;\;
    :\defneq
  \;\;\;
  \begin{tikzcd}
    \mathrm{sSet}\mbox{-}\mathrm{Grpd}
    \ar[
      rr,
      shift right=7pt,
      "{
        \iota
      }"{swap}
    ]
    \ar[
      from=rr,
      shift right=7pt,
      "{
        \mathrm{Loc}
      }"{swap}
    ]
    \ar[
      rr,
      phantom,
      "{ \scalebox{.7}{$\bot_{\mathrlap{\mathrm{Qu}}}$} }"
    ]
    &&
    \mathrm{sSet}\mbox{-}\mathrm{Cat}
    \ar[
      rr,
      shift right=6pt,
      "{
        N
      }"{swap}
    ]
    \ar[
      from=rr,
      shift right=6pt,
      "{
        \mathfrak{C}
      }"{swap}
    ]
    \ar[
      rr,
      phantom,
      "{ \scalebox{.7}{$\bot$} }"
    ]
    &&
    \mathrm{sSet}\;,
  \end{tikzcd}
\end{equation}
such that
\begin{itemize}
    \item
    there exists a natural transformation
    \begin{equation}
      \label{ComparisonFromLocalizationOfRigidificationToDKGroupoids}
      \begin{tikzcd}
        \mathcal{X} \,\in\, \mathrm{sSet}
        \hspace{1cm}
          \vdash
        \hspace{1cm}
        \mathrm{Loc}
        \,\circ\,
        \mathfrak{C}(\mathcal{X})
        \ar[
          rr, 
          "{
            \in \mathrm{W}_{\mathrm{DK}}
          }"{swap}
        ]
        &&
        \mathbf{G}(\mathcal{X})
      \end{tikzcd}
    \end{equation}
    which is a Dwyer-Kan equivalence,
  \item 
    the natural transformation
    \begin{equation}
      \label{ProductProjectionUnderRigidification}
      \mathcal{X},\,\mathcal{Y} \,\in\,
      \mathrm{sSet}
      \hspace{1cm}
        \vdash
      \hspace{1cm}
      \begin{tikzcd}
        \mathfrak{C}(\mathcal{X} \times \mathcal{Y})
        \ar[
          rr,
          "{
            \left(
            \mathfrak{C}(\mathrm{pr}_{\mathcal{X}})
            ,\,
            \mathfrak{C}(\mathrm{pr}_{\mathcal{Y}})
            \right)
          }",
          "{ \in \mathrm{W}_{\mathrm{DK}} }"{swap}
        ]
        &\phantom{A}&
        \mathfrak{C}(\mathcal{X})
        \times
        \mathfrak{C}(\mathcal{Y})
      \end{tikzcd}
    \end{equation}
    is a Dwyer-Kan equivalence.
\end{itemize}
\end{itemize}
\end{proposition}
\begin{proof}
\label{ProofOfKanQuillenModelStructures}
  {\bf (i)}
 The Kan-Quillen model structure on $\mathrm{sSet}$ is, of course, due to \cite[\S II.3]{Quillen67}, see also for instance \cite[\S I.11]{GoerssJardine09}.
 The model structure on $\mathrm{sGrp}$ is due to \cite[\S II 3.7]{Quillen67}, see \cite[\S V]{GoerssJardine09}.
 
 The model structure on $\mathrm{sSet}\mbox{-}\mathrm{Grpd}$ is due to \cite[\S 2.5]{DwyerKan84}. Their result \cite[\S 2.4 with \S 2.3 (i)]{DwyerKan84} asserts that the cofibrations are in particular retracts of degreewise injections of sets (of objects and of morphisms). But since injections of sets are closed under retracts this means that all cofibrations are in particular degreewise injections.
 
 The model structure on 
 $\mathrm{sSet}\mbox{-}\mathrm{Cat}$ is due to 
 \cite{Bergner07}, see also \cite[Thm. A.3.2.4]{Lurie09}. 
 
 \noindent
 {\bf (ii)}
 From this, it is immediate that  the functors in \eqref{DiagramOfDKModelStructures} preserve the structure as stated; the Quillen adjunction on the right 
 of \eqref{DiagramOfDKModelStructures} is also made explicit in \cite[Prop. 2.8]{MRZ23}.
 
 The Quillen equivalence $\mathcal{G} \dashv \overline{\mathrm{W}}$ \eqref{DKGroupoidQuillenEquivalence} is due to \cite[Thm. 3.3]{DwyerKan84} reviewed in \cite[Thm. 7.8]{GoerssJardine09}. 
 
 The adjunction $\mathfrak{C} \dashv N$ on the right of \eqref{RigidificationQuillenAdjunction} is actually a Quillen equivalence with respect to the Joyal model structure on simplicial sets \cite[Thm. 7.8]{Bergner07b}\cite[Thm. 2.2.5.1]{Lurie09}:
 \begin{equation}
   \label{RigidificationAdjunctionForQuasiCategories}
   \begin{tikzcd}
     \mathrm{sSet}\mbox{-}\mathrm{Cat}
     \ar[
       rr,
       shift right=7pt,
       "{ N }"{swap}
     ]
     \ar[
       from=rr,
       shift right=7pt,
       "{ \mathfrak{C} }"{swap}
     ]
     \ar[
       rr,
       phantom,
       "{
         \scalebox{.7}{$\bot_{\mathrlap{\mathrm{Qu}}}$}
       }"       
     ]
     &&
     \mathrm{sSet}_{\mathrm{Joy}}
     \,.
   \end{tikzcd}
 \end{equation} 
 But since the Joyal model structure has the same cofibrations as the Kan-Quillen model structure (the monomorphisms) this implies with \eqref{DiagramOfDKModelStructures} that $\widetilde{\mathbf{G}} \,\defneq\, \mathrm{Loc} \,\circ\, \mathfrak{C}$ preserves cofibrations. To see that it also preserves weak equivalences, and hence is a left Quillen functor as claimed on the left of \eqref{RigidificationQuillenAdjunction}, notice that \eqref{ComparisonFromLocalizationOfRigidificationToDKGroupoids} -- which is due to \cite[Thm. 1.1]{MRZ23} -- implies commuting squares
 \[
  \begin{tikzcd}[row sep=small]
    \mathcal{X}
    \ar[rr, "{ f }", "{ \in \mathrm{W} }"{swap}]
    &&
    \mathcal{X}'
    \\
    \widetilde{\mathbf{G}}(\mathcal{X})
    \ar[
      rr,
      "{
        \widetilde{\mathbf{G}}(f)
      }"
    ]
    \ar[d]
    &&
    \widetilde{\mathbf{G}}(\mathcal{X})
    \ar[d]
    \\
    \mathbf{G}(\mathcal{X})
    \ar[
      rr, 
      "{ \mathbf{G}(f) }"
    ]
    &&
    \mathbf{G}(\mathcal{Y})
    \mathrlap{\,,}
  \end{tikzcd}
\]
where the vertical maps are Dwyer-Kan equivalences. But also the bottom map is a Dwyer-Kan equivalence by Ken Brown's Lemma \ref{KenBrownLemma}, since $\mathbf{G}$ is a left Quillen functor \eqref{DKGroupoidQuillenEquivalence}
on a model category all whose objects are cofibrant, whence also  $\widetilde{\mathbf{G}}(f)$ is a weak equivalence, by the 2-out-of-3 property satisfied by weak equivalences.

Finally, the property \eqref{ProductProjectionUnderRigidification} is due to \cite[Cor. 2.2.5.6]{Lurie09}, see also \cite[Prop. 6.2]{DuggerSpivak11}.
\end{proof}

\begin{remark}[Dwyer-Kan simplicial fundamental groupoids]
  \label{SimplicialFundamentalGroupoid}
  The classical Dwyer-Kan functor $\mathbf{G} : \mathrm{sSet} \longrightarrow \mathrm{sSet}\mbox{-}\mathrm{Grpd}$ \eqref{DKGroupoidQuillenEquivalence} may be thought of as forming simplicial {\it fundamental groupoids} of spaces and hence so may be its Dwyer-Kan-equivalent version $\widetilde{\mathbf{G}}$ \eqref{RigidificationQuillenAdjunction}. 
\end{remark}

\medskip

\noindent
{\bf The integral model structure on simplicial local systems.}
Our interest is now in the pseudofunctor
assigning model categories (Def. \ref{ModCat})
of simplicial local systems (Prop. \ref{ModelStructureOnSimplicialFunctors})
to simplicial fundamental groupoids $\widetilde{\mathbf{G}}(-)$ (Rem. \ref{SimplicialFundamentalGroupoid})  of simplicial sets, which exists as a bivariant pseudofunctor by Ex. \ref{SystemsOfEnrichedFunctorCategories}
and with values in $\mathrm{ModCat}$ (Def. \ref{ModCat}) by
Rem. \ref{BaseChangeBetweenModelStructuresOfSimplicialLocalSystems}:
\begin{equation}
  \label{PseudofunctorAssigningModelCategoriesOfSimplicialLocalSystems}
  \begin{tikzcd}[row sep=-6pt, column sep=small]
    \mathrm{sSet}
    \ar[
      rr,
      "{ \mathbf{G} }"
    ]
    &&
    \mathrm{sSet}\mbox{-}\mathrm{Grpd}
    \ar[
      rr,
      "{
        \mathbf{sCh}_{\mathbb{K}}^{(-)}
      }"
    ]
    &&
    \mathrm{ModCat}
    \\
    \mathcal{X}
    \ar[
      d,
      "{ f }"
    ]
    &\longmapsto&
    \mathbf{X}
    \ar[
      d,
      "{ \mathbf{f} }"
    ]
    &\longmapsto&
    \mathbf{sCh}^{\mathbf{X}}_{\mathbb{K}}
    \ar[
      d,
      shift right=8pt,
      "{ \mathbf{f}_! }"{swap}
    ]
    \ar[
      from=d,
      shift right=8pt,
      "{ \mathbf{f}^\ast }"{swap}
    ]
    \ar[
      d,
      phantom,
      "{ \scalebox{.7}{$\dashv$} }"
    ]
    \\[30pt]
    \mathcal{X}'
    &\longmapsto&
    \mathbf{X}'
    &\longmapsto&
    \mathbf{sCh}^{\mathbf{X}'}_{\mathbb{K}}
  \end{tikzcd}
\end{equation}

\begin{theorem}[Integral model structure on simplicial local systems over varying bases]
\label{GlobalModelStructure} $\,$ 
\newline
  The integral model structures 
  (Def. \ref{IntegralModelStructure})
  on the Grothendieck constructions
  (Def. \ref{GrothendieckConstruction}) on the pseudofunctors \eqref{PseudofunctorAssigningModelCategoriesOfSimplicialLocalSystems} exist and are Quillen equivalent, to be denoted as follows:
  \begin{equation}
  \label{Loc}
  \adjustbox{fbox}{
  $
    \begin{tikzcd}
    \mathbf{Loc}^{\mathrm{sSet}}_{\mathbb{K}}
      \;\;
        :=
      \;\;
    \underset{
      \mathclap{
      \mathcal{X}
      \,\in\,
      \mathrm{sSet}
      }
    }{\displaystyle{\int}}
    \;
    \mathbf{sCh}^{\mathbf{G}(\mathcal{X})}_{\mathbb{K}}
    \;\;\;
    \ar[
      rr,
      shift left=7pt,
      "{
        \widehat{\mathbf{G}}
      }"
    ]
    \ar[
      from=rr,
      shift left=7pt,
      "{
        \widehat{\mathrm{W}}
      }"
    ]
    \ar[
      rr,
      phantom,
      "{
        \simeq_{\mathrm{Qu}}
      }"
    ]
    &&
    \;\;\;
    \underset{
      \mathclap{
      \mathbf{X}
      \,\in\,
      \mathrm{sSet}\mbox{-}\mathrm{Grpd}
      }
    }{\displaystyle{\int}}
    \;
    \mathbf{sCh}^{\mathbf{X}}_{\mathbb{K}}
    \;=:\;
    \mathbf{Loc}_{\mathbb{K}}
    \end{tikzcd}
  $
  }
\end{equation}
\end{theorem}
\begin{proof}

\smallskip 
  \noindent
  First, regarding the existence of the model structures,
  given maps $f : \mathcal{X}' \to \mathcal{X}$ in $\mathrm{sSet}$ (resp. $\mathbf{f} : \mathbf{X}' \to \mathbf{X}$ in $\mathrm{sSet}\mbox{-}\mathrm{Grpd}$),
  we need to show (by Prop. \ref{ExistenceOfIntegralModelStructure}) the following three properties:

  \smallskip 
  \begin{itemize}

  \item[1.] {\bf If $f$ (resp. $\mathbf{f}$) is a weak equivalence then $\mathbf{G}(f)_! \dashv \mathbf{G}(f)^\ast$ (resp. $\mathbf{f}_! \dashv \mathbf{f}^\ast$)
  is a Quillen equivalence.}
  
  Since $\mathbf{G}$ preserves weak equivalences (by Ken Brown's lemma \ref{KenBrownLemma}, being a left Quillen functor \eqref{DKGroupoidQuillenEquivalence} on a category with all objects cofibrant) it is sufficient to see that $\mathbf{sCh}_{\mathbb{K}}^{(-)}$ has this property.

  This is a special case of the following general statement, which may be of interest in its own right.  

  \begin{lemma}Let $\mathbf{C}$ be a combinatorial simplicial model category and $\mathbf{f} \,\colon\, \mathbf{X}' \longrightarrow \mathbf{X}$ a Dwyer-Kan equivalence (Prop. \ref{DwyerKanModelStructures}) of small $\mathrm{KanCplx}$-enriched categories. Then 
  $\mathbf{f}_! : \mathbf{C}^{\mathbf{X}'} \rightleftarrows \mathbf{C}^{\mathbf{X}} : \mathbf{f}^\ast$ is a Quillen equivalence between the projective model structures on the enriched functor categories.
  \end{lemma}
  
\noindent {\it Proof of Lemma.} We demonstrate this claim by appeal to the $\infty$-categories (quasi-categories) \cite{Joyal08}\cite{Lurie09} 
  presented by the model structure, via the homotopy coherent nerve functor $N : \mathrm{sSet}\mbox{-}\mathrm{Cat} \longrightarrow \mathrm{sSet}$ \eqref{RigidificationAdjunctionForQuasiCategories} applied to full simplicial subcategories $(\mbox{-})^{\circ}$ of bifibrant objects in simplicial model categories: By \cite[Prop. 4.2.44]{Lurie09} there are natural transformations as on the top of the following diagram, which restrict on bifibrant objects to weak equivalences of quasi-categories, as shown at the bottom:
  \begin{equation}
    \label{ComparisonMapHomotopCoherentNerve}
   \begin{tikzcd}[row sep=small]
     N\big(\mathbf{C}^{\mathbf{X}}\big)
     \ar[rr]
     &&
     N(\mathbf{C})^{N(\mathbf{X})}
     \\
     N\left(\big(\mathbf{C}^{\mathbf{X}})^{\circ}\right)
     \ar[u, hook]
     \ar[rr, "{ \in\, \mathrm{W}_{\mathrm{Joy}} }"{swap}]
     &&
     N(\mathbf{C}^\circ)^{N(\mathbf{X})}
     \ar[u, hook]
   \end{tikzcd}
  \end{equation}
  Consider then the following diagram of (large) simplicial sets:
  \begin{equation}
    \label{CompatibleBaseChangeForInfinityFunctor}
    \begin{tikzcd}[row sep=10pt]
      N\left(
        \big(\mathbf{C}^{\mathbf{X}}\big)^\circ
      \right)
      \ar[dr, hook]
      \ar[rrrr, "{ \in\, \mathrm{W}_{\mathrm{Joy}} }"{swap}]
      \ar[
        dddd, 
        "{ 
          N\left( \mathbb{R}(\mathbf{f}^\ast) \right) 
        }"{left}
      ]
      &&&&
      N(\mathbf{C}^\circ)^{N(\mathbf{X})}
      \ar[dddd, "{ N(\mathbf{f})^\ast }"{right}]
      \ar[dl, hook]
      \\
      &
      N\big(\mathbf{C}^{\mathbf{X}}\big)
      \ar[rr]
      \ar[
        d, "{ N(\mathbf{f}^\ast) }"{right}
      ]
      &&
      N(\mathbf{C})^{N(\mathbf{X})}      
      \ar[
        d,
        "{ N(\mathbf{f})^\ast }"{left}
      ]
      \\[+10pt]
      &
      N\big(\mathbf{C}^{\mathbf{X}'}\big)
      \ar[d, "{ N(Q) }"{left}]
      \ar[rr]
      \ar[
        dr, 
        "{ \mathrm{id} }"{above}, 
        "{\ }"{name=t, swap}
      ]
      &&
      N(\mathbf{C})^{N(\mathbf{X}')}      
      \\[+10pt]
      &
      N\big(\mathbf{C}^{\mathbf{X}'}\big)
      \ar[r, "{ \mathrm{id} }"{swap}]
      \ar[
        to=t,
        Rightarrow
      ]
      &      
      N\big(\mathbf{C}^{\mathbf{X}'}\big)
      \ar[ur]
      \\
      N\left(
        \big(\mathbf{C}^{\mathbf{X}'}\big)^\circ
      \right)     
      \ar[ur, hook]
      \ar[rrrr, "{ \in \mathrm{W}_{\mathrm{Joy}} }"{swap}]
      &&&&
      N(\mathbf{C}^\circ)^{N(\mathbf{X}')}
      \mathrlap{\,,}
      \ar[uul, hook]
    \end{tikzcd}
  \end{equation}
  where $Q$ denotes a {\it functorial} cofibrant replacement functor (which exists by \cite[Prop. 2.3]{Dugger01}\cite[Prop. 2.5]{Barwick10}
since $\mathbf{C}^{\mathbf{X}'}$ is combinatorial) and the double arrow denotes (the image under the right adjoint functor $N$ of) the corresponding 
natural transformation whose components are the resolution equivalences $Q(\mbox{-}) \xrightarrow[\in \mathrm{W}]{\phantom{-}} (\mbox{-})$.

  In this diagram \eqref{CompatibleBaseChangeForInfinityFunctor}: the top and bottom squares are instances of \eqref{ComparisonMapHomotopCoherentNerve}, 
  the middle square commutes by naturality, the right square commutes
  evidently and the left square commutes by the usual construction of derived functors of right Quillen functors. With this, we have a natural transformation 
  of $\infty$-functors filling the full diagram and we observe that this is a natural equivalence: This follows by \cite[\S 5, Thm C (p. 125)]{Joyal08} 
  from the fact that its objectwise components are (resolution-)equivalences, by construction of $N(Q)$.

  In conclusion, this shows that the right derived functor $\mathbb{R}\mathbf{f}^\ast$ represents the precomposition $\infty$-functor $N(\mathbf{f})^\ast$ up to natural equivalence; 
  in particular, both coincide on homotopy categories up to natural isomorphism. 
  
  But for $\mathbf{f}$ a DK-equivalence between $\mathrm{KanCplx}$-enriched categories the $\infty$-functor $N(\mathbf{f})$ is an equivalence of $\infty$-categories 
  by \cite[Thm. 2.2.5.1]{Lurie09} (with Ken Brown's lemma \ref{KenBrownLemma})
  and therefore $N(\mathbf{f})^\ast$ is an equivalence by \cite[Prop. 1.2.7.3 (3)]{Lurie09}, 
  hence in particular is an equivalence
  of homotopy categories, whence so is $\mathbb{R}\mathbf{f}^\ast$, which finally means that $\mathbf{f}_! \dashv \mathbf{f}^\ast$ is a Quillen equivalence. 

  \hfill $\square$

  \item[2.] {\bf If $f$ (resp. $\mathbf{f}$) is an acyclic fibration then $\mathbf{G}(f)^\ast$ (resp. $\mathbf{f}^\ast$) preserves weak equivalences. }
  
  This is immediate and in fact holds for all maps $\mathbf{f}$, since $\mathbf{f}^\ast$ acts by precomposition and weak equivalences are given objectwise.

  \item[3.] {\bf If $f$ (resp. $\mathbf{f}$) is an acyclic cofibration then $\mathbf{G}(f)_!$ (resp. $\mathbf{f}_!$) preserves weak equivalences. }

  Since $\mathbf{G}$, being a left Quillen functor \eqref{RigidificationQuillenAdjunction},
  preserves acyclic cofibrations, it is sufficient to show the claim for any acyclic cofibration $\mathbf{f}$.

  We will only need that, thereby:
  \begin{equation}
  \label{AssumptionOnAcyclicCofibrationsNeededInExistenceProof}
  \mbox{
  on objects, 
  $\mathbf{f}$ is (i) injective and (ii) essentially surjective}
  \end{equation}
  (a condition which holds also in the variant situation of Thm. \ref{ExternalMonoidalModelStructureOnLocalSystemsOverOneTypes} that we consider in \cref{ExternalModulesOverOneTypes}).
  
  Namely, to check that $\mathbf{f}_!$ preserves equivalences, which are defined objectwise, means we must check that for all objects $x \,\in\, \mathbf{X}$ the map $\iota_{x}^\ast \circ \mathbf{f}_!$ preserves equivalences, where $\iota_x$ is the inclusion of the full subgroupoid $\mathbf{B}\big(\mathbf{X}(x,x)\big)$ at that point. But in fact, since a morphism of simplicial local systems is a weak equivalence over some $x$ iff it is a weak equivalence over any other object in the same connected component (by functoriality and 2-out-of-3), it is sufficient to check that $\iota_{x}^\ast \circ \mathbf{f}_!$ preserves weak equivalences as $x$ ranges over any choice of representatives for each connected component of $\mathbf{X}$. 
  
  Now by the assumption that $\mathbf{f}$ is injective and essentially surjective on objects, we may find such representatives $x$ such that each has a unique preimage $x'$, giving rise to a pullback square of this form:
  \begin{equation}
    \begin{tikzcd}
      \mathbf{B}
      \big(\mathbf{X}'(x',x')\big)
      \ar[
        rr,
        "{
          \mathbf{f}_{\vert \{x'\}}
        }"
      ]
      \ar[
        d, 
        hook,
        "{ \iota_{x'} }"
      ]
      \ar[
        drr,
        phantom,
        "{
          \scalebox{.7}{
            \rm (pb)
          }
        }"{pos=.3}
      ]
      &&
      \mathbf{B}
      \big(
        \mathbf{X}(x,x)
      \big)
      \ar[
        d, 
        hook,
        "{ \iota_x }"
      ]
      \\
      \mathbf{X}'
      \ar[
        rr,
        "{
          \mathbf{f}
        }"
      ]
      &&
      \mathbf{X}
    \end{tikzcd}
  \end{equation}
  But to this diagram the Beck-Chevalley condition \eqref{BeckChevalleyForSimplicialLocalSystemsAlongEmbeddings} applies and implies that it is now equivalent to check that $\big(\mathbf{f}_{\vert x'}\big)_! \circ \iota_{x'}^\ast$ is a weak equivalence for all objects $x'$ of $\mathbf{X}'$. 
  
  Finally, since $\iota_x^\ast$ preserves all weak equivalences (by the previous item), it is now sufficient to show that pushforward along maps of delooping groupoids  
  \[
    \mathbf{f}
    \,:\,
    \mathbf{B}\mathcal{H}
    \xrightarrow{\;}
    \mathbf{B}\mathcal{G}
  \]
  preserves weak equivalences.
  This is the case we check now, using the identification of simplicial local systems over simplicial delooping groupoids with simplicial group representations (Rem. \ref{SimplicialGroupActions}).
  
  On general grounds (reviewed eg. in \cite[Lem. 1.1.7]{EquBundles}), in this case $\mathbf{f}_!$ acts by forming left-induced representations, namely by
  $$
   \mathbf{f}_! 
   \,:\,
   \mathscr{V} \,\mapsto\,
   \mathcal{G} \cdot_{\!{}_{\mathcal{H}}}
   \mathscr{V}
   \,:=\,
   \big(\mathcal{G} \cdot \mathscr{V}\big)
   /\mathcal{H}
   \;:\;
   [n] 
   \,\mapsto\,
   \big(\mathcal{G}_n \cdot \mathscr{V}_n\big)
   / \mathcal{H}_n
   \,,
  $$
  where on the right the tensoring $\mathcal{G} \cdot \mathscr{V}$ is equipped with the diagonal $\mathscr{H}$-action which on $\mathcal{G}$ is given by right inverse multiplication; and we have notationally highlighted that quotients of simplicial objects are computed degreewise.

  Now we use that $\mathbf{f} : \mathbf{B}\mathcal{H} \longrightarrow \mathbf{B}\mathcal{G}$ is simplicial-degreewise a group homomorphism
  \[
    n : \mathbb{N}
    \hspace{1cm}
    \vdash
    \hspace{1cm}
    \phi_n
    \,:\,
    \mathcal{H}_n
    \xrightarrow{\phantom{--}}
    \mathcal{G}_n
    \,.
  \]
  Since the weak equivalences in the local model structure $\mathrm{sCh}_{\mathbb{K}}$ 
  (from Thm. \ref{ModelCategoryOfSimplicialChainComplexes})
  {\it include} 
  (either by definition of left Bousfield localization or else by \cite[Prop. 3.1.5]{Hirschhorn02}) the global
  Reedy equivalences which are the simplicial-degreewise weak equivalences between objects $\mathscr{V}_{\!n} \in \mathrm{Ch}_{\mathbb{K}}$, it is sufficient now to observe that for each $n \in \mathbb{N}$ there is even an isomorphism 
  ({\it not} necessarily natural in $n$, but it does not need to be) of the form
  \begin{equation}
    \label{QuotientIsomorphism}
    \big(
      \mathcal{G}_n 
        \cdot 
      \mathscr{V}_n
    \big)
    /\mathscr{H}_n
    \;\simeq\;
    \big(
      \mathcal{G}_n/\mathcal{H}_n
    \big) 
        \cdot 
    \mathscr{V}_n
    \;\;\;\;
    \in
    \;
    \mathrm{Ch}_{\mathbb{K}}
    \,.
  \end{equation} 
  This concludes the argument, because the tensoring with any set --- as on the right of \eqref{QuotientIsomorphism} --- is a left Quillen functor on $\mathrm{Ch}_{\mathbb{K}}$, 
  and since all objects in $\mathrm{Ch}_{\mathbb{K}}$  are cofibrant (both by Thm. \ref{ModelCategoryOfSimplicialChainComplexes}) so that this Quillen functor preserves all weak equivalences, by Ken Brown's lemma \ref{KenBrownLemma} (or more concretely: because a direct sum of quasi-isomorphisms is itself a quasi-isomorphism). 
 
  \medskip
  
  For completeness, we spell out the isomorphism \eqref{QuotientIsomorphism}.
      Under the identifications \eqref{EquivalentWaysOfTensoringVectorSpaces}, this is the familiar statement from representation theory that the tensor product of any group representation $V$ with the regular $G$-representation is isomorphic to the $\mathrm{dim}(V)$-fold direct sum of the regular representation with itself; but for the record we make the isomorphism explicit by an elementary argument:
  
  Using the Axiom of Choice in our underlying category $\mathrm{Set}$, we may choose a section as follows 
  (and the arbitrariness in this choice makes the construction be non-natural):
  $$
    \begin{tikzcd}[
      row sep=15pt, 
      column sep=large
    ]
      &
      \mathcal{G}_n
      \ar[
        d,
        "{
          [-]
        }"{pos=.4}
      ]
      \\
      \mathcal{G}_n/\mathcal{H}_n
      \ar[r, equals]
      \ar[ur, "\sigma_n", dashed]
      &
      \mathcal{G}_n/\mathcal{H}_n      
    \end{tikzcd}
  $$
  which determines a function
  $$
    \begin{tikzcd}[row sep=-4pt, column sep=small]
      \mathllap{
        \sigma[\mbox{-}]
        \backslash
        (\mbox{-})
      \,:\,\;
      }
      \mathcal{G}_n
      \ar[rr]
      &&
      \phi_n(\mathcal{H}_n)
      \mathrlap{
        \;\subset\,
        \mathcal{G}_n
      }
      \\
      g 
        &\longmapsto& 
      \big(\sigma[g]\big)^{-1} 
        \cdot
      g
      \mathrlap{\,.}
    \end{tikzcd}
  $$
  Using this, the desired isomorphism and its inverse are given by the diagonal morphisms in the following diagram (the top of which shows the coequalizer defining the global quotient, just for context):
  $$
    \begin{tikzcd}[column sep=62pt]
      \mathcal{H}_n
      \cdot
      \big(
      \mathcal{G}_n
        \cdot
      \mathscr{V}_{\!n}
      \big)
      \ar[
        rr,
        shift left=5pt,
        "{
          \scalebox{1.1}{$($}
            h
            ,\,
            (
            g
            ,\,
            v
            )
          \scalebox{1.1}{$)$}
          \,\mapsto\,
          \scalebox{1.1}{$($}
            g
            ,\
            v
          \scalebox{1.1}{$)$}
        }"
      ]
      \ar[
        rr,
        shift right=5pt,
        "{
          \scalebox{1.1}{$($}
            h
            ,\,
            (
            g
            ,\,
            v
            )
          \scalebox{1.1}{$)$}
          \,\mapsto\,
          (
            g \cdot h^{-1}
            ,\
            h \cdot v
          )
        }"{swap}
      ]
      &&
      \mathcal{G}_n
        \cdot
      \mathscr{V}_n
      \ar[
        dd,
        "{
          (g,\,v)
          \,\mapsto\,
          \scalebox{1.3}{$($} 
            [g]
              ,\, 
            \sigma[g] \backslash g \,\cdot\, v
          \scalebox{1.3}{$)$} 
        }"{description}
      ]
      \ar[
        rr,
        "{
          (g,\,v)
          \,\mapsto\,
          [g,\,v]
        }",
        "{
          \mathrm{coeq}
        }"{swap}
      ]
      &&
      (\mathcal{G}_n \cdot \mathscr{V}_n)
        /
      \mathcal{H}_n \;.
      \ar[
        ddll, 
        dashed,
        "{
          [g,\, v]
          \,\mapsto\,
          \scalebox{1.1}{$($} 
            [g]
              ,\, 
             \sigma[g] \backslash g  
               \,\cdot\, 
             v
          \scalebox{1.1}{$)$} 
        }"{description, sloped}
      ]
      \\
      \\
      &&
      (\mathcal{G}_n / \mathcal{H}_n)
      \cdot
      \mathscr{V}_{\!n}
      \ar[
        uurr,
        shift right=17pt,
        shorten <= 10pt,
        shorten >= -20pt,
        "{
          \scalebox{1.1}{$($} 
            [g]
            ,\,
            v
          \scalebox{1.1}{$)$} 
          \,\mapsto\,
          \scalebox{1.1}{$[$}
            \sigma[g]
            ,\, 
            v
          \scalebox{1.1}{$]$}
        }"{swap, sloped}
      ]
    \end{tikzcd}
  $$
  \end{itemize}
This shows that both integral model structures exist.  
It remains to see the Quillen equivalence in \eqref{Loc}:

The underlying pair of adjoint functors $\widehat{\mathbf{G}} \dashv \widehat{\mathrm{W}}$ is given by Ex. \ref{InducedAdjunctionBetweenGrothendieckConstructions}. 
Furthermore, $\widehat{\mathrm{W}}$ is clearly a right Quillen functor because (recall Def. \ref{IntegralModelStructure}) the underlying functor is right Quillen by Prop. \ref{DwyerKanModelStructures} while the action on component morphisms by pullback is right Quillen by Rem. \ref{BaseChangeBetweenModelStructuresOfSimplicialLocalSystems}. Finally to see that this Quillen adjunction is a Quillen equivalence it is sufficient for cofibrant $\mathscr{V}_{\!\!\mathcal{X}} \,\in\, \mathbf{Loc}^{\mathrm{sSet}}_{\mathbb{K}}$ and fibrant $\mathscr{V}'_{\mathbf{X}'} \,\in\, \mathbf{Loc}_{\mathbb{K}}$ to check that
$$
  \mbox{
  $
  \mathscr{V}_{\!\!\mathcal{X}}
  \xrightarrow{ \phi_{f} }
  \widehat{R}\big(
    \mathscr{V}'_{\mathbf{X}'}
  \big)
  $
  \;\;is a weak equivalence iff its adjunct\;\;
  $
  \widehat{\mathbf{G}}\big(
    \mathscr{V}_{\!\!\mathcal{X}}
  \big)
  \xrightarrow{ \phi_{\tilde{f}} }
  \mathscr{V}'_{\mathbf{X}'}
  $
  \;\;
  is a weak equivalence
  }
$$
But on underlying morphisms this is the case because $\mathbf{G} \dashv \overline{\mathrm{W}}$ is a Quillen equivalence \eqref{DKGroupoidQuillenEquivalence}, while on component morphisms both $\widehat{G}(-)$ \eqref{CounitOfLiftOfAdjunctionToGrothendieckConstruction} as well as $\epsilon \circ (-)$ \eqref{CounitOfLiftOfAdjunctionToGrothendieckConstruction} are the identity operation, whence so is, on components, the passage $\epsilon \circ \widehat{\mathbf{G}}(-)$ to adjuncts.
\end{proof}

\begin{definition}[Notation for simplicial local systems]
  We denote the objects and morphisms in the category $\mathbf{Loc}_{\mathbb{K}}$ \eqref{Loc} 
  of simplicial local systems as follows:
  \begin{itemize}
  \item objects are denoted
  \begin{equation}
    \label{NotationForObjectsAmongSimplicialLocalSystems}
    \mathscr{V}_{\mathbf{X}}
    \;\;
    :\equiv
    \;\;
    \Big(
      \mathbf{X}
      \,\in\,
      \mathrm{sSet}\mbox{-}\mathrm{Grpd}
      ,\;
      \mathscr{V}
      \,\in\,
      \mathbf{sCh}^{\mathbf{X}}_{\mathbb{K}}
    \Big)
  \end{equation}
  which is suggestively read as ``$\mathscr{V}$ is a simplicial local system over $\mathbf{X}$'';
  \item
  morphisms are denoted by their components in the underlying {\it contravariant} pseudofunctor
  \begin{equation}
    \label{NotationForMorphismsOfSimplicialLocalSystems}
    \big(
    \phi_{\mathbf{f}}
    \,:\,
    \mathscr{V}_{\mathbf{X}}
    \xrightarrow{\phantom{--}}
    \mathscr{V}'_{\mathbf{X}'}
    \big)
    \;\;\;
      :\equiv
    \;\;\;
    \Bigg(\!\!\!
       \begin{tikzcd}[row sep=-3pt]
        \mathscr{V}
        \ar[r, "{ \phi }"]
        &
        \mathbf{f}^\ast
        \mathscr{V}'
        \\
        \mathbf{X}
        \ar[r, "{ \mathbf{f} }"]
        &
        \mathbf{X}'
      \end{tikzcd}
         \!\!\! \Bigg)
    \mathrlap{\,,}
  \end{equation}
  as opposed to the adjunct form of the component $\phi$, which we indicate by a $\widetilde{\phi}$:
  \begin{equation}
    \label{AdjunctComponentMaps}
    \begin{tikzcd}[row sep=0pt]
      \mathscr{V}
      \ar[r, "{ \phi }"{description}]
      &
      \mathbf{f}^\ast \mathscr{V}'
      \\
      \hline
      \mathbf{f}_!\mathscr{V}
      \ar[
        r, 
        "{ \widetilde{\phi} }"{description}
      ]
      &
      \mathscr{W}      
    \end{tikzcd}
  \end{equation}
  \end{itemize}
\end{definition}

\begin{remark}[Choice of component morphisms]
  The  choice \eqref{NotationForMorphismsOfSimplicialLocalSystems} is motivated by the fact that $\mathbf{f}^\ast$ but not $\mathbf{f}_!$ is a (strong) monoidal 
  functor (Prop. \ref{FrobeniusReciprocityForSimplicialLocalSystems} below), which means that in the notational convention 
  \eqref{NotationForMorphismsOfSimplicialLocalSystems}
  the external tensor product (Def. \ref{ExternalTensorProductOfSimplicialLocalSystems}) below is {\it manifestly} given by the obvious formula. 
  On the other hand, some computation shows (Prop. \ref{PullPushAdjunctOfExternalTensorProducts} below) that with the other convention, the analogous 
  obvious formula will still hold (even if far from manifestly so) so that in the end the choice in \eqref{NotationForMorphismsOfSimplicialLocalSystems} 
  is as arbitrary as one would hope it is.
\end{remark}

\begin{example}[Decomposing group representations]
\label{DecompositionGroupRepresentations}
In $\mathbf{Loc}_{\mathbb{K}}$
\eqref{Loc},
every simplicial group representation
(Rem. \ref{SimplicialGroupActions})
decomposes along a homotopy cartesian square of this form
\begin{equation}
  \label{HomotopyQuotientSquares}
  \left.
  \def\arraystretch{1}
  \begin{array}{l}
    \mathcal{G} \,\in\, \mathrm{sGrp}
    \\
    \mathscr{V} 
    \,\in\,
    \big(
      \mathbf{sCh}_{\mathbb{K}}^{\mathbf{B}\mathcal{G}}
    \big)^{\mathrm{fib}}
  \end{array}
  \right\}
  \hspace{1.2cm}
    \vdash
  \hspace{1.2cm}
  \begin{tikzcd}
    \mathrm{pt}^\ast
    \mathscr{V}
    \ar[r]
    \ar[d]
    \ar[
      dr,
      phantom,
      "{
        \scalebox{.7}{(pb)}
      }"
    ]
    &
    \mathscr{V}_{\mathbf{B}\mathcal{G}}
    \ar[
      d,
      "{
        0_{\mathrm{id}}
      }"{swap},
      "{ \in \mathrm{Fib} }"
    ]
    \\
    0_{\mathrm{pt}}
    \ar[
      r,
      "{
        0_{\mathrm{pt}}
      }"{swap}
    ]
    &
    0_{\mathbf{B}\mathcal{G}}
  \end{tikzcd}
  \;\;\;\;\;\;
  \defneq:
  \;\;\;\;\;\;
  \begin{tikzcd}
    \mathscr{V}
    \ar[r]
    &
    \mathscr{V} \!\sslash\! \mathcal{G}
    \ar[d]
    \\
    & 
    \mathbf{B}\mathcal{G}
  \end{tikzcd}
\end{equation}
exhibiting $\mathscr{V}_{\mathbf{B}\mathcal{G}}$ as a homotopy quotient of $\mathscr{V}_{\mathrm{pt}}$ by a $\mathcal{G}$-action.
\end{example}
\begin{proof}
  First, that the square is Cartesian follows 
  by Prop. \ref{ColimitsInAGrothendieckConstruction}:
  \[
    \begin{tikzcd}
      \mathrm{pt}
      \ar[r]
      \ar[d, "{ \mathrm{id} }"{swap}]
      \ar[
        dr,
        phantom,
        "{ \scalebox{.7}{(pb)} }"
      ]
      &
      \mathbf{B}\mathcal{G}
      \ar[d, "{ \mathrm{id} }"]
      \\
      \mathrm{pt}
      \ar[r]
      &
      \mathbf{B}\mathcal{G}
    \end{tikzcd}
    \hspace{1cm}
    \mbox{and}
    \hspace{1cm}
    \begin{tikzcd}
      \mathrm{pt}^\ast \mathscr{V}
      \ar[r, "{ \mathrm{id} }"]
      \ar[d, "{ 0 }"{swap}]
      \ar[dr, phantom, "{ \scalebox{.7}{(pb)} }"]
      &
      \mathrm{pt}^\ast \mathscr{V}
      \ar[
        d, 
        "{ \mathrm{pt}^\ast 0 }"
      ]
      \\
      0
      \ar[r, "{ 0 }"]
      &
      \mathrm{pt}^\ast 0
    \end{tikzcd}
  \]
  Finally, the right vertical map is a fibration since the identity on $\mathbf{B}\mathcal{G}$ is a fibration in $\mathrm{sSet}\mbox{-}\mathrm{Grpd}$ and since $\mathscr{V} \to 0$ is a fibration in $\mathbf{sCh}^{\mathbf{B}\mathcal{G}}_{\mathbb{K}}$ iff it is in $\mathbf{sCh}_{\mathbb{K}}$, which is the case by assumption.
\end{proof}

\subsection{\texorpdfstring{The external tensor of flat $\infty$-vector bundles}{The external tensor of flat infinity-vector bundles}}
\label{ExternalTensorOnSimplicialLocalSystems}

We discuss here the construction (Def. \ref{ExternalTensorProductOfSimplicialLocalSystems}) 
and its homotopical properties (Thm. \ref{ExternalTensorProductIsHomotopical})
of the external tensor product on $\mathbf{Loc}_{\mathbb{K}}$ \eqref{Loc}, covering the Cartesian product on $\mathrm{sSet}\mbox{-}\mathrm{Grpd}$ (Prop. \ref{CartesianClosureOfSSetEnrichedGroupoids}).

\medskip

\noindent
{\bf Motivic yoga on simplicial local systems.}
The abstract form of the following structures and conditions was essentially all first formulated and named (``Frobenius reciprocity'', ``Beck-Chevalley condition'') 
in discussions of (categorical semantics for) formal logic/type theory \cite{Lawvere70}\cite{Seely83}\cite[\S 1]{Pavlovic91}\cite{Pavlovic96},
even though the same structures govern what came to be known as {\it Grothendieck's yoga of six operations} and as such must have originated at 
around the same time but have been systematically recorded only much later (notably \cite{FHM03}, in whose terminology we are dealing with 
the {\it Wirthm{\"u}ller-form} of the yoga) especially once Grothendieck's idea of ``motives'' was felt to be nailed down by ``motivic homotopy theory'' \cite[\S A.5]{CisinskiDeglise19}\cite[p. 4]{Hoyois17}. 

\smallskip 
On the other hand, the original discussion in logic was entirely in {\it classical} logic, while Grothendieck's ``yoga'' that concerns us 
now always focused on non-cartesian (hence non-classical, i.e.: linear, ``quantum'') monoidal categories dependent on classical base 
spaces --- whence here we speak of the ``motivic yoga'' for short. More recently, essentially the same is referred to as 
``indexed closed monoidal enriched categories with indexed coproducts'' \cite{Shulman13}, which in its category-theoretic sobriety is 
again more suggestive (for the cognoscenti) of the logical/computational meaning of such structures: they serve as categorical semantics 
for the {\it multiplicative fragment} of {\it dependent linear/quantum homotopy type theory}
\cite[\S 3.2]{Schreiber14}\cite[\S 2.4]{Riley22}\cite{QS}; for more on this perspective see \cref{DiscussionAndOutlook} below.

\medskip

Recall from \eqref{KanExtensionAdjointTriple}
that for $\mathbf{f} : \mathbf{X} \xrightarrow{\phantom{-}} \mathbf{Y}$ a morphism in $\mathrm{sSet}\mbox{-}\mathrm{Cat}$ we have an associated adjoint triple of 
$\mathrm{sSet}$-enriched base change functors of simplicial local systems (Def. \ref{MonoidalSimplicialFunctorCategoryIntoSimplicialChainComplexes}):
\begin{equation}
  \label{BaseChangeAdjointTripleOnSimplicialLocalSystems}
  \begin{tikzcd}
    \mathbf{sCh}_{\mathbb{K}}^{\mathbf{X}}
    \ar[
      rr,
      shift left=14pt,
      "{ \mathbf{f}_! }"
    ]
    \ar[
      from=rr,
      "{ \mathbf{f}^\ast }"{description}
    ]
    \ar[
      rr,
      shift right=14pt,
      "{ \mathbf{f}_\ast }"{swap}
    ]
    \ar[
      rr,
      phantom,
      shift left=8pt,
      "{ \scalebox{.7}{$\bot$} }"
    ]
    \ar[
      rr,
      phantom,
      shift right=8pt,
      "{ \scalebox{.7}{$\bot$} }"
    ]
    &&
    \mathbf{sCh}_{\mathbb{K}}^{\mathbf{Y}}
  \end{tikzcd}
\end{equation}
given by precomposition $\mathbf{f}^\ast \,\defneq\, (-) \circ \mathbf{f}$ and its left $\mathbf{f}_!$ and right $\mathbf{f}_\ast$ 
Kan extension, respectively \eqref{CoEndFormulasForKanExtension}; and that for each $\mathbf{X}$ we have symmetric closed monoidal 
category structure \eqref{InternalHomAdjunctionForSimplicialLocalSystems}:
\begin{equation}
  \label{ClosedMonoidalStructureRecalled}
  \mathbf{X}
  \,\in\,
  \mathrm{sSet}\mbox{-}\mathrm{Grpd}
  ,\;\;
  \mathscr{V}
  \,\in\,
  \mathbf{sCh}^{\mathbf{X}}_{\mathbb{K}}
  \hspace{1.1cm}
    \vdash
  \hspace{1.1cm}
  \begin{tikzcd}
    \mathbf{sCh}_{\mathbb{K}}^{\mathbf{X}}
    \ar[
      rr,
      shift left=6pt,
      "{ \mathscr{V} \otimes_{\mathbf{X}} (\mbox{-}) }"
    ]
    \ar[
      from=rr,
      shift left=6pt,
      "{ [\mathscr{V} ,\,\mbox{-}]_{\mathbf{X}} }"
    ]
    \ar[
      rr,
      phantom,
      "{ \scalebox{.7}{$\bot$} }"
    ]
    &&
    \mathbf{sCh}_{\mathbb{K}}^{\mathbf{X}} \;.
  \end{tikzcd}
\end{equation}

\begin{proposition}[Frobenius reciprocity for simplicial local systems]
\label{FrobeniusReciprocityForSimplicialLocalSystems}
  Pullback of simplicial local systems along maps of simplicial groupoids 
  \[
    \begin{tikzcd}[row sep=-1pt]
      \mathbf{X}
      \ar[rr, "{ \mathbf{f} }"]
      &&
      \mathbf{Y}
      \\
      \mathbf{sCh}^{\mathbf{X}}_{\mathbb{K}}
      \ar[from=rr, "{ \mathbf{f}^\ast }"{swap}]
      &&
      \mathbf{sCh}^{\mathbf{Y}}_{\mathbb{K}}
    \end{tikzcd}
  \]
  is: 
  \begin{itemize}
    \item[{\bf(i)}] strong monoidal, in that there are 
      natural isomorphisms of this form:
      \begin{equation}
        \label{PullbackOfSimplicialLocalSystemsPreservesTensorUnit}
        \mathrlap{
          \mathbf{f}^\ast(\mathbbm{1})
          \;\simeq\;
          \mathbbm{1}
        }
      \end{equation}
      \begin{equation}
      \label{PullbackOfSimplicialLocalSystemsIsStrongMonoidal}
      \hspace{-4cm}
      \mathllap{
      \mathscr{V},\, \mathscr{W}
      \;\in\;
      \mathbf{sCh}^{\mathbf{Y}}_{\mathbb{K}}
      \hspace{1cm}
      }
      \vdash
      \mathrlap{
      \hspace{1cm}
      \mathbf{f}^\ast(\mathscr{V} \otimes_{\mathbf{Y}} \mathscr{W})
      \;\;\simeq\;\;
      (\mathbf{f}^\ast \mathscr{V}) 
      \otimes_{\mathbf{X}} (f^\ast\mathscr{W})
      }
      \end{equation}
    \item[{\bf(ii)}]  strong closed, in that there are 
      natural isomorphisms of this form:
      \begin{equation} 
      \label{StrongClosedPullback}
      \hspace{-4cm}
      \mathllap{
      \mathscr{V},\, \mathscr{W}
      \;\in\;
      \mathbf{sCh}^{\mathbf{Y}}_{\mathbb{K}}
      \hspace{1cm}
      }
      \vdash
      \mathrlap{
      \hspace{1cm}
      \mathbf{f}^\ast [\mathscr{V},\,\mathscr{W}]_{\mathbf{Y}}
      \;\;\simeq\;\;
      [\mathbf{f}^\ast \mathscr{V},\, \mathbf{f}^\ast\mathscr{W}]_{\mathbf{X}}
      }
      \end{equation}
    \item[{\bf(iii)}] 
    and satisfies projection, in that there are natural isomorphisms of this form:
    \begin{equation}
      \label{ProjectionFormulaForSimplicialLocalSystems}
      \hspace{-4cm}
      \mathllap{
      \mathscr{R}
      \,\in\,
      \mathbf{sCh}^{\mathbf{X}}_{\mathbb{K}}
      ,\;
      \mathscr{V}
      \,\in\,
      \mathbf{sCh}^{\mathbf{Y}}_{\mathbb{K}}
      \hspace{1cm}
      }
      \vdash
      \mathrlap{
      \hspace{1cm}
     \mathbf{f}_!\big(
      \mathscr{R}
      \otimes_{\mathbf{X}}
      \mathbf{f}^\ast \mathscr{V}
    \big)
    \;\;
    \simeq
    \;\;
    (\mathbf{f}_! \mathscr{R})
    \otimes_{\mathbf{Y}}
    \mathscr{V}\;.
    }
    \end{equation}
  \end{itemize}
\end{proposition}
\begin{proof}
 By the adjoint equivalences
 \eqref{SkeletizationOfLocalSystemsAlongAMap} it is sufficient to check this for maps between skeletal simplicial groupoids.
 That in this case precomposition $\mathbf{f}^\ast$ is a strong monoidal closed functor
 is manifest by Prop. \ref{ClosedMonoidalStructureOfLocalSystemsOverSimplicialDelooping}. The projection formula then follows by
 adjointness (cf. \cite{FHM03}), namely for any $\mathscr{W} \,\in\, \mathbf{sCh}^{\mathbf{Y}}_{\mathbb{K}}$ we have natural isomorphisms
 \[
   \def\arraystretch{2}
   \begin{array}{lll}
     \mathbf{sCh}_{\mathbb{K}}^{\mathbf{Y}}
     \Big(
       \mathbf{f}_!
       \big(
         \mathscr{R} 
           \otimes_{\mathbf{X}}
         \mathbf{f}^\ast
         \mathscr{V}
       \big)
       ,\,
       \mathscr{W}
     \Big)
     &
     \;\simeq\;
     \mathbf{sCh}_{\mathbb{K}}^{\mathbf{X}}
     \Big(
         \mathscr{R} 
           \otimes_{\mathbf{X}}
         \mathbf{f}^\ast
         \mathscr{V}
       ,\,
       \mathbf{f}^\ast
       \mathscr{W}
     \Big)
     &
     \proofstep{
       by \eqref{BaseChangeAdjointTripleOnSimplicialLocalSystems}
     }
     \\
   &  \;\simeq\;
     \mathbf{sCh}_{\mathbb{K}}^{\mathbf{X}}
     \Big(
         \mathscr{R} 
       ,\,
       \big[
         \mathbf{f}^\ast
         \mathscr{V}
         ,\,
         \mathbf{f}^\ast
         \mathscr{W}
       \big]_{\mathbf{X}}
     \Big)
     &
     \proofstep{
       by \eqref{ClosedMonoidalStructureRecalled}
     }
     \\
    & \;\simeq\;
     \mathbf{sCh}_{\mathbb{K}}^{\mathbf{X}}
     \Big(
         \mathscr{R} 
       ,\,
       \mathbf{f}^\ast
       \big[
         \mathscr{V}
         ,\,
         \mathscr{W}
       \big]_{\mathbf{Y}}
     \Big)
     &
     \proofstep{
       by
       \eqref{StrongClosedPullback}
     }
     \\
   &  \;\simeq\;
     \mathbf{sCh}_{\mathbb{K}}^{\mathbf{Y}}
     \Big(
       \mathbf{f}_!
         \mathscr{R} 
       ,\,
       \big[
         \mathscr{V}
         ,\,
         \mathscr{W}
       \big]_{\mathbf{Y}}
     \Big)
     &
     \proofstep{
       by
       \eqref{BaseChangeAdjointTripleOnSimplicialLocalSystems}
     }
     \\
     &   
     \;\simeq\;
     \mathbf{sCh}_{\mathbb{K}}^{\mathbf{Y}}
     \big(
       (
       \mathbf{f}_!
         \mathscr{R} 
       )
         \otimes_{\mathbf{Y}}
         \mathscr{V}
       ,\,
         \mathscr{W}
     \big)
     &
     \proofstep{
       by 
       \eqref{ClosedMonoidalStructureRecalled},
     }
   \end{array}
 \]
  and since these are natural in $\mathscr{W}$, the projection formula \eqref{ProjectionFormulaForSimplicialLocalSystems} follows by the Yoneda Lemma.
\end{proof}

We will need the Beck-Chevalley condition for simplicial local systems, but it will be sufficient to consider the following very special cases 
(the first is item (b) in \cite[p. 511]{Seely83}, for $B = \ast$).

To that end, we denote the projections out of a  Cartesian product of simplicial groupoids as follows:
\[
   \mathbf{X} ,\, \mathbf{Y} 
   \;\in\; 
   \mathrm{sSet}\mbox{-}\mathrm{Grpd}
   \hspace{1.2cm}
     \vdash
   \hspace{1.2cm}
   \begin{tikzcd}[row sep=2pt, column sep=25pt] 
    & 
    \mathbf{X} \times \mathbf{Y}
    \ar[
      dl, 
      "{ \mathrm{pr}_{\mathbf{X}} }"{swap, pos=.35}
    ]
    \ar[
      dr, 
      "{ \mathrm{pr}_{\mathbf{Y}} }"{pos=.35}
    ]
    \\
    \mathbf{X}
    &&
    \mathbf{Y}
    \mathrlap{\,.}
  \end{tikzcd}
\]

\begin{lemma}[Beck-Chevalley for simplicial local systems along product projections]
\label{BeckChevalleyForSimplicialLocalSystemsAlongProductProjections}
Given a diagram in $\mathrm{sSet}\mbox{-}\mathrm{Grpd}$ of the form
\[
  \begin{tikzcd}[column sep=large]
    \mathbf{X} \times \mathbf{Y}
    \ar[
      r,
      "{ \mathbf{f} \times \mathrm{id} }"
    ]
    \ar[
      d,
      "{ \mathrm{pr}_{\mathbf{X}} }"
    ]
    &
    \mathbf{X}' \times \mathbf{Y}
    \ar[
      d,
      "{ \mathrm{pr}_{\mathbf{X}'} }"
    ]
    \\
    \mathbf{X}
    \ar[r, "{ \mathbf{f} }"]
    &
    \mathbf{Y}
  \end{tikzcd}
\]
then the following two ways of pull/push 
\eqref{BaseChangeAdjointTripleOnSimplicialLocalSystems}
of local systems through this diagram  are naturally isomorphic:
\begin{equation}
  \label{BeckChevalleyAlongProducts}
  (f \times \mathrm{id})_!
  \circ 
  \mathrm{pr}_{\mathbf{X}}^\ast
  \;\simeq\;
  \mathrm{pr}_{\mathbf{X}'}^\ast
    \circ 
  f_!
  \;\;
  :
  \;\;
  \mathbf{sCh}^{\mathbf{X}}_{\mathbb{K}}
  \longrightarrow
  \mathbf{sCh}^{\mathbf{X}' \times \mathbf{Y}}_{\mathbb{K}}.
\end{equation}
\end{lemma}
\begin{proof}
By coend calculus \eqref{CoEndFormulasForKanExtension}, we have for $(x',y) \,\in\, \mathbf{X}' \times \mathbf{Y}$ the following sequence of natural isomorphisms:
  \begin{align*}
    \big(
    (\mathbf{f} \times \mathrm{id})_!
    \,\circ\, 
    \mathrm{pr}_{\mathbf{X}}^\ast
    \mathscr{V}
    \big)_{(x', y)}
    &
    \;\simeq\;
    \int^{ (x, y_0) \in \mathbf{X} \times \mathbf{Y} }
    \mathbf{X}\big(\mathbf{f}(x),x'\big)
    \times
    \mathbf{Y}\big(\mathrm{id}(y_0), y\big)
    \cdot
    \big(
      \mathrm{pr}^\ast_{\mathbf{X}} 
      \mathscr{V}
    \big)_{ (x,y_0) }
    \\
  &  \;\simeq\;
    \int^{ (x, y_0) \in \mathbf{X} \times \mathbf{Y} }
    \mathbf{X}\big(\mathbf{f}(x),x'\big)
    \times
    \mathbf{Y}(y_0, y)
    \cdot
    \mathscr{V}_{\!x}
    \\
  &  \;\simeq\;
    \int^{y_0 \in \mathbf{Y}}
    \int^{x \in \mathbf{X}}
    \mbox{\bf{Y}}(y_0, y)
    \times
    \mbox{\bf{X}}\big(\mathbf{f}(x),x'\big)
    \cdot
    \mathscr{V}_{\!x}
    \\
  &  \;\simeq\;
    \int^{y_0 \in \mathbf{Y}}
    \mbox{\bf{Y}}(y_0, y)
    \cdot
    \int^{x \in \mathbf{X}}
    \mbox{\bf{X}}\big(\mathbf{f}(x),x'\big)
    \cdot
    \mathscr{V}_{\!x}
    \\
  &  \;\simeq\;
    \int^{y_0 \in \mathbf{Y}}
    \mbox{\bf{Y}}(y_0, y)
    \cdot
    (\mathbf{f}_! \mathscr{V})_{x'}
    \\
 &   \;\simeq\;
    (\mathbf{f}_! \mathscr{V})_{x'}
    \\
 &   \;\simeq\;
    \big(
    \mathrm{pr}_{\mathbf{X}'}
    (\mathbf{f}_! \mathscr{V})
    \big)_{(x',y)}\;.
    \qedhere
  \end{align*}
\end{proof}

\begin{lemma}[Beck-Chevalley for simplicial local systems along embeddings]
\label{BeckChevalleyForSimplicialLocalSystemsAlongEmbeddings}
Given a diagram in $\SimplicialGroupoids$ of the form
\begin{equation}
  \label{PullbackOfInjectionIntoSimplicialGroupoid}
  \begin{tikzcd}
    \mathbf{X}'
    \ar[rr, "{ \mathbf{f}' }"]
    \ar[
      d, 
      hook, "{ \iota_{\mathbf{X}} }"
    ]
    \ar[
      drr,
      phantom,
      "{
        \scalebox{.7}{
          \rm (pb)
        }
      }"
    ]
    &&
    \mathbf{Y}'
    \ar[
      d, 
      hook, "{ \iota_{\mathbf{Y}} }"
    ]
    \\
    \mathbf{X}
    \ar[rr, "{ \mathbf{f} }"]
    &&
    \mathbf{Y}
    \,,
  \end{tikzcd}
\end{equation}
where the vertical maps are full simplicial sub-groupoid embedding (injective on objects, fully faithful on hom-objects), then the following two ways of pull/push of simplicial local systems through this diagram \eqref{BaseChangeAdjointTripleOnSimplicialLocalSystems} are naturally isomorphic:
\begin{equation}
  \label{BeckChevalleyAlongEmbeddings}
  \mathbf{f}'_!
  \,\circ\,
  \iota_{\mathbf{X}}^\ast
  \;\;
  \simeq
  \;\;
  \iota_{\mathbf{Y}}
  \,\circ\,
  \mathbf{f}_!
  \;\;:\;\;
  \mathbf{sCh}^{\mathbf{X}}_{\GroundField}
  \xrightarrow{\phantom{--}}
  \mathbf{sCh}^{\mathbf{Y}'}_{\GroundField}
  \,.
\end{equation}
\end{lemma}
\begin{proof}
It is clear that the statement follows as soon as it holds for $\mathbf{X}$ and $\mathbf{Y}$ replaced by the connected components of the images of $\iota_{\mathbf{X}}$ and $\iota_{\mathbf{Y}}$, respectively (this is where the assumption enters that \eqref{PullbackOfInjectionIntoSimplicialGroupoid} is a pullback, hence that $\mathbf{f}'$ is the corestriction of $\mathbf{f}$ to $\mathbf{Y}'$). Therefore we may assume, without restriction, that $\iota_{\mathbf{X}}$ and $\iota_{\mathbf{Y}}$ are {\it in addition} also essentially surjective on objects. 

In this case, since they are also assumed to be fully faithful, they are in fact equivalences of enriched categories and as such are {\it final} as enriched functors, meaning that their precomposition preserves weighted colimits such as coends (\cite[\S 4.5]{Kelly82}, see also \cite{Campion20}):
\[
  \def\arraystretch{1.6}
  \begin{array}{ll}
    \big(
      \mathbf{f}'_!
      \circ
      \iota_{\mathbf{X}}^\ast(\mathscr{V})
    \big)_{y'}
    \\
    \;\simeq\; 
    \int^{x' \in \mathbf{X}'}
    \mathbf{Y}'\big(\mathbf{f}'(x'), y'\big)
    \cdot
    (\iota_{\mathbf{X}}^\ast\mathscr{V})_{x'}
    &
    \proofstep{
      \eqref{CoEndFormulasForKanExtension}
    }
    \\
    \;\simeq\; 
    \int^{x' \in \mathbf{X}'}
    \mathbf{Y}\Big(
      \iota_{\mathbf{Y}} 
      \big( 
        \mathbf{f}'(x')
      \big)
      ,
    \iota_{\mathbf{Y}}(y')
    \Big)
    \cdot
    \mathscr{V}_{\iota_{\mathbf{X}}(x')}
    &
    \proofstep{
      $\iota_{\mathbf{Y}}$
      fully faithful
    }
    \\
    \;\simeq\;
    \int^{x' \in \mathbf{X}'}
    \mathbf{Y}\Big(
      \mathbf{f}\big(
        \iota_{\mathbf{X}}(x')
      \big), 
      \iota_{\mathbf{Y}}(y')
    \Big)
    \cdot
    \mathscr{V}_{\iota_{\mathbf{X}}(x')}
    &
    \proofstep{
      \eqref{PullbackOfInjectionIntoSimplicialGroupoid}
    }
    \\
    \;\simeq\;
    \int^{x \in \mathbf{X}}
    \mathbf{Y}\Big(
      \mathbf{f}(
        x
      ), 
      \iota_{\mathbf{Y}}(y')
    \Big)
    \cdot
    \mathscr{V}_{x}
    &
    \proofstep{ finality of $\iota_{\mathbf{X}}$ }
    \\
    \;\simeq\;
    \big(
      \iota_{\mathbf{X}}^\ast
      \circ
      \mathbf{f}_! \mathscr{V}
    \big)_{ y' }
    &
    \proofstep{
      \eqref{CoEndFormulasForKanExtension}.
    }
  \end{array}
\]
These natural isomorphisms establish the claim.
\end{proof}

\noindent
{\bf External tensor of simplicial local systems.}

\begin{definition}[External tensor product of simplicial local systems]
\label{ExternalTensorProductOfSimplicialLocalSystems}
$\,$ \newline
The {\it external tensor product} on the category of simplicial local systems \eqref{Loc} is the following functor:
\begin{equation}
  \label{FormulaForExternalTensorOfSimplicialLocalSystems}
  \begin{tikzcd}[row sep=-25pt, column sep=large]
    \mathbf{Loc}_{\mathbb{K}}
    \times
    \mathbf{Loc}_{\mathbb{K}}
    \ar[rr, "{ \boxtimes }"]
    &&
    \mathbf{Loc}_{\mathbb{K}}
    \\[20pt]
    \big(
      \mathscr{V}_{\!\mathbf{X}}
      ,\,
      \mathscr{W}_{\!\mathbf{Y}}
    \big)
    \ar[
      d,
      "{
        (
          \phi_f
          ,\,
          \gamma_g
        )
      }"{left}
    ]
    &\longmapsto&
    \Big(
    \big(
      \mathrm{pr}_{\mathbf{X}}^\ast
        \mathscr{V}
    \big)
    \otimes_{\mathbf{X} \times \mathbf{Y}}
    \big(
      \mathrm{pr}_{\mathbf{Y}}^\ast
        \mathscr{W}
    \big)    
    \Big)_{\mathbf{X} \times \mathbf{Y}}
    \ar[
      d,
      "{
        \big(
        (
          \mathrm{pr}_{\mathbf{X}}^\ast
          \phi
        )
        \otimes_{\mathbf{X} \times \mathbf{Y}}
        (
          \mathrm{pr}_{\mathbf{Y}}^\ast
          \gamma
        )
        \big)_{ f \times g }
      }"{right}
    ]
    \\[+50pt]
    \big(
      \mathscr{V}'_{\!\mathbf{X}'}
      ,\,
      \mathscr{W}'_{\!\mathbf{Y}'}
    \big)
    &\longmapsto&
    \Big(
    \big(
      \mathrm{pr}_{\mathbf{X}'}^\ast
        \mathscr{V}'
    \big)
    \otimes_{\mathbf{X}' \times \mathbf{Y}'}
    \big(
      \mathrm{pr}_{\mathbf{Y}'}^\ast
        \mathscr{W}'
    \big)    
    \Big)_{
      \mathbf{X}' \times \mathbf{Y}'
      \mathrlap{\,,}
    }
  \end{tikzcd}
\end{equation}
where on the right we are leaving the strong monoidal structure isomorphism notationally implicit. In more detail, 
the component map of the morphism on the right is this composite:
\[
  \begin{tikzcd}[column sep=53pt]
  \big(
    \mathrm{pr}_{\mathbf{X}}^\ast
      \mathscr{V}
  \big)
  \otimes_{\mathbf{X} \times \mathbf{Y}}
  \big(
    \mathrm{pr}_{\mathbf{Y}}^\ast
      \mathscr{W}
  \big)      
  \ar[
    r,
    "{
      (
      \mathrm{pr}_{\mathbf{X}}^\ast
      \phi
      )
      \otimes
      (
      \mathrm{pr}_{\mathbf{Y}}^\ast
      \gamma
      )
    }"{yshift=2pt}
  ]
  &
  \Big(\!
  \big(\,
  \underset{
    \mathclap{
      (f\times g)^\ast 
      \mathrm{pr}_{\mathbf{X}'}^\ast   
    }
  }{
  \underbrace{
    \mathrm{pr}_{\mathbf{X}}^\ast
    f^\ast
  }
  }
    \mathscr{V}'
  \big)
  \otimes_{\mathbf{X} \times \mathbf{Y}}
  \big(\,
  \underset{
    \mathclap{
      (f\times f')^\ast 
      \mathrm{pr}_{\mathbf{Y}'}^\ast   
    }
  }{
  \underbrace{
    \mathrm{pr}_{\mathbf{Y}}^\ast
    g^\ast
  }
  }
      \mathscr{W}'
  \big)      
 \! \Big)
  \,\simeq\,
  (f \times g)^\ast
  \Big(\!
  \big(
    \mathrm{pr}_{\mathbf{X}'}^\ast
      \mathscr{V}'
  \big)
  \otimes_{\mathbf{X}' \times \mathbf{Y}'}
  \big(
    \mathrm{pr}_{\mathbf{Y}'}^\ast
      \mathscr{W}'
  \big)      
  \!\Big)
  \,,
  \end{tikzcd}
\]
where the under-braces indicate equalities exhibiting the commutativity of this diagram
\begin{equation}
  \label{MapOfProductsOfSimplicialGroupoids}
  \begin{tikzcd}[
    row sep=15pt,
    column sep=40pt
  ]
    \mathbf{X}
    \ar[r, "{ f }"]
    &
    \mathbf{X}'
    \\
    \mathbf{X}\times \mathbf{Y}
    \ar[u, "{ \mathrm{pr}_{\mathbf{X}} }"{left}]
    \ar[d, "{ \mathrm{pr}_{\mathbf{Y}} }"{left}]
    \ar[
      r,
      "{ f \times g }"
    ]
    &
    \mathbf{X}' \times \mathbf{Y}'
    \ar[u, "{ \mathrm{pr}_{\mathbf{X}'} }"{right}]
    \ar[d, "{ \mathrm{pr}_{\mathbf{Y}'} }"{right}]
    \\
    \mathbf{Y}
    \ar[r, "{ g }"]
    &
    \mathbf{Y}'
    \mathrlap{\,,}
  \end{tikzcd}
\end{equation}
while the isomorphism ``$\simeq$'' is the strong monoidal structure of $(f \times f)^\ast$ 
(which is in fact still an actual equality, objectwise).
\end{definition}

\begin{remark}[External tensor with a unit system]
Since pullback preserves tensor units \eqref{PullbackOfSimplicialLocalSystemsPreservesTensorUnit}, 
the pullback of a simplicial local system to a Cartesian product is isomorphic to its external tensor product \eqref{FormulaForExternalTensorOfSimplicialLocalSystems}
with the unit system on the other factor:
\begin{equation}
  \label{ExternalTensorWithUnitIsPullback}
  \mathscr{V}_{\mathbf{X}} 
    \boxtimes 
  \mathbb{K}_{\mathbf{Y}}
  \;\simeq\;
  \big(
  \mathrm{pr}_{\mathbf{X}}^\ast\mathscr{V}
  \big)_{\mathbf{X} \times \mathbf{Y}}
  \,.
\end{equation}
In particular, the external tensor product in general may be expressed in terms of external tensoring with tensor units, as:
\begin{equation}
  \label{ExternalTensorFromExternalTensorWithUnits}
  \mathscr{V}_{\mathbf{X}}
  \,\boxtimes\,
  \mathscr{W}_{\mathbf{Y}}
  \;\;\simeq\;\;
  \big(
    \mathscr{V}_{\mathbf{X}}
    \,\boxtimes\,
    \mathbb{K}_{\mathbf{Y}}
  \big)_{\mathbf{X} \times \mathbf{Y}}
  \;\otimes_{\mathbf{X} \times \mathbf{Y}}\;
  \big(
    \mathbb{K}_{\mathbf{X}}
    \,\boxtimes\,
    \mathscr{W}_{\mathbf{Y}}
  \big)_{\mathbf{X} \times \mathbf{Y}}\;.
\end{equation}
\end{remark}

\begin{lemma}[Pull/push of external tensors along product maps]
Given maps of simplicial groupoids
$\mathbf{f} \,:\, \mathbf{X} \to \mathbf{X}'$
and
$\mathbf{g} \,:\, \mathbf{Y} \to \mathbf{Y}'$
there are natural isomorphisms
\footnote{The isomorphism \eqref{PullbackOfExternalTensorAlongProductOfFunctions} is a fairly immediate consequence of
\eqref{PullbackOfSimplicialLocalSystemsPreservesTensorUnit}, but \eqref{PushOfExternalTensorAlongProductOfFunctions} 
is not so immediate, it appears mentioned in generality but without proof in \cite[p. 624]{Shulman13}, while in models for parameterized spectra it appears in \cite[Rem. 2.5.8, Prop. 13.7.2]{MaySigurdsson06}, \cite[Lem. 3.4.1]{Malkiewich19}, \cite[Lem. 2.5.1]{Malkiewich23}.}
\begin{align}
  \label{PullbackOfExternalTensorAlongProductOfFunctions}
  (\mathbf{f} \times \mathbf{g})^\ast
  \big(
    \mathscr{V} \boxtimes \mathscr{W}
  \big)
  \;\;
 & \simeq
  \;\;
  (\mathbf{f}^\ast \mathscr{V})
  \boxtimes
  (\mathbf{g}^\ast \mathscr{W})\;,
\\
  \label{PushOfExternalTensorAlongProductOfFunctions}
  (\mathbf{f} \times \mathbf{g})_!
  \big(
    \mathscr{V} \boxtimes \mathscr{W}
  \big)
  \;\;
   & \simeq
  \;\;
  (\mathbf{f}_! \mathscr{V})
  \boxtimes
  (\mathbf{g}_! \mathscr{W})\;.
\end{align}
\end{lemma}

\begin{proof}
  The isomorphism \eqref{PullbackOfExternalTensorAlongProductOfFunctions} may be obtained as the following composite of natural isomorphisms:
  \[
    \def\arraystretch{1.6}
    \begin{array}{lll}
      (\mathbf{f} \times \mathbf{g})^\ast
      \big(
        \mathscr{V} 
        \,\boxtimes\,
        \mathscr{W}
      \big)
      &
      \;\simeq\;
      (\mathbf{f} \times \mathbf{g})^\ast
      \Big(
        \big(
          \mathrm{pr}_{\mathbf{Y}}^\ast 
          \mathscr{V}
        \big)
        \,\otimes_{\mathbf{Y} \times \mathbf{Y}'}\,
        \big(
          \mathrm{pr}_{\mathbf{Y}'}^\ast 
          \mathscr{W}
        \big)        
      \Big)
      &
      \proofstep{
        by 
        \eqref{FormulaForExternalTensorOfSimplicialLocalSystems}
      }
      \\
   &   \;\simeq\;
      \big(
        (\mathbf{f} \times \mathbf{g})^\ast
        \mathrm{pr}_{\mathbf{Y}}^\ast
        \mathscr{V}
      \big)
      \otimes_{\mathbf{X} \times \mathbf{X}'}
      \big(
        (\mathbf{f} \times \mathbf{g})^\ast
        \mathrm{pr}_{\mathbf{Y}'}^\ast
        \mathscr{W}
      \big)
      &
      \proofstep{
        by
        \eqref{PullbackOfSimplicialLocalSystemsIsStrongMonoidal}
      }
      \\
   &   \;\simeq\;
      \big(
        \mathrm{pr}_{\mathbf{X}}^\ast
        \mathbf{f}^\ast
        \mathscr{V}
      \big)
      \otimes_{\mathbf{X} \times \mathbf{X}'}
      \big(
        \mathrm{pr}_{\mathbf{X}'}^\ast
        \mathbf{g}^\ast
        \mathscr{W}
      \big)
      &
      \proofstep{
        by \eqref{MapOfProductsOfSimplicialGroupoids}
        \&
        \eqref{CompositionalCoherenceIsomorphism}
      }
      \\
    &  \;\simeq\;
      \big(\mathbf{f}^\ast \mathscr{V}\big)
      \boxtimes
      \big(\mathbf{g}^\ast \mathscr{W}\big)
      &
      \proofstep{
        by 
        \eqref{FormulaForExternalTensorOfSimplicialLocalSystems}
      }.
    \end{array}
  \]
    For the second isomorphism \eqref{PushOfExternalTensorAlongProductOfFunctions}, first observe the special case
  \begin{equation}
    \label{PushExternalTensorAlongFTimesId}
    \def\arraystretch{1.7}
    \begin{array}{lll}
      (\mathbf{f} \times \mathrm{id})_!
      \big(
        \mathscr{V}_{\mathbf{X}}
        \boxtimes
        \mathbb{K}_{\mathbf{Y}}
      \big)
      &
      \;\simeq\;
      \Big(
      (\mathbf{f} \times \mathrm{id})_!
      \big(
      \mathrm{pr}_{\mathbf{X}}^\ast
      \mathscr{V}
      \big)_{\mathbf{X}' \times \mathbf{Y}'}
      \Big)
      &
      \proofstep{
        by \eqref{ExternalTensorWithUnitIsPullback}
      }
      \\
     & \;\simeq\;
      \big(
      \mathrm{pr}_{\mathbf{Y}}^\ast
      (\mathbf{f}_! \mathscr{V})
      \big)_{\mathbf{X}' \times \mathbf{Y}}
      & 
      \proofstep{
        by \eqref{BeckChevalleyAlongProducts}
      }
      \\
     & \;\simeq\;
      (\mathbf{f}_! \mathscr{V})_{\mathbf{X}'}
        \boxtimes 
      \mathbb{K}_{\mathbf{Y}}
      &
      \proofstep{
        by \eqref{ExternalTensorWithUnitIsPullback}
      }
    \end{array}
  \end{equation}
  from which we obtain the general isomorphism as the following composite:
  \[
    \def\arraystretch{1.7}
    \begin{array}{lll}
      (\mathbf{f} \times \mathbf{g})_!
      (\mathscr{V} \boxtimes \mathscr{W})
      &
      \;\simeq\;
      (\mathbf{f} \times \mathbf{g})_!      
      \big(
        (\mathscr{V} \boxtimes \mathbb{K}_{\mathbf{Y}})
        \otimes
        (\mathbb{K}_{\mathbf{X}} \boxtimes \mathscr{W})
      \big)
      &
      \proofstep{
        by \eqref{ExternalTensorFromExternalTensorWithUnits}
      }
      \\
&    \;\simeq\;
      (\mathbf{f} \times \mathbf{g})_!
      \Big(
        \big(
          (\mathrm{id} \times g)^\ast
          (\mathscr{V} \boxtimes \mathbb{K}_{\mathbf{Y}'})
        \big)
        \otimes
        (\mathbb{K}_{\mathbf{X}} \boxtimes \mathscr{W})
      \Big)
      &
      \proofstep{
        by
        \eqref{PullbackOfExternalTensorAlongProductOfFunctions}
        \&
        \eqref{PullbackOfSimplicialLocalSystemsPreservesTensorUnit}
      }
      \\
    &  \;\simeq\;
      (\mathbf{f} \times \mathrm{id})_!
      (\mathrm{id} \times \mathbf{g})_!
      \Big(
        \big(
          (\mathrm{id} \times \mathbf{g})^\ast
          (\mathscr{V} \boxtimes \mathbb{K}_{\mathbf{Y}'})
        \big)
        \otimes
        (\mathbb{K}_{\mathbf{X}} \boxtimes \mathscr{W})
      \Big)
      &
      \proofstep{
        by \eqref{CompositionalCoherenceIsomorphism}
      }
      \\
    &  \;\simeq\;
      (\mathbf{f} \times \mathrm{id})_!
      \Big(
          (\mathscr{V} \boxtimes \mathbb{K}_{\mathbf{Y}'})
        \otimes
        \big(
        (\mathrm{id} \times \mathbf{g})_!
        (\mathbb{K}_{\mathbf{X}} \boxtimes \mathscr{W})
        \big)
      \Big)
      &
      \proofstep{
       by \eqref{ProjectionFormulaForSimplicialLocalSystems}
      }
      \\
   &   \;\simeq\;
      (\mathbf{f} \times \mathrm{id})_!
      \Big(
          (\mathscr{V} \boxtimes \mathbb{K}_{\mathbf{Y}'})
        \otimes
        \big(
        \mathbb{K}_{\mathbf{X}} \boxtimes (\mathbf{g}_! \mathscr{W})
        \big)
      \Big)
      &
      \proofstep{
        by \eqref{PushExternalTensorAlongFTimesId}
      }
      \\
   &   \;\simeq\;
      (\mathbf{f} \times \mathrm{id})_!
      \Big(
          (\mathscr{V} \boxtimes \mathbb{K}_{\mathbf{Y}'})
        \otimes
        (\mathbf{f} \times \mathrm{id})^\ast
        \big(
        \mathbb{K}_{\mathbf{X}'} \boxtimes (\mathbf{g}_! \mathscr{W})
        \big)
      \Big)
      &
      \proofstep{
        by
        \eqref{PullbackOfExternalTensorAlongProductOfFunctions}
        \& 
        \eqref{PullbackOfSimplicialLocalSystemsPreservesTensorUnit}
      }
      \\
  &    \;\simeq\;
      \big(
          (\mathbf{f} \times \mathrm{id})_!
          (\mathscr{V} \boxtimes \mathbb{K}_{\mathbf{Y}'})
      \big)
        \otimes
      \big(
        \mathbb{K}_{\mathbf{X}'} \boxtimes (\mathbf{g}_! \mathscr{W})
      \big)
      &
      \proofstep{
        by
        \eqref{ProjectionFormulaForSimplicialLocalSystems}
      }
      \\
 &     \;\simeq\;
      \big(
        (\mathbf{f}_!\mathscr{V}) \boxtimes \mathbb{K}_{\mathbf{Y}'}
      \big)
        \otimes
      \big(
        \mathbb{K}_{\mathbf{X}'} \boxtimes (\mathbf{g}_! \mathscr{W})
      \big)
      &
      \proofstep{
        by 
        \eqref{PushExternalTensorAlongFTimesId}
      }
      \\
  &    \;\simeq\;
      (\mathbf{f}_! \mathscr{V})
      \boxtimes
      (\mathbf{g}_! \mathscr{W})
      &
      \proofstep{
        by
        \eqref{ExternalTensorFromExternalTensorWithUnits}.
      }
    \end{array}
  \]
  This concludes the proof.
\end{proof}

As a direct consequence:
\begin{proposition}[Pull-push adjunct of external tensor products]
\label{PullPushAdjunctOfExternalTensorProducts}
The push/pull adjunct of an external tensor product of morphisms into pullbacks
\[
  \phi \boxtimes \gamma
  \,:\,
  \mathscr{V} \boxtimes \mathscr{W}
  \xrightarrow{\phantom{--}}
  (\mathbf{f}^\ast \mathscr{V}')
  \boxtimes 
  (\mathbf{g}^\ast \mathscr{W}')
  \,\simeq\,
  (\mathbf{f} \times \mathbf{g})^\ast
  (\mathscr{V}' \boxtimes \mathscr{W}')
\]
is the external tensor product of the separate adjuncts:
\begin{equation}
  \label{PullPushAdjuncOfExternalTensorMap}
  \widetilde{\phi \boxtimes \gamma}
  \;\simeq\;
  \widetilde \phi \,\boxtimes\, \widetilde \gamma
  \;:\;
  (\mathbf{f} \times \mathbf{g})_!
  (\mathscr{V} \boxtimes \mathscr{W})
  \,\simeq\,
  (\mathbf{f}_! \mathscr{V}) \boxtimes (\mathbf{g}_! \mathscr{W})
  \xrightarrow{\phantom{--}}
  \mathscr{V}' \,\boxtimes\, \mathscr{W}'
  \,.
\end{equation}
\end{proposition}

\begin{proposition}[External pushout-product]
  \label{ExternalPushoutProduct}
  The pushout-product (Def. \ref{PushoutProduct}) of the external tensor (Def. \ref{ExternalTensorProductOfSimplicialLocalSystems}) 
  in $\mathbf{Loc}_{\mathbb{K}}$ is given by the following formula:
  \[
    (\phi_{\mathbf{f}})
    \,\widehat{\boxtimes}\,
    (\gamma_{\mathbf{g}})
    \;\;
      \simeq
    \;\;
    \Big(
    \widetilde{
    \big(
      (\mathrm{pr}_{\mathrm{X}'})^\ast 
      \widetilde{\phi}
    \, \big)
    \,\widehat{\otimes}\,
    \big(
      (\mathrm{pr}_{\mathrm{Y}'})^\ast 
      \widetilde{\gamma}
   \, \big)
    }
    \Big)_{ \mathbf{f} \,\widehat{\times}\, \mathbf{g}}
    \,.
  \]
\end{proposition}
\begin{proof}
  By the general formula for colimits in Grothendieck constructions (Prop. \ref{ColimitsInAGrothendieckConstruction}), the underlying colimit 
  is the Cartesian pushout-product of simplicial groupoids
  \begin{equation}
    \label{UnderlyingCartesianPushoutProductOfExternalPushoutProduct}
    \begin{tikzcd}[row sep=large, column sep=huge]
      \mathbf{X} \times \mathbf{Y}
      \ar[r, "{ \mathrm{id} \,\times\, \mathbf{g} }"{description}]
      \ar[d, "{ \mathbf{f} \,\times\, \mathrm{id} }"{description}]
      \ar[dr, phantom, "{ \scalebox{.6}{(po)} }"]
      &
      \mathbf{X} \times \mathbf{Y}'
      \ar[d, "{ q_r }"{description}]
      \ar[ddr, bend left=15, "{ \mathbf{f} \,\times\, \mathrm{id} }"{sloped, description}]
      &[-30pt]
      \\
      \mathbf{X}' \times \mathbf{Y}
      \ar[r, "{ q_l}"{description}]
      \ar[drr, bend right=15, "{ \mathrm{id} \,\times\, \mathbf{g} }"{sloped, description}]
      &
      \mathbf{f} \,\widehat{\times}\, \mathbf{g}
      \ar[dr, dashed]
      \\[-30pt]
      && \mathbf{X}' \times \mathbf{Y}'
    \end{tikzcd}
  \end{equation}
  and the linear colimiting component map over the dashed morphism is obtained by pushing the separate linear 
components along the  coprojections $q_\cdot$ to form a cospan in $\mathbf{sCh}_{\mathbb{K}}^{\mathbf{f} \,\widehat{\times}\, \mathbf{g}}$, 
whose universal pushout-product morphism, in turn, is the further pushforward of that cocone along the dashed morphism:
\[
  \def\arraystretch{2.15}
  \begin{array}{ll}
  \big(\mathbf{f} \widehat{\times} \mathbf{g}\big)_!
  \Bigg(\!\!
    \bigg(\!\!
      (q_l)_! 
      \Big(
        \big(
          (\mathrm{pr}_{\mathbf{X}'})^\ast 
          \widetilde{\phi} \,
        \big)
        \otimes
        \big(
          (\mathrm{pr}_{\mathbf{Y}})^\ast
          \mathrm{id}_{\scalebox{.7}{$\mathscr{W}$}}
        \big)
      \Big)
   \!\! \bigg)
    \wedge
    \bigg(\!\!
      (q_r)_! 
      \Big(
        \big(
          (\mathrm{pr}_{\mathbf{X}})^\ast
          \mathrm{id}_{\scalebox{.7}{$\mathscr{V}$}}
        \big)
        \otimes
        \big(
          (\mathrm{pr}_{\mathbf{Y}'})^\ast 
          \widetilde{\gamma}
     \,   \big)
      \Big)
   \!\! \bigg)
  \!\! \Bigg)
  &
  \proofstep{by \eqref{PullPushAdjuncOfExternalTensorMap}}
  \\
  \;\simeq\;
  \Bigg(\!\!
    \bigg(\!\!
      \big(\mathbf{f} \widehat{\times} \mathbf{g}\big)_!
      (q_l)_! 
      \Big(
        \big(
          (\mathrm{pr}_{\mathbf{X}'})^\ast 
          \widetilde{\phi}\,
        \big)
        \otimes
        \big(
          (\mathrm{pr}_{\mathbf{Y}})^\ast
          \mathrm{id}_{\scalebox{.7}{$\mathscr{W}$}}
        \big)
      \Big)
   \!\! \bigg)
    \wedge
    \bigg(\!\!
      \big(\mathbf{f} \widehat{\times} \mathbf{g}\big)_!
      (q_r)_! 
      \Big(
        \big(
          (\mathrm{pr}_{\mathbf{X}})^\ast
          \mathrm{id}_{\mathscr{V}}
        \big)
        \otimes
        \big(
          (\mathrm{pr}_{\mathbf{Y}'})^\ast 
          \widetilde{\gamma}
      \,  \big)
      \Big)
   \!\! \bigg)
 \!\! \Bigg)
  &
  \proofstep{by \eqref{BaseChangeAdjointTripleOnSimplicialLocalSystems}}
  \\
  \;\simeq\;
  \Bigg(\!\!
    \bigg(\!\!
      (\mathrm{id} \otimes \mathbf{g})_!
      \Big(\!
        \big(
          (\mathrm{pr}_{\mathbf{X}'})^\ast 
          \widetilde{\phi}\,
        \big)
        \otimes
        \big(
          (\mathrm{pr}_{\mathbf{Y}})^\ast
          \mathrm{id}_{\scalebox{.6}{$\mathscr{W}$}}
        \big)
   \!\!   \Big)
   \!\! \bigg)
    \wedge
    \bigg(\!\!
      (\mathbf{f} \otimes \mathrm{id})_!
      \Big(\!\!
        \big(
          (\mathrm{pr}_{\mathbf{X}})^\ast
          \mathrm{id}_{\scalebox{.7}{$\mathscr{V}$}}
        \big)
        \otimes
        \big(
          (\mathrm{pr}_{\mathbf{Y}'})^\ast 
          \widetilde{\gamma}
       \, \big)
   \!\!   \Big)
   \!\! \bigg)
 \!\! \Bigg)
  &
  \proofstep{by \eqref{UnderlyingCartesianPushoutProductOfExternalPushoutProduct}}
  \\
  \;\simeq\;
  \Bigg(\!\!
    \bigg(\!\!
      \Big(\!\!
            \big(
          (\mathrm{pr}_{\mathbf{X}'})^\ast 
          \widetilde{\phi}\,
        \big)
        \otimes
        \big(
          (\mathrm{pr}_{\mathbf{Y}})^\ast
          \mathrm{id}_{\scalebox{.7}{$\mathbf{f}_!\mathscr{W}$}}
        \big)
      \!\!\Big)
    \!\!\bigg)
    \wedge
    \bigg(\!\!
      \Big(\!\!
        \big(
          (\mathrm{pr}_{\mathbf{X}})^\ast
          \mathrm{id}_{\scalebox{.7}{$\mathbf{f}_!\mathscr{V}$}}
        \big)
        \otimes
        \big(
          (\mathrm{pr}_{\mathbf{Y}'})^\ast 
          \widetilde{\gamma}
       \, \big)
      \Big)
  \!\!  \bigg)
  \!\!\Bigg)
  &
  \proofstep{by 
    \eqref{PushOfExternalTensorAlongProductOfFunctions}
  }
  \\[-3pt]
  \;\simeq\;
  \big(
    (\mathrm{pr}_{\mathbf{X}'})^\ast 
    \widetilde{\phi}
  \, \big)
  \,\displaystyle{\widehat{\otimes}}\,
  \big(
    (\mathrm{pr}_{\mathbf{Y}'})^\ast 
    \widetilde{\gamma}
  \,\big)
  &
  \proofstep{
    by def.
  }
  \end{array}
\]
This concludes the proof.
\end{proof}

\medskip

\noindent
{\bf External internal hom of simplicial local systems.} We discuss the right adjoint to external tensoring with a simplicial local system. Since right adjoints to tensoring functors are called {\it internal homs} this would by default be named the {\it external internal hom}, for better or worse.

\begin{proposition}
  \label{ExternalTensorWithLocalSystemOverDiscreteSpacePreservesColimits}
  The external tensor product (Def. \ref{ExternalTensorProductOfSimplicialLocalSystems})
  with a simplicial local system preserves colimits:
  \begin{equation}
    \label{ExternalTensorWithSystemOverDiscreteSpacePreservingColimits}
    \left.
    \def\arraystretch{1.4}
    \begin{array}{l}
    \mathscr{W}_{\mathbf{Y}} \,\in\, \mathbf{Loc}_{\mathbb{K}},
    \\
    \mathcal{V}_{\mathbf{X}}
    \,:\, I
    \xrightarrow{\;}
    \mathbf{Loc}_{\mathbb{K}}
    \end{array}
    \right\}
    \hspace{1cm}
      \vdash
    \hspace{1cm}
    \Big(
      \underset{\underset{i \in I}{\longrightarrow}}{\lim}
      \mathscr{V}(i)_{\mathbf{X}_i}
    \Big)
    \,\boxtimes\,
    \mathscr{W}_{\mathbf{Y}}
    \;\;\;
    \simeq
    \;\;\;
    \underset{\underset{i \in I}{\longrightarrow}}{\lim}    
    \big(
      \mathscr{V}(i)_{\mathbf{X}_i}
     \,\boxtimes\,
    \mathscr{W}_{\mathbf{Y}}
    \big).
  \end{equation}
\end{proposition}
\begin{proof}
First notice that the statement holds for the underlying colimit in $\mathrm{sSet}\mbox{-}\mathrm{Grpd}$, since $(-) \times \mathbf{Y}$ is a left adjoint (Prop. \ref{CartesianClosureOfSSetEnrichedGroupoids}):
\[
  \big(
  \underset{\underset{j \in I}{\longrightarrow}}{\lim}
  \;
  \mathbf{X}_j
  \big)
  \times
  \mathbf{Y}
  \;\;\;\;
  \simeq
  \;\;\;\;
  \underset{\underset{j \in I}{\longrightarrow}}{\lim}
  \big(
    \mathbf{X}_j
    \times
    \mathbf{Y}
  \big)
  \,.
\]
Now denoting the coprojections of this  underlying colimit by:
\begin{equation}
  \label{TowardsProvingExternalTensorPreservesColimits}
  \begin{tikzcd}
    \mathbf{X}_i
    \ar[
      r,
      "{
        q^{\mathbf{X}_i}
      }"
    ]
    &
    \underset{
      \underset{j \in I}{\longrightarrow}
    }{\lim}
    \,
      \mathbf{X}_j
    \,,
  \end{tikzcd}
  \hspace{2cm}
  \begin{tikzcd}[row sep=15pt]
    \mathbf{X}_i 
    \times
    \mathrm{Y}
    \ar[
      rr,
      "{
        q^{\mathbf{X}_i}
        \,\times\,
        \mathrm{id}
      }"
    ]
    \ar[
      d,
      "{
        \mathrm{pr}_{\mathbf{X}_i}
      }"
    ]
    &&
    \big(
    \underset{\underset{j \in I}{\longrightarrow}}{\lim} \mathbf{X}_j
    \big)
    \times 
    \mathrm{Y}
    \ar[
      d,
      shorten <=-10pt,
      "{
        \scalebox{.7}{$
        \mathrm{pr}_{\underset{\longrightarrow}{\lim} \mathbf{X}}
        $}
      }"
    ]
    \\
    \mathbf{X}_i 
    \ar[
      rr,
      "{
        q^{\mathbf{X}_i}
      }"
    ]
    &&
    \big(
    \underset{\underset{j \in I}{\longrightarrow}}{\lim} \mathbf{X}_j
    \big)
  \end{tikzcd}
\end{equation}
we identify, via Prop. \ref{ColimitsInAGrothendieckConstruction}, the full colimit by the  following sequence of natural isomorphisms:
\[
  \def\arraystretch{1.8}
  \begin{array}{lll}
    \underset{\longrightarrow}{\lim}
    \big(
      \mathscr{V}_{\mathbf{X}}
    \big)
    \,\boxtimes\,
    \mathscr{W}_{\mathbf{Y}}
    &
    \;\simeq\;
    \Big(
      \underset{\longrightarrow}{\lim}
      \,
      (q^{\mathbf{X}})_! \mathscr{V}
    \Big)_{\!
      \underset{\longrightarrow}{\lim}
      \mathbf{X}
    }
    \,\boxtimes\,
    \mathscr{W}_{\mathbf{Y}}
    &
    \proofstep{
      by \eqref{ColimitsInAGrothendieckConstruction}
    }
    \\
 &   \;\simeq\;
    \Big(
    \big(
      (\mathrm{pr}_{
        \underset{\longrightarrow}{\lim}
        \mathbf{X}})^\ast
      \underset{\longrightarrow}{\lim}
      \,
      (q^{\mathbf{X}})_! \mathscr{V}
    \big)
    \otimes
    \big(
    (\mathrm{pr}_{\mathbf{Y}})^\ast
    \mathscr{W}
    \big)
    \Big)_{
      \underset{\longrightarrow}{\lim}
      \mathbf{X}
      \times \mathbf{Y}
    }
    &
    \proofstep{
      by def. \eqref{ExternalTensorProductOfSimplicialLocalSystems}
    }
    \\
 &   \;\simeq\;
    \Big(
    \big(
      \underset{\longrightarrow}{\lim}
      \,
      (\mathrm{pr}_{
        \underset{\longrightarrow}{\lim}
        \mathbf{X}})^\ast
      \,
      (q^{\mathbf{X}})_! 
      \mathscr{V}
    \big)
    \otimes
    \big(
    (\mathrm{pr}_{\mathbf{Y}})^\ast
    \mathscr{W}
    \big)
    \Big)_{
      \underset{\longrightarrow}{\lim}
      \mathbf{X}
      \times \mathbf{Y}
    }
    &
    \proofstep{
      since $(\mbox{-})^\ast$ is left adjoint
      \eqref{BaseChangeAdjointTripleOnSimplicialLocalSystems}
    }
    \\
  &  \;\simeq\;
    \Big(
    \big(
      \underset{\longrightarrow}{\lim}
      \,
      (
        q^{\mathbf{X}}
        \,\times\,
        \mathrm{id}_{\mathbf{Y}}
      )_! 
      \,
      (\mathrm{pr}_{\mathbf{X}})^\ast
      \mathscr{V}
    \big)
    \otimes
    \big(
    (\mathrm{pr}_{\mathbf{Y}})^\ast
    \mathscr{W}
    \big)
    \Big)_{
      \underset{\longrightarrow}{\lim}
      \mathbf{X}
      \times \mathbf{Y}
    }
    &
    \proofstep{
      by Beck-Chevalley \eqref{BeckChevalleyAlongProducts}
      for \eqref{TowardsProvingExternalTensorPreservesColimits}
    }
    \\
 &   \;\simeq\;
    \bigg(
      \underset{\longrightarrow}{\lim}
    \Big(
    \big(
      (
        q^{\mathbf{X}}
        \,\times\,
        \mathrm{id}_{\mathbf{Y}}
      )_! 
      \,
      (\mathrm{pr}_{\mathbf{X}})^\ast
      \mathscr{V}
    \big)
    \otimes
    \big(
    (\mathrm{pr}_{\mathbf{Y}})^\ast
    \mathscr{W}
    \big)
    \Big)
   \! \bigg)_{
      \underset{\longrightarrow}{\lim}
      \mathbf{X}
      \times \mathbf{Y}
    }
    &
    \proofstep{
      since $(\mbox{-})\otimes\cdots$ preserves colimits
    }
    \\
&    \;\simeq\;
    \bigg(
      \underset{\longrightarrow}{\lim}
      \,
      (
        q^{\mathbf{X}}
        \,\times\,
        \mathrm{id}_{\mathbf{Y}}
      )_! 
    \Big(
    \big(
      \,
      (\mathrm{pr}_{\mathbf{X}})^\ast
      \mathscr{V}
    \big)
    \otimes
    \big(
    (\mathrm{pr}_{\mathbf{Y}})^\ast
    \mathscr{W}
    \big)
    \Big)
    \!\bigg)_{
      \underset{\longrightarrow}{\lim}
      \mathbf{X}
      \times \mathbf{Y}
    }
    &
    \proofstep{
      projection formula 
      \eqref{ProjectionFormulaForSimplicialLocalSystems}
    }
    \\
 &   \;\simeq\;
      \underset{\longrightarrow}{\lim}
    \bigg(\!
    \Big(
    \big(
      \,
      (\mathrm{pr}_{\mathbf{X}})^\ast
      \mathscr{V}
    \big)
    \otimes
    \big(
    (\mathrm{pr}_{\mathbf{Y}})^\ast
    \mathscr{W}
    \big)
    \Big)_{
      \mathbf{X}
      \times 
      \mathbf{Y}
    }    
    \bigg)
    &
    \proofstep{
      by \eqref{ColimitsInAGrothendieckConstruction}
    }
    \\
 &   \;\simeq\;
     \underset{\longrightarrow}{\lim}
     \big(
       \mathscr{V}_{\mathbf{X}}
       \,\boxtimes\,
       \mathscr{W}_{\mathbf{Y}}
     \big)
     &
    \proofstep{
      by def. \eqref{ExternalTensorProductOfSimplicialLocalSystems}\,.
    }
  \end{array}
\]
This concludes the proof.
\end{proof}

\begin{proposition}[Distributive coproducts of simplicial local systems]
  \label{ExternalTensorDistributesOverCoproductsOfSystemsOverDiscreteBaseSpaces}
  Any simplicial local system over a skeletal simplicial groupoid is the coproduct in $\mathbf{Loc}_{\mathbb{K}}$ of 
  its restrictions to the connected components:
  \begin{equation}
    \label{LocalSystemOverDiscreteSpaceIsCoproduct}
    \begin{array}{l}
    \mathbf{Y}
    \,\defneq\,
    \underset{i \in I}{\coprod}
    \mathbf{Y}_i
    \,\in\,
    \mathrm{sSet}\mbox{-}\mathrm{Grpd}_{\mathrm{skl}}
    \,,
    \\
    \mathscr{W}
    \,\in\,
    \mathbf{sCh}_{\mathbb{K}}^{\mathbf{Y}}
    \end{array}
    \hspace{1.2cm}
      \vdash
    \hspace{1.2cm}
    \mathscr{W}_{\mathbf{Y}}
    \;\;
    \simeq
    \;\;
    \underset{i \in I}{\coprod}
    \,
    \mathscr{W}_{\mathbf{Y}_i}
    \;\;\;
    \in
    \;
    \mathbf{Loc}_{\mathbb{K}}
    \,.
  \end{equation}
  and the external tensor product 
  (Def. \ref{ExternalTensorProductOfSimplicialLocalSystems})
  with any 
  $\mathscr{V}_{\mathbf{X}} \,\in\, \mathbf{Loc}_{\mathbb{K}}$ 
  distributes over (these) coproducts:
  \begin{equation}
    \label{ExternalTensorWithSystemOverDiscreteSpace}
    \def\arraystretch{1.6}
    \begin{array}{lll}
      \mathscr{V}_{\mathbf{X}}
      \,\boxtimes\,
      \underset{i \in I}{\coprod}
      \mathscr{W}_{\mathbf{Y}_i}
      \;\simeq\;
      \underset{i \in I}{\coprod}
      \big(
      \mathscr{V}_{\mathbf{X}}
      \,\boxtimes\,
      \mathscr{W}_{\mathbf{Y}_i}
      \big).
    \end{array}
  \end{equation}
\end{proposition}
\begin{proof}
  This is essentially the general phenomenon of free coproduct completion of connected objects (Prop. \ref{CategoriesBeingCoproductCompletionOfTheirConnectedObjects})
  only that base sets are now replaced by skeletal simplicial groupoids. 
  For the record, we spell it out. In the following we discuss binary coproducts just for convenience of notation; the argument immediately generalizes to
  general set-indexed coproducts.
  
First, observe --- readily by adjointness \eqref{BaseChangeAdjointTripleOnSimplicialLocalSystems}, alternatively by the Kan extension formula
\eqref{CoEndFormulasForKanExtension} --- that pushforward of simplicial local systems along a coprojection into a coproduct of simplicial groupoids
is extension by zero to the other connected component:
\[
  \left.
  \begin{array}{l}
    \mathbf{X}_1,\,\mathbf{X}_2
    \,\in\,
    \mathrm{sSet}\mbox{-}\mathrm{Grpd}
    \\
    \mathscr{V} 
      \,\in\,
    \mathbf{sCh}_{\mathbb{K}}^{\mathbf{X}_1}
    ,\,
    \mathscr{V}' 
      \,\in\,
    \mathbf{sCh}_{\mathbb{K}}^{\mathbf{X}_2}
  \end{array}
  \right\}
  \hspace{.6cm}
    \vdash
  \hspace{.6cm}
  \begin{tikzcd}[column sep=35pt]
    \mathbf{X}_1
    \ar[d, hook', "{ q_1 }"{swap}]
    \ar[drr, "{ \mathscr{V} }"]
    \\
    \mathbf{X}_1
    \sqcup
    \mathbf{X}_2
    \ar[
      rr,
      "{
        (q_1)_!
        \mathscr{V}
      }"{pos=.4, description}
    ]
    &&
    \mathbf{sCh}_{\mathbb{K}}    
    \\
    \mathbf{X}_2
    \ar[u, hook, "{ q_2 }"]
    \ar[urr, "{ 0 }"{swap}]
  \end{tikzcd}
  \;\;\;\;
  \mbox{and}
  \;\;\;\;
  \begin{tikzcd}[column sep=35pt]
    \mathbf{X}_1
    \ar[d, hook', "{ q_1 }"{swap}]
    \ar[drr, "{ 0 }"]
    \\
    \mathbf{X}_1
    \sqcup
    \mathbf{X}_2
    \ar[
      rr,
      "{
        (q_2)_!
        \mathscr{V}
      }"{pos=.4, description}
    ]
    &&
    \mathbf{sCh}_{\mathbb{K}}    
    \\
    \mathbf{X}_2
    \ar[u, hook, "{ q_2 }"]
    \ar[urr, "{ \mathscr{V}' }"{swap}]
  \end{tikzcd}
\]
from which it follows that the coproduct of a pair of such push-forwards is given by
\[
  \begin{tikzcd}[column sep=55pt]
    \mathbf{X}_1
    \ar[d, hook', "{ q_1 }"{swap}]
    \ar[drr, "{ \mathscr{V} }"]
    \\
    \mathbf{X}_1
    \sqcup
    \mathbf{X}_2
    \ar[
      rr,
      "{
        (q_1)_!
        \mathscr{V}
        \;\sqcup\;
        (q_2)_!
        \mathscr{V}'
      }"{pos=.4, description}
    ]
    &&
    \mathbf{sCh}_{\mathbb{K}}    
    \,.
    \\
    \mathbf{X}_2
    \ar[u, hook, "{ q_2 }"]
    \ar[urr, "{ \mathscr{V}' }"{swap}]
  \end{tikzcd}
\]
This implies the first formula \eqref{ExternalTensorWithSystemOverDiscreteSpace} by the general formula for colimits in Grothendieck constructions (Prop. \ref{ColimitsInAGrothendieckConstruction}), which gives that 
\begin{equation}
  \label{ComponentFormulaForCoproductsInLoc}
  \def\arraystretch{1.6}
  \begin{array}{ll}
    \big(
      (q_1)_! \mathscr{V}
      \,\sqcup\,
      (q_2)_! \mathscr{V}'
    \big)_{ \mathbf{X}_1 \sqcup \mathbf{X}_2}
    \;\simeq\;
    \mathscr{V}_{\mathbf{X}_1}
    \,\coprod\,
    \mathscr{V}'_{\mathbf{X}_2}
  \end{array}
\end{equation}
Moreover, in the situation
\[
  \begin{tikzcd}[column sep=large, row sep=small]
    & &[-30pt] 
    \mathbf{X} \,\times\, \mathbf{Y}_i
    \ar[dl, "{ \; q_i }"
    ]
    \ar[dr, "{ \mathrm{pr}_{\mathbf{Y}_i} }"]
    \ar[ddll, bend right=33, "{ \mathrm{pr}_{\mathbf{X}} }"{swap}]
    &[-10pt]
    \\
    &
    \mathbf{X} 
      \times
    \big( 
      \mathbf{Y}_1 \sqcup \mathbf{Y}_2
    \big)
    \ar[dr, "{\mathrm{pr}_{\mathbf{Y}_1 \sqcup \mathbf{Y}_2}}"{swap}]
    \ar[dl, "{\;\; \mathrm{pr}_{\mathbf{X}} }"]
    &&
    \mathbf{Y}_i
    \ar[dl, "{\; q_i }"{swap}]
    \\[+10pt]
    \mathbf{X}
    &&
    \mathbf{Y}_1 \sqcup \mathbf{Y}_2
  \end{tikzcd}
\]
it follows that the corresponding Beck-Chevalley condition is satisfied
\begin{equation}
  \label{BeckChevalleyForProductsOfCoproductInclusions}
  (\mathrm{pr}_{\mathbf{Y}_1 \sqcup \mathbf{Y}_2})^\ast
  \circ
  (q_i)_!
  \;\;
  \simeq
  \;\;
  (q_i)_! 
  \circ
  (\mathrm{pr}_{\mathbf{Y}_i})^\ast
\end{equation}
and that
\begin{equation}
  \label{TensorProductWithExtensionByZero}
  \big(
    (\mathrm{pr}_{\mathbf{X}})^\ast
   \mathscr{V}
  \big)
  \,\otimes\,
  \big(
    (q_i)_! (\mathrm{pr}_{\mathbf{Y}_i})^\ast
    \mathscr{W}
  \big)
  \;\;
  \simeq
  \;\;
  (q_i)_! 
  \Big(\!
  \big(
    (\mathrm{pr}_{\mathbf{X}})^\ast
   \mathscr{V}
  \big)
  \,\otimes\,
  \big(
    (\mathrm{pr}_{\mathbf{Y}_i})^\ast
    \mathscr{W}
  \big)
  \!\Big).
\end{equation}
With this, we establish the second statement:
\[
  \def\arraycolsep{-5pt}
  \def\arraystretch{1.7}
  \begin{array}{ll}
  \mathscr{V}_{\mathbf{X}} \,\boxtimes\, 
  \big(
    \mathscr{W}_{\mathbf{Y}_1}
    \coprod
    \mathscr{W}'_{\mathbf{Y}_2}
  \big)
  \\
  \;\simeq
  \Big(\!
    \big(
      (\mathrm{pr}_{\mathbf{X}})^\ast
      \mathscr{V}
    \big)
    \otimes
    (\mathrm{pr}_{\mathbf{Y}_1 \sqcup \mathbf{Y}_2})^\ast
    \big(
      (q_1)_! \mathscr{W}
      \,\sqcup\,
      (q_2)_! \mathscr{W}'
    \big)
 \! \Big)_{ \mathbf{X} \times \big( \mathbf{Y}_1 \sqcup \mathbf{Y}_2 \big) }
  &
 \;\;\;\;   \proofstep{
    by 
    \eqref{ComponentFormulaForCoproductsInLoc}
  }
  \\
  \;\simeq
  \bigg(\!\!
    \big(
      (\mathrm{pr}_{\mathbf{X}})^\ast
      \mathscr{V}
    \big)
    \otimes
    \Big(
      \big(
      (\mathrm{pr}_{\mathbf{Y}_1 \sqcup \mathbf{Y}_2})^\ast
      (q_1)_! \mathscr{W}
      \big)
      \,\sqcup\,
      \big(
      (\mathrm{pr}_{\mathbf{Y}_1 \sqcup \mathbf{Y}_2})^\ast
      (q_2)_! \mathscr{W}'
      \big)
    \Big)
  \!\! \bigg)_{ \mathbf{X} \times \mathbf{Y}_1 \,\sqcup\, \mathbf{X} \times \mathbf{Y}_1 }
  &
  \;\;\;\;  \proofstep{
    $(\mbox{-})^\ast$ is left adjoint
  }
  \\
  \;\simeq
  \bigg(\!\!
    \Big(\!
      \big(
        (\mathrm{pr}_{\mathbf{X}})^\ast
        \mathscr{V}
      \big)
      \otimes
      \big(
      (\mathrm{pr}_{\mathbf{Y}_1 \sqcup \mathbf{Y}_2})^\ast
      (q_1)_! \mathscr{W}
      \big)
    \! \Big)
      \sqcup
     \Big(\!
      \big(
        (\mathrm{pr}_{\mathbf{X}})^\ast
        \mathscr{V}
      \big)
      \otimes
      \big(
      (\mathrm{pr}_{\mathbf{Y}_1 \sqcup \mathbf{Y}_2})^\ast
      (q_2)_! \mathscr{W}'
      \big)
    \!\Big)
  \!\!\bigg)_{\! \mathbf{X} \times \mathbf{Y}_1 \,\sqcup\, \mathbf{X} \times \mathbf{Y}_1 }
  &
 \;\;\;\; \proofstep{
    $\cdots \otimes (\mbox{-})$ is left adjoint
  }
  \\
  \;\simeq
  \bigg(\!\!
    \Big(
      \big(
        (\mathrm{pr}_{\mathbf{X}})^\ast
        \mathscr{V}
      \big)
      \otimes
      \big(
        (q_1)_!
        (\mathrm{pr}_{\mathbf{Y}_1})^\ast
        \mathscr{W}
      \big)
   \!  \Big)
      \,\sqcup\,
     \Big(\!
      \big(
        (\mathrm{pr}_{\mathbf{X}})^\ast
        \mathscr{V}
      \big)
      \otimes
      \big(
        (q_2)_!
        (\mathrm{pr}_{\mathbf{Y}_2})^\ast
        \mathscr{W}'
      \big)
    \Big)
  \!\!\bigg)_{ \mathbf{X} \times \mathbf{Y}_1 \,\sqcup\, \mathbf{X} \times \mathbf{Y}_1 }
  &
 \; \;\;\;  \proofstep{
    by 
    \eqref{BeckChevalleyForProductsOfCoproductInclusions}
  }
  \\
  \;\simeq
  \bigg(\!\!
    (q_1)_!
    \Big(
      \big(
        (\mathrm{pr}_{\mathbf{X}})^\ast
        \mathscr{V}
      \big)
      \otimes
      \big(
        (\mathrm{pr}_{\mathbf{Y}_1})^\ast
        \mathscr{W}
      \big)
     \!\Big)
      \,\sqcup\,
     (q_2)_!
     \Big(\!
      \big(
        (\mathrm{pr}_{\mathbf{X}})^\ast
        \mathscr{V}
      \big)
      \otimes
      \big(
        (\mathrm{pr}_{\mathbf{Y}_2})^\ast
        \mathscr{W}'
      \big)
    \Big)
  \!\!\bigg)_{ \mathbf{X} \times \mathbf{Y}_1 \,\sqcup\, \mathbf{X} \times \mathbf{Y}_1 }
  &
 \;  \;\;\; \proofstep{
    by 
    \eqref{TensorProductWithExtensionByZero}
  }
  \\
  \;\simeq
  \big(
  \mathscr{V}_{\mathbf{X}}
  \,\boxtimes\,
  \mathscr{W}_{\mathbf{Y}_1}
  \big)
  \,\coprod\,
  \big(
  \mathscr{V}_{\mathbf{X}}
  \,\boxtimes\,
  \mathscr{W}'_{\mathbf{Y}_2}
  \big)
  &
  \;\;\;\;  \proofstep{
    by 
    \eqref{ComponentFormulaForCoproductsInLoc}\;.
  }
  \end{array}
\]
This concludes the proof.
\end{proof}

\begin{proposition}[External internal hom of simplicial local systems]
  \label{ExternalInternalHomOfSimplicialLocalSystems}
  $\,$

\begin{itemize}
\item[{\bf (i)}] Forming the external tensor product with a simplicial local system over a discrete space 
 has a right adjoint functor:
  \begin{equation}
    \label{ExternalInternalHomAdjunction}
    \begin{array}{l}
    \mathrm{Y}
    \,\in\,
    \mathrm{Set}
    \xhookrightarrow{\quad}
    \mathrm{sSet}\mbox{-}\mathrm{Grpd}_{\mathrm{skl}}
    \\
    \mathscr{R} \,\in\,
    \mathbf{sCh}_{\mathbb{K}}^{\mathrm{Y}}
    \end{array}
    \hspace{1.2cm}
      \vdash
    \hspace{1.2cm}
    \begin{tikzcd}
      \mathbf{Loc}_{\mathbb{K}}
      \ar[
        rr,
        shift left=5pt,
        "{
          \mathscr{R}_{\mathrm{Y}}
          \,\boxtimes\,
          (\mbox{-})
        }"
      ]
      \ar[
        from=rr,
        shift left=5pt,
        "{
          \scalebox{1}{$
          \mathscr{R}_{\mathrm{Y}}
          \,\extmap\,
          (\mbox{-})
          $}
        }"
      ]
      \ar[
        rr,
        phantom,
        "{
          \scalebox{.7}{$\bot$}
        }"
      ]
      &&
      \mathbf{Loc}_{\mathbb{K}}\;.
    \end{tikzcd}
  \end{equation}

\item[{\bf (ii)}] 
This is given by
  \begin{equation}
    \label{ExternalInternalHomFormula}
    \mathscr{W}_{\mathbf{Z}}
    \,\in\,
    \mathbf{Loc}_{\mathbb{K}}
    \hspace{1cm}
    \vdash
    \hspace{1cm}
    \mathscr{R}_{\mathrm{Y}}
    \,\extmap\,
    \mathscr{W}_{\mathbf{Z}}
    \;\;
      :\defneq
    \;\;
    \bigg(\,
    \underset{y \in \mathrm{Y}}{\prod}
    \mathrm{ev}_y^\ast
    \big[
      (p_{\mathbf{Z}})^\ast
      \mathscr{R}_{\{y\}}
      ,\,
      \mathscr{W}
    \big]
   \! \bigg)_{\mathbf{Z}^{\mathrm{Y}}}
    \,,
  \end{equation}

\item[{\bf (iii)}]  which is such that over a map
  $f : \mathrm{Y} \longrightarrow \mathbf{Z}$, hence when restricted along $\{f\} \xhookrightarrow{\;} \mathbf{Z}^{\mathrm{Y}}$, it is given by
  \begin{equation}
    \label{ExternalInternalHomOverASinleMap}
    \{f\}^\ast
    \big(
    \mathscr{R}_{\mathrm{Y}}
    \,\extmap\,
    \mathscr{W}_{\mathbf{Z}}
    \big)
    \;\;
    \simeq
    \;\;
    (p_{\mathrm{Y}})_\ast
    \,
    \big[
      \mathscr{R}
      ,\,
      f^\ast
      \mathscr{W}
    \big]
    \,.
  \end{equation}
\end{itemize}
\end{proposition}
\begin{proof}
  First, consider the special case where $\mathrm{Y}$ is a singleton set $\{y\}$.
  Then we have the following sequence of natural isomorphisms:
\begin{equation}
  \label{AdjointnessForExternalTensorWithSystemOverSingleton}
  \def\arraystretch{1.6}
  \begin{array}{lll}
    \mathbf{Loc}_{\mathbb{K}}
    \big(
    \mathscr{V}_{\mathbf{X}}
    \,\boxtimes\,
    \mathscr{R}_{\{y\}}
    ,\,
    \mathscr{W}_{\mathbf{Z}}
    \big)
    &
    \;\simeq\;
    \big(
      f \,\in\,
      \mathrm{sSet}\mbox{-}\mathrm{Grpd}(\mathbf{X},\mathbf{Z})
    \big)
    \,\times\,
    \mathbf{sCh}_{\mathbb{K}}^{\mathbf{X}}
    \Big(
      \mathscr{V}
      \otimes
      \big(
      (\mathrm{pr}_{\{y\}})^\ast
      \mathscr{R}
      \big)
      ,\,
      f^\ast \mathscr{W}
    \Big)
    &
    \proofstep{
      by \eqref{HomSetsOfContravariantGrothendieckConstruction}
    }
    \\
  &  \;\simeq\;
    \big(
      f \,\in\,
      \mathrm{sSet}\mbox{-}\mathrm{Grpd}(\mathbf{X},\mathbf{Z})
    \big)
    \,\times\,
    \mathbf{sCh}_{\mathbb{K}}^{\mathbf{X}}
    \Big(
      \mathscr{V}
      \otimes
      \big(
      (p_{\mathbf{X}})^\ast
      \mathscr{R}
      \big)
      ,\,
      f^\ast \mathscr{W}
    \Big)
    &
    \proofstep{
      by 
      \eqref{TowardsRightAdjointOfexternalTensor}
    }
    \\
 &   \;\simeq\;
    \big(
      f \,\in\,
      \mathrm{sSet}\mbox{-}\mathrm{Grpd}(\mathbf{X},\mathbf{Z})
    \big)
    \,\times\,
    \mathbf{sCh}_{\mathbb{K}}^{\mathbf{X}}
    \Big(
      \mathscr{V}
      ,\,
      \big[
        (p_{\mathbf{X}})^\ast
        \mathscr{R}
        ,\,
        f^\ast \mathscr{W}
      \big]
    \Big)
    &
    \proofstep{
      by
      \eqref{InternalHomAdjunctionForSimplicialLocalSystems}
    }
    \\
 &   \;\simeq\;
    \big(
      f \,\in\,
      \mathrm{sSet}\mbox{-}\mathrm{Grpd}(\mathbf{X},\mathbf{Z})
    \big)
    \,\times\,
    \mathbf{sCh}_{\mathbb{K}}^{\mathbf{X}}
    \Big(
      \mathscr{V}
      ,\,
      \big[
        f^\ast
        (p_{\mathbf{Z}})^\ast
        \mathscr{R}
        ,\,
        f^\ast \mathscr{W}
      \big]
    \Big)
    &
    \proofstep{
      by
      \eqref{TowardsRightAdjointOfexternalTensor}
    }
    \\
 &   \;\simeq\;
    \big(
      f \,\in\,
      \mathrm{sSet}\mbox{-}\mathrm{Grpd}(\mathbf{X},\mathbf{Z})
    \big)
    \,\times\,
    \mathbf{sCh}_{\mathbb{K}}^{\mathbf{X}}
    \Big(
      \mathscr{V}
      ,\,
      f^\ast
      \big[
        (p_{\mathbf{Z}})^\ast
        \mathscr{R}
        ,\,
        \mathscr{W}
      \big]
    \Big)
    &
    \proofstep{
      by 
      \eqref{StrongClosedPullback}
    }
    \\
&    \;\simeq\;
    \mathbf{Loc}_{\mathbb{K}}
    \big(
      \mathscr{V}_{\mathbf{X}}
      ,\,
      \big[
        (p_{\mathbf{Z}})^\ast
        \mathscr{R}
        ,\,
        \mathscr{W}
      \big]_{\mathbf{Z}}
    \big)
    &
    \proofstep{
      by 
      \eqref{HomSetsOfContravariantGrothendieckConstruction},
    }
  \end{array}
\end{equation}
where we have made use of notation corresponding to the following commuting diagram:
\begin{equation}
  \label{TowardsRightAdjointOfexternalTensor}
  \begin{tikzcd}[row sep=small]
    \mathbf{X}
    \times \{y\}
    \ar[
      d,
      "{
        \mathrm{pr}_{\{y\}}
      }"{swap}
    ]
    \ar[r, "{ \sim }"]
    &
    \mathbf{X}
    \ar[r, "{ f }"]
    \ar[
      d, 
      "{
        p_{\mathbf{X}} 
      }"{swap}
    ]
    &
    \mathbf{Z}
    \ar[
      d,
      "{
        p_{\mathbf{Z}}
      }"
    ]
    \\
    \{y\}
    \ar[r, "{ \sim }"]
    & 
    \ast
    \ar[r, equals]
    &
    \ast
    \mathrlap{\,.}
  \end{tikzcd}
\end{equation}
From this, we obtain the general statement \eqref{ExternalInternalHomFormula} as follows:
\[
  \def\arraystretch{1.7}
  \begin{array}{lll}
    \mathbf{Loc}_{\mathbb{K}}
    \big(
    \mathscr{V}_{\mathbf{X}}
    \boxtimes
    \mathscr{R}_{\mathrm{Y}}
    ,\,
    \mathscr{W}_{\mathbf{Z}}
    \big)
    &
    \;\simeq\;
    \mathbf{Loc}_{\mathbb{K}}
    \bigg(
    \mathscr{V}_{\mathbf{X}}
    \boxtimes
    \Big(
    \underset{y \in Y}{\coprod}
    \mathscr{R}_{\{y\}}
    \Big)
    ,\,
    \mathscr{W}_{\mathbf{Z}}
    \bigg)
    &
    \proofstep{
      by
      \eqref{LocalSystemOverDiscreteSpaceIsCoproduct}
    }
    \\
   & \;\simeq\;
    \mathbf{Loc}_{\mathbb{K}}
    \Big(
    \underset{y \in Y}{\coprod}
    \big(
    \mathscr{V}_{\mathbf{X}}
    \boxtimes
    \mathscr{R}_{\{y\}}
    \big)
    ,\,
    \mathscr{W}_{\mathbf{Z}}
    \Big)
    &
    \proofstep{
      by \eqref{ExternalTensorWithSystemOverDiscreteSpace}
    }
    \\
   & \;\simeq\;
    \underset{y \in Y}{\prod}
    \mathbf{Loc}_{\mathbb{K}}
    \Big(
    \mathscr{V}_{\mathbf{X}}
    \boxtimes
    \mathscr{R}_{\{y\}}
    ,\;
    \mathscr{W}_{\mathbf{Z}}
    \Big)
    &
    \proofstep{
      hom-functors preserve limits
    }
    \\
  &  \;\simeq\;
    \underset{y \in Y}{\prod}
    \mathbf{Loc}_{\mathbb{K}}
    \Big(
    \mathscr{V}_{\mathbf{X}}
    ,\;
    \big[
      (p_{\mathbf{Z}})^\ast
      \mathscr{R}_{\{y\}}
      ,\,
      \mathscr{W}
    \big]_{\mathbf{Z}}
    \Big)
    &
    \proofstep{
      by
      \eqref{AdjointnessForExternalTensorWithSystemOverSingleton}
    }
    \\
 &   \;\simeq\;
    \mathbf{Loc}_{\mathbb{K}}
    \bigg(\!
    \mathscr{V}_{\mathbf{X}}
    ,\;
    \underset{y \in Y}{\prod}
    \Big(
    \big[
      (p_{\mathbf{Z}})^\ast
      \mathscr{R}_{\{y\}}
      ,\,
      \mathscr{W}
    \big]_{\mathbf{Z}}
    \Big)
   \! \bigg)
    &
    \proofstep{
      hom-functors preserve limits
    }
    \\
 &   \;\simeq\;
    \mathbf{Loc}_{\mathbb{K}}
    \bigg(\!
    \mathscr{V}_{\mathbf{X}}
    ,\;
    \Big(
    \underset{y \in \mathrm{Y}}{\prod}
    \mathrm{ev}_y^\ast
    \big[
      (p_{\mathbf{Z}})^\ast
      \mathscr{R}_{\{y\}}
      ,\,
      \mathscr{W}
    \big]
    \Big)_{\mathbf{Z}^{\mathrm{Y}}}
    \!\bigg)
    &
    \proofstep{
      by Prop. \ref{ColimitsInAGrothendieckConstruction}.
    }
  \end{array}
\]
Finally, the formula \eqref{ExternalInternalHomOverASinleMap} is obtained as follows:
\[
  \def\arraystretch{1.4}
  \begin{array}{lll}
    \{f\}^\ast
    \underset{y \in \mathrm{Y}}{\prod}
    \mathrm{ev}_y^\ast
    \big[
      (p_{\mathbf{Z}})^\ast
      \mathscr{R}_{\{y\}}
      ,\,
      \mathscr{W}
    \big]
    &
    \;\simeq\;
    \underset{y \in \mathrm{Y}}{\prod}
    \{f\}^\ast
    \mathrm{ev}_y^\ast
    \big[
      (p_{\mathbf{Z}})^\ast
      \mathscr{R}_{\{y\}}
      ,\,
      \mathscr{W}
    \big]
    &
    \proofstep{
      since $(\mbox{-})^\ast$
      is right adjoint \eqref{BaseChangeAdjointTripleOnSimplicialLocalSystems}
    }
    \\
&    \;\simeq\;
    \underset{y \in \mathrm{Y}}{\prod}
    \{f(y)\}^\ast
    \big[
      (p_{\mathbf{Z}})^\ast
      \mathscr{R}_{\{y\}}
      ,\,
      \mathscr{W}
    \big]    
    &
    \proofstep{
     by
     \eqref{AFunctionEvaluation}
    }
    \\
 &   \;\simeq\;
    \underset{y \in \mathrm{Y}}{\prod}
    \big[
      \{f(y)\}^\ast
      (p_{\mathbf{Z}})^\ast
      \mathscr{R}_{\{y\}}
      ,\,
      \{f(y)\}^\ast
      \mathscr{W}
    \big]    
    &
    \proofstep{
      by
      \eqref{StrongClosedPullback}
    }
    \\
 &   \;\simeq\;
    \underset{y \in \mathrm{Y}}{\prod}
    \big[
      \mathscr{R}_{\{y\}}
      ,\,
      \mathscr{W}_{\{f(y)\}}
    \big]  
    &
    \proofstep{
      by
      \eqref{AFunctionEvaluation}
    }
    \\
&    \;\simeq\;
    \underset{y \in \mathrm{Y}}{\int}
    \big[
      \mathscr{R}_{\{y\}}
      ,\,
      (f^\ast \mathscr{W})_{\{y\}}
    \big]  
    \\
  &  \;\simeq\;
    (p_{\mathrm{Y}})_\ast
    \,
    \big[
      \mathscr{R}
      ,\,
      f^\ast \mathscr{W}
    \big]  
    &
    \proofstep{
      by
      \eqref{CoEndFormulasForKanExtension}\,,
    }
  \end{array}
\]
where we made use of the notation in the following commuting diagram
\begin{equation}
  \label{AFunctionEvaluation}
  \begin{tikzcd}[row sep=3pt]
    \{f\}
    \ar[rr, hook]
    \ar[dr, "{f(y) }"{swap}]
    &&
    \mathbf{Z}^{\mathrm{Y}}
    \ar[dl, "{ \mathrm{ev}_y }"]
    \\
    & 
    \mathbf{Z}
    \ar[dd, "{ p_{\mathbf{Z}} }"{left}]
    \\
    \\[5pt]
    &
    \ast
  \end{tikzcd}
\end{equation}
This concludes the proof.
\end{proof}

\begin{example}[External internal hom of vector bundles over discrete spaces]
  \label{ExternalInternalHomOverDiscreteSpaces}
  In the special case when the local systems in question take values in plain vector spaces
  \[
    \mathscr{R}
    ,\,
    f^\ast \mathscr{W}
    \,:\,
    \mathrm{Y}
    \longmapsto 
    \mathrm{Mod}_{\mathbb{K}}
    \xhookrightarrow{\quad}
    \mathrm{Ch}_{\mathbb{K}}
    \xhookrightarrow{\; \mathrm{const} \;}
    \mathrm{sCh}_{\mathbb{K}}
  \]
  then the expression in \eqref{ExternalInternalHomOverASinleMap} reduces to the vector space of vector bundle morphisms 
  $\mathscr{R} \to f^\ast \mathscr{W}$
  over $\mathrm{Y}$. This way we recover the expression for the $\boxtimes$-adjoint internal hom given in \cite[p. 6]{QS} 
  (there denoted  ``$\multimap$'' instead of ``$\extmap$'').
\end{example}

\begin{proposition}
\label{ExternalTensorPreservesHomotopyQuotientSquares}
  The external tensor product preserves the Cartesian squares \eqref{HomotopyQuotientSquares}.
  \[
    \begin{tikzcd}
      \mathscr{V}_{\mathrm{pt}}
      \,\boxtimes\,
      \mathscr{W}_{\mathbf{Y}}
      \ar[r]
      \ar[d]
      \ar[
        dr,
        phantom,
        "{
          \scalebox{.7}{\rm(pb)}
        }"
      ]
      &
      \mathscr{V}_{\mathbf{B}\mathcal{G}}
      \,\boxtimes\,
      \mathscr{W}_{\mathbf{Y}}
      \ar[d]
      \\
      0_{\mathrm{pt}}
      \,\boxtimes\,
      \mathscr{W}_{\mathbf{Y}}
      \ar[r]
      &
      0_{\mathbf{B}\mathcal{G}}
      \,\boxtimes\,
      \mathscr{W}_{\mathbf{Y}}
    \end{tikzcd}
 \]
\end{proposition}

\medskip

\noindent
{\bf Homotopical properties of the external tensor product.}
We establish in Thm. \ref{ExternalTensorProductIsHomotopical} homotopical properties of the external tensor product $\boxtimes$ of simplicial 
local systems (from Def. \ref{ExternalTensorProductOfSimplicialLocalSystems}). The first key point is that $\boxtimes$ preserves all weak 
equivalences and in this sense already coincides with its derived functor.

\begin{lemma}[Homotopical properties of $\mathrm{Set}$-tensoring of simplicial groupoids]
\label{HomotopicalPropertiesOfSetTensoringOfSimplicialGroupoids}
The tensoring of $\mathrm{sSet}\mbox{-}\mathrm{Grpd}$ over sets
$$
  (-) \cdot (-)
  \;:\;
  \begin{tikzcd}
  \mathrm{Set}
  \,\times\,
  \mathrm{sSet}\mbox{-}\mathrm{Grpd}
  \ar[r, hook]
  &
  \mathrm{sSet}\mbox{-}\mathrm{Grpd}
  \,\times\,
  \mathrm{sSet}\mbox{-}\mathrm{Grpd}
  \ar[rr, "{ (-) \times (-) }"]
  &&
  \mathrm{sSet}\mbox{-}\mathrm{Grpd}
  \end{tikzcd}
$$
is a restricted Quillen bifunctor in that (recalling from Prop. \ref{DwyerKanModelStructures} that the cofibrations in 
$\mathrm{sSet}\mbox{-}\mathrm{Grpd}$ are in particular injections on sets of objects and weak equivalences are in particular
bijections on connected components):
\[
  \mathrm{X} \xhookrightarrow{f} \mathrm{X}'
  \;\in\;
  \mathrm{Inj}(\mathrm{Set}),
  \;
  \mathbf{Y} \xrightarrow{\mathbf{g}} \mathbf{Y}'
  \;\in\;\;
  \mathrm{Cof}\big(\mathrm{sSet}\mbox{-}\mathrm{Grpd}\big)
  \hspace{1.2cm}
  \vdash
  \hspace{1.2cm}
  f \widehat{\times} \mathbf{g}
  \;\in\;
  \mathrm{Cof}\big(\mathrm{sSet}\mbox{-}\mathrm{Grpd}\big)
\]
and the pushout-product on the right is in addition a weak equivalence if $\mathbf{g}$ is in addition a weak equivalent or if $f$ is an isomorphism.

In particular, the Cartesian product with a fixed set is a left Quillen functor:
  \begin{equation}
    \label{SetTensoringQuillenAdjunctionOnSSetGrpd}
    \begin{tikzcd}[column sep=large]
       S \,\in\, \mathrm{Set}
       \xhookrightarrow{\quad}
       \mathrm{sSet}\mbox{-}\mathrm{Grpd}
       \hspace{1cm}
       \vdash
       \hspace{1cm}
      \mathrm{sSet}\mbox{-}\mathrm{Grpd}
      \ar[
        rr,
        shift left=7pt,
        "{ 
           S \times (-)
           \;\simeq\;
           \underset{s \in S}{\coprod}(-)
         }"
      ]
      \ar[
        from=rr,
        shift left=7pt,
        "{ 
          (-)^S
           \;\simeq\;
           \underset{s \in S}{\prod}(-)
        }"
      ]
      \ar[
        rr,
        phantom,
        "{
          \scalebox{.7}{$\bot_{\mathrlap{\mathrm{Qu}}}$}
        }"
      ]
      &&
      \mathrm{sSet}\mbox{-}\mathrm{Grpd} \;. 
    \end{tikzcd}
  \end{equation}
\end{lemma}
\begin{proof}
It is immediate that \eqref{SetTensoringQuillenAdjunctionOnSSetGrpd} is a Quillen adjunction, since the class of (acyclic) 
cofibrations is closed under coproducts in the arrow category. Moreover, the pushout-product diagram in question is of the form
$$
  \begin{tikzcd}[
    column sep=50pt
  ]
    \mathrm{X} \!\times\! \mathbf{Y} 
    \ar[
      rr,
      "{
        \mathrm{id}_{\mathrm{X}} \,\times\, \mathbf{g}
      }"
    ]
    \ar[d, hook]
    &&
    \mathrm{X} \!\times\! \mathbf{Y}'
    \ar[d, hook]
    \ar[
      ddr,
      bend left=30pt,
      "{
        f \,\times\, \mathrm{id}_{\mathbf{Y}'}
      }"{sloped}
    ]
    \\
    \mathrm{X} \!\times\! \mathbf{Y}
    \mathrlap{
      \;\coprod\;
      (\mathrm{X'} \setminus \mathrm{X})
      \!\times\! 
      \mathbf{Y}
    }
    \ar[rr, shorten <=70pt]
    \ar[
      drrr, 
      bend right=15, 
      "{ 
        \mathrm{id} \times \mathbf{g} 
      }"{sloped}
    ]
    &&
    \mathrm{X} \!\times\! \mathbf{Y}'
    \mathrlap{
      \;\coprod\;
      (\mathrm{X}'\setminus\mathrm{X})
      \!\times\! 
      \mathbf{Y}
    }
    \ar[
      dr, 
      dashed, 
      shorten <=14pt, 
      end anchor={[yshift=3pt]}
    ]
    \\
    & & &
    \mathrm{X} \!\times\! \mathbf{Y}'
    \;\coprod\;
    (\mathrm{X}'\setminus\mathrm{X})
    \!\times\! 
    \mathbf{Y}'
  \end{tikzcd}
$$
showing that the dashed pushout-product morphism is a coproduct (in the arrow category) of $\mathrm{id}_{X \times \mathbf{Y}'}$ 
(which is trivially an acyclic cofibration) with copies of $\mathbf{g}$, hence is itself a cofibration, by the previous comment, 
and an acyclic cofibration if $\mathbf{g}$ is.
\end{proof}

\begin{theorem}[Homotopical properties of external tensor product on simplicial local systems]
  \label{ExternalTensorProductIsHomotopical}
  The external tensor product (Def. \ref{ExternalTensorProductOfSimplicialLocalSystems}) on the integral model category of simplicial local
  systems (Prop. \ref{GlobalModelStructure})
  \[
    \begin{tikzcd}
      \mathbf{Loc}_{\mathbb{K}}
        \times
      \mathbf{Loc}_{\mathbb{K}}
      \ar[
        r,
        "{ \boxtimes }"
      ]
      &
      \mathbf{Loc}_{\mathbb{K}}
    \end{tikzcd}
  \]
  has the following properties:
  \begin{itemize}
    \item[{\bf (i)}]\!\!
    It is a homotopical functor, in that it sends weak equivalences in $\mathbf{Loc}_{\mathbb{K}} \times \mathbf{Loc}_{\mathbb{K}}$ 
    to weak equivalences in $\mathbf{Loc}_{\mathbb{K}}$:
    \begin{equation}
      \label{ExternalTensorIsHomotopical}
      \mathrm{W} \boxtimes \mathrm{W}
      \;\subset\;
      \mathrm{W}
      \,,
    \end{equation}
    hence it passes immediately to its derived functor.
    \item[{\bf (ii)}]\!\! 
    It is ``linear-componentwise a Quillen bifunctor'' in that it satisfies the pushout-product axiom \eqref{PushoutProductAxiom} 
    on linear component maps:
    \begin{equation}
      \label{LinearComponentwiseQuillenPropertyOfExtenralTensor}
      \def\arraystretch{1.5}
      \begin{array}{lcl}
      \phi_{\mathbf{f}} 
        \,\in\, 
      \mathrm{Cof}_{\mathbf{f}},
      \;\;
      \gamma_{\mathbf{g}} 
        \,\in\, 
      \mathrm{Cof}_{\mathbf{g}}
      &
      \hspace{1cm}
      \vdash
      \hspace{1cm}
      &
      (\phi_{\mathbf{f}})
        \,\widehat{\boxtimes}\, 
      (\gamma_{\mathbf{g}})  
      \;\in\;
      \mathrm{Cof}_{
        \mathbf{f} 
          \,\widehat{\times}\, 
        \mathbf{g}
      }
      \\
      \phi_{\mathbf{f}} 
        \,\in\, 
      \mathrm{Cof}_{\mathbf{f}},
      \;\;
      \gamma_{\mathbf{g}} 
        \,\in\, 
      (\mathrm{Cof}\cap \mathrm{W})_{\mathbf{g}}
      &
      \vdash
      &
      (\phi_{\mathbf{f}})
        \,\widehat{\boxtimes}\, 
      (\gamma_{\mathbf{g}})  
      \;\in\;
      (\mathrm{Cof}\cap \mathrm{W})_{
        \mathbf{f} 
          \,\widehat{\times}\, 
        \mathbf{g}
     }\;.
     \end{array}
  \end{equation}
  In particular, for fixed base objects this means that:
  \begin{equation}
    \label{ExternalTensorIsQuillenBifunctorOverFixedBaseObjects}
    \mathbf{X}
    ,\,
    \mathbf{Y}
    \;\;
    \in
    \;
    \mathrm{sSet}\mbox{-}\mathrm{Grpd}_{\mathrm{skl}}
    \hspace{1cm}
    \vdash
    \hspace{1cm}
    \begin{tikzcd}
    \mathbf{sCh}^{\mathbf{X}}_{\mathbb{K}}
    \times
    \mathbf{sCh}^{\mathbf{Y}}_{\mathbb{K}}
    \ar[r, "{ \boxtimes }"]
    &
    \mathbf{sCh}^{\mathbf{X} \times \mathbf{Y}}_{\mathbb{K}}    
    \end{tikzcd}
    \;\;
    \mbox{is a Quillen bifunctor}.
  \end{equation}

  \item[{\bf (iii)}]\!\! It is a left Quillen functor 
    when restricted in one argument to local systems over a discrete space, hence with right Quillen adjoint \eqref{ExternalInternalHomAdjunction}:
    \[
      \def\arraystretch{1.3}
      \begin{array}{l}
      \mathrm{X} \,\in\,
      \mathrm{Set}
      \hookrightarrow
      \mathrm{sSet}\mbox{-}\mathrm{Grpd}
      ,
      \\
      \mathscr{V}_{\mathrm{X}}
      \in \mathbf{sCh}^{X}_{\mathbb{K}}
      \end{array}
      \hspace{1.2cm}
      \vdash
      \hspace{1.2cm}
      \begin{tikzcd}[column sep=large]
        \mathbf{Loc}_{\mathbb{K}}
        \ar[
          rr,
          shift left=6pt,
          "{
            \mathscr{V}_{\mathrm{X}}
            \,\boxtimes\,
            (\mbox{-})
          }",
        ]
        \ar[
          from=rr,
          shift left=6pt,
          "{
            \mathscr{V}_{\mathrm{X}}
            \,\extmap\,
            (\mbox{-})
          }",
        ]
        \ar[
          rr,
          phantom,
          "{
            \scalebox{.7}{$\bot_{\mathrlap{\mathrm{Qu}}}$}
          }"
        ]
        &&
        \mathbf{Loc}_{\mathbb{K}} \,.
      \end{tikzcd}
   \]
\end{itemize}
\end{theorem}
\begin{proof}
\noindent {\bf (i)} By symmetry of the external tensor product and definition of product categories, it is sufficient to check 
that external tensor product 
 $\mathscr{W}_{\mathbf{X}} \boxtimes (-)$ with a fixed object preserves weak equivalences in the other argument.
 Now, due to the objectwise definition of the weak equivalences and fibrations in the projective model structure on each
 $\mathbf{sCh}^{\mathbf{Y}}_{\mathbb{K}}$, the precomposition functors $\mathbf{f}^\ast$ preserve fibrant replacements.  
 Therefore a morphism $\phi_{\mathbf{f}} \,:\, \mathscr{V}_{\mathbf{X}} \longrightarrow \mathscr{V}'_{\mathbf{X}'}$ is an integral weak equivalence 
 (Def. \ref{IntegralModelStructure}) iff $\mathbf{f} \,:\, \mathbf{X} \longrightarrow \mathbf{X}'$ is a weak equivalence in
 $\mathrm{sSet}\mbox{-}\mathrm{Grpd}$ and $\phi : \mathscr{V} \longrightarrow \mathbf{f}^\ast \mathscr{V}$ is 
 a weak equivalence in $\mathbf{sCh}_{\mathbb{K}}^{\mathbf{X}}$.
 Given such an integral weak equivalence, we need to see that also:
 \begin{itemize}
  \item[(a)]
  $
    \mathbf{f} \times \mathrm{id}_{\mathbf{Y}}
    \,:\,
    \mathbf{X} \times \mathbf{Y}
    \xrightarrow{\;\;}
    \mathbf{X}' \times \mathbf{Y}
  $
  is a weak equivalence in $\mathrm{sSet}\mbox{-}\mathrm{Grpd}$, hence a Dwyer-Kan equivalence. This is hom-wise the condition that 
  Cartesian product with a simplicial set preserves weak equivalences of simplicial sets,
  which is the case (for instance by Ken Brown's Lemma \ref{KenBrownLemma} using that $\mathrm{sSet}$ is cartesian monoidal model and 
  all objects are cofibrant).

  (Beware that even so $\mathbf{X} \times (-)$ is homotopical, it is far from being left Quillen on $\mathrm{sSet}\mbox{-}\mathrm{Grpd}$.)

  \item[(b)]
  $
  \big((\mathrm{pr}_{\mathbf{X}})^\ast\phi\big) \otimes \mathrm{id}_{\scalebox{.7}{$(\mathrm{pr}_{\mathbf{Y}})^\ast \mathscr{W}$}}$ 
  is a weak equivalence 
  in $\mathbf{sCh}^{\mathbf{X}}_{\mathbb{K}}$.

  First, $(\mathrm{pr}_{\mathbf{X}})^\ast\phi$ itself is a weak equivalence, since these are defined objectwise and hence preserved by the 
  precomposition functors $(-)^\ast$. Similarly, the tensor product over $\mathbf{X} \times \mathbf{Y}$ is a weak equivalence if and only if for all 
  $(x,y) \in \mathbf{X} \times \mathbf{Y}$ its component $\phi_x \otimes \mathrm{id}_{\scalebox{.7}{$\mathscr{W}_y$}}$ is a weak equivalence, hence if 
  tensoring $\mathscr{W} \otimes (-)$ preserves weak equivalences in $\mathbf{sCh}_{\mathbb{K}}$. This follows by Ken Brown's Lemma \ref{KenBrownLemma} 
  since $\mathscr{W} \otimes (-)$ is a left Quillen functor on and all objects are cofibrant in $\mathbf{sCh}_{\mathbb{K}}$, 
  by Prop. \ref{ModelCategoryOfSimplicialChainComplexes}.
 \end{itemize}

 \noindent {\bf (ii)} By Prop. \ref{ExternalPushoutProduct}, we are immediately reduced to proving the special case  
  \eqref{ExternalTensorIsQuillenBifunctorOverFixedBaseObjects}.
  For this, it is sufficient to check the pushout-product axiom on generating (acyclic) cofibrations. But for these 
  \eqref{GeneratingCofibrationsOfSimplicialChainComplexes},
  the relevant diagram (where we are abbreviating representable simplicial copresheaves by 
  $\underline{x} := \mathbf{X}(x,-) : \mathbf{X} \longrightarrow \mathbf{sSet}$)
$$ 
  \begin{tikzcd}[row sep=25pt, column sep=60pt]
    \big(
      \underline{x}
      \cdot
      \mathscr{V}
    \big)
    \boxtimes
    \big(
    \underline{y}
    \cdot
    \mathscr{W}
    \big)
    \ar[
      r, 
      "{
        \mathrm{Id}
        \,\boxtimes\,
        (\mathrm{id} \cdot g)
      }"
    ]
    \ar[
      d,
      "{
        (\mathrm{id}_{\underline{x}} \cdot f)
        \,\boxtimes\,
        \mathrm{Id}
      }"{swap}
    ]
    \ar[
      dr,
      phantom,
      "{ \scalebox{.6}{(po)} }"
    ]
    &
    \big(
      \underline{x} 
      \cdot
      \mathscr{V}
    \big)
    \boxtimes
    \big(
    \underline{y}
    \cdot
    \mathscr{W}'
    \big)
    \ar[d]
    \ar[
      ddr,
      bend left=15,
      "{
        (\mathrm{id}_{\underline{x}}\cdot g)      
        \,\boxtimes\,
        \mathrm{Id}
      }"{sloped}
    ]
    &[-40pt]
    \\
    \big(
      \underline{x}
      \cdot
      \mathscr{V}'
    \big)
    \boxtimes
    \big(
    \underline{y}
    \cdot
    \mathscr{W}
    \big)
    \ar[r]
    \ar[
      drr,
      bend right=15,
      "{
        \mathrm{Id}
        \,\boxtimes\,
        (\mathrm{id}_{\underline{y}}\cdot g)
      }"{sloped}
    ]
    &
    \big(
      \mathrm{id}_{\underline{x}}
      \cdot
      f
    \big)
    \widehat{\boxtimes}
    \big(
      \mathrm{id}_{\underline{y}}
      \cdot
      g
    \big)
    \ar[dr, dashed]
    \\[-20pt]
    &&
    \big(
      \underline{x}
      \cdot
      \mathscr{V}'
    \big)
    \,\boxtimes\,
    \big(
      \underline{y}
      \cdot
      \mathscr{W}'
    \big)
  \end{tikzcd}
$$
reduces to the tensoring with $\underline{(x,y)}$ of the analogous diagram for the tensor product in $\mathbf{sCh}_{\mathbb{K}}$:
$$ 
  \begin{tikzcd}[
    sep=32pt
  ]
    \underline{(x,y)}
    \cdot
    \mathscr{V} \otimes \mathscr{W}
    \ar[
      r, 
      "{
        \mathrm{id} 
          \cdot
        \mathrm{id}
        \otimes g
      }"
    ]
    \ar[
      d,
      "{
        \mathrm{id}
          \cdot
        f
          \otimes
        \mathrm{id}
      }"{swap}
    ]
    \ar[
      dr,
      phantom,
      "{ \scalebox{.6}{(po)} }"
    ]
    &
    \underline{(x,y)}
    \cdot
    \mathscr{V} 
      \otimes
    \mathscr{W}'
    \ar[d]
    \ar[
      ddr,
      bend left=15,
      "{
        \mathrm{id}
        \cdot
        f
        \otimes
        \mathrm{id}
      }"{sloped}
    ]
    &[-40pt]
    \\
    \underline{(x,y)}
    \cdot
    \mathscr{V}'
    \otimes
    \mathscr{W}
    \ar[r]
    \ar[
      drr,
      bend right=15,
      "{
        \mathrm{id}
        \cdot
        \mathrm{id}
        \otimes 
        g
      }"{sloped}
    ]
    &
    \underline{(x,y)}
    \cdot 
    (f \,\widehat{\otimes}\, g)
    \ar[dr, dashed]
    \\[-20pt]
    &&
    \underline{(x,y)}
    \cdot
    \mathscr{V}'
    \otimes
    \mathscr{W}'
  \end{tikzcd}
  \hspace{-.2cm}
  =
  \quad 
  \underline{(x,y)}
  \cdot
  \left(\!\!\!
  \begin{tikzcd}[
    sep=32pt
  ]
    \mathscr{V} \otimes \mathscr{W}
    \ar[
      r, 
      "{
        \mathrm{id}
        \otimes g
      }"
    ]
    \ar[
      d,
      "{
        f
          \otimes
        \mathrm{id}
      }"{swap}
    ]
    \ar[
      dr,
      phantom,
      "{ \scalebox{.6}{(po)} }"
    ]
    &
    \mathscr{V} 
      \otimes
    \mathscr{W}'
    \ar[d]
    \ar[
      ddr,
      bend left=15,
      "{
        f
        \otimes
        \mathrm{id}
      }"{sloped}
    ]
    &[-40pt]
    \\
    \mathscr{V}'
    \otimes
    \mathscr{W}
    \ar[r]
    \ar[
      drr,
      bend right=15,
      "{
        \mathrm{id}
        \otimes 
        g
      }"{sloped}
    ]
    &
    f \,\widehat{\otimes}\, g
    \ar[dr, dashed]
    \\[-20pt]
    &&
    \mathscr{V}'
    \otimes
    \mathscr{W}'
  \end{tikzcd}
  \!\!\!\!\right).
$$
Since $(\mathbf{sCh}_{\mathbb{K}}, \otimes)$ is $\mathrm{sSet}$-enriched monoidal model (by Prop. \ref{ModelCategoryOfSimplicialChainComplexes}), 
the tensoring $\underline{(x,y)}\cdot(-)$ on the right is a left Quillen functor and hence the claim follows.

\medskip

\noindent {\bf (iii)} The previous item (ii) establishes that the only obstacle to $\boxtimes$ being a left Quillen bifunctor is the 
failure of the underlying Cartesian product being left Quillen bifunctorial in $\mathrm{sSet}\mbox{-}\mathrm{Grpd}$. But restricted 
in one argument to discrete groupoids (sets) it is by Lem. \ref{HomotopicalPropertiesOfSetTensoringOfSimplicialGroupoids}, 
and so claim (iii) follows.
\end{proof}

\begin{remark}[Issue of full external monoidal model structure]
\label{IssueOfFullExternalMonoidalStructure}
Beyond Thm. \ref{HomotopicalPropertiesOfSetTensoringOfSimplicialGroupoids}, one  would wish that $\boxtimes$ were a Quillen bifunctor (Def. \ref{LeftQuillenBifunctor}) so that also its right adjoint 
internal hom were homotopically well behaved. However, on $\mathbf{Loc}_{\mathbb{K}}$ this fails in general simply because the underlying category $\mathrm{sSet}\mbox{-}\mathrm{Grpd}$ is, while cartesian monoidal as a category, not cartesian monoidal as a {\it model category}, i.e.,
the Cartesian product here in general fails the Quillen bifunctor property (already the product of two free simplicial groupoids is in general not itself free). 

However, the second item of Thm. \ref{ExternalTensorProductIsHomotopical}  shows that 
this is the only problem in that $\boxtimes$ generally satisfies the Quillen bifunctor property on linear components and hence satisfies 
it genuinely whenever the Cartesian product on $\mathrm{sSet}\mbox{-}\mathrm{Grpd}$ does so, which is the case at least when one argument 
is restricted to simplicial local systems over discrete spaces. 
\footnote{For the Quillen equivalent model category $\mathbf{Loc}^{\mathrm{sSet}}_{\mathbb{K}}$ \eqref{Loc}
the situation is somewhat complementary: Here the underlying Cartesian product is a Quillen bifunctor but now there is little control over 
the action on linear components, since the transfer functor $\mathbf{G} : \mathrm{sSet} \to \mathrm{sSet}\mbox{-}\mathrm{Grpd}$ 
does not preserve products.}

Moreover, with a little tweak to the base model structure, we do get full external monoidal model structure over base spaces which are 1-types (meaning: 1-groupoids, i.e., simplicial groupoids whose $\mathrm{sSet}$-enrichment happens to be in $\mathrm{Set}$), this we discuss in \cref{ExternalModulesOverOneTypes}, see Thm. \ref{ExternalMonoidalModelStructureOnLocalSystemsOverOneTypes} there.

While less strong than what one might have hoped for, the results of Thm. \ref{ExternalTensorProductIsHomotopical} still compare favorably with 
the state of the art in the literature on parameterized spectra. 
\end{remark}

\medskip

\subsection{External monoidal model over homotopy 1-types}
\label{ExternalModulesOverOneTypes}

We now restrict the discussion to $\infty$-local systems over just homotopy 1-types, by
restricting the base category from simplicial to plain groupoids, keeping the fiber category that of chain complexes. While evidently somewhat more restrictive, notice that this is exactly the infrastructure sufficient for interpreting the formalization of anyonic braid gates in \cite{TQP} (where the 1-types are those of configuration spaces of points whose $\mathbf{B}\ComplexNumbers^\times$-fibered mappings into chain complexes concentrated in degree $n$ constitute anyonic ground state wavefunctions, cf. \cite[Thm. 6.8]{TQP}).

\medskip

In this special case we may make use of the ``canonical'' model structure on $\Groupoids$ (Def. \ref{CanonicalModelStructureOnGroupoids}) which, in contrast to the structure on $\SimplicialGroupoids$, is cartesian monoidal as a model category. This finally makes the external tensor product of $\infty$-local systems over 1-types into a monoidal model category structure, in fact equivariantly so (Thm. \ref{ExternalMonoidalModelStructureOnLocalSystemsOverOneTypes}).

\medskip 
The upshot is that for every (equivariant) commutative monoid internal to $\infty$-local systems over 1-types we obtain a model category of (equivariant $\infty$-local systems which are) {\it modules} over this monoid (Cor. \ref{ModelStructureOnMonoidsInEquivariantLocalSystems}), thus producing quite a proliferation of, if you will, models of parameterized quantum types, albeit restricted to base 1-types.
In particular, taking the equivariant monoid to be the complex numbers regarded as an $\RealNumbers$-module equipped with $\ZTwo$-equivariance via complex conjugation, $\ZTwo \acts \, \ComplexNumbers$, its modules are the ``Real'' local systems including the (flat) Real vector bundles in the sense of KR-theory (Ex. \ref{TheRealLocalSystems}).
In \cite{Reality} we explain how this is a good context for interpreting finite-dimensional quantum types not just as linear types but as linear types with Hermitian inner product structure, hence with Hilbert-space structure, see the outlook at the end of \cref{DiscussionAndOutlook}.

\medskip

\begin{proposition}[Canonical model structure on groupoids]
\label{CanonicalModelStructureOnGroupoids}
The category $\Groupoids$ carries a model category structure where a functor $f \,\isa\,\AGroupoid{W} \to \AGroupoid{W}'$ is
\begin{itemize}
  \item a weak equivalences iff it is an equivalence of categories, hence
  a homotopy equivalence of groupoids, 
  hence
  inducing isomorphisms on connected components and on fundamental groups;
  \item a fibration iff for every morphism in $\mathcal{W}'$ and every lift of its domain to $\mathcal{W}$ also
  the morphism itself has a lift with that lifted domain:
  \[
    \begin{tikzcd}[
      row sep=4pt
    ]
      \AGroupoid{W}
      \ar[
        dd,
        "{ \; f }"
      ]
      &
      \mathllap{\forall}
      \,
      \widehat{w}_1
      \ar[
        r,
        dashed,
        "{ 
          \exists 
          \,
          \widehat{\phi}
        }"
      ]
      &
      {\color{gray} w_2 }
      \\
      \\
      \AGroupoid{W}'
      &
      \mathllap{
        \forall \,
      }
      w_1 
      \ar[
        r,
        "{ \phi }"
      ]
      &
      w_2
    \end{tikzcd}
  \]
  \item 
  a cofibration iff it is injective on objects.
\end{itemize}
This model structure is
\begin{itemize}
  \item[{\bf (i)}]
  proper,
  \item[{\bf (ii)}]
  combinatorial,
  \item[{\bf (iii)}]
  cartesian monoidal.\footnote{Meaning that the Cartesian product is a left Quillen bifunctor, Def. \ref{LeftQuillenBifunctor}.}
\end{itemize}
\end{proposition}
Since the proof is rather elementary, this model structure (and its analog for categories instead of groupoids) had been folklore
before being recorded in publications, whence some authors came to refer to it as the ``folk model structure''. 
\begin{proof}
The existence of the canonical model structure was first noted by \cite[p. 783]{Anderson78}, published proofs are due to \cite[Thm. 2]{JoyalTierney91}\cite[Thm. 6.7]{Strickland00}, see also the unpublished note \cite{Rezk96}. For the record, 
we spell out the proof of the above three properties. To start with, notice that:
\begin{equation}
\label{ImmediatePropertiesOfCanonicalModelStructure}
\begin{minipage}{15cm}
\begin{itemize}
\item[{\bf (1)}] All objects in the model structure are evidently bifibrant (both cofibrant as well as fibrant).
\item[{\bf (2)}] The acyclic fibrations are precisely the surjective-on-objects equivalences of categories.
\item[{\bf (3)}] The acyclic cofibrations are precisely the injective-on-objects equivalences of categories.
\end{itemize}
\end{minipage}
\end{equation}

What drives the proof of monoidal model structure below is that forming the Cartesian product with a fixed groupoid preserves 
all equivalences of groupoids. In fact, since all objects are fibrant we have, more generally, that the {\it fiber product} 
with a fibration preserves equivalences between fibrations (\cite[p. 428, 431]{Brown73}):

\newpage 
\begin{equation}
  \label{CartesianProductOfGroupoidsPreservesEquivalences}
  \left.
  \def\arraystretch{1.2}
  \begin{array}{l}
  \mathcal{B} \,\in\, \Groupoids
  \,,
  \\
  p_{\mathcal{X}}, 
  p_{\mathcal{Y}},
  p_{\mathcal{Z}}
  \,\in\,
  \big(
    \Groupoids_{/\mathcal{B}}
  \big)^{\mathrm{fib}}
  \,,
  \\
  \hspace{15pt}
  \begin{tikzcd}[
    column sep=17pt,
    row sep=10pt
  ]
  \mathllap{
    F \,:\;\;
  }
  \mathcal{Y} 
  \ar[rr, "{ \mathrm{equiv} }"{swap}]
  \ar[
    dr,
    "{ p_{\mathcal{X}} }"{swap}
  ]
  &&
  \mathcal{Z}
  \ar[
    dl,
    "{ p_{\mathcal{Y}} }"
  ]
  \\
  &
  \mathcal{B}
  \end{tikzcd}
  \end{array}
  \right\}
  \hspace{.7cm}
  \vdash
  \hspace{2cm}
  \begin{tikzcd}[
    column sep=16pt,
    row sep=6pt
  ]
    \mathllap{
      \mathcal{X}
      \times_{{}_{\mathcal{B}}} 
      F
      \,:\;    
    }
    \mathcal{X}
    \times_{{}_{\mathcal{B}}}
    \mathcal{Y}
    \ar[rr, "{ \mathrm{equiv} }"{swap}]
    \ar[dr]
    &&
    \mathcal{X}
    \times_{{}_{\mathcal{B}}}
    \mathcal{Z}
    \ar[dl]
    \\
    &
    \mathcal{B}
  \end{tikzcd}
\end{equation}

\noindent
For {\bf (i)}: Since all objects are bifibrant \eqref{ImmediatePropertiesOfCanonicalModelStructure}, properness 
follows by \cite[Prop. 13.1.2]{Hirschhorn02}.

\noindent
For {\bf (ii)}: Since $\Groupoids$  is the category of models of a limit sketch (this being the diagram shapes 
for internal groupoids, cf. \cite[\S XII.1]{MacLane97}) it is locally presentable (by \cite[Cor. 1.52]{AdamekRosicky94}).
For the sets of generating cofibrations ($I$) and generating acyclic cofibrations ($J$) we may take the following evident functors
\begin{equation}
  \label{GeneratingCofibrationsOfCanonicalModelStructure}
  I
  \;:\defneq\;
  \left\{
  \def\arraystretch{1.2}
  \begin{array}{ccc}
    \varnothing &\longrightarrow& \{0\},
    \\
    \{0\} \sqcup \{1\} &\longrightarrow& \{0 \xleftrightarrow{\sim} 1\},
    \\
    \{
      \hspace{-4pt}
      \begin{tikzcd}[
        column sep=13pt
      ]
        0\!
        \ar[
          r,
          <->,
          shift left=4pt,
          "{ \sim }"
        ]
        \ar[
          r,
          <->,
          shift right=4pt,
          "{ \sim }"{swap}
        ]
        & 
        \!1
      \end{tikzcd}
      \hspace{-4pt}
    \} &\longrightarrow& \{0 \xleftrightarrow{\sim} 1\} 
  \end{array}
  \right\}
  ,
  \hspace{1cm}
  J \,:\defneq\,
  \big\{ 
    \{0\}
    \longrightarrow
    \{ 0 \xleftrightarrow{\sim} 1 \}
  \big\}
  \,,
\end{equation}
because the right lifting property against $I$ clearly characterizes the functors which are (a) surjective on objects, (b) full,
and (c) faithful, which are exactly the acyclic fibrations \eqref{ImmediatePropertiesOfCanonicalModelStructure}; 
while the right lifting property against $J$ characterizes the isofibrations by definition. 

\noindent
For {\bf (iii)}: Given a diagram in $\Groupoids$ of the form
\begin{equation}
  \label{PushoutProductForGroupoids}
  \begin{tikzcd}[row sep=small, column sep=large]
    &
    \AGroupoid{X}
    \times
    \AGroupoid{Y}
    \ar[
      dl,
      "{
        F \times \mathrm{Id}
      }"{swap}
    ]
    \ar[
      dr,
      "{
        \mathrm{Id} \times G
      }"
    ]
    \ar[
      dd,
      phantom,
      "{
        \scalebox{.65}{
          (po)
        }
      }"
    ]
    \\
    \AGroupoid{X}' 
      \times 
    \AGroupoid{Y}
    \ar[
      dr,
      "{ l }"{description}
    ]
    \ar[
      ddr,
      "{
        \mathrm{Id} \times G
      }"{swap}
    ]
    &&
    \AGroupoid{X}
      \times 
    \AGroupoid{Y}'
    \ar[
      dl,
      "{ r }"{description}
    ]
    \ar[
      ddl,
      "{
        F \times \mathrm{Id}
      }"
    ]
    \\
    &
    \AGroupoid{X}'\!\times\!\AGroupoid{Y}
      \,
      \underset{
        \mathclap{
          \AGroupoid{X} \times \AGroupoid{Y}
        }
      }{\coprod}
      \,
    \AGroupoid{X}\!\times\!\AGroupoid{Y}'
    \ar[
      d,
      dashed,
      "{
        F \widehat{\times} G
      }"{pos=.4, description}
    ]
    \\[10pt]
    &
    \AGroupoid{X}'
      \times
    \AGroupoid{Y}'
  \end{tikzcd}
\end{equation}
we need to show (Def. \ref{LeftQuillenBifunctor}):
\begin{itemize}
\item[(1)] That $F \widehat{\times} G$ is a cofibration if $F$ and $G$ are, hence that $(F \widehat{\times} G)_0$ is an injection of sets of objects if $F_0$ and $G_0$ are. But $(-)_0 \,:\, \Categories \to \Sets$ preserves limits and colimits, so that we are reduced to checking the pushout-product axiom on sets, which holds by \eqref{FibersOfPushoutProductMapInSets}.
\item[(2)] That, moreover, $F \widehat{\times} G$ is an equivalence of categories if $F$ is. But with $F$ clearly also $F \times \mathrm{Id}$ is an injective-on-objects equivalence of categories \eqref{CartesianProductOfGroupoidsPreservesEquivalences}, hence an acyclic cofibration. Therefore the model category structure implies that also its pushout $r$ in \eqref{PushoutProductForGroupoids} is an acyclic cofibration, in particular an equivalence, and then --- by the 2-out-of-3 property applied to the bottom right triangle --- that also $F \widehat{\times} G$ is an equivalence.\qedhere
\end{itemize}
\end{proof}

\begin{theorem}[External monoidal model structure on simplicial local systems over 1-types]
\label{ExternalMonoidalModelStructureOnLocalSystemsOverOneTypes}
  The integral model structure on 
  \begin{equation}
    \label{IntegralModelStructureOverGroupoids}
    \mathrm{Loc}^{\Groupoids}_{\GroundField}
    \;\;\;
    \defneq
    \;\;
    \underset{
      \AGroupoid{X}
      \,\isa\,
      \Groupoid
    }{\int}
    \mathrm{sCh}^{\AGroupoid{X}}_{\GroundField}
  \end{equation}
  (as in Thm. \ref{GlobalModelStructure} but now based on the canonical model structure on groupoids,
  Prop. \ref{CanonicalModelStructureOnGroupoids})
  \begin{itemize}
  \item[{\bf (i)}] exists and is 
  \item[{\bf (ii)}] combinatorial,
  \item[{\bf (iii)}]
  monoidal model with respect to the external tensor $\boxtimes$ (as in Def. \ref{ExternalTensorProductOfSimplicialLocalSystems}).
  \end{itemize}
More generally, for $G \,\in\, \Groups(\Sets)$, these statements generalize to the projective model structure 
of $G$-actions in local systems (cf. Rem. \ref{SimplicialGroupActions}):
\begin{equation}
  \label{ModelStructureOnEquivariantLocalSystemsOvergroupoids}
  \Actions{G}\big(
    \mathrm{Loc}_{\GroundField}^{\Groupoids}
  \big)
  \;\;
    \defneq
  \;\;
  \mathrm{Func}\big(
    \mathbf{B}G
    ,\,
    \mathrm{Loc}_{\GroundField}^{\Groupoids}    
  \big).
\end{equation}
\end{theorem}
\begin{proof} 
  For {\bf (i)}:
  Existence of the model structure follows verbatim as in the proof of Thm. \ref{GlobalModelStructure}, noticing that:
\begin{itemize}

\item[{\bf (a)}] The condition for weak equivalences is strictly a special case of that in the proof of Thm. \ref{GlobalModelStructure} using that the inclusion $\Groupoids \hookrightarrow \SimplicialGroupoids$ preserves weak equivalences (an ordinary equivalence of ordinary groupoids is a Dwyer-Kan equivalence when regarded in simplicial groupoids).

\item[{\bf (b)}] The condition for acyclic fibrations in the proof of Thm. \ref{GlobalModelStructure} actually holds irrespective of the properties of acyclic fibrations (it only relies on the fact that precomposition always preserves projective equivalences),

\item[{\bf (c)}] The only assumption on acyclic cofibrations in the base category that are actually used in the proof of Thm. \ref{GlobalModelStructure},
  namely \eqref{AssumptionOnAcyclicCofibrationsNeededInExistenceProof}, is satisfied also in the canonical model structure on $\Groupoids$, by \eqref{ImmediatePropertiesOfCanonicalModelStructure}.
\end{itemize}

  \smallskip

  \noindent
  For {\bf (ii)}: We need to show local presentability and cofibrant generation:
  
  \noindent{\bf  (a)}
  Local presentability follows by \cite[Prop. 5.3.1(4)]{MakkaiPare89}\footnote{Compare \href{https://mathoverflow.net/a/102083/381}{MO:102083/381}.} from the fact that $\Groupoids$ is locally presentable (by Prop. \ref{CanonicalModelStructureOnGroupoids}), each $\mathrm{Ch}_{\GroundField}^{\AGroupoid{X}}$ is locally presentable (by \cite[Cor. 1.54]{AdamekRosicky94}, since $\mathrm{Ch}_{\GroundField}$ is locally presentable by Thm. \ref{ModelCategoryOfSimplicialChainComplexes}), and the contravariant functor 
  \[
    \begin{tikzcd}[
      row sep=0pt
    ]
      \Groupoids^{\mathrm{op}}
      \ar[rr]
      &&
      \Categories
      \\
      \AGroupoid{X}
      &\mapsto&
      \mathrm{sCh}^{\AGroupoid{X}}
      \mathrlap{
        \;=\;
        \Categories\big(
          \AGroupoid{X}
          ,\,
          \mathrm{sCh}_{\GroundField}
        \big)
      }
    \end{tikzcd}
  \]
  (being the hom-functor of the 2-category of categories)
  sends colimits in $\Groupoids$ to 2-limits in $\Categories$. 

 \noindent {\bf  (b) }
  We claim that (acyclic) generating cofibrations are given by covering the generating (acyclic) cofibrations of $\Groupoids$ \eqref{GeneratingCofibrationsOfCanonicalModelStructure} with those of $\mathrm{sCh}_{\GroundField}$ (which exist by Thm. \ref{ModelCategoryOfSimplicialChainComplexes}, now to be denoted $i : \mathscr{D}_i \to \mathscr{C}_i$
  and $j \,:\, \mathscr{D}_j \to \mathscr{C}_j$, respectively)
  as follows:
\begin{equation}
  \label{GeneratingCofibrationsForIntegralModelStructureOverGroupoids}
  \def\arraystretch{1}
  \begin{array}{l}
   I_{
     {}_{ 
       \scalebox{.65}{$
         \mathrm{Loc}_\GroundField^{\Groupoids}
       $}
    }
  }
  \;\;
  :\defneq
  \;\;
  \left\{
  \begin{tikzcd}[
    column sep=20pt,
    row sep=5pt  
  ]  
    &[-30pt]
    0
    \ar[rrrr]
    &
    && 
    &[-30pt]
    0
    \\
    \varnothing
    \ar[rrrr]
    &&&&
    \{0\}
  \end{tikzcd}
  ,\,
  \begin{tikzcd}[
    column sep=30pt,
    row sep=5pt,
  ]
    &[-40pt]
    &[+20pt]
    &[-40pt]
    0
    \ar[rrrr]
    &[-15pt]
    &[-40pt]
    &[+20pt]
    &[-40pt]
    0
    \\
    &&
    1 \}
    \ar[
      from=ddll,
      shorten =-2pt,
      shift left=4pt,
      "{ \sim }"{sloped, pos=.75, swap},
    ]
    \ar[
      from=ddll,
      shorten =-2pt,
      shift right=4pt
    ]
    \ar[
      rrrr,
      |->
    ]
    &&&&
    1 \}
    \\
    & 
    0
    \ar[
      uurr,
      shorten=-2pt,
      equals,
      shift left=3pt
    ]
    \ar[
      uurr,
      shorten=-2pt,
      equals,
      shift right=3pt
    ]
    \ar[
      rrrr,
      shorten <=-2pt,
      crossing over
    ]
    &&&&
    0
    \ar[
      uurr,
      crossing over,
      equals
    ]
    \\
    \{ 0
    \ar[
      rrrr,
      |->
    ]
    &&&&
    \{ 0
    \ar[
      uurr,
      "{ \sim }"{sloped}
    ]
  \end{tikzcd}
  \right\}
  \\
  \\
  \hspace{2cm}
  \coprod
  \;
  \left\{
  \begin{tikzcd}[
    column sep=30pt,
    row sep=5pt
  ]
    &[-40pt]
    &
    & [-40pt]
    0
    \ar[rrrr]
    &[-20pt]
    &[-40pt]
    &[20pt]
    &[-40pt]
    \mathscr{C}_i
    \\
    &&
    \{ 1 \}
    \ar[
      rrrr,
      |->,
      shorten >=-3pt
    ]
    &&&& 
    1 \}
    \\
    &
    \mathscr{D}_i
    \ar[
      rrrr,
      "{ i }"{description}
    ]
    &&&&
    \mathscr{C}_i
    \ar[
      uurr,
      equals,
      crossing over
    ]
    \\
    \{ 0 \}
    \ar[
      rrrr,
      |->
    ]
    &&&&
    \{ 0
    \ar[
      uurr,
      "{ \sim }"{sloped}
    ]
  \end{tikzcd}
   \middle\vert\;
   i 
     \in 
   I_{
     {}_{ 
       \scalebox{.65}{$
         \mathrm{sCh}_\GroundField
       $}
    }
  }
  \right\}
  \\
  \\
   J_{
     {}_{ 
       \scalebox{.65}{$
         \mathrm{Loc}_\GroundField^{\Groupoids}
       $}
    }
  }
  \;=\;
 \left\{
  \begin{tikzcd}[
    column sep=30pt,
    row sep=7pt
  ]
    &[-40pt]
    &
    &[-20pt]
    &
    &[-40pt]
    &[+10pt]
    &[-40pt] 
    \mathscr{C}_j
    \\
    &&&&&& 
    1 \}
    \\
    &
    \mathscr{D}_j
    \ar[rrrr, "{ j }"]
    &&& &
    \mathscr{C}_j
    \ar[uurr, equals]
    \\
    \{ 0 \}
    \ar[
      rrrr,
      |->
    ]
    &&&&
    \{ 0
    \ar[
      uurr,
      "{ \sim }"{sloped, swap}
    ]
  \end{tikzcd}
  \;\middle\vert\;
   \;
   j
     \in
   J_{{}_{\scalebox{.65}{$\mathrm{sCh}_\GroundField$}}}
   \;
  \right\}
  \end{array}
\end{equation}
Namely, by existence of the model structure and local presentability it is now sufficient to show that a morphism on $\mathrm{Loc}_{\GroundField}^{\Groupoids}$ is a fibration (acyclic fibration) iff it has the right lifting property against $J$ ($I$). 
It is immediately clear that this is the case on the underlying base morphisms in $\Groupoid$. That it also holds for the linear 
component maps is due to the fact that the codomains of all the generating (acyclic) cofibrations of $\Groupoids$ are codiscrete 
groupoids which in 
\eqref{GeneratingCofibrationsForIntegralModelStructureOverGroupoids} 
are covered by identity morphisms between linear components: Together this implies that the relevant lifts are uniquely fixed by 
lifts over a single object, where the lifting condition is thus the objectwise one in the projective model structure.

Concretely, a lifting problem of $J_{{}_{\scalebox{.65}{$\mathrm{Loc}_\GroundField^\Groupoids$}}}$ against a morphism
$\phi_f \,:\, \mathscr{V}_{\mathcal{X}} \to \mathscr{W}_{\mathcal{Y}}$ covering an isofibration $f$ is, after lifting
the underlying morphism in $\Groupoids$, in components of this form:
\[
  \begin{tikzcd}[
    column sep=30pt,
    row sep=5pt
  ]
    &[-35pt]
    &[20pt]
    &[-35pt]
    &
    &[-35pt]
    &
    &[-35pt]     
    \mathscr{V}_{x_1}
    \ar[
      dddddd,
      "{ \phi_{y_1} }"{pos=.3}
    ]
    \\[+10pt]
    &
    &
    &
    &
    &
    &
    x_1 \}
    \ar[
      from=dddll,
      shorten <=-5pt,
      "{ \sim }"{sloped, pos=.75},
      "{ g }"{swap, pos=.5}
    ]
    \ar[
      dddddd,
      |->,
      shift right=3pt,
      shorten >=-6pt
    ]
    &[-40pt]     
    \\
    \\
    &
    \mathscr{D}_j
    \ar[
      rrrr
    ]
    \ar[
      ddddd,
      "{ j }"{description}
    ]
    &
    &
    &
    &
    \mathscr{V}_{\!x_0}
    \ar[
      uuurr,
      shorten =-5pt,
      "{
        \mathscr{V}_{\!g}
      }"
    ]
    &
    & 
    \\[+10pt]
    \{0\}
    \ar[
      ddddd,
      |->,
    ]
    \ar[
      rrrr,
      shorten <=-2pt,
      crossing over,
      |->
    ]
    &&&&
    \{ x_0
    \\
    \\
    &&& 
    \mathscr{C}_j
    \ar[
      rrrr,
      crossing over,
      shorten >=-4pt,
      "{
        \mathscr{W}_{\!h}(w)
      }"{pos=.75}
    ]
    &&&&
    \mathscr{W}_{\!y_1}
    \\[+10pt]
    &&
    1 \}
    \ar[
      rrrr,
      |->,
      shorten >=-4pt
    ]
    \ar[
      from=ddll,
      "{ 
        \sim 
      }"{sloped, pos=.8}
    ]
    &&&&
    y_1 \}
    \ar[
      from=ddll,
      "{ 
        \sim  
      }"{sloped, pos=.8},
      "{
        h
      }"{swap}
    ]
    \ar[
      from=uuuuuu,
      -,
      shift right=3pt,
      crossing over,
      shorten >=20pt
    ]
    \\
    & 
    \mathscr{C}_j
    \ar[
      uurr,
      equals
    ]
    \ar[
      rrrr,
      crossing over,
      shorten >=-6pt,
      "{
        w
      }"
    ]
    &&&&
    \mathscr{W}_{\!y_0}
    \ar[
      from=uuuuu,
      crossing over,
      shorten <=-3pt,
      "{
        \phi_{y_0}
      }"{pos=.3}
    ]
    \ar[
      uurr,
      shorten =-5pt,
      crossing over,
      "{
        \mathscr{W}_{\!h}
      }"{pos=.65}
    ]
    \\[+10pt]
    \mathllap{\{} 0
    \ar[
      rrrr,
      |->
    ]
    &&&&
    \{ y_0
    \ar[
      from=uuuuu,
      |->,
      crossing over
    ]
    &
    &
  \end{tikzcd}
\]
and if a lift $\widehat{w} \,:\, \mathbb{D}^n \to \mathscr{V}_{x_0}$ of $w$ (in the front) exists, then a compatible lift of 
$\mathcal{W}_g(g)$ (in the rear) exists uniquely, given by $\mathscr{V}_{\! g}(\widehat w)$. Therefore all such lifts exist 
iff $\phi \,:\, \mathscr{V}_{\mathcal{X}} \to f^\ast \mathscr{W}_{\mathcal{Y}}$ is objectwise and degreewise surjective, 
hence projectively fibrant.

Similarly, a lifting problem of the non-trivially covered generating cofibration in \eqref{GeneratingCofibrationsForIntegralModelStructureOverGroupoids} against a morphism $\phi_f$ covering an isofibration $f$ is, after lifting on underlying groupoids, in components of this form:
\[
  \begin{tikzcd}[
    column sep=30pt,
    row sep=5pt
  ]
    &[-35pt]
    &[20pt]
    &[-35pt]
    0
    \ar[
      rrrr
    ]
    \ar[
      dddddd
    ]
    &
    &[-35pt]
    &
    &[-35pt]     
    \mathscr{V}_{x_1}
    \ar[
      dddddd,
      "{ \phi_{y_1} }"{pos=.3}
    ]
    \\[+10pt]
    &
    &
    \{1\}
    \ar[
      dddddd,
      |->
    ]
    \ar[
      rrrr,
      |->,
      shorten =-3pt,
      crossing over
    ]
    &
    &
    &
    &
    x_1 \}
    \ar[
      from=dddll,
      shorten <=-5pt,
      "{ \sim }"{sloped, pos=.75},
      "{ g }"{swap, pos=.5}
    ]
    \ar[
      dddddd,
      |->, 
      shift right=3pt,
      shorten >=-6pt
    ]
    &[-40pt]     
    \\
    \\
    &
    \mathscr{D}_i
    \ar[
      rrrr,
      crossing over
    ]
    \ar[
      ddddd,
      "{ i }"{description}
    ]
    &
    &
    &
    &
    \mathscr{V}_{\!x_0}
    \ar[
      uuurr,
      shorten =-5pt,
      shift left=2pt,
      crossing over,
      "{
        \mathscr{V}_{\!g}
      }"{pos=.75}
    ]
    &
    & 
    \\[+10pt]
    \{0\}
    \ar[
      ddddd,
      |->,
    ]
    \ar[
      rrrr,
      shorten <=-2pt,
      crossing over,
      |->
    ]
    &&&&
    \{ x_0
    \\
    \\
    &&& 
    \mathscr{C}_i
    \ar[
      rrrr,
      crossing over,
      shorten >=-4pt,
      "{
        \mathscr{W}_{\!h}(w)
      }"{pos=.75}
    ]
    \ar[
      ddll,
      equals,
      shorten =-3pt
    ]
    &&&&
    \mathscr{W}_{\!y_1}
    \\[+10pt]
    &&
    1 \}
    \ar[
      rrrr,
      |->,
      shorten >=-4pt
    ]
    \ar[
      from=ddll,
      "{ 
        \sim 
      }"{sloped, pos=.8}
    ]
    \ar[
      from=uuuuuu,
      shorten <=65pt,
      crossing over
    ]
    &&&&
    y_1 \}
    \ar[
      from=ddll,
      "{ 
        \sim  
      }"{sloped, pos=.8},
      "{
        h
      }"{swap}
    ]
    \ar[
      from=uuuuuu,
      -,
      shift right=3pt,
      crossing over,
      shorten >=20pt
    ]
    \\
    & 
    \mathscr{C}_i
    \ar[
      rrrr,
      crossing over,
      shorten >=-6pt,
      "{
        w
      }"
    ]
    &&&&
    \mathscr{W}_{\!y_0}
    \ar[
      from=uuuuu,
      crossing over,
      shorten <=-3pt,
      "{
        \phi_{y_0}
      }"{pos=.3}
    ]
    \ar[
      uurr,
      shorten =-5pt,
      crossing over,
      "{
        \mathscr{W}_{\!h}
      }"{pos=.65}
    ]
    \\[+10pt]
    \mathllap{\{} 0
    \ar[
      rrrr,
      |->
    ]
    &&&&
    \{ y_0
    \ar[
      from=uuuuu,
      |->,
      crossing over
    ]
    &
    &
  \end{tikzcd}
\]
and by the same argument as before any lift of $w$ (in the front) uniquely induces a lift of $\mathscr{W}_{h}(w)$ (in the rear). 
Therefore all such lifts exists iff $\phi \,:\, \mathscr{V}_{\mathcal{X}} \to f^\ast \mathscr{W}_{\mathcal{Y}}$ is objectwise 
a surjective weak equivalence, hence an acyclic fibration.

\smallskip

  \noindent
  For {\bf (iii)}:
  For the monoidal model structure, we need to check the pushout-product property (Def. \ref{LeftQuillenBifunctor}). So consider a diagram in $\mathrm{Loc}^\Groupoids_{\GroundField}$ of the form
  \[
    \begin{tikzcd}[column sep=large]
      & 
      \mathscr{V}_{\!\mathcal{X}}
      \otimes
      \mathscr{W}_{\mathcal{Y}}
      \ar[
        dl,
        "{
          \phi_{f}
          \,\boxtimes\,
          \mathrm{Id}
        }"{swap}
      ]
      \ar[
        dr,
        "{
          \mathrm{Id}
          \,\boxtimes\,
          \gamma_{g}
        }"
      ]
      \ar[
        dd,
        phantom,
        "{ \scalebox{.7}{(po)} }"
      ]
      \\
      \mathscr{V}'_{\!\mathcal{X}'}
      \otimes
      \mathscr{W}_{\mathcal{Y}}      
      \ar[
        dr,
        "{ 
          \lambda_l
        }"
      ]
      \ar[
        ddr,
        "{
          \mathrm{id}
          \,\boxtimes\,
          \gamma_g
        }"{swap}
      ]
      &&
      \mathscr{V}_{\!\mathcal{X}}
      \otimes
      \mathscr{W}'_{\mathcal{Y}'}
      \ar[
        dl,
        "{ \rho_r }"{swap}
      ]
      \ar[
        ddl,
        "{
          \phi_f
          \,\boxtimes\,
          \mathrm{id}
        }"
      ]
      \\
      &
      \mathscr{V}'_{\mathcal{X}'}
      \boxtimes
      \mathscr{W}_{\mathcal{Y}}
      \;\;
      \underset{
        \mathclap{
          \mathscr{V}_{\mathcal{X}}
          \boxtimes
          \mathscr{W}_{\mathcal{Y}}      
        }
      }{\coprod}
      \;\;
      \mathscr{V}'_{\mathcal{X}'}
      \boxtimes
      \mathscr{W}_{\mathcal{Y}}
      \ar[
        d,
        dashed,
        "{
          \phi_f 
          \widehat{\boxtimes}
          \gamma_g
        }"{description, pos=.4}
      ]
      \\[+20pt]
      &
      \mathscr{V}'_{\!\mathcal{X}'}
      \boxtimes
      \mathscr{W}'_{\mathcal{Y}'}
      \,,
    \end{tikzcd}
  \]
where $\phi_f$ and $\gamma_g$ are integral cofibrations, which means (Def. \ref{IntegralModelStructure}) that 
$f$ and $g$ are cofibrations in $\Groupoids$, while $\phi : f_!(\mathscr{V}) \to \mathscr{V}'$ is a cofibration in 
$\mathrm{Ch}_{\GroundField}^{\mathcal{X}}$ and $\gamma : g_!(\mathscr{W}) \to \mathscr{W}'$ is a cofibration in 
$\mathrm{Ch}_{\GroundField}^{\mathcal{Y}}$.
We need to show that $\phi_f \,\widehat{\boxtimes}, \gamma_g$ is a cofibration, which, by the same token and using Prop. \ref{ColimitsInAGrothendieckConstruction}, means to show that:
\begin{itemize}
\item[{\bf (a)}]
$f \widehat{\times} g$ is a cofibration in $\Groupoids$.
This holds by Prop. \ref{CanonicalModelStructureOnGroupoids}, see around \eqref{PushoutProductForGroupoids}.
\item[{\bf (b)}]
$\big( (l_! \phi) \otimes \mathrm{id}\big)  \,\widehat{\otimes}\, \big( \mathrm{id} \otimes (r_! \gamma) \big)$ is a cofibration in $\mathrm{Ch}_{\GroundField}^{\mathcal{X}' \times \mathcal{Y}'}$.
This follows by Thm. \ref{ExternalTensorProductIsHomotopical}(ii), since the functors $l_!$ and $r_!$ are left Quillen (by Rem. \ref{BaseChangeBetweenModelStructuresOfSimplicialLocalSystems}) and hence preserve cofibrations. 
\end{itemize}  
By the directly analogous argument, it follows that $\phi_f \,\boxtimes\, \gamma_g$ is moreover a weak equivalence if $\phi_f$ or $\gamma_g$ is.

\smallskip

\noindent
Finally to see that these statements generalize to $G$-actions: Existence and cofibrant generation is implied from {\bf (i)}
and {\bf (ii)} by general facts \cite[Thm. 11.6.1]{Hirschhorn02} about the projective model structure on functors, 
and monoidal model structure is implied from {\bf (iii)} by the argument of \cite[p. 6]{BergerMoerdijk06} that we already reviewed in the proof of Prop. \ref{MonoidalModelStructureOnLinearSimplicialGroupRepresentations}, now using just ordinary groups and the default $\Set$-enrichment of $\mathrm{Loc}_{\GroundField}^{\Groupoids}$ (for lack of an established simplicial model structure on $\mathrm{Loc}_{\GroundField}^{\Groupoids}$).
\end{proof}

\begin{corollary}[Model structure on modules internal to equivariant local systems]
  \label{ModelStructureOnMonoidsInEquivariantLocalSystems}
  For a group $G \,\in\, \Groups(\Sets)$ and 
  \[
    G \acts \;
    \mathscr{R}_{\mathcal{R}}
    \;\;
    \in
    \;\;
    \mathrm{CMon}
    \Big(\!
      \Actions{G}\big(
        \mathrm{Loc}^{\Groupoids}_{\GroundField}
      \big)
    \!\! \Big)
  \]
  a commutative monoid internal to $G$-actions on local systems over groupoids \eqref{IntegralModelStructureOverGroupoids} whose underlying object is cofibrant,
  its category of modules
  \begin{equation}
    \label{ModelStructureOnModulesOnEquivariantLocalSystems}
    \Modules{
      \scalebox{.7}{$
        G \acts \, \mathscr{R}_{\mathcal{R}}
      $}
    }
    \;\;
    \defneq
    \;\;
    \Modules{
      \scalebox{.7}{$
        G \acts \, \mathscr{R}_{\mathcal{R}}
      $}
    }
    \Big(\!
      \Actions{G}\big(
        \mathrm{Loc}^{\Groupoids}_{\GroundField}
      \big)
   \!\! \Big)
  \end{equation}
  carries a cofibrantly generated model structure whose fibrations and weak equivalences are those of the underlying integral model structure $\mathrm{Loc}_{\GroundField}^{\Groupoids}$ \eqref{ExternalMonoidalModelStructureOnLocalSystemsOverOneTypes}, and which is monoidal model 
  with respect to the induced tensor product of modules.
\end{corollary}
\begin{proof}
  By the fact that the monoidal model category $\Actions{G}\big( \mathrm{Loc}_{\GroundField}^{\Groupoids} \big)$ is cofibrantly generated (Thm. \ref{ExternalMonoidalModelStructureOnLocalSystemsOverOneTypes}) this is the statement of \cite[Thm. 3.1(1)(2) \& Rem. 3.2]{SchwedeShipley00}.
\end{proof}

\begin{example}[Parameterized dg-modules over homotopy 1-types]
  \label{ParameterizedDGModulesOverHomotopyOneTypes}
  A commutative monoid in chain complexes
  \begin{equation}
    \label{MonoidInChainComplexes}
    \mathscr{R}
    \,\in\,
    \mathrm{Mon}\big(
      \mathrm{Ch}_{\GroundField}
    \big)
    \longhookrightarrow
    \mathrm{Mon}
    \big(
      \mathrm{Loc}_{\GroundField}^\Groupoids
    \big)
  \end{equation} 
  is a graded-commutative dg-algebra over $\GroundField$, and its modules in $\mathrm{Loc}_{\GroundField}^{\Groupoids}$ are local systems of dg-$\mathscr{R}$-modules:
  \begin{equation}
    \label{ModelStructureOnDGModules}
    \mathrm{Loc}
      _{\scalebox{.7}{$\mathscr{R}$}}
      ^{\Groupoids}
    \;\;
    \simeq
    \;\;
    \Modules{
      \scalebox{.7}{$\mathscr{R}_\ast$}
    }\big(
     \mathrm{Loc}_{\GroundField}^{\Groupoids}
    \big)
    \,.
  \end{equation}
  Since the underlying object $\mathscr{R}_\ast$ in $\mathrm{Loc}_{\GroundField}^{\Groupoids}$ is always cofibrant (because its base space 
  is the point and since every object of $\mathrm{Ch}_{\GroundField}$ is cofibant, by Thm. \ref{ModelCategoryOfSimplicialChainComplexes})
  the model structure from Cor. \ref{ModelStructureOnMonoidsInEquivariantLocalSystems} on \eqref{ModelStructureOnDGModules} exists and 
  presents the homotopy theory of $H\mathscr{R}$-module spectra \eqref{ChainComplexesAsModuleSpectra} parameterized over homotopy 1-types.
\end{example}

\begin{example}[Equivariant parameterized dg-modules over equivariant parameterized 1-types]
  More generally, for $G \,\in\, \Groups(\Sets)$ acting on a monoid $\mathscr{R}$ \eqref{MonoidInChainComplexes} as in Ex. \ref{ParameterizedDGModulesOverHomotopyOneTypes}, the underlying equivariant object $G \acts \mathscr{R}_{\ast}$ is rarely cofibrant,
  but a useful cofibrant replacement is given by its tensor product with the normalized chain complex $N_\bullet(-)$ of the degree-wise 
  $\mathbb{R}$-linear span of the simplicial group $\mathbf{E}G \,\defneq\, G^{\times^{\bullet + 1}}$:
  \begin{equation}
    \label{EquivariantCofibrantReplacement}
    G \acts \;
    \big(
    N_\bullet \mathbb{R}[\mathbf{E}G]
    \,\otimes\,
    \mathscr{R}_\ast
    \big)
    \;\;\in\;\;
    \Actions{G}\big(
      \mathrm{Loc}_{\GroundField}^\Groupoids
    \big)_{\mathrm{cof}}
    \,.
  \end{equation}
  and via the Eilenberg-Zilber map on $N_\bullet$ this carries a compatible monoid structure
  \[
    G \acts \;
    \big(
    N_\bullet \mathbb{R}[\mathbf{E}G]
    \,\otimes\,
    \mathscr{R}_\ast
    \big)
    \;\;\in\;\;
    \mathrm{Mon}\Big( \!
    \Actions{G}\big(
      \mathrm{Loc}_{\GroundField}^\Groupoids
    \big)
    \!\! \Big)
    \,.
  \]
  Therefore, Cor. \ref{ModelStructureOnMonoidsInEquivariantLocalSystems} now provides a model structure on the corresponding
  category of modules, which presents the homotopy theory of $H\mathscr{R}$-modules spectra parameterized over 1-types with
  $G$-equivariant structure (both on the base space and compatibly on the spectra). 
\end{example}

Finally, as a simple but interesting special case of these general examples:

\begin{example}[The Real local systems]
  \label{TheRealLocalSystems}
  Consider the complex numbers, equipped with their  involution by complex conjugation, as a monoid in $\ZTwo$-equivariant 
  $\mathbb{R}$-chain complexes:
  \[
    \ZTwo \acts \; \ComplexNumbers
    \;\in\;
    \mathrm{Mon}\Big(
      \Actions{\ZTwo}\big(
        \mathrm{Loc}_{\GroundField}^{\Groupoids}
      \big)
   \!\! \Big)
    \,.
  \] 
  Its modules are the Real local systems subsuming, as their heart, the flat Real vector bundles \cite[p. 368]{Atiyah67}  (with a capital ``$R$'') 
  in the sense of Atiyah. We will discuss these in \cite{Reality} as good categorical models for finite-dimensional Hilbert spaces (see also the outlook \cref{DiscussionAndOutlook}).

  In this case, the cofibrant replacement \eqref{EquivariantCofibrantReplacement} has a simple explicit description, 
  thus giving rise to a decent model of the homotopy theory of Real local systems over homotopy 1-types.
\end{example}

\section{Applications and Outlook}
\label{DiscussionAndOutlook}

We close by highlighting some applications of our results to the field of mathematical quantum physics. 

\medskip

\noindent
{\bf $\infty$-Local systems and Realistic Topological Quantum Field Theory.}
The seminal understanding of the applicability of homotopy theory to  \emph{topological quantum field theory} has been so successful (since \cite{Lurie09TQGT}), that among all the abstract mathematical TQFTs considered now it may be easy to miss that much remains to be understood regarding the description of physically realistic TQFTs, for instance describing the experimentally observed \emph{fractional quantum Hall effect} (cf. the algebro-topological discussion in \cite{SS25-FQH}). In these realistic situations (arising notably in condensed matter theory), the \emph{quantum adiabatic theorem} implies (cf. \cite[p. 21]{TQP}) that Hilbert spaces of gapped quantum ground states form local systems over the space of external parameters of a quantum system. Moreover, these Hilbert spaces typically arise under a quantization procedure from $\infty$-local systems of chain complexes (jargon: \emph{hypergeometric construction of KZ-connections}, \cite{EFK98}), as made explicit in the main theorem of \cite{TQP}. Generally, such $\mathbb{C}$-linear $\infty$-local systems appear in the quantization of topological sectors of higher gauge theories ([\S 2.3.1]{SS25-Complete}), and are part of extended TQFTs in Lurie's sense by the results of \cite{Stefanich25}. 

Beyond the dynamics encoded by quantum field theory, also the \emph{quantum information theory} of realistic quantum states (such as notably their measurement processes) is controlled by the 6-functor motivic yoga of parameterized Hilbert spaces (\cite{QS}) and quantum circuits modeled on chain complexes of Hilbert spaces have been considered \cite{Zucchini25}.

This motivates the question for a formal quantum programming language which is a formal language for homotopy types the way HoTT is (\cite{Shulman21}) but enhanced to parameterized stable homotopy types. Based on our original suggestions along these lines \cite{Schreiber14}, such a \emph{Linear Homotopy Type Theory} (LHoTT)  has been developed \cite{RFL21}\cite{Riley22} and is now in need of formal ``semantics'' tying it to the traditional set-theoretic formulation of parameterized stable homotopy theory:

\medskip

\noindent
{\bf Candidate Categorical Semantics for a Fragment of LHoTT.}
Broadly speaking, higher categorical semantics for the homotopically-typed quantum programming language {\tt LHoTT} \cite{RFL21}\cite{Riley22}
ought to be given by $\infty$-categories of parameterized $R$-module spectra, regarded as 
equipped with the bireflective sub-$\infty$-category of plain homotopy types and as doubly closed monoidal with respect to (external-)cartesian 
and external $R$-tensor product. 
But available notions of formalization of this idea (for pointers see \cite[around (107)]{TQP})
must proceed through suitable model 1-categories which on the one hand present such
$\infty$-categories under simplicial localization while on the other hand compatibly providing ordinary 1-categorical semantics 
(the only kind of categorical semantics currently understood) for the type theory.

\medskip 
The model category $\mathbf{Loc}_{\mathbb{K}}$ constructed in Thm. \ref{GlobalModelStructure} and equipped with the homotopical external 
tensor product of Thm. \ref{ExternalTensorProductIsHomotopical}
provides candidate such semantics for the ``Motivic Yoga fragment'' of {\tt LHoTT} (cf. p. \pageref{ExternalTensorOnSimplicialLocalSystems} and \cite[Def. 2.17]{QS} --- essentially the 
content of \cite[\S 2.4]{Riley22} as envisioned in \cite[\S 3.2]{Schreiber14}) --- and (only) rudimentarily for the remaining classical
cartesian fragment (cf. Prop. \ref{ExternalInternalHomOfSimplicialLocalSystems}). The reason for (both of) these is the use of
the model of classical base types by $\mathrm{sSet}$-enriched groupoids instead of simplicial sets (cf. Rem. \ref{IssueOfFullExternalMonoidalStructure}): 
\begin{equation}
  \begin{tikzcd}
    \scalebox{.7}{  
      \color{gray}
      \begin{tabular}{l}
      Classical 
      \\
      homotopy types
      \\
      (Prop. \ref{DwyerKanModelStructures})
      \end{tabular}
    }
    &[-30pt]
    \mathrm{sSet}\mbox{-}\mathrm{Grpd}
    \ar[
      from=rr,
      shift left=14pt
    ]
    \ar[
      rr,
      hook
    ]
    \ar[
      from=rr,
      shift right=14pt
    ]
    \ar[
      rr,
      phantom,
      shift left=8pt,
      "{
        \scalebox{.7}{$\bot_{\mathrlap{\mathrm{Qu}}}$}
      }"
    ]
    \ar[
      rr,
      phantom,
      shift right=8pt,
      "{
        \scalebox{.7}{$\bot_{\mathrlap{\mathrm{Qu}}}$}
      }"
    ]
    &&
    \mathbf{Loc}_{\mathbb{K}}
    \ar[
        out=-38, 
        in=+38, 
        decorate,
        looseness=3.5,
        shift left=+2pt,
        "{ \natural }"{description}
    ]
    &[-20pt]
    \scalebox{.7}{
      \color{gray}
      \begin{tabular}{l}
        Classical modality on
        \\
        parameterized quantum
        \\
        homotopy types
      \end{tabular}
    }
  \end{tikzcd}
\end{equation}
This choice makes the theory of parameterized linear/quantum homotopy types modeled as simplicial local systems {\it over}
classical homotopy types flow  naturally via simplicial model category theory, but the model category $\mathrm{sSet}\mbox{-}\mathrm{Grpd}$ 
(in contrast to $\mathrm{sSet}$) is not itself cartesian monoidal 
model: only its  $\mathrm{Set}$-tensoring remains a Quillen bifunctor.
(In the Quillen equivalent model $\mathbf{Loc}^{\mathrm{sSet}}_{\mathbb{K}}$ 
\eqref{Loc} the situation is complementary:
Here the pushout-product axiom holds on 
the underlying base types, but now it fails for the linear components since the fundamental simplicial groupoid functor $\mathbf{G}(-)$ 
does not preserve products.)

\medskip 
On the other hand, discrete parameter base spaces are all that traditional quantum information theory has ever used so far (cf. Ex. \ref{ExternalInternalHomOverDiscreteSpaces}), so that the semantics 
for the homotopically-multiplicative \& rudimentarily-classical fragment of {\tt LHoTT} provided by $\mathbf{Loc}_{\mathbb{K}}$ is still a 
considerable homotopy-theoretic generalization of previously existing models of dependent linear types
(such as for the Proto-{\tt Quipper} language, cf. \cite{RiosSelinger18}\cite{FuKishidaSelinger20}).

\smallskip

Beyond that, the variant model structure from \cref{ExternalModulesOverOneTypes} (Thm. \ref{ExternalMonoidalModelStructureOnLocalSystemsOverOneTypes})
\begin{equation}
  \begin{tikzcd}
    \scalebox{.7}{  
      \color{gray}
      \begin{tabular}{l}
      Classical 
      \\
      homotopy 1-types
      \end{tabular}
    }
    &[-30pt]
    \Groupoids
    \ar[
      from=rr,
      shift left=14pt
    ]
    \ar[
      rr,
      hook
    ]
    \ar[
      from=rr,
      shift right=14pt
    ]
    \ar[
      rr,
      phantom,
      shift left=8pt,
      "{
        \scalebox{.7}{$\bot_{\mathrlap{\mathrm{Qu}}}$}
      }"
    ]
    \ar[
      rr,
      phantom,
      shift right=8pt,
      "{
        \scalebox{.7}{$\bot_{\mathrlap{\mathrm{Qu}}}$}
      }"
    ]
    &&
    \mathbf{Loc}
      ^\Groupoids
      _{\mathbb{K}}
    \ar[
        out=-38, 
        in=+38, 
        decorate,
        looseness=3.5,
        shift left=+2pt,
        "{ \natural }"{description}
    ]
    &[-20pt]
    \scalebox{.7}{
      \color{gray}
      \begin{tabular}{l}
        Classical modality on
        \\
        parameterized quantum
        \\
        homotopy types
        \\
        over classical 1-types
      \end{tabular}
    }
  \end{tikzcd}
\end{equation}
does make the external tensor product into a Quillen bifunctor, now at the cost of restricting the entire model to parameter 
spaces which are homotopy 1-types.
This model category $\mathrm{Loc}_{\GroundField}^\Groupoids$ may hence be thought of as a linear enhancement of the original 
``groupoid interpretation'' \cite{HofmannStreicher94}\cite{HofmannStreicher98}\cite{LewisBohnet17} of homotopy type theory.

\medskip 
While again less general than what one might hope for, notice that $\mathrm{Loc}^{\Groupoids}_{\mathbb{K}}$ provides sufficient 
infrastructure for interpreting the formalization of topological quantum gates described in \cite[Thm. 6.8]{TQP}: The parameter 
base spaces there are homotopy types of configurations spaces of points in the plane, which indeed are homotopy 1-types 
(equivalent to the delooping groupoids of braid groups). This groupoid model of parameterized quantum information is further discussed in \cite{EoS}.

\medskip

\noindent
{\bf Hilbert space structure.}
The linear structure of quantum data reflected in fiberwise stable homotopy types is actually only one of the characteristics of quantum data: The other half is what one might call the {\it metricity} structure, ultimately encoded in the Hermitian inner product on spaces of quantum states (Hilbert spaces). 
But the natural category-theoretic model for inner products (on finite-dimensional state spaces), namely by tensor self-duality, {\it fails} when applied naively in the category of complex vector spaces (or complex $\infty$-local systems) because here this yields complex bi-linear instead of the required Hermitian sesqui-linear pairings. This subtlety is ultimately the reason why the categorical quantum information community takes recourse to ``dagger-category''-structure (cf. \cite{StehouwerSteinebrunner23}), which however serves more to axiomatize the problem than to provide a reason for why to choose sesquilinear structure.

\smallskip 
With our final Corollary \ref{ModelStructureOnMonoidsInEquivariantLocalSystems} we get access to a wealth of variant models of local systems, by imposing module structure. In \cite{Reality} we explain that particularly the ``Real modules'' of Ex. \ref{TheRealLocalSystems} are relevant for quantum information theory: Internal to the resulting model category for Real $\infty$-local systems  
$\Modules{ \scalebox{.7}{$\ZTwo \acts \, \ComplexNumbers$} }\big( \Actions{\ZTwo}\big(\mathrm{Loc}_{\mathbb{R}}^{\Groupoids}\big) \! \big)$ from \eqref{ModelStructureOnModulesOnEquivariantLocalSystems}, Hilbert spaces (finite-dimensional) do exist as self-dual objects, and such that 
the operator (``dagger''-)adjoints are subsumed as data of what internally are just orthogonal maps between these. 

\smallskip 
This way, all the above statements about categorical semantics for quantum programming language may be enhanced from pure states with quantum gates between them to mixed states with quantum channels between them, along the lines of \cite[\S 2.5]{QS}. We discuss this further in \cite{Reality}.

\appendix

\section{Appendix: Some definitions and facts}

For reference, we record some basic facts from the literature
and highlight some immediate examples that we use in the main text.

\medskip

\noindent
{\bf Categories, groupoids and simplicial enrichment.} We use basic concepts from category theory (cf. \cite{MacLane97}) and enriched
category theory (cf. \cite{Kelly82}).

\begin{definition}[Categories and groupoids]
 \label{GroupoidsAndCategories}
 
With respect to any fixed Grothendieck universe $\mathfrak{U}$ of sets \cite[\S 3.2]{Schubert72} of which we assume at least two 
$\mathfrak{U} < \mathfrak{U}'$, cf. \cite[p. 4]{Levy18}\cite[p. 18]{Shulman08}:

 \begin{itemize}
 \item[{\bf (i)}] We write
 \begin{equation}
   \label{LocalizationAndCore}
   \begin{tikzcd}[column sep=large]
     \mathrm{Grpd}
     \ar[from=r, shift right=12pt, "{ \mathrm{Loc} }"{swap}]
     \ar[r, hook]
     \ar[from=r, shift left=12pt, "{ \mathrm{core} }"]
     \ar[r, phantom, shift left=6pt, "{ \scalebox{.7}{$\bot$} }"]
     \ar[r, phantom, shift right=6pt, "{ \scalebox{.7}{$\bot$} }"]
     &
     \mathrm{Cat}
   \end{tikzcd}
 \end{equation}
 for the full inclusion of the 1-category of $\mathfrak{U}$-small groupoids into the 1-category of $\mathfrak{U}$-small categories (cf. \cite[\S 3]{Schubert72}), with left adjoint $\mathrm{Loc}$ being the {\it localization}-construction that universally inverts all morphisms \cite[\S 1.5.4]{GabrielZisman67}.

(The $\mathfrak{U}'$-small categories are called {\it $\mathfrak{U}$-large}, whence $\mathrm{Cat}$ in this case is ``very large'' \cite[p. 18]{Shulman08}.

\item[{\bf (ii)}] 
More generally, we write
 \begin{equation}
   \label{sSetLocalizationAndCore}
   \begin{tikzcd}[column sep=large]
     \mathrm{sSet}\mbox{-}\mathrm{Grpd}
     \ar[from=r, shift right=12pt, "{ \mathrm{Loc} }"{swap}]
     \ar[r, hook]
     \ar[from=r, shift left=12pt, "{ \mathrm{core} }"]
     \ar[r, phantom, shift left=6pt, "{ \scalebox{.7}{$\bot$} }"]
     \ar[r, phantom, shift right=6pt, "{ \scalebox{.7}{$\bot$} }"]
     &
     \mathrm{sSet}\mbox{-}\mathrm{Cat}
   \end{tikzcd}
 \end{equation}
 for the categories of $\mathcal{V}$-enriched categories
 \cite{Kelly82} over \cite{DwyerKan80} the category $\mathcal{V} = \mathrm{sSet}$ of simplicial sets \cite[\S II]{GabrielZisman67} (review includes \cite{Riehl14})
 and for the enriched groupoids 
 \cite[\S 3]{EPR21}
 among these,
 traditionally regarded as ``simplicial groupoids'' with discrete simplicial sets of objects
 \cite[\S 5.5]{DwyerKan80}\cite{DwyerKan84}\cite[\S V.7]{GoerssJardine09}\cite[\S 9.3]{Jardine15}.
 Here the three functors in \eqref{sSetLocalizationAndCore} are degreewise those of \eqref{LocalizationAndCore}, cf. \cite[Def. 2.7]{MRZ23}.
 \end{itemize}
\end{definition}

\begin{proposition}[Cartesian closure of $\mathrm{sSet}$-enriched groupoids]
  \label{CartesianClosureOfSSetEnrichedGroupoids}
  Both $\mathrm{sSet}\mbox{-}\mathrm{Cat}$ and $\mathrm{sSet}\mbox{-}\mathrm{Grpd}$ are cartesian closed \cite[\S IV.6]{MacLane97}, with cartesian product given by forming enriched product categories \cite[\S 1.4]{Kelly82} and internal hom given by enriched functor categories \cite[\S 2.2]{Kelly82}.
\end{proposition}
\begin{proof}
  For $\mathrm{sSet}\mbox{-}\mathrm{Cat}$ this is the statement of \cite[\S 2.3]{Kelly82}. One readily checks that both constructions restrict to simplicial groupoids.
\end{proof}

\noindent
{\bf Pseudofunctors and the Grothendieck construction.}
Given a ``coherent system of categories and functors`` -- namely a pseudo-functorial diagram of categories, Def. \ref{Pseudofunctor} 
below -- the {\it Grothendieck construction} (Def. \ref{GrothendieckConstruction} below) is the natural way of merging this data into 
a single category whose morphisms subsume those of the individual categories but also transfers from one category to the other along one of the given functors.

\begin{definition}[Pseudofunctor {\cite[\S A.1]{Grothendieck60}, cf. \cite[Def. 3.10]{Vistoli05}}] 
  \label{Pseudofunctor}
  $\,$ \newline 
 \noindent {\bf (i)} For $\mathcal{B}$ a category, a {\it covariant pseudofunctor} to $\mathrm{Cat}$
  \begin{equation}
    \label{CovariantPseudofunctor}
    \begin{tikzcd}[row sep=-5pt, column sep=20pt]
      \mbox{\bf C}_{(-)}
      \;\colon\;
      &
      \mathcal{B}
      \ar[rr]
      &&
      \mathrm{Cat}
      \\
      &
      X_1 
      \ar[d, "{f}"]
      &\longmapsto&
      \mbox{\bf C}_{X_1}
      \ar[d, "{f_!}"]
      \\[+20pt]
      &
      X_2 
      &\longmapsto& 
      \mbox{\bf C}_{X_2}
    \end{tikzcd}
  \end{equation}
  is an assignment that sends 
  \begin{itemize} 
 \item  each object $B \in \mathrm{Obj}(\mathcal{B})$ to a category $\mathbf{C}_B$, 
 \item  each morphism $f\colon X_0 \to X_1$ to a functor $f_! \,\colon\,\mathbf{C}_{X_0} \to \mathbf{C}_{X_1}$, 
 \item  each pair of composable morphisms $X_0 \overset{f_{01}}{\to} X_1 \overset{f_{12}}{\to} X_2$ to a natural 
  isomorphism $(f_{12})_! \circ (f_{01})_! \Rightarrow (f_{12} \circ f_{01})_!$ 
  \begin{equation}
    \label{CompositionalCoherenceIsomorphism}
   \begin{tikzcd}
     & X_1
     \ar[
       dr, 
       "{ f_{12} }",
       "{\ }"{swap, pos=.1, name=s}
     ]
     \\
     X_0
     \ar[ur, "{f_{01}}"]
     \ar[
       rr, 
       "{ f_{02} }"{swap},
       "{\ }"{name=t}
      ]
     &&
     X_2
     \\
     \ar[from=s, to=t, shorten=3pt, equals]
   \end{tikzcd}
   \quad 
   \longmapsto
   \quad 
   \begin{tikzcd}
     & 
     \mathbf{C}_{X_1}
     \ar[
       dr, 
       "{ (f_{12})_! }",
       "{\ }"{swap, pos=.1, name=s}
     ]
     \\
     \mathbf{C}_{X_0}
     \ar[ur, "{ (f_{01})_! }"]
     \ar[
       rr, 
       "{ (f_{02})_! }"{swap},
       "{\ }"{name=t}
      ]
     &&
     \mathbf{C}_{X_2}
     \\
     \ar[
       from=s, 
       to=t, 
       shift left=10pt,
       shorten=3pt, 
       Rightarrow, 
       "{ \mu_{f_{01}, f_{12}} }"{swap},
       "{\sim}"{sloped, swap }
     ]
   \end{tikzcd}
  \end{equation}
 
  \item  and, finally, each 
  identity morphism $\mathrm{id}_X : X \to X$ to a natural isomorphism $(\mathrm{id}_X)_! \Rightarrow \mathrm{id}_{\mathbf{C}_X}$ 
  \[
    \begin{tikzcd}
      X
      \ar[rr, "{ \mathrm{id}_X }"]
      &&
      X
    \end{tikzcd}
    \;\;\;\;
    \mapsto
    \;\;\;\;
    \begin{tikzcd}[row sep=small]
      \mathbf{C}_X
      \ar[
        rr, 
        bend left=50, 
        "{ (\mathrm{id}_B)_! }",
        "{\ }"{name=s, swap}
      ]
      \ar[
        rr, 
        bend left=00, 
        "{\ }"{name=t},
        "{ \mathrm{id}_{\mathbf{C}_X} }"{swap}
      ]
      &&
      \mathbf{C}_X
      \ar[from=s, to=t, Rightarrow, "{\sim}"{sloped}]
    \end{tikzcd}    
  \]
  such that these natural isomorphisms satisfy evident associativity and unitality coherences.
\end{itemize} 
 \noindent {\bf (ii)}   Similarly, a contravariant pseudofunctor is such a pseudofunctor on the opposite category $\mathcal{B}^{\mathrm{op}}$.
  \begin{equation}
    \label{ContravariantPseudofunctor}
    \begin{tikzcd}[row sep=-5pt, column sep=20pt]
      \mbox{\bf C}_{(-)}
      \;\colon\;
      &
      \mathcal{B}^{\mathrm{op}}
      \ar[rr]
      &&
      \mathrm{Cat}
      \\
      &
      X_1 
      \ar[d, "{f}"]
      &\longmapsto&
      \mbox{\bf C}_{X_1}
      \ar[from=d, "{f^\ast}"]
      \\[+20pt]
      &
      X_2 
      &\longmapsto& 
      \mbox{\bf C}_{X_2}
    \end{tikzcd}
  \end{equation}
\end{definition}

\begin{definition}[Grothendieck construction {\cite[\S VI.8]{Grothendieck71}, cf.  \cite[\S 3.1.3]{Vistoli05}}]
  \label{GrothendieckConstruction}
 $\,$  

\begin{itemize}
\item[{\bf (i)}]  
The {\it Grothendieck construction} on a covariant pseudofunctor
$\mathbf{C}_{(-)} : \mathcal{B} \longrightarrow \mathrm{Cat}$ 
\eqref{CovariantPseudofunctor} 
is the category $\int_{X \in\mathcal{B}} \mathbf{C}_X$ whose 

  \begin{itemize}
    \item objects $\mathscr{V}_X$ are pairs $(X, \mathscr{V})$ with $X \in \mathrm{Obj}(\mathcal{B})$ and $\mathscr{V} \,\in\, \mathrm{Obj}(\mathbf{C}_{X})$,
    \item morphisms $\phi_f \,\colon\, \mathscr{V}_X \longrightarrow \mathscr{W}_{\mathrm{Y}}$ are pairs $(f,\phi)$ with $f : X \longrightarrow Y$ in $\mathcal{B}$ and 
    \fbox{$\phi \,:\, f_! \mathscr{V} \longrightarrow \mathscr{W}$ in $\mathbf{C}_{Y}$}\,,
  \end{itemize}
  hence the hom-sets of the 
  covariant Grothendieck construction
  are these dependent products:
  \begin{equation}
    \label{HomSetsOfCovariantGrothendieckConstruction}
    \Big(
    \int_{\mathcal{B}}
    \mathbf{C}
    \Big)
    \Big(
      \mathscr{V}_X
      ,\,
      \mathscr{W}_Y      
    \Big)
    \;\;:\defneq\;\;
    \big(
      f \in \mathcal{B}(X,Y)
    \big)
    \times
    \mathbf{C}_Y
    \big(
      f_! X
      ,\,
      Y
    \big).
  \end{equation}

  item[{\bf (ii)}]
  Dually, the  {\it Grothendieck construction} on a contra-variant pseudofunctor $\mathbf{C}_{(-)} : \mathcal{B}^{\mathrm{op}} \longrightarrow \mathrm{Cat}$
  \eqref{ContravariantPseudofunctor} is the category $\int_{X \in\mathcal{B}} \mathbf{C}_X$ whose 

  \begin{itemize}
    \item objects $\mathscr{V}_X$ are pairs $(X, \mathscr{V})$ with $X \in \mathrm{Obj}(\mathcal{B})$ and $\mathscr{V} \,\in\, \mathrm{Obj}(\mathbf{C}_{X})$,
    \item morphisms $\phi_f \,\colon\, \mathscr{V}_X \longrightarrow \mathscr{W}_{Y}$ are pairs $(f,\phi)$ with $f : X \longrightarrow Y$ in $\mathcal{B}$ 
    and \fbox{$\phi \,:\, \mathscr{V} \longrightarrow f^\ast \mathscr{W}$ in $\mathbf{C}_{X}$}\,,
  \end{itemize}
  hence the hom-sets of the 
  contravariant Grothendieck construction
  are these dependent products:
  \begin{equation}
    \label{HomSetsOfContravariantGrothendieckConstruction}
    \Big(
    \int_{\mathcal{B}}
    \mathbf{C}
    \Big)
    \Big(
      \mathscr{V}_X
      ,\,
      \mathscr{W}_Y      
    \Big)
    \;\;:\defneq\;\;
    \big(
      f \in \mathcal{B}(X,Y)
    \big)
    \times
    \mathbf{C}_Y
    \big(
      X
      ,\,
      f^\ast Y
    \big)
    \,.
  \end{equation}

\item[{\bf (iii)}] 
Finally, composition of morphisms 
$\!\!
  \begin{tikzcd}[sep=15pt]
    \mathscr{V}_X
    \ar[r, "{ \phi_f }"]
    &
    \mathscr{W}_{Y}
    \ar[r, "{ \psi_g }"]
    &
    \mathscr{R}_{Z}
  \end{tikzcd}
\!\!$
in the Grothendieck construction is defined by using the  pseudo-functoriality of $\mathbf{C}_{(-)}$ 
to coherently push (or pull) morphisms into the codomain or domain category:
\[
  \psi_g \circ \phi_f
  \;:=\;
  \big( 
    \psi 
      \,\circ\, 
    g_!(\phi) 
      \,\circ\,
    \mu_{f,g}(\mathscr{V})
  \big)_{g \circ f}
  \hspace{1cm}
  \mbox{or}
  \hspace{1cm}
  \psi_g \circ \phi_f
  \;:=\;
  \big(
    \mu_{f,g}(\mathscr{R})
      \,\circ\,
    g^\ast(\psi) 
      \,\circ\, 
    \phi 
  \big)_{g \circ f}
  \,.
\]
Here one is using the coherence isomorphisms \eqref{CompositionalCoherenceIsomorphism} 
to adjust for the identification of composite functors:
\[
  \begin{tikzcd}[sep=20pt]
    (g \,\circ\, f )_!(\mathscr{V})
    \ar[
      r, 
      "{ \mu_{f,g} }"{yshift=1pt}, 
      "{\sim}"{swap}
    ]
    &
    g_!\big(f_! (\mathscr{V}) \big)
    \ar[
      r,
      "{ g_!(\phi) }"
    ]
    &
    g_!\big( \mathscr{W} \big)
    \ar[r, "{ \psi }"]
    &
    \mathscr{R}
  \end{tikzcd}
  \hspace{.4cm}
  \mbox{or}
  \hspace{.4cm}
  \begin{tikzcd}[sep=20pt]
    \mathscr{V}
    \ar[
      r,
      "{ \phi }"
    ]
    &
    f^\ast\big( \mathscr{W} \big)
    \ar[
      r, 
      "{ f^\ast(\psi) }"
    ]
    &
    f^\ast
    \big(
      g^\ast(
        \mathscr{R}
      )
    \big)
    \ar[
      r,
      "{
        \mu_{f,g}
      }"{yshift=1pt}
    ]
    &
    (g \,\circ\, f)^\ast
    \mathscr{R}\;.
  \end{tikzcd}
\]
\end{itemize}
\end{definition}
\begin{remark}[Grothendieck fibration]
A key aspect of the Grothendieck construction is that it is a {\it fibered category} over the original diagram shape, 
and as such an equivalent incarnation of the pseudo-functor that induced it. While important, here we do not need this 
aspect and will regard the Grothendieck construction as a plain category, this being the domain category of 
the corresponding Gorthendieck fibration.
\end{remark}

\begin{example}[Categories of indexed sets of objects {\cite[\S 3]{Benabou85}}, Free coproduct completion {\cite[\S 2]{HT95}}]
\label{CategoriesOfIndexedSetsOfObjects}
$\,$
\newline
  For $\mathcal{C}$ any category,  there is the contravariant pseudofunctor (Def. \ref{Pseudofunctor}) on $\mathrm{Set}$ 
  which to a set $S$ assigns the $S$-fold product category of $\mathcal{C}$ with itself:
  \begin{equation}
    \label{PseudofunctorOfProductCategories}
    \hspace{-1cm}
    \mathcal{C} \,\in\, \mathrm{Cat}
    \hspace{1cm}
      \vdash
    \hspace{1cm}
    \begin{tikzcd}[row sep=-6pt, column sep=small]
      \mathrm{Set}^{\mathrm{op}}
      \ar[rr]
      &&
      \mathrm{Cat}
      \\[3pt]
      S
      \ar[
        dd,
        "{ f }"
      ]
      &\longmapsto&
      \mathrm{Func}(S,\,\mathcal{C})
      \ar[
        from=dd,
        "{ f^\ast }"{swap}
      ]
      \ar[r, phantom, "{ \defneq }"]
      &
      \mathcal{C}^S
      \;\simeq\;
      \underset{s \in S}{\prod}
      \mathcal{C}
      \\[+20pt]
      \\
      T &\longmapsto&
      \mathrm{Func}(T,\,\mathcal{C})
      \ar[r, phantom, "{ \defneq }"]
      &
      \mathcal{C}^T
      \;\simeq\;
      \underset{t \in T}{\prod}
      \mathcal{C}
      \,.
    \end{tikzcd}
  \end{equation}
  equivalently, the 
  functor category into $\mathcal{C}$ out of the discrete category on $S$:
  \[
    \begin{tikzcd}[row sep=-4pt, column sep=small]
    \mathrm{Func}(
      S
      ,\,
      \mathcal{C}
    )
    \ar[rr, "{ \sim }"]
    &&
    \underset{s \in S}{\prod}
    \mathcal{C}
    \\
    (
      s 
      \,\mapsto\,
      \mathscr{V}_s
    )
    &\longmapsto&
    (
      \mathscr{V}_s
    )_{s \in S}
    \,,
   \end{tikzcd}
  \]
  and whose base change functors are given by precomposition with, hence re-indexing by, the given map of sets:
  \[
  \hspace{-1cm} 
    f \,:\, S \longrightarrow T
    \hspace{1cm}
      \vdash
    \hspace{1cm}
    \begin{tikzcd}[row sep=-4pt, column sep=small]
      \mathrm{Func}(S,\,\mathcal{C})
      \ar[
        from=rr,
        "{ f^\ast }"{swap}
      ]
      &&
      \mathrm{Func}(T,\,\mathcal{C})
      \\
      \big(
        \mathscr{V}_{f(s)}
      \big)_{s \in S}
      &\longmapsfrom&
      (
        \mathscr{V}_{t}
      )_{t \in T}
      \,.
    \end{tikzcd}
  \]
  Accordingly, the Grothendieck construction (Def. \ref{GrothendieckConstruction}) on this pseudofunctor,

  \noindent
  \fbox{$\underset{S \in \mathrm{Set}}{\int} \!  \mathcal{C}^S$} has the following description:
  \begin{itemize}
    \item objects $\mathscr{V}_{\!S}$ are dependent pairs consisting of a set $S \in \mathrm{Set}$ and an $S$-tuple 
    $\big( \mathscr{V}_s \big)_{s \in S}$ of objects in $\mathcal{C}$,

    \item morphisms $\phi_f \,:\, \mathscr{V}_S \longrightarrow \mathscr{W}_T$ are $S$-tuples 
    $
      \big(
        \phi_s
        \,:\,
        \mathscr{V}_s
        \xrightarrow{\;\;}
        \mathscr{W}_{f(s)}
      \big)_{s \in S}
    $
    of morphisms in $\mathcal{C}$.
  \end{itemize}
  Independently of whether or how $\mathcal{C}$ has co-products,  this category has set-indexed coproducts
  ${\coprod}_i
    \mathscr{V}(i)_{S_i}$ with underlying set $\coprod_i S_i$ and components
  $\big({\coprod}_i
    \mathscr{V}(i)_{S_i}\big)_{s_j} \,=\, \mathscr{V}(j)_{s_j}$ for $s_j \in S_j$.
  \medskip
  
  But if the
  category $\mathcal{C}$ is {\it extensive}, in that it already has coproducts itself and the coproduct-functors between (products of)
  slice categories are equivalences
  \[
    S \,\in\, \mathrm{Set}
    \hspace{1cm}
      \vdash
    \hspace{1cm}
    \begin{tikzcd}[row sep=-4pt, column sep=small]
      \underset{s \in S}{\prod}
      \mathcal{C}_{/X_s}
      \ar[
        rr,
        "{ \sim }"
      ]
      &&
      \mathcal{C}_{/ \coprod_s X_s }
      \\
      \left(\!\!
      \begin{array}{c}
        E_s
        \\
        \downarrow
        \\
        X_s
      \end{array}
     \!\! \right)_{\!\!\! s \in S}
      &\longmapsto&
    \left(\!\!
      \begin{array}{c}
        \coprod_{\, s}
        E_s
        \\
        \downarrow
        \\
        \coprod_{\, s} X_s
      \end{array}
      \!\!\right)
    \end{tikzcd}
  \]
  then the construction yields the category of bundles in $\mathcal{C}$ over sets, the latter understood via the unique coproduct-preserving 
  inclusion $\iota_{\mathrm{Set}} \,\colon\, \mathrm{Set} \hookrightarrow \mathcal{C}$, hence the comma category $(\mathrm{id}_{\mathcal{C}},\,\iota_{\mathrm{Set}})$:
\[
  \mbox{$\mathcal{C}$ extensive}
  \hspace{1cm}
    \vdash
  \hspace{1cm}
  \underset{S \in \mathrm{Set}}{\int}
  \,
  \underset{s \in S}{\prod}
  \,
  \mathcal{C}
  \;\;\;
    \simeq
  \;\;\;
  (\mathrm{id}_{\mathcal{C}},\,\iota_{\mathrm{Set}})  
  \mathrlap{\,,}
\]
whose morphisms $\phi_f \,:\, X_S \longrightarrow Y_T$ are commuting diagrams in $\mathcal{C}$ of this form:
\[
  \begin{tikzcd}[row sep=small]
    \underset{s \in S}{\coprod}
    X_s
    \ar[rr, "\phi"]
    \ar[d]
    &&
    \underset{t \in T}{\coprod}
    Y_t
    \ar[d]
    \\
    S 
    \ar[rr, "{f}"]
      &&
    T
    \mathrlap{\,.}
  \end{tikzcd}
\]
Conversely, if $\mathcal{C}$ is not extensive, then we may understand $\int_{S \in \mathrm{Set}} \mathcal{C}^S$ as the stand-in for 
the would-be category of ``$\mathcal{C}$-fiber bundles'' over sets.
\end{example}
\begin{proposition}[{\cite[Lem. 4.2]{CarboniVitale98}}]
  \label{CategoriesBeingCoproductCompletionOfTheirConnectedObjects}
  A category $\mathcal{C}$ with all set-indexed coproducts each of whose objects is a coproduct of connected objects
  is the free coproduct completion (Ex. \ref{CategoriesOfIndexedSetsOfObjects}) of its full subcategory of connected objects 
  (i.e., of those objects $X \in \mathcal{C}$ for which $\mathcal{C}(X,-) : \mathcal{C} \to \mathcal{C}$ preserves coproducts).
\end{proposition}
\begin{proof}
  Since, by assumption, every object is already presented by an indexed set of connected objects, it remains to see that also the morphisms 
  $\big(\coprod_{\, s} X_s\big) \longrightarrow \big(\coprod_{\, t} Y_t\big)$ are in bijection to indexed sets of morphisms of connected 
  objects. This follows by
  $$
    \def\arraystretch{1.4}
    \begin{array}{ll}
      \mathcal{C}
      \big(
        \coprod_{\, s} X_s
        ,\,
        \coprod_{\, t} Y_t
      \big)
      &      \;\simeq\;
      \prod_{\, s} 
      \mathcal{C}
      \big(
        X_s
        ,\,
        \coprod_{\, t} Y_t
      \big)
      \\
    &  \;\simeq\;
      \underset{s \in S}{\prod} 
      \;
      \underset{t_s \in T}{\coprod}
      \mathcal{C}
      \big(
        X_s
        ,\,
        Y_{t_s}
      \big)
      \\
  &    \;\simeq\;
      \underset{f : S \to T}{\coprod}
      \;
      \underset{s \in S}{\prod}
      \mathcal{C}
      \big(
        X_s
        ,\,
        Y_{f(s)}
      \big)\,,
    \end{array}
  $$
  where the first bijection is by general properties of Hom-functors and the second is by the assumption that all $X_s$ are connected.
\end{proof}

\begin{example}[Induced adjunctions between Grothendieck constructions]
  \label{InducedAdjunctionBetweenGrothendieckConstructions}
  Given a contravariant pseudofunctor and a left adjoint functor into its domain
  \[
    \begin{tikzcd}
      \mathcal{C}
      \ar[
        rr,
        shift left=5pt,
        "{
          L
        }"
      ]
      \ar[
        from=rr,
        shift left=5pt,
        "{
          R
        }"
      ]
      \ar[
        rr,
        phantom,
        "{ \scalebox{.7}{$\bot$} }"
      ]
      &&
      \mathcal{B}
    \end{tikzcd}
    \hspace{1cm}
    \begin{tikzcd}
      \mathcal{B}^{\mathrm{op}}
      \ar[r, "{ \mathbf{C}_{(-)} }"]
      &
      \mathrm{Cat}
    \end{tikzcd}
  \]
  there is an induced adjunction between the Grothendieck constructions on $\mathbf{C}_{(-)}$ and on $\mathbf{C}_{L(-)}$, covering the given adjunction:
  \begin{equation}
    \label{}
    \begin{tikzcd}
    \Big(\,
    \underset{
      c \in \mathcal{C} 
    }{\int}  
    \mathbf{C}_{L(c)}
    \Big)
    \ar[
      rr,
      shift left=5pt,
      "{ \hat L }"
    ]
    \ar[
      from=rr,
      shift left=5pt,
      "{ \hat R }"
    ]
    \ar[
      rr,
      phantom,
      "{ \scalebox{.7}{$\bot$} }"
    ]
    \ar[d]
    &&
    \Big(\,
    \underset{
      b \in \mathcal{B} 
    }{\int}  
    \mathbf{C}_{b}
    \Big)
    \ar[d]
    \\
    \mathcal{C}
    \ar[
      rr,
      shift left=5pt,
      "{
        L
      }"
    ]
    \ar[
      from=rr,
      shift left=5pt,
      "{
        R
      }"
    ]
    \ar[
      rr,
      phantom,
      "{ \scalebox{.7}{$\bot$} }"
    ]
    &&
    \mathcal{B}
    \end{tikzcd}
  \end{equation}
  where on components in $\mathbf{C}_{(-)}$ the functor $\widehat{L}$ is the identity 
  while $\widehat{R}$ is pullback along the underlying adjunction counit $\epsilon^{L \dashv R}  \,:\, L \circ R \to \mathrm{id}$:
  \begin{equation}
    \label{LiftOfAdjunctionToGrothendieckConstruction}
    \begin{tikzcd}
      \mathscr{V}_{c}
      \ar[d, "{ \phi_f }"]
      \\
      \mathscr{V}'_{c'}
    \end{tikzcd}
    \;\;\;
    \overset{\widehat{L}}{\longmapsto}
    \;\;\;
    \begin{tikzcd}
      \mathscr{V}_{L(c)}
      \ar[d, "{ \phi_{L(f)} }"]
      \\
      \mathscr{V}'_{L(c')}
    \end{tikzcd}
    \hspace{1cm}
    \mbox{and}
    \hspace{1cm}
    \begin{tikzcd}
      \mathscr{V}_{b}
      \ar[d, "{ \phi_f }"]
      \\
      \mathscr{V}'_{b'}
    \end{tikzcd}
    \;\;\;
    \overset{\widehat{R}}{\longmapsto}
    \;\;\;
    \begin{tikzcd}
      \mathscr{V}_{R(b)}
      \ar[
        d, 
        "{ 
          (\epsilon^{L \dashv R}_{b})^\ast
          \phi_{L(f)} 
        }"
      ]
      \\
      \mathscr{V}'_{R(b')}
    \end{tikzcd}
  \end{equation}
  The counit of this adjunction is given by the identity component map covering the underlying counit:
  \begin{equation}
    \label{CounitOfLiftOfAdjunctionToGrothendieckConstruction}
    \epsilon^{\widehat{L}\dashv \, \widehat{R}}_{\mathscr{V}_{\mathbf{x}}}
    \;:\;
    \widehat{L}\, \widehat{R}
    \big(
      \mathscr{V}_{b}
    \big)
    \,=\,
    \big(
      \epsilon^{L \dashv \, R}_{L R(B)}
      \mathscr{v}
    \big)_{}.
  \end{equation}
\end{example}

\begin{proposition}[Colimits in a Grothendieck construction {\cite[\S 3.2, Thm. 2]{TBG91}\cite[Prop. 2.4.4]{HarpazPrasma15}}] 
\label{ColimitsInAGrothendieckConstruction}
$\,$

\begin{itemize}
 \item[{\bf (i)}] 
 The Grothendieck construction $\int_{X \in \mathcal{B}} \mathbf{C}_X$ (Def. \ref{GrothendieckConstruction})
 on a covariant pseudofunctor $\mathbf{C}_{(-)} : \mathcal{B} \xrightarrow{\phantom{-}} \mathrm{Cat}$ \eqref{Pseudofunctor} 
 is cocomplete as soon as the base category $\mathcal{B}$ 
  as well as all the fiber categories $\mathbf{C}_X$, $X \in \mathcal{B}$ are cocomplete.
  In this case the colimit of a small diagram
  \[
    \begin{tikzcd}[sep=0pt]
      I
      \ar[rr]
      && 
      \int_{X \in \mathcal{B}}\mathbf{C}_X
      \\
      i &\mapsto& \mathscr{V}(i)_{X_i}
    \end{tikzcd}
  \]
  is given by
  \begin{equation}
    \label{FormulaForColimitInGrothendieckConstruction}
    \underset{\underset{i \in I}{\longrightarrow}}{\lim}
    \big(
      \mathscr{V}(i)_{X_i}
    \big)
    \;\;\;\;
    \simeq
    \;\;\;\;
    \Big(
     \underset{\longrightarrow}{\mathrm{lim}}_i
    \big(
      q(i)_! \mathscr{V}(i)
    \big)
    \Big)_{
      \underset{\longrightarrow}{\mathrm{lim}}_i  
      X_i
    }
    \;\;\;\;\;
    \in
    \;
    \int_{X \in \mathcal{B}}
    \mathbf{C}_X
    \,,
  \end{equation}
  where
  \[
    i\,\in\, I
    \hspace{1cm}
    \vdash
    \hspace{1cm}
    q(i)
    \;:\;
    X_i 
    \xrightarrow{\phantom{--}}
    \underset{\longrightarrow}{\lim}_i X_i
    \;\;\;
    \in
    \;
    \mathcal{B}
  \]
  denote the coprojections into the underlying colimit in $\mathcal{B}$.
  
  \item[{\bf (ii)}] 
  The analogous dual statement holds for limits.
  \end{itemize}
\end{proposition}
\begin{example}[External cartesian product]
  \label{ExternalCartesianProduct}
  Given a contravariant pseudofunctor $\mathbf{C}_{(-)} : \mathcal{B}^{\mathrm{op}} \xrightarrow{\phantom{-}} \mathrm{Cat}$
  such that both $\mathcal{B}$ as well as all the $\mathbf{C}_{(-)}$ have Cartesian products, then its Gorthendieck 
  construction has cartesian products given by
  \begin{equation}
    \label{FormulaForExternalCartesianProduct}
    \mathscr{V}_{X}
    \times
    \mathscr{W}_{Y}
    \;\;
    \simeq
    \;\;
    \Big(\!
    \big( 
      (\mathrm{pr}_{X})^\ast \mathscr{V}
    \big)
    \times
    \big( 
      (\mathrm{pr}_{Y})^\ast \mathscr{W}
    \big)
    \!\Big)_{X \times Y}
    \,.
  \end{equation}
  More explicitly, the components of the external Cartesian product are
  \[
    \def\arraystretch{1.3}
    \begin{array}{lll}
      \big(
        \mathscr{V}_X \times \mathscr{W}_Y
      \big)_{(x,y)}
      &      \;\simeq\;
      \{(x,y)\}^\ast
      \Big(
        \big(
          (\mathrm{pr}_X)^\ast
          \mathscr{V}
        \big)
        \times
        \big(
          (\mathrm{pr}_Y)^\ast
          \mathscr{W}
        \big)
      \Big)
      \\
   &   \;\simeq\;
      \Big(
        \big(
          \{(x,y)\}^\ast
          (\mathrm{pr}_X)^\ast
          \mathscr{V}
        \big)
        \times
        \big(
          \{(x,y)\}^\ast
          (\mathrm{pr}_Y)^\ast
          \mathscr{W}
        \big)
      \Big)
      \\
    &  \;\simeq\;
      \big(
        \{x\}^\ast \mathscr{V}
      \big)
      \times
      \big(
        \{y\}^\ast \mathscr{W}
      \big)
      \\
  &    \;\simeq\;
      \mathscr{V}_{\!x}
      \times
      \mathscr{W}_{\!y}
    \end{array}
    \hspace{.7cm}
    \begin{tikzcd}
      \{x\}
      \ar[d]
      &
      \{(x,y)\}
      \ar[l, "{\sim}"{swap}]
      \ar[r, "{\sim}"]
      \ar[d, hook]
      &
      \{y\}
      \ar[d, hook]
      \\
      X 
      &
      X \times Y
      \ar[r, "{ \mathrm{pr}_Y }"]
      \ar[l, "{ \mathrm{pr}_X }"{swap}]
      &
      Y
      \mathrlap{\,.}
    \end{tikzcd}
  \]
\end{example}
This gives the following elementary fact, which is crucial in the main text:
\begin{proposition}[Free coproduct completion]
 If a category $\mathcal{C}$ has Cartesian products, then its free coproduct completion \textup{(Ex. \ref{CategoriesOfIndexedSetsOfObjectsWithCoproducts})} 
also has Cartesian products and those distribute over the coproducts.
\end{proposition}

\medskip

\noindent
{\bf The 2-category of categories with adjoint functors between them.}
We extract the gist of the discussion in \cite[p. 97-103]{MacLane97}.
\begin{definition}[Conjugate transformation of adjoints {\cite[p. 98]{MacLane97}}]
  \label{ConjugateTransformationOfAdjoints}
  Given a pair of pairs of adjoint functors between the same categories
  \[
    \begin{tikzcd}
      \mathcal{C}
      \ar[
        rr,
        shift left=5pt,
        "{ L_i }"
      ]
      \ar[
        from=rr,
        shift left=5pt,
        "{ R_i }"
      ]
      \ar[
        rr,
        phantom,
        "{ \scalebox{.7}{$\bot$} }"
      ]
      &&
      \mathcal{D}
    \end{tikzcd}
    \hspace{1cm}
    i \,\in\, \{1,2\}\,,
  \]
  then a {\it conjugate transformation} between them
  \[
    (\lambda ,\, \rho)
    \;:\;
    \begin{tikzcd}
    (L_1 \dashv R_1) 
    \ar[
      r,
      Rightarrow
    ]
    &
    (L_2 \dashv R_2)
    \end{tikzcd}
  \]
  is a pair of natural transformations of the form
  \[
    \lambda : L_1 \Rightarrow L_2
    ,\,
    \;\;\;
    \rho : R_2 \Rightarrow R_1
  \]
  such that they make the following square of natural transformations of hom-sets commute, 
  where the horizontal maps refer to the given hom-isomorphisms:
  \begin{equation}
    \label{ConjugacyOfTransformations}
    \begin{tikzcd}[row sep=small]
      \mathcal{C}\big(
        L_2(-)
        ,\,
        -
       \big)
     \ar[
       rr,
       "{ \sim }"
     ]
     \ar[
       d,
       "{
         \mathcal{C}\left(
           \lambda_{(-)} 
           ,\, 
           \mathrm{id}_{(-)}
        \right)
       }"{swap}
     ]
     &&
     \mathcal{D}\big(
       -
       ,\,
       R_2(-)
     \big)
     \ar[
       d,
       "{
         \mathcal{D}\left(
           \mathrm{id}_{(-)}
           ,\, 
           \rho_{(-)} 
        \right)
       }"
     ]
     \\[+10pt]
      \mathcal{C}\big(
        L_1(-)
        ,\,
        -
       \big)
     \ar[
       rr,
       "{ \sim }"
     ]
     &&
     \mathcal{D}\big(
       -
       ,\,
       R_1(-)
     \big)     
    \end{tikzcd}
  \end{equation}
  Such conjugate transformations compose via composition of their components $(\lambda, \rho)$, yielding a category of adjoint functors 
  with conjugate transformations between them, which we denote as follows:
  \begin{equation}
    \label{CategoryOfAdjointFunctorsWithConjugateTransformations}
    \mathcal{C},\, \mathcal{D}
    \,\in\,
    \mathrm{Cat}
    \;\;\;\;\;\;\;\;\;\;\;\;
    \vdash
    \;\;\;\;\;\;\;\;\;\;\;\;
    \mathrm{Cat}_{\mathrm{adj}}
    (\mathcal{C},\,\mathcal{D})
    \;\in\;
    \mathrm{Cat}
    \,.
  \end{equation}
\end{definition}

\begin{proposition}[Uniqueness of conjugate transformations {\cite[p. 98]{MacLane97}}]
  \label{UniquenessOfConjugateTransformations}

  Given $ L_i \dashv R_i \,:\, \mathcal{C} \rightleftarrows \mathcal{D}$ and $\lambda$
  in Def. \ref{ConjugateTransformationOfAdjoints}, there is a \emph{unique} $\rho$ that completes this data to a conjugate transformation.
  In other words, the forgetful functor from \eqref{CategoryOfAdjointFunctorsWithConjugateTransformations} to the functor category is 
  a fully faithful sub-category inclusion:
  \begin{equation}
    \label{FullEmbeddingOfCatAdhCD}
    \hspace{1.5cm} 
    \begin{tikzcd}[row sep=-4pt, column sep=small]
    \mathllap{
    \mathcal{C},\, \mathcal{D}
    \;\in\;
    \mathrm{Cat}
    \;\;\;\;\;\;\;\;\;\;\;\;
    \vdash
    \;\;\;\;\;\;\;\;\;\;\;\;    
    }
    \mathrm{Cat}_{\mathrm{adj}}(\mathcal{C},\,\mathcal{D})
    \ar[rr, hook]
    &&
    \mathrm{Cat}(\mathcal{C},\,\mathcal{D})
    \\
    (L_1 \dashv R_1)
    \ar[d, "{ (\lambda,\, \rho) }"]
    &\longmapsto&
    L_1
    \ar[
      d,
      "{ \lambda }"
    ]
    \\[23pt]
    (L_2 \dashv R_2)
    &\longmapsto&
    L_2
    \mathrlap{\,.}
    \end{tikzcd}
  \end{equation}
\end{proposition}

\begin{proposition}[Horizontal composition of conjugate transformations {\cite[p. 102]{MacLane97}}]
  \label{HorizontalCompositionOfConjugateTransformations} 
  The horizontal composition $(-)\cdot(-)$ of 
  the underlying natural transformations of a pair of conjugate transformations Def. \ref{ConjugateTransformationOfAdjoints} is itself a conjugate transformation, so that the composition functor on functor categories restricts along the inclusions \eqref{FullEmbeddingOfCatAdhCD}:
  \[
    \hspace{3cm} 
    \begin{tikzcd}[sep=0pt]
      \mathllap{
    \mathcal{C}
    ,\,
    \mathcal{D}
    ,\,
    \mathcal{E}
    \;\in\;
    \mathrm{Cat}
    \;\;\;\;\;\;\;\;\;\;\;
    \vdash
    \;\;\;\;\;\;\;\;\;\;\;      
      }
      \mathrm{Cat}_{\mathrm{adj}}(\mathcal{D},\,\mathcal{E})
      \times
      \mathrm{Cat}_{\mathrm{adj}}(\mathcal{C},\,\mathcal{D})
      \ar[
        rr
      ]
      &&
      \mathrm{Cat}_{\mathrm{adj}}(\mathcal{C},\,\mathcal{E})
      \\
     \scalebox{0.9}{$   \big(
        (\lambda,\,\rho)
        ,\,
        (\lambda',\,\rho')
      \big)
      $}
      &\longmapsto&
   \scalebox{0.9}{$   \big(
        \lambda' \cdot \lambda
        ,\,
        \rho \cdot \rho'
      \big).
      $}
    \end{tikzcd}
  \]
\end{proposition}

Via Prop. \ref{HorizontalCompositionOfConjugateTransformations}, we have: 

\begin{definition}[2-category of categories, adjoint functors and conjugate transformations {\cite[p. 102]{MacLane97}}] 
\label{CatAdj}
  Write
  \begin{equation}
    \label{FunctorFromCatAdjToCat}
    \begin{tikzcd}
      \mathrm{Cat}_{\mathrm{adj}}
      \ar[r]
      &
      \mathrm{Cat}
    \end{tikzcd}
  \end{equation}
  for the (very large) locally full sub-2-category of $\mathrm{Cat}$ whose
  \begin{itemize}
    \item objects are categories,
    \item hom-categories are those \eqref{CategoryOfAdjointFunctorsWithConjugateTransformations}
    of adjoint functors with conjugate transformations between them.
  \end{itemize}
\end{definition}

\begin{proposition}[Bivariant pseudofunctors,
{cf. \cite[Lem. 9.1.2]{Jacobs98}\cite[Prop. 2.2.1]{HarpazPrasma15}\cite[pp. 10]{CagneMellies20}}] 
\label{RecognitionOfBivariantPseudofunctors}
  Given a covariant pseudofunctor $\mathbf{C}_{(-)}$ (Def. \ref{Pseudofunctor}) such that each component 
  functor $f_! : \mathbf{C}_X \longrightarrow \mathbf{C}_{Y}$ has a right adjoint
  \begin{equation}
    \label{BivariantPseudofunctor}
    \begin{tikzcd}[row sep=-2pt, column sep=20pt]
      \mbox{\bf C}_{(-)}
      \;\colon\;
      &
      \mathcal{B}
      \ar[rr]
      &&
      \mathrm{Cat}
      \\
      &
      X_1 
      \ar[d, "{f}"]
      &\longmapsto&
      \mbox{\bf C}_{X_1}
      \ar[d, shift right=7pt, "{f_!}"{swap}]
      \ar[from=d, shift right=7pt, "{f^\ast}"{swap}]
      \ar[d, phantom, "{ \scalebox{.8}{$\dashv$} }"]
      \\[+20pt]
      &
      X_2 
      &\longmapsto& 
      \mbox{\bf C}_{X_2}
    \end{tikzcd}
  \end{equation}
  then:
  \begin{itemize}
    \item[{\bf (i)}] it factors essentially uniquely through $\mathrm{Cat}_{\mathrm{adj}}$ \eqref{FunctorFromCatAdjToCat},
    \item[{\bf (ii)}]  hence it induces a contravariant pseudofunctor with component functors $f^\ast$,
    \item[{\bf (iii)}] 
    such that the Grothendieck construction (Def. \ref{GrothendieckConstruction}) on the covariant pseudofunctor is equivalent to 
    that on the corresponding contravariant pseudofunctor via the functor that is the identity on objects and on morphisms is 
    the hom-isomorphism of the given adjoint pairs:
    \[
      \begin{tikzcd}[row sep=-2pt]
        \mathscr{f_! \mathscr{V}}
        \ar[r, "{\widetilde{\phi}}"]
        &
        \mathscr{W}
        \\
        X \ar[r, "f"]
        &
        Y
      \end{tikzcd}
      \;\;\;\;\;\;\;\;\;\;
      \leftrightarrow
      \;\;\;\;\;\;\;\;\;\;
      \begin{tikzcd}[row sep=-2pt]
        \mathscr{\mathscr{V}}
        \ar[r, "{{\phi}}"]
        &
        f^\ast \mathscr{W}
        \\
        X \ar[r, "f"]
        &
        Y
        \mathrlap{\,.}
      \end{tikzcd}
    \]
    Therefore, both construction are still unambiguously denoted by $\int_{X \in \mathcal{B}} \mathbf{C}_{X}$.
    \end{itemize}
\end{proposition}
\begin{proof}
  The first statement is a direct consequence of Prop. \ref{UniquenessOfConjugateTransformations}, the second then follows 
  by Prop. \ref{HorizontalCompositionOfConjugateTransformations} and finally the third by the property \eqref{ConjugacyOfTransformations} 
  in Def. \ref{ConjugateTransformationOfAdjoints}.
\end{proof}

In refinement of Ex. \ref{CategoriesOfIndexedSetsOfObjects}, we have:

\begin{example}[Categories of indexed sets of objects with coproducts]
\label{CategoriesOfIndexedSetsOfObjectsWithCoproducts}
If a category $\mathcal{C}$ already has all coproducts, then the pseudofunctor \eqref{PseudofunctorOfProductCategories}
of its product categories has left adjoint component functors given by forming coproducts over fibers of base maps
\begin{equation}
\label{BaseChangeFunctorsBetweenProductCategories}
  f \,:\, S \longrightarrow T
  \hspace{1cm}
    \vdash
  \hspace{1cm}
  \begin{tikzcd}[row sep=-3pt, column sep=10pt]
 \scalebox{0.8}{$   \big(
      \mathscr{V}_s
    \big)_{ \in S}
    $}
    &\xmapsto{\qquad}&
   \scalebox{0.8}{$  \bigg(\,
      \underset{
        s \in f^{-1}(\{t\})
      }{\coprod}
      \mathscr{V}_s
    \bigg)_{t \in T}
    $}
    \\
    \mathrm{Func}
    \big(
      S
      ,\,
      \mathcal{C}
    \big)
    \ar[
      rr,
      shift left=5pt,
      "{ f_! }"
    ]
    \ar[
      from=rr,
      shift left=5pt,
      "{ f^\ast }"
    ]
    \ar[
      rr,
      phantom,
      "{ \scalebox{.7}{$\bot$} }"
    ]
    &&
    \mathrm{Func}
    \big(
      T
      ,\,
      \mathcal{C}
    \big)    
    \\[6pt]
  \scalebox{0.8}{$   \big(
      \mathscr{V}_{f(s)}
    \big)_{s \in S}
    $}
    \ar[
      rr,
      phantom,
      "{ \xmapsfrom{\qquad} }"
    ]
    &&
 \scalebox{0.8}{$    \big(
      \mathscr{V}_{t}
    \big)_{t \in T}
    $}
  \end{tikzcd}
\end{equation}
Consequently, here Prop. \ref{RecognitionOfBivariantPseudofunctors} says that we have in fact a bivariant pseudofunctor: 
\[
    \begin{tikzcd}[row sep=-3pt, column sep=15pt]
      \mathrm{Set}
      \ar[rr]
      &&
      \mathrm{Cat}_{\mathrm{adj}}
      \\
      S
      \ar[
        dd,
        "{ f }"
      ]
      &\longmapsto&
      \mathrm{Func}(S,\,\mathcal{C})
      \ar[
        from=dd,
        shift right=7pt,
        "{ f^\ast }"{swap}
      ]
      \ar[
        dd,
        shift right=7pt,
        "{ f_! }"{swap}
      ]
      \ar[
        dd,
        phantom,
        "{ \scalebox{.7}{$\dashv$} }"
      ]
      \ar[r, phantom, "{ \simeq }"]
      &[1pt]
      \underset{s \in S}{\prod}
      \mathcal{C}
      \\[+20pt]
      \\
      T &\longmapsto&
      \mathrm{Func}(T,\,\mathcal{C})
      \ar[r, phantom, "{ \simeq }"]
      &
      \underset{t \in T}{\prod}
      \mathcal{C}
      \,.
    \end{tikzcd}
\]
\end{example}

More generally:
\begin{example}[Systems of enriched functor categories]
  \label{SystemsOfEnrichedFunctorCategories}
  Let $\mathbf{C}$ be an $\mathrm{sSet}$-enriched bicomplete and $\mathrm{sSet}$-(co)tensored category. 
    Then for every $\mathrm{sSet}$-enriched functor $\mathbf{f} \,:\, \mathbf{X} \xrightarrow{\;} \mathbf{Y}$
    between small $\mathrm{sSet}$-enriched categories
  the precomposition functors $\mathbf{f}^\ast$ between the $\mathrm{sSet}$-enriched functor categories
  into $\mathbf{C}$ 
  \[
    \mathbf{C}^{(\mbox{-})}
    \,:\defneq\,
    \mathbf{sFunc}(\mbox{-},\mathbf{C})
    \,,
  \]
  has a left adjoint $\mathbf{f}_!$ and a right adjoint $\mathbf{f}_\ast$:
  \begin{equation}
    \label{KanExtensionAdjointTriple}
    \begin{tikzcd}
       \mathbf{C}^{\mathbf{X}}
       \ar[
         rr,
         shift left=14pt,
         "{ \mathbf{f}_! }"
       ]
       \ar[
         rr,
         shift right=14pt,
         "{ \mathbf{f}_\ast }"{swap}
       ]
       \ar[
         from=rr,
         "{
           \mathbf{f}^\ast
         }"{description}
       ]
       \ar[
         rr,
         phantom,
         shift left=8pt,
         "{ \scalebox{.7}{$\bot$} }"
       ]
       \ar[
         rr,
         phantom,
         shift right=8pt,
         "{ \scalebox{.7}{$\bot$} }"
       ]
       &&
       \mathbf{C}^{\mathbf{Y}}
       \,,
    \end{tikzcd}
  \end{equation}
  given by enriched left and right Kan extension, respectively, 
  expressed by the following (co)end formulas \cite[(4.24), (4.25)]{Kelly82}:
  \begin{equation}
    \label{CoEndFormulasForKanExtension}
    \mathscr{V} \,\in\,
    \mathbf{C}^{\mathbf{X}}
    \hspace{1cm}
      \vdash
    \hspace{1cm}
    \left\{
    \def\arraystretch{2}
    \begin{array}{l}
      (f_! \mathscr{V})
      \;:\;
      y \,\mapsto\,
      \int^{x \in \mathbf{X}}
      \,
      \mathbf{Y}\big(\mathbf{f}(x),\,y\big)
      \cdot
      \mathscr{V}_{\!x}
      \\
      (\mathbf{f}_\ast \mathscr{V})
      \;:\;
      y \,\mapsto\,
      \int_{x \in \mathbf{X}}
      \,
      \big(
        \mathscr{V}_{\!x}
      \big)^{
        \scalebox{.7}{$
         \mathbf{Y}\big(y, \, \mathbf{f}(x)\big)
        $}
      }
    \end{array}
    \right.
  \end{equation}  
  In particular, this gives a bivariant pseudo-functor on small $\mathrm{sSet}$-enriched categories:
  \[
    \begin{tikzcd}[sep=0pt]
      \mathllap{
        \mathbf{C}^{(\mbox{-})}
        \;:\;\;
      }
      \mathrm{sSet}\mbox{-}\mathrm{Cat}_{\mathrm{sm}}
      \ar[rr]
      &&
      \mathrm{Cat}_{\mathrm{adj}}
      \\
      \mathbf{X} 
      \ar[d, "{\mathbf{f}}"]
        &\mapsto& 
      \mathbf{C}^{\mathbf{X}}
      \ar[
        d,
        shift right=8pt,
        "{
          \mathbf{f}_!
        }"{swap}
      ]
      \ar[
        from=d,
        shift right=8pt,
        "{
          \mathbf{f}^\ast
        }"{swap}
      ]
      \ar[
        d,
        phantom,
        "{ \scalebox{.7}{$\dashv$} }"
      ]
      \\[30pt]
      \mathbf{Y} 
        &\mapsto& 
      \mathbf{C}^{\mathbf{Y}}   
      \mathrlap{\,.}
    \end{tikzcd}
  \]
\end{example}

\medskip

\noindent
{\bf Model category theory.}

\begin{lemma}[Ken Brown's Lemma {\cite[Lem. 1.1.12]{Hovey99}, based on \cite[p. 421]{Brown73}}]
  \label{KenBrownLemma}
  Let $F \,:\, \mathcal{D} \longrightarrow \mathcal{C}$ be a functor between (underlying categories of) model categories. Then:
  \begin{itemize}
    \item[{\bf (i)}] If $F$ sends acyclic fibrations to weak equivalences then it sends all weak equivalences between fibrant 
    objects to weak equivalences.
    \item[{\bf (ii)}]  If $F$ sends acyclic cofibrations to weak equivalences then it sends all weak equivalences between 
    cofibrant objects to weak equivalences.
  \end{itemize}
\end{lemma}

As a simple but important special case of right transfer:
\begin{proposition}[Model structure transfer along adjoint equivalence]
  \label{ModelStructureTransferAlongAdjointEquivalence}
  Given an adjoint equivalence of categories
  $
    \begin{tikzcd}
    \mathcal{D}
      \ar[r, shift right=7pt, "{R}"{description}]
      \ar[from=r, shift right=7pt, "{ L }"{description}]
      \ar[r, phantom, "{ \scalebox{.6}{$\bot$} }"]
      &
    \mathcal{C}
    \end{tikzcd}
  $
  and a model structure on $\mathcal{C}$, then $\mathcal{D}$ becomes a model category and the adjunction becomes a Quillen equivalence
  by setting $\mathrm{W}(\mathcal{D}) \,\defneq\, R^{-1}\big(\mathrm{W}(\mathcal{C})\big)$, $\mathrm{Fib}(\mathcal{D}) \,\defneq\, R^{-1}\big(\mathrm{Fib}(\mathcal{C})\big)$,
  $\mathrm{Cof}(\mathcal{D}) \,\defneq\, R^{-1}\big(\mathrm{Fib}(\mathcal{C})\big)$.
\end{proposition}

\medskip

\noindent
{\bf Model category structures on Grothendieck constructions.} We recall the main point of \cite{HarpazPrasma15}\cite{CagneMellies20}, which goes back to \cite{Roig94}\cite{Stanulescu12}.

\begin{definition}[2-category of model categories {\cite[p. 24]{Hovey99}, cf. \cite[Def. 2.5.3]{HarpazPrasma15}}] 
\label{ModCat}
Write 
  \begin{equation}
    \label{FunctorFromModCatAdjToCat}
    \begin{tikzcd}
      \mathrm{ModCat}
      \ar[r]
      &
      \mathrm{Cat}_{\mathrm{adj}}
      \ar[r]
      &
      \mathrm{Cat}
    \end{tikzcd}
  \end{equation}
  for the (very large) 2-category whose
\begin{itemize}
  \item objects are model categories,
  \item 1-morphisms are Quillen adjunctions regarded in the direction of the left adjoint,
  \item 2-morphisms are conjugate transformations (Def. \ref{ConjugateTransformationOfAdjoints})
  between the underlying adjoint functors,
\end{itemize}
equipped with its forgetful 2-functor to $\mathrm{Cat}_{\mathrm{adj}}$ (Def. \ref{CatAdj}).
\end{definition}

\begin{definition}[Integral model structure {\cite[Def. 3.0.4]{HarpazPrasma15}}] 
\label{IntegralModelStructure}
  Given a model category $\mathcal{B}$ 
  and a pseudofunctor (Def. \ref{Pseudofunctor}) on $\mathcal{B}$ with values in model categories (Def. \ref{ModCat})
  \[
    \mathbf{C}_{(-)} : \mathcal{B} \longrightarrow \mathrm{ModCat} \longrightarrow \mathrm{Cat}_{\mathrm{adj}}
  \]
  then we call a morphism 
  \[
    \phi_f 
      \,:\, 
    \mathscr{V}_{X} 
      \longrightarrow 
    \mathscr{V}'_{X'}
    \;\;\;\;\;\;\;
    \in
    \;
    \int_{X \in \mathcal{B}} \mathbf{C}_X
  \]
  in its Grothendieck construction (Def. \ref{GrothendieckConstruction}):
  \begin{itemize}
    \item[{\bf (i)}] 
    an {\it integral weak equivalence} if
    \begin{itemize}
      \item[{\bf (a)}] $f : X \longrightarrow X'$ is a weak equivalence in $\mathcal{B}$,
      \item[{\bf (b)}] $f_!(\mathscr{V}^{\mathrm{cof}}) \overset{f_!(p)}{\longrightarrow} f_!(\mathscr{V}) \overset{\tilde \phi}{\longrightarrow} \mathscr{V}'$ is a weak equivalence in $\mathbf{C}_{X'}$,
      for $p : \mathscr{V}^{\mathrm{cof}} \to \mathscr{V}$ a cofibrant replacement in $\mathbf{C}_{X}$,
      \item[{\bf($\overline{\bf{b}}$)}] which, when $f_! \dashv f^\ast$ is a Quillen equivalence, is equivalent to:
      
      $
          \mathscr{V}
          \xrightarrow{\phi}
          f^\ast(\mathscr{V}')
          \xrightarrow{ f^\ast(q) }
          f^\ast(\mathscr{V}'_{\mathrm{fib}})
      $ is a weak equivalence in $\mathbf{C}_X$,
      for $q : \mathscr{V}' \to \mathscr{V}'_{\mathrm{fib}}$ a fibrant replacement in $\mathbf{C}_{X'}$;
    \end{itemize}
    \item[{\bf (ii)}]  an {\it integral fibration} if
    \begin{itemize}
      \item[{\bf (a)}] $f : X \longrightarrow X'$ is a fibration in $\mathcal{B}$;
      \item[{\bf (b)}] $\phi : \mathscr{V} \longrightarrow 
        f^\ast(\mathscr{V}')$ is a fibration in $\mathbf{C}_X$,
    \end{itemize}    
    \item[{\bf (iii)}]  an {\it integral cofibration} if
    \begin{itemize}
      \item[{\bf (a)}] $f : X \longrightarrow X'$ is a cofibration in $\mathcal{B}$,
      \item[{\bf (b)}] $\widetilde{\phi} : f_!(\mathscr{V}) \longrightarrow \mathscr{V}$ is a cofibration in $\mathbf{C}_{X'}$.
    \end{itemize}        
  \end{itemize}
\end{definition}

\begin{proposition}[Existence of integral model structures {\cite[Thm. 3.0.12]{HarpazPrasma15}}]
\label{ExistenceOfIntegralModelStructure}
  The classes of morphisms in Def. \ref{IntegralModelStructure} constitute a model category structure if, given $f : X \to X'$ 
  in $\mathcal{B}$, the following conditions are satisfied:
  \begin{itemize}
    \item[{\bf (i)}]  if $f$ is a weak equivalence then $f_! \dashv f^\ast : \mathbf{C}_X \rightleftarrows\mathbf{C}_{X'}$ 
    is a Quillen equivalence,
    \item[{\bf (ii)}]  if $f$ is an acyclic fibration then 
      $f^\ast : \mathbf{C}_{X'} \longrightarrow \mathbf{C}_X$ preserves weak equivalences,
    \item[{\bf (iii)}]  if $f$ is an acyclic cofibration then 
      $f_! : \mathbf{C}_X \longrightarrow \mathbf{C}_{X'}$ preserves weak equivalences.
  \end{itemize}
\end{proposition}

\begin{example}[Integral model structure over trivial model structure]
\label{IntegralModelStructureOvertrivialmodelStructure}
  In the case that the base category $\mathcal{B}$  in Def. \ref{IntegralModelStructure}
  is the ``trivial'' model structure on a bicomplete category -- whose weak equivalences are just the isomorphisms and all whose morphisms 
  are fibrations and cofibrations -- then the conditions in Prop. \ref{ExistenceOfIntegralModelStructure} are satisfied, and hence the 
  integral model structure on any pseudofunctor $\mathbf{C}_{(-)} \,:\,\mathcal{B} \longrightarrow \mathrm{ModCat}$ exists.
\end{example}

\begin{lemma}[Enhanced enrichment of functor categories]
\label{EnhancedEnrichmentOfFunctorCategories}
Consider
\begin{itemize}
  \item $\mathbf{V}$ a bicomplete symmetric closed monoidal category, regarded as canonically enriched over itself
  via its internal hom $[-,-]$;
  \item $(\mathbf{C}, \otimes)$ a complete $\mathbf{V}$-enriched category that is also $\mathbf{V}$-(co)tensored;
  \item $\mathbf{X}$ a small $\mathbf{V}$-enriched category.
\end{itemize}

\begin{itemize}
\item[{\bf (i)}]
Then the enriched functor category
$
  \mbox{\bf{V}}^{\scalebox{.7}{\bf{X}}} 
  :=\,
  \mathrm{Func}(\mathbf{X},\mathbf{V})
$
carries symmetric closed monoidal category structure with respect to the $\mathbf{X}$-objectwise tensor product
\[
  \mathcal{S}_{\mathbf{X}},
  \mathcal{T}_{\mathbf{X}}
  \,\in\,
  \mathbf{V}^{\mathbf{X}}
  \hspace{.7cm}
  \vdash
  \hspace{.7cm}
  \mathcal{S}_{\mathbf{X}}
  \otimes
  \mathcal{T}_{\mathbf{X}}
  \;:=\;
  \big(
    x 
      \,\mapsto\,
    \mathcal{S}_x \otimes \mathcal{T}_x
  \big)
  \;\;\;
  \in
  \;
  \mathbf{V}^{\mathbf{T}},
\]
whose corresponding internal hom is given by
\[
  \mathcal{S}_{\mathbf{X}},
  \mathcal{T}_{\mathbf{X}}
  \,\in\,
  \mathbf{V}^{\mathbf{X}}
  \hspace{.7cm}
  \vdash
  \hspace{.7cm}
  \big[
  \mathcal{S}_{\mathbf{X}}
  \,,\,
  \mathcal{T}_{\mathbf{X}}
  \big]
  \;:=\;
  \Big(
    x 
      \,\mapsto\,
    \int_{x \in \mathbf{X}}
    \mathbf{V}\big(
      \mathbf{X}(x,x')
      \otimes
      \mathcal{S}_{x'}
      \;,\;
      \mathcal{T}_{x'}
    \big)
  \Big)
  \;\;\;
  \in
  \;
  \mathbf{V}^{\mathbf{T}}
  \mathrlap{\,.}
\]

\item[{\bf (ii)}]
Furthermore, $\mathbf{C}^{\mathbf{X}} :=\, \mathrm{Func}(\mathbf{X},\mathbf{C})$ becomes $\mathbf{V}^{\mathbf{X}}$-enriched, 
-tensored and -cotensored via the following end-formulas, respectively:
\begin{equation}
  \def\arraystretch{1.8}
  \def\arraycolsep{10pt}
  \begin{array}{rcl}  
    \mathscr{V}_{\scalebox{.7}{\bf{X}}}
    ,\,
    \mathscr{W}_{\scalebox{.7}{\bf{X}}}
    \,\in\,
    \mbox{\bf{C}}^{\scalebox{.7}{\bf{X}}}
    &
    \vdash
    &
    \mbox{\bf{C}}^{\scalebox{.7}{\bf{X}}}
    \big(
      \mathscr{V}_{\scalebox{.7}{\bf{X}}}
      ,\,
      \mathscr{W}_{\scalebox{.7}{\bf{X}}}
    \big)
    \;:=\;
    \Big(
      x 
      \,\mapsto\,
      \int_{\scalebox{.7}{$x' \!\in\! \mbox{\bf{X}}$}}
      \mbox{\bf{C}}
      \big(
        \mbox{\bf{X}}(x,x')
        \cdot
        \mathscr{V}_{\!x'}
        ,\,
        \mathscr{W}_{\!x'}
      \big)
   \! \Big)
    \;\;\;
    \in
    \;
    \mbox{\bf{V}}^{\scalebox{.7}{\bf{X}}}
    \\
    \mathcal{S}_{\scalebox{.7}{\bf{X}}}
    \,\in\,
    \mbox{\bf{V}}^{\scalebox{.7}{\bf{X}}}
    ,\;
    \mathscr{W}_{\scalebox{.7}{\bf{X}}}
    \,\in\,
    \mbox{\bf{C}}^{\scalebox{.7}{\bf{X}}}
    &
    \vdash
    &
  \mathcal{S}_{\scalebox{.7}{\bf{X}}}
  \cdot
  \mathscr{V}_{\scalebox{.7}{\bf{X}}}
  \;:=\;
  \big(
  x 
    \,\mapsto\,
  \mathcal{S}_{\!x}
  \cdot
  \mathscr{V}_{\!x}
  \big)
  \;\;\;
  \in
  \;
  \mbox{\bf{C}}^{\scalebox{.7}{\bf{X}}}
  \\
    \mathcal{S}_{\scalebox{.7}{\bf{X}}}
    \,\in\,
    \mbox{\bf{V}}^{\scalebox{.7}{\bf{X}}}
    ,\;
    \mathscr{W}_{\scalebox{.7}{\bf{X}}}
    \,\in\,
    \mbox{\bf{C}}^{\scalebox{.7}{\bf{X}}}
    &
    \vdash
    &
    \big(\mathscr{V}_{\scalebox{.7}{\bf{X}}}\big)^{
      \mathcal{S}_{\scalebox{.5}{\bf{X}}}
    }
    \;:=\;
    \Big(
      x 
      \,\mapsto\,
      \int_{\scalebox{.7}{$x' \!\in\! \mbox{\bf{X}}$}}
      \big(\mathscr{W}_{\!x'}\big)^{
        \scalebox{.7}{$
        \mbox{\bf{X}}(x,x')
        \cdot
        \mathcal{S}_{\!x'}      
        $}
      }
    \Big)
    \;\;\;
    \in
    \;
    \mbox{\bf{C}}^{\scalebox{.7}{\bf{X}}}.
  \end{array}
\end{equation}
\end{itemize}
\end{lemma}
\begin{proof}
This follows by standard manipulations and may be folklore but hard to cite from the literature; we have spelled out the details at: 
\href{https://ncatlab.org/nlab/show/enriched+functor+category\#EnhancedEnrichment}{\tt ncatlab.org/nlab/show/enriched+functor+category\#EnhancedEnrichment}.
\end{proof}

\medskip

\noindent
{\bf Monoidal and enriched model categories.} 
A monoidal category with a model category structure is a {\it monoidal model category} (cf. \cite[\S 4]{Hovey99}\cite[\S 2]{SchwedeShipley00}\cite[\S A.3.1.2]{Lurie09}) and an enriched category with a model structure is an  {\it enriched model category} (\cite[\S II.2]{Quillen67}) if\footnote{
  For a monoidal model category one requires in addition a ``unit axiom'' on the tensor unit. But this axiom is automatically satisfied as soon as the tensor unit is a cofibrant object, which is the case in all cases in the main text.
} the tensor product or tensoring, respectively, is a left Quillen bifunctor (Def. \ref{LeftQuillenBifunctor} below):
\begin{definition}[Pushout-product]
  \label{PushoutProduct}
  Given a bifunctor $\otimes \,:\, \mathcal{D}_1 \times \mathcal{D}_2 \xrightarrow{\phantom{-}} \mathcal{C}$ into a category 
  with pushouts, then the corresponding {\it pushout-product} of a pair of morphisms $f : X \to X'$ in $\mathcal{D}_1$ 
  and $g : Y \to Y'$ in $\mathcal{D}_2$ is the universal dashed map in $\mathcal{C}$ given by the following diagram:
  \begin{equation}
    \label{PushoutProductDiagram}
    \begin{tikzcd}[row sep=large, column sep=huge]
      X \otimes Y
      \ar[
        r,
        "{
          \mathrm{id} \,\otimes\, g
        }"
      ]
      \ar[
        d,
        "{
          f \,\otimes\, \mathrm{id}
        }"{swap}
      ]
      \ar[
        dr,
        phantom,
        "{
          \scalebox{.7}{(po)}
        }"
      ]
      &
      X \otimes Y'
      \ar[
        d,
        "{ q_r }"
      ]
      \ar[
        ddr,
        bend left=20,
        "{ f \,\otimes\, \mathrm{id} }"{sloped}
      ]
      &[-15pt]
      \\
      X' \otimes Y
      \ar[
        r,
        "{ q_l }"
      ]
      \ar[
        drr,
        bend right=20,
        "{ 
           \mathrm{id} \,\otimes\, g 
        }"{sloped}        
      ]
      &
      X' \otimes Y
      \underset{
        X\otimes Y
      }{\coprod}
      X\otimes Y'
      \ar[
        dr,
        dashed,
        shorten <=-16pt,
        "{
          f \,\widehat{\otimes}\, g
        }"{description}
      ]
      \\[-15pt]
      &&
      X' \otimes Y'\;.
    \end{tikzcd}
  \end{equation}
\end{definition}
\begin{example}[Cartesian pushout-products of sets]
  \label{CartesianPushoutProductOfSets}
  In the category $\mathrm{Set}$ with respect to the Cartesian product $\mathrm{Set} \times \mathrm{Set} \xrightarrow{\times} \mathrm{Set}$, 
  the pushout of $\mathrm{id} \times g$ along $f \times \mathrm{id}$ is the quotient set
  \[
    f \,\widehat{\times}\,\ g
    \;\;
    \simeq
    \;\;
    \big\{
      (x,y')
      ,\,
      (x',y)
    \big\}_{\big/ \big( 
      (
        f(x),\, y 
      )
      \,\sim\,
      (
        x,\, g(y)
      )
    \big)}
  \]
  (where all variables range over the sets denoted by the corresponding capital letters)
  whose equivalence classes we denote by $[x,y']$ and  $[x',y]$,
  on which the pushout-product map \eqref{PushoutProductDiagram} is given by
  \[
    \begin{tikzcd}[row sep=-3pt]
      f \,\widehat{\times}\, g
      \ar[rr, dashed]
      &&
      X' \times Y'
      \\
    {[x,y']} 
    &\longmapsto&  
    \big[f(x),\,y'\big] 
    \\
    {[x',y]} 
   &\longmapsto&  
   \big[x',\,g(y)\big]
  \end{tikzcd}
  \]
  whose fibers are as follows:
  \begin{equation}
    \label{FibersOfPushoutProductMapInSets}
    \big(
      f \,\widehat{\times}\, g
    \big)_{(x',\,y')}
    \;\;
      \simeq
    \;\;
    \left\{
    \def\arraystretch{1.4}
    \begin{array}{cll}
      \ast 
      &\big\vert& 
      (x',y') \,\in\, \mathrm{im}(f) \times \mathrm{im}(g)
      \\
      f^{-1}(
        \{x'\}
      )
      \;\simeq\;
      \big\{
        [x,y']
        \;
        \big\vert
        \;
        x 
        \,\in\,
        f^{-1}(
          \{x'\}
        )
      \big\}
      &
      \big\vert
      &
      y' \,\in\, Y' \setminus \mathrm{im}(g)
      \\
      g^{-1}(
        \{y'\}
      )
      \;\simeq\;
      \big\{
        [x',y]
        \;
        \big\vert
        \;
        y 
        \,\in\,
        g^{-1}(
          \{y'\}
        )
      \big\}
      &
      \big\vert
      &
      x' \,\in\, X' \setminus \mathrm{im}(f)\;.
    \end{array}
    \right.
  \end{equation}
\end{example}
\begin{example}[Pushout-product with an identity]
  With respect to any bifunctor $(-)\otimes(-)$, 
  forming the pushout-product (Def. \ref{PushoutProduct})
  with an identity morphism yields the identity morphism on the codomain:
  \begin{equation}
    \label{PushoutProductWithAnIdentity}
    f \,\widehat{\otimes}\, \mathrm{id}_{Y}
    \;\simeq\;
    \mathrm{id}_{X' \otimes Y}
    \hspace{1cm}
    \mbox{and}
    \hspace{1cm}
    \mathrm{id} 
      \,\widehat{\otimes}\, 
    g
    \;\simeq\;
    \mathrm{id}_{X \otimes Y'}
    \,.
  \end{equation}
  Moreover, if $\mathcal{C}$ has initial objects $\varnothing$ and the bifunctor preserves the initial object 
in each argument separately, then the pushout-product with an initial morphism is given by $\otimes$:
  \begin{equation}
    \label{PushoutProductWithAnInitialMorphism}
    \big(
      \varnothing
      \to
      X'
    \big)
    \,\widehat{\otimes}\,
    g
    \;\simeq\;
    \mathrm{id}_{X'} \otimes g
    \hspace{1cm}
    \mbox{and}
    \hspace{1cm}
    f
    \,\widehat{\otimes}\,
    \big(
      \varnothing
      \to
      Y'
    \big)
    \;\simeq\;
    f
      \otimes 
    \mathrm{id}_{Y'} \;.
  \end{equation}
\end{example}

\begin{definition}[Left Quillen bifunctor]
  \label{LeftQuillenBifunctor}
  Given model categories $\mathcal{D}_1, \mathcal{D}_2$ and $\mathcal{C}$ a functor of the form
  \[
    \otimes
    \,:\,
    \begin{tikzcd}
      \mathcal{D}_1
      \times
      \mathcal{D}_2
      \ar[r]
      &
      \mathcal{C}
    \end{tikzcd}
  \]
  is called a (left) {\it Quillen bifunctor}
  if
  \begin{itemize}
    \item[{\bf (i)}] ({\it two-variable cocontinuity}):
     $\otimes$ preserves colimits in each argument separately;
    \item[{\bf (ii)}] ({\it pushout-product axiom}):  
    given a pair of cofibrations, their $\otimes$-pushout product is also a cofibration
    \begin{equation}
     \label{PushoutProductAxiom}
     \def\arraystretch{1.3}
     \begin{array}{l}
       f \;\in\; \mathrm{Cof}(\mathcal{D}_1)
       \color{gray}{ 
       \,\cap \mathrm{W}(\mathcal{D}_1) }
       \\
       g \;\in\; \mathrm{Cof}(\mathcal{D}_2)
       \color{gray}{ 
       \, \cap \mathrm{W}(\mathcal{D}_2) }
     \end{array}
     \hspace{1cm}
       \vdash
     \hspace{1cm}
       f \,\widehat{\otimes}\, g \,\in\, \mathrm{Cof}(\mathcal{C})
       \color{gray}{ 
       \, \cap \mathrm{W}(\mathcal{C}) }
     \end{equation}
   and if, {\it moreover}, either is a weak equivalence, then so is the pushout product.
  \end{itemize}
\end{definition}

\medskip

\noindent
{\bf $\mathrm{sSet}$-Enriched categories (aka: simplicial categories).}
Most of what we say here applies to enriched categories over more general symmetric monoidal enriching categories than just $\mathrm{sSet}$, 
but we focus on this case for brevity of notation, since this is what we use in the main text.

\begin{definition}[$\mathrm{sSet}$-enriched monoidal category {\cite[Def. 1]{BataninMarkl12}\cite[Def. 2.1]{MorrisonPenneys19}\cite[\S 1.6]{LurieAlgebra}}]
\label{sSetEnrichedMonoidalCategory} $\,$ \newline
  A {\it $\mathrm{sSet}$-enriched monoidal category} 
  $(\mathbf{C}, \otimes)$
  is an $\mathrm{sSet}$-enriched category equipped with 
  a tensor product given by an $\mathrm{sSet}$-enriched functor
  \[
    X,X', Y, Y' \,\in\, \mathrm{Obj}(\mathbf{C})
    \;\;\;\;\;\;\;\;\;\;\;
    \vdash
    \;\;\;\;\;\;\;\;\;\;\;    
    \begin{tikzcd}
      \mathbf{C}(X,Y)
      \times
      \mathbf{C}(X',Y')
      \ar[
        rr,
        "{ \otimes_{ X \otimes X',\, Y \otimes Y'} }"
      ]
      &&
      \mathbf{C}\big(
        X \otimes X'
        ,\,
        Y \otimes Y'
      \big)
    \end{tikzcd}
  \]
  satisfying its coherence laws by $\mathrm{sSet}$-enriched natural isomorphisms.
\end{definition}

\medskip

\noindent
{\bf Rigidification of quasi-categories.}
\begin{definition}[Rigidification of quasi-categories]
  We write, as usual,
  \begin{equation}
    \label{HomotopyCoherentNerveAndRigidification}
    \begin{tikzcd}
      \mathrm{sSet}\mbox{-}\mathrm{Cat}
      \ar[
        from=rr,
        shift right=5pt,
        "{ \mathfrak{C} }"{swap}
      ]
      \ar[
        rr,
        shift right=5pt,
        "{ N }"{swap}
      ]
      \ar[
        rr,
        phantom,
        "{ \scalebox{.7}{$\bot$} }"
      ]
      &&
      \mathrm{sSet}
    \end{tikzcd}
  \end{equation}
  for the {\it homotopy coherent nerve} $N$
  of simplicial categories and its left adjoint $\mathfrak{C}$ \cite[\S 1.1.5, \S 2.2]{Lurie09},
  which on quasi-categories $\mathrm{QCat} \hookrightarrow \mathrm{sSet}$ may be understood as {\it rigidification} \cite{DuggerSpivak11}.
\end{definition}
\begin{proposition}[Rigidification preserves products up to DK-equivalence {\cite[Cor. 2.2.5.6]{Lurie09}\cite[Prop. 6.2]{DuggerSpivak11}}]
  \label{RigidificationPreservesProductsUpToDwyerKanEquivalence}
  For $S,\, S' \,\in\, \mathrm{sSet}$ there is a natural isomorphism from the rigidification \eqref{HomotopyCoherentNerveAndRigidification} 
  of their Cartesian product to the product $\mathrm{sSet}$-categories of their  rigidifications 
  \[
    \begin{tikzcd}
      \mathfrak{C}(S \times S')
      \ar[r]
      &
      \mathfrak{C}(S)
      \,\times\,
      \mathfrak{C}(S')
      \mathrlap{\,,}
    \end{tikzcd}
  \]
  which is a Dwyer-Kan equivalence (Prop. \ref{DwyerKanModelStructures}).
\end{proposition}
\begin{proposition}[Comparing rigidification to Dwyer-Kan fundamental groupoids {\cite[Thm. 1.1]{MRZ23}}]
  \label{ComparingRigidificationToDKFundamentalGroupoids}
  For $S \,\in\, \mathrm{sSet}$ there is a natural transformation from the localization \eqref{LocalizationAndCore} of the rigidification \eqref{HomotopyCoherentNerveAndRigidification} to the Dwyer-Kan fundamental simplicial groupoid
  \[
    \begin{tikzcd}
      \mathrm{Loc}
        \circ
      \mathfrak{C}(S)
      \ar[r]
      &
      \mathcal{G}(S)
    \end{tikzcd}
  \]
  which is a Dwyer-Kan equivalence (Prop. \ref{DwyerKanModelStructures}).
\end{proposition}

\medskip


\end{document}